\documentclass{jams-l}

\usepackage{amssymb}
\usepackage{graphicx}
\usepackage[cmtip,all]{xy}
\usepackage[usenames,dvipsnames]{color}
\usepackage{upgreek}
\usepackage{multirow}
\usepackage{hyperref}
\definecolor{linkcolour}{rgb}{0,0.2,0.6}
\hypersetup{colorlinks,breaklinks,urlcolor=linkcolour, linkcolor=linkcolour}
\usepackage{colortbl}

\newtheorem{theorem}{Theorem}[section]
\newtheorem{lemma}[theorem]{Lemma}
\newtheorem{proposition}[theorem]{Proposition}
\newtheorem{corollary}[theorem]{Corollary}

\theoremstyle{definition}
\newtheorem{definition}[theorem]{Definition}

\newtheorem{convention}[theorem]{Convention}
\newtheorem{example}[theorem]{Example}

\theoremstyle{remark}
\newtheorem{remark}[theorem]{Remark}

\numberwithin{equation}{section}

\usepackage{geometry}
\begin{document}

\title[Elim. of quot. in various loc. of premodels into models]{Elimination of quotients in various localisations of premodels into models}

\author{R\'{e}my Tuy\'{e}ras}
\address{epartment of Mathematics, Massachusetts Institute of Technology, 77 Massachusetts Avenue, Cambridge, MA 02139-4307, USA}
\curraddr{}
\email{rtuyeras@mit.edu}
\thanks{This research was carried out under an international Macquarie University Research Excellence Scholarship (iMQRES)}

\date{}

\dedicatory{}

\begin{abstract}
The contribution of this article is quadruple. It (1) unifies various schemes of premodels/models including situations such as presheaves/sheaves, sheaves/flabby sheaves, prespectra/$\Omega$-spectra, simplicial topological spaces/(complete) Segal spaces, pre-localised rings/localised rings, functors in categories/strong stacks and, to some extent, functors from a limit sketch to a model category \emph{versus} the homotopical models for the limit sketch; (2) provides a general construction from the premodels to the models; (3) proposes technics that allows one to assess the nature of the universal properties associated with this construction; (4) shows that the obtained localisation admits a particular presentation, which organises the structural and relational information into bundles of data. This presentation is obtained via a process called an \emph{elimination of quotients} and its aim is to facilitate the handling of the relational information appearing in the construction of higher dimensional objects such as weak $(\omega,n)$-categories, weak $\omega$-groupoids and higher moduli stacks.
\end{abstract}

\maketitle

\section{Introduction}\label{sec:introduction}
\subsection{Motivation 1} \label{ssec:motivation_1}
There is an abundant literature on how to construct an algebraic object from one of its presentations \cite{Lawvere1963, Kelly, FreydKelly, AdamekRosicky, Bourceux}---this process will be referred to as a \emph{localisation}. It is also well-known that the category of algebraic objects will satisfy strict universal properties if the objects themselves can be distinguished from their presentations by strict properties and, similarly, the category will usually satisfy weak universal properties if the objects can be distinguished from their presentations by weak properties, but little is known about how to derive strict universal properties for the category when the algebraic objects are only characterised by weak properties. One of the goals of the present paper is to address this lack.

 If we think of an algebraic object as a model for a limit sketch \cite{AdamekRosicky}, then algebraic objects can usually be distinguished from their presentations by lifting properties. Specifically, in the case of a limit sketch $D$, the presentations are given by the functors $D \to \mathbf{Set}$ while the models are given by those presentations $D \to \mathbf{Set}$ that preserve the chosen limits of $D$; as shown in \cite{FreydKelly}, this type of property can be expressed in terms of a lifting property in the functor category $\mathbf{Set}^{D}$. On the other hand, the localisation of a presentation $X$ into a model $Q(X)$ is endowed with a \emph{reflection} property, which equips $X$ with a map $i:X \to Q(X)$ such that for every arrow $f:X \to M$ where $M$ is a model, there exists an arrow $f':Q(X) \to M$ making the following diagram commute.
 
\[
\xymatrix{
X\ar[d]_{i}\ar[r]^{f}&M\\
Q(X)\ar[ru]_{f'}&
}
\]

If the lifting properties characterising the models are strict, then one is able to show that the reflection is strict, that is to say that the arrow $f':Q(X) \to M$ is unique for any given $f:X \to M$. For instance, in \cite{FreydKelly}, 
one starts by characterising the models via strict lifting properties and the strictness of these is naturally carried over to the reflection property. This is the same idea in \cite{Kelly} where the author is able to construct a (strict) reflection from the strict lifting properties inherently associated with well-pointed endofunctors.

On the other hand, if the lifting properties are weak, then one is usually only able to show that the reflection is weak, in which case the arrow $f':Q(X) \to M$ is only proven to exist. For instance, in \cite{AdamekTholen}, the \emph{small object argument} (recall that this argument comes from Homotopy Theory, which mostly, if not only, deals with weak lifting properties; see \cite{Hirschhorn,Quillen67}) is used to construct weak reflections for subcategories of injective objects.
  Similarly, in Garner's framework \cite{Garner09, Garner10}, the small object argument is generalised to construct weak homomorphisms of $\infty$-categories \emph{\`{a} la} Batanin \cite{Batanin98} while the possibility to construct $\infty$-categories is assumed: the reason being that $\infty$-categories are objects that can be characterised by strict lifting properties \cite[Corollary 1.19]{Berger} while weak homomorphisms between these do not require such a strictness.

However, to the best of my knowledge, there has not been any published work explaining how to obtain strict reflection properties from weak lifting constructions such as the small object argument. In fact, it is not even clear how to obtain strict universal properties from weak characterisations in general. For instance, in \cite{EssentialWeakTholen}, essential weak factorisation systems were introduced to study injective and projective hulls, which are meant to capture canonical envelops of injective and projective objects, with the goal of strengthening the lifting properties associated with the usual associated replacements (see intro. \emph{ibid.}), but it is not said if these hulls can satisfy strict universal properties; in fact \cite{TholenInjnotNat} gives a hint that this is unlikely and states that only an \emph{almost} reflection property can be shown. The paper even emphasises the need of methods to pass from a weak setting to a strict one in its last section \cite[Section 4]{TholenInjnotNat}, in which it is asked if it is possible to know when strict universal properties, such as naturality and functoriality, can be shown to be satisfied by a given weak reflection.

In an area of Mathematics in which the weakening of definitions and theories (e.g., $\infty$-topos theory, univalent homotopy type theory, devired algebraic geometry, etc.) have now taken more and more importance,
but whose language---Category Theory---also takes advantage of strict universal properties, it is, indeed, of interest to know if there are theorems that allow one to determine whether a set of weak lifting properties defining a type of algebraic object can provide the associated category with a strict universal property---at least stricter than the expected one.

The present paper is an effort to provide a set of technics and theorems  showing that such a scheme is possible. Precisely, one of the main contributions of this paper is to propose a language (or context) in which it is possible to say if a category of algebraic objects that are characterised by weak lifting properties can be shown to possess a strict universal property (see Section \ref{ssec:result_1}). We will even show that the proposed argument is a generalisation of Quillen's small object argument (see Corollary \ref{cor:factorisation_SOA}) and will thus answer one of our earlier questions. The theorems given herein are meant to be generalised in future work (in which the boundary between strictness and weakness will become blurrier), the purpose being to pave the way for the construction of models taking their values in higher categorical structures.

\subsection{Motivation 2}\label{ssec:motivation_2}
The second matter that motivates the present paper is the so-called \emph{elimination of quotients} mentioned in the title, which basically comes down to conclude that the way we encode an object is as important as its inherent properties. For instance, it is this same type of ideas that motivated

$\triangleright$ the introduction of the \emph{elimination of imaginaries}, in Shelah's Model Theory \cite{Poizat83,Shelah90}, in which quotients are eliminated in the form of definable quotient maps by using the various sorts available from the ambient (multi-sorted) theory;

$\triangleright$ the development of the concept of \emph{covering space}, in Algebraic Topology \cite{Hatcher}, that provides ways to blow up the quotients acting on a space and to bring out its homotopical properties by studying the automorphisms acting on the resulting quotient maps;

$\triangleright$ the definition of \emph{stack}, in Algebraic Geometry \cite{Letter_Grothendieck}, due to the existence of non-trivial automorphisms that may occur because of the different ways a moduli space can be represented. 

To really understand how the coding of objects, and, even that of sets, matters from the point of view of their algebraic structures, let us consider an example. Take a set $X$ and consider the coproduct $E := X+X$ encoded by the following logical specification. 
\[
\{(i,x)~|~x \in X\text{ and }i \in \{0,1\}\}
\]

If one takes $R$ to be the binary relation on $X+X$ that identifies $(0,x)$ with $(1,x)$ for every $x \in X$, then the quotient $E/R$ is obviously isomorphic to $X$. However, in much the same way as it is fundamental to not confuse an isomorphism with an identity, it is, here, important to understand that $E/R$ is not \emph{same} as $X$. From the point of view of the present paper, the difference between $X$ and $E/R$ lies in the implicit algebraic structure with which $E/R$ is equipped. 
This object can indeed be seen as a surjection $p:E \to X$ equipped with two sections $s_0,s_1:X \hookrightarrow E$ whose cospan structure defines a universal cocone, and this structure is noticeable even thought $E/X$ is isomorphic to a mere set. In other words, the quotient $E/R$ can be seen as living way beyond the category of sets, for the simple reason that isomorphisms are not the same as identities.

All this shows that the way we construct algebraic objects matters quite substantially, mainly because the algebraic properties coming along with their representations can turn out to be either very useful or extremely cumbersome (e.g., $X$ versus $E/R$).

The goal of our so-called `elimination of quotients' will be to eliminate the cumbersome quotients that may occur in the representation of algebraic objects and organise, in the form of quotient maps, the useful ones. Here, I feel important to mention that such a re-organisation is possible because our objects are characterised by \emph{weak} lifting properties, which allow more freedom than strict ones.

If we look at how Kelly \cite[Theorem 10.2]{Kelly} constructs algebraic objects, and to be more specific, models for some limit sketch $(D,K)$, where $K$ denotes the set of limit cones associated with $D$, we see that he isolates each cone $c \in K$ and constructs, for each of these and every presentation $X:D \to \mathbf{Set}$, a well-pointed endofunctor $i_c:X \to P_c(X)$ where the object $P_c(X)$ completes the presentation $X$ with operations required by the sub-theory $(D,\{c\})$ of $(D,K)$. To complete $X$ with respect to the operations required by the whole theory $(D,K)$, he pushes out the wide span made of the arrows $i_c$, for all $c \in K$, to obtain a well-pointed endofunctor $i:X \to P(X)$. In particular, each cone $c \in K$ is equipped with a factorisation as follows.
\[
\xymatrix@R-20pt{
X\ar@{->}[rd]_{i_c}\ar@{->}[rr]^i&&P(X)\\
&*+!L(.0){P_c(X)}\ar@{-->}[ru]_{j_c}&
}
\]

Finally, the reflector $X \to Q(X)$ associated with the theory $(D,K)$ is computed through a transfinite composition of the following form.
\[
\xymatrix{
X \ar[r]^{i}&P(X) \ar[r]^{P(i)} & P^2(X)\ar[r]^{P^2(i)} & P^3(X)\ar[r]^{P^3(i)} & \dots
}
\]

Isolating each cone $c$ in $K$ and proceeding to a pushout of the well-pointed endofunctors $X \to Q_c(X)$ is a necessity if one wants to use the very neat and compact formalism of well-pointed endofunctors. However, this pushout procedure, as elegant as it may be,
adds more cumbersome quotients than useful ones. Precisely, the wide pushout of the objects $Q_c(X)$ looks more like the type $(X+X)/R \cong X$ because it mostly identifies all the copies of $X$ living in each $Q_c(X)$ through the maps $X \to Q_c(X)$. 

As we can imagine, these cumbersome quotients become much more abundant when enriching our algebraic objects to other categories than $\mathbf{Set}$ and it would not be imaginable to be willing to do combinatorics with representations that repeat and contract the same information over and over.
Not only do the results proposed in the present article avoid these cumbersome quotients, but they also bring out the hidden algebraic structure of the useful ones, where, here, the term `algebraic structure' is used in the sense previously discussed for the quotient $E/R$.

In fact, our results go in the direction of Lawvere's work \cite{Lawvere1963}, in which the concept of \emph{congruence} is used to construct a reflector from the category of presentations to that of models by showing how the quotients act on the free algebra functor applied on the presentations \cite[Theorem 5.1]{Lawvere1963}. It is worth noting that the concept of congruence has given rise to a very rich theory regarding the characterisation of congruence lattices for varieties of algebras \cite{KissKearnes,McKenzie}. Our results can therefore be seen as a refined extension of Lawvere's work. This refinement is presented in the form of a formal language that could be seen as suitable for a generalisation of Congruence Lattice Theory to more general objects than those proposed by Lawvere.

\subsection{Results for motivation 1}\label{ssec:result_1}
In the same fashion as there are categories of models for a theory \cite{AdamekRosicky}, or categories of fibrants objects \cite{Brown} or even systems of fibrant objects \cite{ProWar}, it is, here, proposed the definition of \emph{system of premodels} (see Definition \ref{def:System:premodels}), which gathers in the same structure a category of \emph{presentations} together with maps along which the \emph{models} are defined via weak lifting properties. An interesting feature of this structure is that it encompasses many examples that are meant to be captured \emph{operibus citatis}; particular examples can also be found in \cite{Maltsi,Rezk,Joyal_Tierney,Stanculescu}. There is also a novelty in the fact that the maps along which the weak lifting properties are defined are not maps in the category of values or that of presentations, but in a category whose level of definition allows one to verify whether the subcategory of the resulting models possesses a strict reflection property.
For instance, this allows us to retrieve and explain the strict reflection property associated with the models for a limit sketch. 

If we restrict ourselves to algebraic objects defined by limit-preserving functors, say valued in a category in which choices of colimits are obvious, a \emph{system of premodels} is given by
\begin{itemize}
\item[(1)] a limit sketch $(D,K)$;
\item[(2)] a category $\mathcal{C}$ with enough limits and pushouts, if not all;
\item[(3)] a subcategory $\mathcal{P} \hookrightarrow \mathcal{C}^D$;
\item[(4)] for every cone $c \in K$, a set $\mathtt{V}_c$ of commutative squares in $\mathcal{C}$, say as follows.
\[
\xymatrix{
\mathbb{S}\ar[r]^{\gamma_1}\ar[d]_{\gamma_2}&\mathbb{D}_1\ar[d]^{\beta_1}\\
\mathbb{D}_2\ar[r]_{\beta_2}&\mathbb{D}'
}
\]
\end{itemize}

Before giving the definition of a model for this structure, we need to recall that a cone $c$ in $K$ is a natural transformation $\Delta_{A}(\mathtt{ou}(c)) \Rightarrow \mathtt{in}(c)$ where $\mathtt{ou}(c)$ is an object in $D$, $\mathtt{A}$ is a small category, $\Delta_{A}(\mathtt{ou}(c))$ is the obvious constant functor $\mathtt{A} \to \mathbf{1} \to D$ picking out the object $\mathtt{ou}(c)$ in $D$ and $\mathtt{in}(c)$ is some functor $\mathtt{A} \to D$.
Now, a \emph{model} for the previous structure is a functor $D \to \mathcal{C}$ in $\mathcal{P}$ such that for every $c \in K$, the canonical arrow 
\[
P(\mathtt{ou}(c)) \to \mathrm{lim}\, P \circ \mathtt{in}(c),
\]
for which we shall prefer the more compact notation $P[c]:=\mathrm{lim}\, P \circ \mathtt{in}(c)$, is orthogonal in the arrow category $\mathcal{C}^{\mathbf{2}}$ to every commutative square in $\mathtt{V}_c$ (as shown below).
\[
\xymatrix@C-1.75pc{
&*+!R(.3){P\mathtt{ou}(c)}\ar[rr]^-{}&&P[c]\\
\mathbb{S}\ar[ur]^x\ar[d]_{\gamma_2}\ar[rr]^{\gamma_1}&&\mathbb{D}_1\ar[ru]^y\ar[d]_{\beta\delta_1}&\\
\mathbb{D}_2\ar@{-->}[ruu]\ar[rr]_{\beta\delta_2}&&\mathbb{D}'\ar@{-->}[ruu]&
}
\]

In the case of limits sketches, we retrieve the usual definition of model by taking, for every cone $c \in K$, the following pair of commutative squares in $\mathbf{Set}$; the leftmost one encodes the surjectiveness of the map $P(\mathtt{ou}(c)) \to P[c]$ while the other one encodes its injectiveness.
\[
\xymatrix{
\emptyset\ar[r]^{\gamma_1}\ar[d]_{\gamma_2}&\mathbf{1}\ar[d]^{\beta_1}\\
\mathbf{1}\ar[r]_{\beta_2}&\mathbf{1}
}
\quad\quad\quad\quad\quad\quad
\xymatrix{
\mathbf{1}+\mathbf{1}\ar[r]^{\gamma_1}\ar[d]_{\gamma_2}&\mathbf{1}\ar[d]^{\beta_1}\\
\mathbf{1}\ar[r]_{\beta_2}&\mathbf{1}
}
\]

One of the very advantages of this language is to allow the specification of more general arrows than bijections such as weak equivalences (see characterisation in \cite[Lemma 7.5.1]{Simpson}). This explains why this language is expected to be generalised to higher categorical structures in the future.

Now, our main result, given in Theorems \ref{th:localisation_universal} and  \ref{th:admissible_quotiented_factorisable_model}, can be simplified in terms of Theorem \ref{theo:elimitation_reflection_simple_case}, in which items (i) and (ii) are in fact redundant. The statement makes use of the arrow $\beta:\mathbb{S}' \to \mathbb{D}'$, which denotes, for every commutative square contained in $\mathtt{V}_c$ and every $c \in K$, the universal arrow induced by the pair of arows $\beta_1$ and $\beta_2$ under the pushout (denoted by $\mathbb{S}'$) of the arrows $\gamma_1$ and $\gamma_2$.

\begin{theorem}\label{theo:elimitation_reflection_simple_case}
Suppose that $\mathcal{P} \hookrightarrow \mathcal{C}^D$ is an identity. For every object $A$ in $\mathcal{P}$, there exists an arrow $i:A \to Q(A)$ in $\mathcal{P}$ (Theorem \ref{th:admissible_quotiented_factorisable_model}) such that for every arrow $f:A \to X$ in $\mathcal{P}$ where $X$ is a model for the system of premodels, if
\begin{itemize}
\item[(i)] the map $\beta$ is an epimorphism for every square in $\mathtt{V}_c$ and every $c \in K$;
\item[(ii)] the arrow $X(\mathtt{ou}(c)) \to X[c]$ is a monomorphism in $\mathcal{C}$;
\item[(iii)] the arrow $\beta_1$ is an epimorphism for every square in $\mathtt{V}_c$ and every $c \in K$,
\end{itemize}
then there exists a unique arrow $g:Q(A) \to X$ making the following diagram commute (Theorems \ref{th:localisation_universal} and \ref{th:admissible_quotiented_factorisable_model}).
\[
\xymatrix{
A\ar[r]^{f}\ar[d]_{i}&X\\
Q(A)\ar[ru]_{g}&
}
\]
\end{theorem}
As one can see, the previous theorem explains, in the language of systems of premodels, why one can expect a strict reflection property in the case of set-valued models for a limit sketch. 

In Theorem \ref{theo:elimitation_reflection_simple_case}, the assumption that the inclusion $\mathcal{P} \hookrightarrow \mathcal{C}^D$ is an identity will be replaced, in Theorems \ref{th:localisation_universal} and  \ref{th:admissible_quotiented_factorisable_model}, with the notion of \emph{effectiveness}, which translates a variation of the concept of \emph{definability} in $\mathcal{P}$ (notice the parallelism with the concept of elimination of imaginaries given in Section \ref{ssec:motivation_2}). As will be shown in Theorem \ref{th:admissible_quotiented_effective_model}, this concept of definability becomes trivial if $\mathcal{P}$ is taken to be equal to $\mathcal{C}^D$.

\subsection{Results for Motivation 2}\label{ssec:result_2}
From the point of view of motivation 2, the present paper mainly focus on models for limit sketches in $\mathbf{Set}$, so that we will mostly state our results from the perspective of these objects. This will nevertheless give an idea of what our theorems look like when generalised to other categories. The proof of the results stated below will be recapitulated in the conclusion of the present paper (Section \ref{sec:conclusion}).

We now consider a limit sketch $(D,K)$, where, for simplicity only, $K$ is supposed to be a finite set of finite-limit cones. The proposition given below states that it is possible to construct the reflector of any presentation in a very specific way, which is not visible from Kelly's construction \cite{Kelly}.

\begin{proposition}\label{prop:Elimination_quotient_intro_1}
For every presentation $X$ in $\mathbf{Set}^{D}$ and ordinal $i\in \omega$, there exist a pair of objects $E_i(X)$ and $B_i(X)$ and an epimorphism $p_i:B_i(X)+E_i(X) \to B_{i+1}(X)$ such that the reflector of $X$ for the theory $(D,K)$ is given by the transfinite composition of the following sequence of arrows  in $\mathbf{Set}^{D}$.
\[
\xymatrix{
B_0(X)+E_0(X) \ar[r]^-{p_0}& B_1(X)+E_1(X) \ar[r]^-{p_1} & B_2(X)+E_2(X) \ar[r]^-{p_2} & \dots
}
\]

In addition, the mappings $X \mapsto E_i(X)$ and $X \mapsto B_i(X)$ are functorial and the arrow $p_i:B_i(X)+E_i(X) \to B_{i+1}(X)$ is natural in $X$.
\end{proposition}

Of course, one could argue that the map $X \to P(X)$ coming from Kelly's construction can be factorised into an epimorphism and a monomorphism $X \twoheadrightarrow B(X) \hookrightarrow P(X)$, so that we might recover the previous form, but it is not obvious whether $P(X)$ can be decomposed into a functorial sum $B(X)+E(X)$ in $\mathbf{Set}^{D}$, mainly because the quotients that acts on $P(X)$ might prevent from doing so. In fact, there is a much stronger way to assess the difference between Kelly's construction and the previous one, which is given below.

\begin{proposition}\label{prop:Elimination_quotient_intro_2}
For every presentation $X$ in $\mathbf{Set}^{D}$, there exist a sequence of epimorphism $(p_i:B_i(X)+E_i(X) \to B_{i+1}(X))_{i \in \omega}$, as given in Proposition \ref{prop:Elimination_quotient_intro_1}, for which there is a natural transformation of transfinite sequences
\[
\xymatrix{
B_0(X)+E_0(X) \ar[r]^-{p_0} \ar[d]_{\alpha_0}& B_1(X)+E_1(X) \ar[r]^-{p_1} \ar[d]_{\alpha_1}& B_2(X)+E_2(X) \ar[r]^-{p_2} \ar[d]_{\alpha_2}& \dots\\
X \ar[r]^-{i}& P(X) \ar[r]^-{P(i)} & P^2(X) \ar[r]^-{P^2(i)} & \dots\\
}
\]
for which $\alpha_0$ is the identity on $X$ and if there exists a pair of dashed arrows making the following triangle commute for $n>1$, then the front arrow must factorise through the canonical map $B_n(X) \to B_n(X)+E_n(X)$ and the object $P^{n-1}(X)$ is a model for the limit sketch $(D,\{c\})$.
\[
\xymatrix@C-20pt@R-5pt{
&&B_{n-1}(X)+E_{n-1}(X)\ar[rr]^-{p_{n-1}}\ar[d]^{\alpha_{n-1}}&&*+!L(.7){B_n(X)+E_n(X)}\ar[d]^{\alpha_n}\\
&&P^{n-1}(X)\ar[rr]|\hole&&P^{n}(X)\\
P^{n-1}(X)\ar@{-->}[rruu]\ar[rru]^{\mathrm{id}}\ar[rr]_-{i_c}&&*+!L(.9){P_cP^{n-1}(X)}\ar@{-->}[rruu]\ar[rru]_-{j_c}&&
}
\]
\end{proposition}

In other words, Kelly's construction has too many quotients to be non-trivially lifted to the elimination of quotients, and if a lift exists, then it cannot be in the free part $E_n(X)$, which means that, at rank $n$, the free operations added to satisfy the theory $(D,\{c\})$ are superfluous. 

Even though the natural transformation $\alpha$ is to identify free operations between each other, note that it cannot identify too much information either as the universal property of a reflector implies that the transfinite colimit of $\alpha$ provides an isomorphism between the two underlying reflectors of $X$.
\[
\xymatrix@R-10pt{
X\ar[r]\ar@{=}[d]&*+!L(.6){Q_{\mathrm{Elim}(X)}}\ar[d]^{\cong}\\
X\ar[r]&*+!L(.6){Q_{\mathrm{Kelly}}(X)}
}
\]

In fact, we will show that, in the case of models for a limit sketch, the so-called elimination of quotients takes the form given in Theorem \ref{theo:Elimination_quotient_intro}, in which every cone $c$ in $K$ is again viewed as a natural transformation $\rho:\Delta_{A}(\mathtt{ou}(c)) \Rightarrow \mathtt{in}(c)$ where $\mathtt{ou}(c)$ is an object in $D$, $\mathtt{A}$ is a small category, $\Delta_{A}(\mathtt{ou}(c))$ is the obvious constant functor $\mathtt{A} \to \mathbf{1} \to D$ picking out $\mathtt{ou}(c)$ in $D$ and $\mathtt{in}(c)$ is some functor $\mathtt{A} \to D$.

\begin{theorem}\label{theo:Elimination_quotient_intro}
For every presentation $X$ in $\mathbf{Set}^{D}$, there exist a sequence of epimorphisms as given in Proposition \ref{prop:Elimination_quotient_intro_2},  for which we will denote the coproduct object $B_i(X)+E_i(X)$ as a functor $S_i:D \to \mathbf{Set}$, such that
\begin{itemize}
\item[-] $B_0(X) = X$ and $E_0(X) =\emptyset$;
\item[-] $E_{i+1}(X)$ is the left Kan extension of the functor
\[
\begin{array}{lllll}
\hat{S}_i[\_]&:&K& \to &\mathbf{Set}\\
&&c& \mapsto &\mathrm{lim}\, S_i \circ \mathtt{in}(c)\\
\end{array}
\]
along the functor $\mathtt{ou}:K \to D$, where $K$ is seen as a discrete small category;
\[
\xymatrix{
D\ar@{-->}[rd]^{E_{i+1}(X)}&\\
K\ar@{}[ru]|<<<<{\rotatebox[origin=c]{115}{$\Rightarrow$}a_i}\ar[u]^{\mathtt{ou}}\ar[r]_-{\hat{S}_i[\_]}&\mathbf{Set}
}
\]
\item[-] the epimorphism $p_i$ is the quotient map $S_i \to B_{i+1}(X)$ making the following identifications: 
\begin{itemize}
\item[(1)] for every object $d$ in $D$, it identifies a pair $x,y \in S_i(d)$ if there exists a cone $c \in K$ and an arrow $t:\mathtt{ou}(c) \to d$ in $D$ for which the pushout of the canonical arrow $S_i(\mathtt{ou}(c)) \to \hat{S}_i[c]$ along $S_i(t):S_i(\mathtt{ou}(c)) \to S(d)$ maps $x$ and $y$ to the same element;
\[
~\quad\quad\quad\quad\quad\quad\quad\quad
\xymatrix{
S_i(\mathtt{ou}(c))\ar[d]\ar[r]^{S_i(t)}&S_i(d)\ar@{-->}[d]^{\pi}&x,y\ar@{|->}[d]\\
\hat{S}_i[c]\ar@{-->}[r]&\ast&\pi(x) = \pi(y)
}
\]
\item[(2)] for every object $d$ in $D$, it identifies a pair $x,y \in S_i(d)$ where $x \in E_{i}(X)(d)$ and $y \in B_{i}(x)(d)$ if there exists a cone $c \in K$, an object $z$ in the diagram $\mathtt{A}$ of $c$ and a morphism $t:\mathtt{in}(c)(z) \to d$ in $D$ such that $x$ and $y$ can be lifted to a common element in $\hat{S}_{i-1}[c]$ via the span $E_{i}(X)(d) \leftarrow \hat{S}_{i-1}[c] \rightarrow B_{i}(x)(d)$ made of the following composites.
\[
\xymatrix@C+10pt@R-15pt{
\hat{S}_{i-1}[c] \ar[r]^-{a_{i-1}}& E_{i}(X)(\mathtt{ou}(c)) \ar[r]^-{E_{i}(X)(\rho_z)}& E_{i}(X)(\mathtt{in}(c)(z))\ar[r]^-{E_{i}(X)(t)}& E_{i}(X)(d)\\
\hat{S}_{i-1}[c] \ar[r]_-{\mathrm{proj}_z} & S_{i-1}\mathtt{in}(c)(z) \ar[r]_-{S_{i-1}(t)}& S_{i-1}(d)\ar[r]_-{p_{i-1}} & B_{i}(X)(d)\\
}
\]
\end{itemize}
\end{itemize}
\end{theorem}

Even though we have only discussed the finite-limit case, all of the previous propositions hold for non-finite limit-sketches. In this case, the ordinal $\omega$ becomes the cardinality of the limit-sketch (see the end of Section \ref{sssec:Croquis}) and the transfinite sequence of arrows $B_i(X)+E_i(X) \to B_{i+1}(X)$ needs to be defined such that $B_{\alpha}(X)$ is the transfinite colimits of all the arrows preceding the rank $\alpha$.

\subsection{Road Map}
The main results of the paper start to be developed from Section \ref{sec:Category_of_models_for_a_croquis}, while Sections \ref{sec:Background_notations_conventions} and  \ref{sec:Convergent_functor} give an account of various notations, conventions and technicalities. Specifically, Section \ref{sec:Background_notations_conventions} introduces a set of conventions meant to facilitate our notations while Section \ref{sec:Combinatorial_cat} focuses on a notion of smallness that will only be used in Section \ref{sec:Combinatorial_cat}.

Even if Section \ref{sec:Background_notations_conventions} does not sound so attractive, the reader might want to skim through this section to get used to specific notations such as $\iota_{\kappa}$ (Section \ref{ssec:Ordinals}); $\mathrm{col}_D$ (Section \ref{ssec:Limits_colimits}); $\xi_i$ as well as $\mathrm{col}_i$ (Section \ref{sssec:Universal_shifting}) and $\mathbf{B}_d^D(\_)$ (Section \ref{ssec:Bounded_diag}).

Section \ref{sec:Convergent_functor} defines a notion of smallness that generalises the usual one. Recall that one usually says that an object $D$ in some category $\mathcal{C}$ is small if for any functor \footnote{or, sometimes, any functor belonging to a certain classes of functors. This restriction generally arises in non-accessible categories such as in the category of topological spaces}.  defined from  the ordinal category $\omega$ to $\mathcal{C}$, say $F:\omega \to \mathcal{C}$, the following canonical map is a bijection. 
\[
\mathrm{col}_{i \in \omega}\mathcal{C}(D,F(i)) \to \mathcal{C}(D,\mathrm{col}_{i \in \omega} F(i)) 
\]

On the other hand, the smallness condition defined in Section \ref{sec:Convergent_functor} would be more of the following type. The property is now centred on the functor $F$ and not on the object $D$ any more; we then consider a set of objects $G$ in $\mathcal{C}$ and say that a functor $F:\omega +1 \to \mathcal{C}$ is $G$-convergent if the following canonical map is a bijection for every object $D \in G$.
\begin{equation}\label{eq:description_of_the_idea_of_convergence}
\mathrm{col}_{i \in \omega}\mathcal{C}(D,F(i)) \to \mathcal{C}(D,F(\omega)) 
\end{equation}

The reason for this change is that the image $F(\omega)$ will not always be a colimit of the form $\mathrm{col}_{i \in \omega} F(i)$. 

Then comes Section \ref{sec:Category_of_models_for_a_croquis}, in which is defined the notion of system of premodels. The difference with the simplified version given in Section \ref{ssec:result_1} and that of Section \ref{sec:Category_of_models_for_a_croquis} is that the canonical map $X\mathtt{ou}(c) \to X[c]$ is now constructed from various parts of the system of premodel structure, so that it is now of the form $X\mathtt{ou}(c) \to RX[c]$ where $R$ is a right adjoint endofunctor on $\mathcal{C}$. This right adjoint $R$ will often be an identity functor in this paper, save for $\Omega$-spectra, in which case it will be equal to the loop space functor $\Omega$. In the future, the functor $R$ will however take multiple forms.

Sections \ref{sec:Constructors} and  \ref{sec:From_narratives_to_combinatorial_cat} work together to formalise the idea of algebraic structure associated with a quotient. Recall that completing a presentation with operations usually requires the adding of free operations along with certain quotients. In our case, the free structure will be added to the presentations, but the quotient structure will be \emph{resolved} in a separate object $\mathfrak{q}$ (see Section \ref{ssec:Quotiented-arrows}). The term \emph{resolved} here refers to the concept of resolution developed in \cite{Resoution_poly}, which should be viewed as a way of passing from what looks like a set $E/X$ to a higher dimensional structure, such as category or a quotient map $E \to X$. 

The purpose of Section \ref{sec:From_narratives_to_combinatorial_cat}, alone, is to give a theoritical generalisation of Quillen's small object argument \cite{Quillen67} while Section \ref{sec:Constructors} focuses on applying the formalism of Section \ref{sec:From_narratives_to_combinatorial_cat} to systems of premodels. 

The difference between our argument and Quillen's one is that one does no longer consider strict pushouts at every step and the lifts meant to be produced by these pushouts only commutes in the subsequent steps. These differences arise for two reasons. The first one is the desired elimination of quotients and the second one is due to the fact that the pushouts used in the usual argument do not necessarily commute with the right adjoints (including the limits) involved in the construction of the object $RX[c]$.

To be able to formalise the previous ideas, we will introduce the concept of tome, whose goal is to gather all the squares that one would like to force to admit a lift through the small object argument.
This will take the form of a functor $\varphi: \mathtt{S} \to \mathcal{C}^{\mathbf{2}}/h$, where $h$ is an object in the arrow category $\mathcal{C}^{\mathbf{2}}$. Note that this tool will mainly find its use in the way the category $\mathtt{S}$ is encoded. 

Specifically, in Section \ref{sec:Constructors}, this category $\mathtt{S}$ will be discrete and will take the form of a coproduct of what could look like two left Kan extensions.
\[
\mathtt{S}:=\Big(\sum_{\vartheta \in J_A} D(\upepsilon(\vartheta),d) \times \Lambda_A[\vartheta]\Big) + \Big(\sum_{\vartheta \in J_Q} D(\upchi(\vartheta),d) \times \Lambda_Q[\vartheta]\Big)
\]

The left-hand sum will allow us to parameterise all those squares that are to force the adding of the structural information to the presentations while the right-hand sum will allow us to handle all of the quotients that the adding of this information is supposed to generate.
Note that the rightmost sum of $\mathtt{S}$ is only meant to quotient out what has been added at a previous step, leaving free the information added by the current leftmost sum and thus producing the elimination of quotients discussed in Section \ref{ssec:result_2}. All the data needed to talk about an elimination of quotients such as $J_A$, $J_Q$, $\upepsilon$, $\upchi$, $\Lambda_A[\_]$, $\Lambda_Q[\_]$ (and some more) will be gathered into the notion of constructor (see Section \ref{sssec:Constructor}). Remarks \ref{rem:form_of_elt_I_J} and  \ref{rem:encoding_of_analytic_functors} might be helpful in seeing what all those left Kan extension-like constructions actually parameterise. 

Finally, the small object argument is carried out in Section \ref{sec:Combinatorial_cat} where the smallness condition is used to prove the usual lifting properties. The universal property satisfied by our construction is discussed in Section \ref{sec:Universal_property} via Theorems \ref{th:admissible_quotiented_factorisable_model} (existential part) and  \ref{th:localisation_universal} (uniqueness). The latter mainly focus on the properties required to prove Theorem \ref{theo:Elimination_quotient_intro}, whose proof is recapitulated in the conclusion (see Section \ref{ssec:conclusion_motivation_2}).

\subsection{Acknowledgments}
I would like to thank Steve Lack and the referees for comments that allowed the improvement of the earlier versions of this text. 
I would also like to thank the members of the Australian Category Seminar for various remarks regarding the content of this paper.

\section{Background, Notations and Conventions}\label{sec:Background_notations_conventions}

\subsection{Ordinals}\label{ssec:Ordinals}
Any ordinal will be identified with the preorder category it induces. For every ordinal $\kappa$, the inclusion functor $\kappa \hookrightarrow \kappa+1$ will be denoted by $\iota_{\kappa}$. For convenience, the preorder category of one and two objects will be denoted by $\mathbf{1}$ and $\mathbf{2}$, respectively. We shall also use the notation $\omega$ to denote the least infinite ordinal.

\subsection{Wide Subcategories} Let $\mathcal{C}$ be a category. A subcategory $\mathcal{A} \subseteq \mathcal{C}$ will be said to be \emph{wide} if the inclusion functor $\mathcal{A} \hookrightarrow \mathcal{C}$ is surjective on objects. Put simply, this means that $\mathcal{A}$ contains all the objects of $\mathcal{C}$.

\subsection{Limits and Colimits}\label{ssec:Limits_colimits}
For every category $\mathcal{C}$ and small category $D$, the obvious functor $\mathcal{C}^{\mathbf{1}} \to \mathcal{C}^{D}$ mapping an object $X:\mathbf{1} \to \mathcal{C}$ to the pre-composition of $X:\mathbf{1} \to \mathcal{C}$ with the canonical functor $D \to \mathbf{1}$ will be denoted by $\Delta_D$. For convenience, the category $\mathcal{C}^{\mathbf{1}}$ will often be identified with the category $\mathcal{C}$. 
If they exist, the left and right adjoints of $\Delta_D$ will be denoted by $\mathrm{col}_D$ and $\mathrm{lim}_D$, respectively. Recall that the images of these two functors are understood as the colimits and limits of $\mathcal{C}$ over $D$, respectively.
As usual, in the case where the functor $\mathrm{lim}_D:\mathcal{C}^D \to \mathcal{C}^{\mathbf{1}}$ exists, the category $\mathcal{C}$ will be said to be \emph{complete over $D$}. Similarly, the category $\mathcal{C}$ will be said to be \emph{cocomplete over $D$} when the functor $\mathrm{col}_D:\mathcal{C}^D \to \mathcal{C}^{\mathbf{1}}$ exists.

\begin{proposition}\label{prop:co_complete_category_then_functor_category_too}
If a category $\mathcal{C}$ is complete (resp. cocomplete), then so is $\mathcal{C}^D$ for any small category $D$ where
the limits (resp. colimits) are defined objectwise in $\mathcal{C}$.
\end{proposition}
\begin{proof}
Suppose that $\mathcal{C}$ is complete. For every object $d$ in $D$, the restriction functor $\nabla_d:\mathcal{C}^D \to \mathcal{C}$ mapping $X$ to  $X(d)$ has a right adjoint whose images are given by the Right Kan extensions along the functor $\mathbf{1} \to D$ picking out $d$ \cite{MacLane}. This implies that $\nabla_d$ commutes with limits. By duality, the other statement regarding colimits follows.
\end{proof}

\subsection{Cardinality} 
Let $A$ be an object in $\mathbf{Set}$. The \emph{cardinality of $A$} is the least ordinal $\kappa$ such that there is a bijection between $A$ and $\kappa$. In ZFC, the axiom of choice ensures that the cardinality of a set $A$ always exists, which will be denoted by 
$|A|$.

For any small category $D$, the \emph{cardinality of $D$} is the cardinality of the following coproduct of sets, where $\mathrm{Obj}(D)$ is the set of objects of $D$.
\[
\mathrm{Ar}(D):=\sum_{a,b \in \mathrm{Obj}(D)}D(a,b)
\]

The cardinality of $D$ will be denoted by $|D|$. Below is given a well-known result on the 
commutativity of limits and colimits.

\begin{proposition}\label{prop:SGA_limits_colimits_commute}
For every small category $D$ and limit ordinal $\kappa \geq |D|$, the canonical natural transformation 
$\mathrm{col}_{\kappa}\,\mathrm{lim}_{D} \Rightarrow \mathrm{lim}_{D}\,\mathrm{col}_{\kappa}$ valued in $\mathbf{Set}$ over $\mathbf{Set}^{\kappa \times D}$ is an isomorphism.
\end{proposition}
\begin{proof}
See Appendix \ref{Appendix:first}.
\end{proof}

Similarly, for every complete category $\mathcal{C}$ and small category $D$, the functor $\Delta_D:\mathcal{C} \to \mathcal{C}^D$ commutes with colimits (see Proposition \ref{prop:co_complete_category_then_functor_category_too}). In fact, it follows from Proposition \ref{prop:SGA_limits_colimits_commute} that the unit of the adjuncion $\Delta_D \vdash \mathrm{lim}_D$ commutes with colimits in $\mathbf{Set}$ as stated in the next proposition.

\begin{proposition}\label{prop:eta_commutes_with_colimits_O_kappa}
For every small category $D$ and limit ordinal $\kappa \geq |D|$, denote by the letter $\eta$ the units of the two adjunctions $\Delta_D \vdash \mathrm{lim}_D$ in $\mathbf{Set}$ and $\mathbf{Set}^{\kappa}$. The following diagram of canonical arrows in $\mathbf{Set}$ commutes for any functor $F:\kappa \to \mathbf{Set}$.
\[
\xymatrix@C-3pc{
*+!R(.3){\mathrm{col}_{\kappa}F}\ar[d]_-{\mathrm{col}_{\kappa}\eta F}\ar@{=}[rr]&& *+!L(.3){\mathrm{col}_{\kappa}F} \ar[d]^-{\eta_{\mathrm{col}_{\kappa}F}}\\
*+!R(.7){\mathrm{col}_{\kappa}\mathrm{lim}_D\Delta_DF}\ar[rr]_-{\cong}&&*+!L(.7){\mathrm{lim}_D\Delta_D(\mathrm{col}_{\kappa}F)}
}
\]
\end{proposition}
\begin{proof}
See Appendix \ref{Appendix:first}.
\end{proof}

\subsection{Universal Shiftings} \label{sssec:Universal_shifting}
Let $i:\mathtt{T} \to \mathtt{S}$ be a functor between small categories. The pre-composition with $i$ induces an obvious functor $\_ \circ i:\mathcal{C}^{\mathtt{T}} \to \mathcal{C}^{\mathtt{S}}$. Mostly for convenience, the composition of this functor with the colimit functor $\mathrm{col}_{\mathtt{S}}:\mathcal{C}^{\mathtt{S}} \to \mathcal{C}$ will later be denoted by $\mathrm{col}_{i}:\mathcal{C}^{\mathtt{T}} \to \mathcal{C}$. The obvious canonical natural transformation $\xi_{i}:\mathrm{col}_{i} \Rightarrow \mathrm{col}_{\mathtt{S}}$ will be called the \emph{universal shifting along $i$}.
Similarly, the composition of the functor $\_ \circ i:\mathcal{C}^{\mathtt{T}} \to \mathcal{C}^{\mathtt{S}}$ with the limit functor $\mathrm{lim}_{\mathtt{S}}:\mathcal{C}^{\mathtt{S}} \to \mathcal{C}$ will be denoted by $\mathrm{lim}_i:\mathcal{C}^{\mathtt{T}} \to \mathcal{C}$.

\subsection{Right Lifting Property} Let $\mathcal{C}$ be a category and $\mathcal{A}$ be a class of arrows in $\mathcal{C}$. The class of arrows of $\mathcal{C}$ that have the right lifting property (abbrev. \emph{rlp}) with respect to the arrows of $\mathcal{A}$ will be denoted by $\mathbf{rlp}(\mathcal{A})$.

\subsection{Sequential Functors}\label{ssec:sequential_functor}
 Let $\kappa$ be some ordinal and $\mathcal{C}$ be a category. A functor $F:\kappa+1 \to \mathcal{C}$ will be said to be \emph{sequential} if for any limit ordinal $\alpha$ in $\kappa+1$, the object $F(\alpha)$ may be identified with the colimit of the functor $F \circ \iota_{\alpha}:\alpha \to \mathcal{C}$ such that, for every ordinal $\beta$ in $\alpha$, the morphism $F(\beta < \alpha):F(\beta) \to F(\alpha)$ corresponds to the arrow of the universal cocone of $\mathrm{col}_{\alpha} F \circ \iota_{\alpha}$ associated with $\beta$.

\begin{proposition}\label{prop:strictly_combinatorial_llp_fib_f}
If a morphism $f:X \to Y$ has the rlp with respect to every arrow $F(k < k+1)$ for every $k \in \kappa$, then $f$ belongs to $\mathbf{rlp}(\{F(0 < k)~|~k \in \kappa+1\})$.
\end{proposition}
\begin{proof}
It is straightforward to show that if a morphism $f$ has the rlp with respect to two composable arrows $i$ and $j$, then it has the rlp with respect to the composition $i \circ j$. A direct generalisation to the transfinite case shows the proposition.
\end{proof}

\subsection{Limit Sketches}
A \emph{limit sketch} is a small category $\mathtt{S}$ equipped with a subset $Q$ of its cones \footnote{Recall that these are, by definition, natural transformations of the form $\Delta_A(d) \Rightarrow U$ in $\mathtt{S}$ where $A$ is a small category, $U$ is a functor $A \to \mathtt{S}$ and $d$ an object in $\mathtt{S}$, called the \emph{peak}}. The cones in $Q$ will be said to be \emph{chosen}. A \emph{model for a limit sketch $\mathtt{S}$ in a category $\mathcal{C}$} is a functor $\mathtt{S} \to \mathcal{C}$ that sends the chosen cones in $Q$ to universal cones \footnote{`Universal' here means that the cone, say $\Delta_A(d) \Rightarrow U$, defines a limit of the functor $U:A \to \mathtt{S}$} in $\mathcal{C}$. The models of a limit sketch $\mathtt{S}$ in $\mathcal{C}$ define the objects of a category $\mathbf{Mod}_{\mathcal{C}}(\mathtt{S})$ whose morphisms are natural transformations in $\mathcal{C}$ over $\mathtt{S}$. For any limit sketch $\mathtt{S}$, the category of models for $\mathtt{S}$ in $\mathbf{Set}$ will be denoted by $\mathbf{Mod}(\mathtt{S})$.

\begin{example}[Limit sketch for monoids]
The category of monoids in $\mathbf{Set}$ may be defined as a category of models for a certain limit sketch $\mathtt{Mon}$. The underlying small category of $\mathtt{Mon}$ is freely generated over a set of arrows and quotiented by commutativity relations. Specifically, the category $\mathtt{Mon}$ has four objects $g_0$, $g_1$, $g_2$ and $g_3$, where $g_1$ is called the \emph{underlying object} of the sketch, and a set of arrows as follows, where the identities have been forgotten.
\[
g_2 \mathop{\longrightarrow}\limits^{\mu} g_1\quad g_0 \mathop{\longrightarrow}\limits^{\eta} g_1
\quad g_3 \mathop{\longrightarrow}\limits^{p_{\underline{1}2}} g_1\quad g_3 \mathop{\longrightarrow}\limits^{p_{1\underline{2}}} g_2
\quad g_3 \mathop{\longrightarrow}\limits^{p_{\underline{2}1}} g_1\quad g_3 \mathop{\longrightarrow}\limits^{p_{2\underline{1}}} g_2
\]
\[
g_2 \mathop{\longrightarrow}\limits^{p_1} g_1 \quad g_2 \mathop{\longrightarrow}\limits^{p_2} g_1 \quad g_3 \mathop{\longrightarrow}\limits^{\mu^*} g_2 \quad g_3 \mathop{\longrightarrow}\limits^{\mu_*} g_2\quad g_1 \mathop{\longrightarrow}\limits^{\eta^*} g_2 \quad g_1 \mathop{\longrightarrow}\limits^{\eta_*} g_2\quad g \mathop{\longrightarrow}\limits^{!} g_0
\]

The commutativity relations are given by the diagrams
\[
\begin{array}{ccc}
\xymatrix@C-0.2pc@R-1.6pc{
&g_1\ar[r]^-{1}&g_1\\
g_3\ar[ru]^{p_{\underline{1}2}}\ar[rd]_{p_{1\underline{2}}}\ar[r]|-{\mu^*}&g_2\ar[ru]|{p_1}\ar[rd]|{p_2}&\\
&g_2\ar[r]_-{\mu}&g_1\\
}&
\xymatrix@C-0.2pc@R-1.6pc{
&g_1\ar[r]^-{1}&g_1\\
g_3\ar[ru]^{p_{2\underline{1}}}\ar[rd]_{p_{\underline{2}1}}\ar[r]|-{\mu_*}&g_2\ar[ru]|{p_1}\ar[rd]|{p_2}&\\
&g_2\ar[r]_-{\mu}&g_1\\
}&
\xymatrix@C-0.2pc@R-0.2pc{
g_3\ar[r]^{\mu_*}\ar[d]_{\mu^*}&g_2\ar[d]^{\mu}\\
g_2\ar[r]_{\mu}&g_1
}
\\
\xymatrix@C-0.2pc@R-1.6pc{
&g_1\ar[r]^-{1}&g_1\\
g_1\ar[ru]^{1}\ar[rd]_{!}\ar[r]|-{\eta^*}&g_2\ar[ru]|{p_1}\ar[rd]|{p_2}&\\
&g_0\ar[r]_-{\eta}&g_1\\
}&
\xymatrix@C-0.2pc@R-1.6pc{
&g_1\ar[r]^-{1}&g_1\\
g_1\ar[ru]^{1}\ar[rd]_{!}\ar[r]|-{\eta_*}&g_2\ar[ru]|{p_2}\ar[rd]|{p_1}&\\
&g_0\ar[r]_-{\eta}&g_1\\
}&
\xymatrix@C-0.2pc@R-0.2pc{
g_1\ar[r]^{\eta_*}\ar[dr]_{1}&g_2\ar[d]^{\mu}&\ar[ld]^{1}\ar[l]_{\eta^*}g_1\\
&g_1&
}
\end{array}
\]
while the chosen cones are given by the trivial cone, of peak $g_0$, defined over the empty category and the following spans.
\[
g_1 \mathop{\longleftarrow}\limits^{p_{1}} g_2 \mathop{\longrightarrow}\limits^{p_{2}} g_1
\quad\quad 
g_1 \mathop{\longleftarrow}\limits^{p_{2\underline{1}}} g_3 \mathop{\longrightarrow}\limits^{p_{\underline{2}1}} g_2
\quad\quad 
g_1 \mathop{\longleftarrow}\limits^{p_{1\underline{2}}} g_3 \mathop{\longrightarrow}\limits^{p_{\underline{1}2}} g_2
\]

The astute reader might have noticed that $\mu$ and $\eta$ stand for the multiplication and unit of the monoid structure. It will come in handy to denote the preceding limit sketch by $\mathtt{Mon}(\mu,\eta)$. Note that other limit sketches can give rise to the same models, so that the previous limit sketch is only an example among other possible presentations of the theory of monoids.
\end{example}

\begin{example}[Limit sketch for commutative monoids]
It is possible to add one more arrow and two diagrams to the limit sketch $\mathtt{Mon}(\mu,\eta)$ so that the category of models associated with the resulting limit sketch, say $\mathtt{Cmon}(\mu,\eta)$, is that of commutative monoids. Precisely, this would imply the adding of an arrow $\sigma:g_2 \to g_2$ that makes the following diagrams commute.
\[
\begin{array}{cccc}
\xymatrix@C-0.2pc@R-1.6pc{
&g_1\ar[r]^-{1}&g_1\\
g_2\ar[ru]^{p_1}\ar[rd]_{p_2}\ar[r]|-{\sigma}&g_2\ar[ru]|{p_2}\ar[rd]|{p_1}&\\
&g_1\ar[r]_-{1}&g_1\\
}&&&
\xymatrix@C-0.2pc@R-0.2pc{
g_2\ar[r]^{\sigma}\ar[d]_{\mu}&g_2\ar[d]^{\mu}\\
g_1\ar[r]_{1}&g_1
}
\end{array}
\]
\end{example}

\begin{example}[Limit sketch for abelian groups]
It is also possible to add three more arrows and two diagrams to the limit sketch $\mathtt{Cmon}(\mu,\eta)$ so that the category of models associated with the resulting limit sketch, which will later be denoted by $\mathtt{Ab}(\mu,\eta,\delta)$ for the notations given below, is that of abelian groups. Precisely, this would imply the adding of three arrows $\delta:g_1 \to g_2$ , $\alpha:g_1 \to g_1$ and $\alpha_*:g_2 \to g_2$ that makes the following diagrams commute.
\[
\begin{array}{cccc}
\xymatrix@C-0.2pc@R-1.6pc{
&g_1\ar[r]^-{\alpha}&g_1\\
g_2\ar[ru]^{p_1}\ar[rd]_{p_2}\ar[r]|-{\alpha_*}&g_2\ar[ru]|{p_2}\ar[rd]|{p_1}&\\
&g_1\ar[r]_-{1}&g_1\\
}
&&&
\xymatrix@C-0.2pc@R-0.2pc{
g_1\ar[r]^{\delta}\ar[d]_{!}&g_2\ar[r]^{\alpha_*}&g_2\ar[d]^{\mu}\\
g_0\ar[rr]_{\eta}&&g_1
}
\end{array}
\]
\end{example}

\begin{example}[Limit sketch for rings]\label{exa:rings_sketch_Rg}
By definition, the subcategory of $\mathtt{Ab}(\mu,\eta,\delta)$ generated by $p_1$, $p_2$, $p_{\underline{1}2}$, $\dots$, $p_{1\underline{2}}$ and $!$ is also included in $\mathtt{Mon}(\mu',\eta')$. The pushout of $\mathtt{Ab}(\mu,\eta,\delta)$ and $\mathtt{Mon}(\mu',\eta')$ along these underlying inclusions provides a certain limit sketch $\mathtt{pRg}(\mu,\mu',\eta,\eta')$ that contains five objects and all the arrows and cones appearing in $\mathtt{Ab}(\mu,\eta,\delta)$ and $\mathtt{Mon}(\mu',\eta')$; the associated limit sketch combines the structure of a monoid with the structure of a commutative monoid. One thus recovers the theory of rings if one adds an object $g_4$, a chosen cone
$g_2 \mathop{\longleftarrow}\limits^{q_{1}} g_4 \mathop{\longrightarrow}\limits^{q_{2}} g_2$
and the following arrows and commutativity relations to $\mathtt{pRg}(\mu,\mu',\eta,\eta')$.
\[
g_3 \mathop{\longrightarrow}\limits^{\delta^*} g_4 \quad\quad
g_3 \mathop{\longrightarrow}\limits^{\delta_*} g_4 \quad\quad
g_4 \mathop{\longrightarrow}\limits^{\pi_1} g_2 \quad\quad
g_4 \mathop{\longrightarrow}\limits^{\pi_2} g_2 \quad\quad
g_4 \mathop{\longrightarrow}\limits^{\mu^{\prime\prime}} g_2 
\]
\[
\begin{array}{cccc}
\xymatrix@C-0.2pc@R-1.6pc{
&g_1\ar[r]^-{\delta}&g_2\\
g_3\ar[ru]^{p_{\underline{1}2}}\ar[rd]_{p_{1\underline{2}}}\ar[r]|-{\delta^*}&g_4\ar[ru]|{q_1}\ar[rd]|{q_2}&\\
&g_2\ar[r]_-{1}&g_2\\
}&
\xymatrix@C-0.2pc@R-1.6pc{
&g_2\ar[r]^-{1}&g_2\\
g_3\ar[ru]^{p_{\underline{2}1}}\ar[rd]_{p_{2\underline{1}}}\ar[r]|-{\delta_*}&g_4\ar[ru]|{q_1}\ar[rd]|{q_2}&\\
&g_1\ar[r]_-{\delta}&g_2\\
}
&
\xymatrix@C-0.2pc@R-1.6pc{
&g_2\ar[r]^-{p_1}&g_1\\
g_4\ar[ru]^{q_1}\ar[rd]_{q_2}\ar[r]|-{\pi_1}&g_2\ar[ru]|{p_1}\ar[rd]|{p_2}&\\
&g_2\ar[r]_-{p_1}&g_1\\
}&
\xymatrix@C-0.2pc@R-1.6pc{
&g_2\ar[r]^-{p_2}&g_1\\
g_4\ar[ru]^{q_1}\ar[rd]_{q_2}\ar[r]|-{\pi_2}&g_2\ar[ru]|{p_1}\ar[rd]|{p_2}&\\
&g_2\ar[r]_-{p_2}&g_1\\
}
\end{array}
\]
\[
\begin{array}{ccc}
\xymatrix@C-0.2pc@R-1.6pc{
&g_2\ar[r]^-{\mu'}&g_1\\
g_4\ar[ru]^{\pi_1}\ar[rd]_{\pi_2}\ar[r]|-{\mu^{\prime\prime}}&g_2\ar[ru]|{p_1}\ar[rd]|{p_2}&\\
&g_2\ar[r]_-{\mu'}&g_1\\
}&
\xymatrix@C-0.2pc@R-0.2pc{
g_3\ar[r]^{\delta^*}\ar[d]_{\mu^*}&g_4\ar[r]^{\mu^{\prime\prime}}&g_2\ar[d]^{\mu}\\
g_2\ar[rr]_{\mu'}&&g_1
}
&
\xymatrix@C-0.2pc@R-0.2pc{
g_3\ar[r]^{\delta_*}\ar[d]_{\mu_*}&g_4\ar[r]^{\mu^{\prime\prime}}&g_2\ar[d]^{\mu}\\
g_2\ar[rr]_{\mu'}&&g_1
}
\end{array}
\]

The resulting limit sketch $\mathtt{Rg}(\mu,\mu',\eta,\eta')$ then defines a sketch for which the models are rings. The limit sketch $\mathtt{Rg}(\mu,\mu',\eta,\eta')$ to which the identity morphism $1_{g_1}:g_1 \to g_1$ is added to the set of chosen cones---when seen as a trivial cone---will later be denoted by $\mathtt{Rg}$.
\end{example}

\subsection{Subfunctors}
Let $D$ be a small category and $F:D \to \mathbf{Set}$ be a functor. A \emph{subfunctor of $F$} is a functor $G:D \to \mathbf{Set}$ such that 
(1) for every object $d$ in $D$, the inclusion $G(d) \subseteq F(d)$ holds and
(2) for every morphism $t:d \to d'$ in $D$, the function $G(t):F(d) \to F(d')$ is the restriction of $F(t)$ along the respective inclusions of the domain and codomain.

\subsection{Overcategories}\label{ssec:Overcategories}
Let $\mathcal{C}$ be a category and $X$ be an object in $\mathcal{C}$. The obvious functor $\mathcal{C}/X \to \mathcal{C}$ mapping
an arrow $f:A \to X$ in $\mathcal{C}$ to the object $A$ in $\mathcal{C}$ will be denoted by $\partial$. 

\begin{remark}\label{rem:functor_in_overcats_give_nat_transf}
Let $\mathtt{T}$ be a small category.
Any functor $F:\mathtt{T} \to \mathcal{C}/X$ may be seen as a natural transformation in $\mathcal{C}$
over $\mathtt{T}$ of the form $h:\partial F \Rightarrow \Delta_{\mathtt{T}}(X)$. The converse is also true.
\end{remark}

Let now $G:\mathcal{A} \to \mathcal{C}$ be a functor. It will come in handy to denote by $\mathcal{C}\rotatebox[origin=c]{-90}{$\multimap$}G$
the obvious functor on $\mathcal{A}$ satisfying the following mapping rule on the objects.
\[
\begin{array}{ccc}
X &\mapsto& \mathcal{C}/G(X)\\
(u:X \to Y) &\mapsto& \mathcal{C}/G(u)\\
\end{array}
\]

\subsection{Covering Families}\label{sssec:Sieves}
Let $D$ be a small category and $d$ be an object in $D$. A \emph{covering family} on $d$ is a collection $S:=\{u_i:d_i \to d\}_{i \in A}$ of arrows in $D$.
For every morphism $f:c \to d$ in $D$, we shall speak of the \emph{pullback of $S$ along $f$} to refer to a collection of arrows $f^*S:=\{v_i:c_i \to c\}_{i \in A}$ where the arrow $v_i$ is a pullback of $u_i$ along $f$. Also, note that every morphism $g:d \to c$ gives rise to a family $g \circ S:=\{g \circ u_i\}_{i \in A}$. This last operation is used to define a more complex operation on $S$ as follows. For every $i \in A$, take a covering family $T_{i}$ on $d_i$. We will denote by $S \circ \{T_{i}\}_{i \in A}$ the covering family on $d$ obtained by the disjoint union of families $u_i \circ T_{i}$ for every $i \in A$.

\subsection{Grothendieck Pretopologies}
Let $D$ be a small category. A \emph{Grothendieck pretopology} on $D$ consists, for every object $d$ in $D$, of a collection $J_d$ of covering families $S$ on $d$ such that
\begin{itemize}
\item[(1)] (Stability) for every arrow $f:c \to d$ in $D$, the pullback $f^*S$ exists in $J_c$;
\item[(2)] (Locality) for every $i \in A$ and $T_{i}$ in $J_{d_i}$, the covering family $S \circ  \{T_{i}\}_{i \in A}$ is in $J_d$;
\item[(3)] (Identity) for every object $d$ in $D$, the singleton $\{\mathrm{id}_d:d \to d\}$ is in $J_d$.
\end{itemize}

Such a collection will usually be denoted by $J$. A category $D$ equipped with a Grothen\-dieck pretopology $J$ on $D$ will be called a \emph{site}.

\begin{remark}\label{rem:Grothendieck_pretopologies_rem_stabilisation}
Every covering family $S=\{u_i:d_i \to d\}_{i \in A}$ on an object $d$ in $J_d$ may be seen as a functor $A \to D/d$ if $A$ is seen as a discrete category. It follows from the stability and locality axioms that this functor extends to a product-preserving functor $A' \to D/d$ where $A'$ is the completion of $A$ under products. This functor will be called the \emph{stabilisation of $S$}.
\end{remark}

\subsection{Families}
For any category $\mathcal{C}$, the notation $\mathbf{Fam}(\mathcal{C})$ will be used to denote the category whose objects
are pairs $(S,F)$ where $S$ is a discrete category and $F$ is a functor $F:S \to \mathcal{C}$ and whose morphisms $(S,F) \Rightarrow (S',F')$ are given by pairs $(a,\alpha)$ where $a$ is a functor $a:S \to S'$ and $\alpha$ is a natural transformation $\alpha:F \Rightarrow F'a$.

\subsection{Bounded Diagrams}\label{ssec:Bounded_diag}
Let $D$ be a small category, $d$ be an object in $D$ and $\mathcal{C}$ be a category. We will denote by
$\mathbf{B}_d^{D}\mathcal{C}$ the category whose objects are triples $(P,e,Q)$ where $P$ and $Q$ are functors $D \to \mathcal{C}$ and 
$e$ is an arrow $P(s) \to Q(s)$ in $\mathcal{C}$ and whose morphisms, say $(P,e,Q) \to (P',e',Q')$, are given by pairs of natural transformations $(\alpha,\alpha')$ of respective forms $P \Rightarrow P'$ and $Q \Rightarrow Q'$ making the following square commute.
\[
\xymatrix{
P(d)\ar[r]^{\alpha(d)}\ar[d]_e&P'(d)\ar[d]^f\\
Q(d)\ar[r]_{\alpha'(d)}&Q'(d)
}
\]

Note that $\mathbf{B}_d^{D}\mathcal{C}$ is also a functor category $\mathcal{C}^{\mathbf{2}_d[D]}$ where $\mathbf{2}_d[D]$ is the
smallest subcategory of $\mathbf{2} \times D$ consisting of the two copies of $D$ and the arrow linking the two copies of $d$.

\section{Convergent Functors}\label{sec:Convergent_functor}

This section aims to define the notion of convergent functor, which is to replace the notion of `small object' that is usually used in transfinite constructions.

\subsection{Emulations}
Let $\mathtt{S}$ and $\mathtt{T}$ be two small categories and $\mathcal{C}$ be a category. A pair of functors $g:\mathcal{C}^{\mathtt{S}} \to \mathcal{C}^{\mathtt{T}}$ and $h:\mathbf{Set}^{\mathtt{S}} \to \mathbf{Set}^{\mathtt{T}}$ will be called an \emph{$(\mathtt{S}\downarrow \mathtt{T})$-emulation} in $\mathcal{C}$ if it is equipped with a natural isomorphism as follows.
\[
\xymatrix{
\mathcal{C}^{\mathrm{op}} \times\mathcal{C}^{\mathtt{S}}\ar@{}[rrd]|{\text{\rotatebox[origin=c]{45}{$\cong$}}}\ar[rr]^{\mathrm{id}_{\mathcal{C}} \times g}\ar[d]_{\mathcal{C}(\_,\_)}&&\mathcal{C}^{\mathrm{op}} \times\mathcal{C}^{\mathtt{T}}\ar[d]^{\mathcal{C}(\_,\_)}\\
\mathbf{Set}^{\mathtt{S}}\ar[rr]_h&&\mathbf{Set}^{\mathtt{T}}
}
\]

In terms of an equation, the previous diagram means that $(g,h)$ is equipped with a natural isomorphism (in the variables $X \in \mathcal{C}$, $Y \in \mathcal{C}^{\mathtt{S}}$ and $t \in \mathtt{T}$) as follows.
\begin{equation}\label{eq:g_hat_g_natura;_siomorphism_convergence}
\mathcal{C}\Big(X,g(Y)(t)\Big) \cong h\Big(\mathcal{C}(X,Y(\_))\Big)(t)
\end{equation}

\begin{example}\label{exam:emulation:identity}
Let $\mathtt{T}$ be a small category and $\mathcal{C}$ be a category. Take $g$ to be the identity functor $\mathrm{id}:\mathcal{C}^{\mathtt{T}} \to \mathcal{C}^{\mathtt{T}}$ and $h$ to be the identity functor $\mathrm{id}:\mathbf{Set}^{\mathtt{T}} \to \mathbf{Set}^{\mathtt{T}}$. By definition, the pair $(g,h)$ is a $(\mathtt{T}\downarrow \mathtt{T})$-emulation.
\end{example}

\begin{example}\label{exam:emulation:precomposition_functor}
Let  $U:\mathtt{T} \to \mathtt{S}$ be a functor between small categories and $\mathcal{C}$ be a category. Take $g$ to be the pre-composition functor $\mathcal{C}^{\mathtt{S}} \to \mathcal{C}^{\mathtt{T}}$ induced by $U$ and $h$ to be the equivalent version of $g$ in $\mathbf{Set}$. It suffices a few lines of calculation to show that the following isomorphism holds, which implies that the pair $(g,h)$ defines an $(\mathtt{S}\downarrow \mathtt{T})$-emulation.
\[
\mathcal{C}\Big(X,Y \circ U(t)\Big) \cong \Big(\mathcal{C}(X,Y(\_))\Big)(U(t))
\]
\end{example}

\begin{example}\label{exam:emulation:delta}
Let $\mathtt{T}$ be a small category and $\mathcal{C}$ be a category. Take $g$ to be the functor $\Delta_{\mathtt{T}}:\mathcal{C}^{\mathbf{1}} \to \mathcal{C}^{\mathtt{T}}$ and $h$ to be the functor $\Delta_{\mathtt{T}}:\mathbf{Set}^{\mathbf{1}} \to \mathbf{Set}^{\mathtt{T}}$. It follows from Example \ref{exam:emulation:precomposition_functor} that the pair $(g,h)$ is a $(\mathbf{1}\downarrow \mathtt{T})$-emulation.
\[
\mathcal{C}\Big(X,\Delta_{\mathtt{T}}(Y)(t)\Big) \cong \Delta_{\mathtt{T}}\Big(\mathcal{C}(X,Y)\Big)(t)
\]
\end{example}

\begin{example}\label{exam:emulation:case_of_limits}
Let $\mathtt{S}$ be a small category and $\mathcal{C}$ be a category complete over $\mathtt{S}$. Take $g$ to be the limit functor $\mathrm{lim}_{\mathtt{S}}:\mathcal{C}^{\mathtt{S}} \to \mathcal{C}^{\mathbf{1}}$ and $h$ to be the limit functor $\mathrm{lim}_{\mathtt{S}}:\mathbf{Set}^{\mathtt{S}} \to \mathbf{Set}^{\mathbf{1}}$. It is a well-known fact following from Yoneda's Lemma that the pair $(g,h)$ is an $(\mathtt{S}\downarrow \mathbf{1})$-emulation.
\end{example}

\begin{example}\label{exam:emulation:case_of_unit_of_limits}
Let $\mathtt{S}$ be a small category and $\mathcal{C}$ be a category complete over $\mathtt{S}$. We will denote by $\eta$ the unit of the adjunction $\Delta_{\mathtt{S}} \vdash \mathrm{lim}_{\mathtt{S}}$ valued in any category. Now, take $g$ to be the obvious functor $\mathcal{C} \to \mathcal{C}^{\mathbf{2}}$ mapping an object $X$ in $\mathcal{C}$ to the arrow $\eta_X:X \to \mathrm{lim}_{\mathtt{S}} \Delta_{\mathtt{S}}(X)$ in $\mathcal{C}$ and $h$ to be the equivalent version of $g$ in the category $\mathbf{Set}$ (which is complete over $\mathtt{S}$). It follows from Yoneda's Lemma that the following diagram commutes, which implies that the pair $(g,h)$ is an $(\mathtt{S}\downarrow \mathbf{2})$-emulation.
\[
\xymatrix{
\mathcal{C}(X,Y)\ar[r]^{=}\ar[d]_{\eta_{\mathcal{C}(X,Y)}}&\mathcal{C}(X,Y)\ar[d]^{\mathcal{C}(X,\eta_Y)}\\
*+!R(.45){\mathrm{lim}_{\mathtt{S}}\Delta_{\mathtt{S}}\mathcal{C}(X,Y)}\ar[r]_{\cong}&*+!L(.5){\mathcal{C}(X,\mathrm{lim}_{\mathtt{S}}\Delta_{\mathtt{S}}(Y))}
}
\]
\end{example}

\begin{example}\label{exam:emulation:unit_plus_cone}
Let $\mathtt{S}$ be a small category. For this example, we shall additionally need a small category $A$ together a cone $r:\Delta_{A}(s) \Rightarrow U$ in $\mathtt{S}^A$. Let now $\mathcal{C}$ denote a complete category over $A$. The unit of the adjunction $\Delta_{A} \vdash \mathrm{lim}_{A}$ in $\mathcal{C}$ will be denoted by $\eta$. Now, to define our emulation, take $g$ to be the obvious functor $\mathcal{C}^{\mathtt{S}} \to \mathcal{C}^{\mathbf{2}}$ mapping a functor $P:\mathtt{S} \to \mathcal{C}$ to the arrow
\[
\xymatrix{
P(s) \ar[rr]^-{\eta_{P(s)}}&&\mathrm{lim}_{A} \Delta_{A} (P(s)) \ar[rr]^-{\mathrm{lim}_{A} P r}&& \mathrm{lim}_{A} PU
}
\]
in $\mathcal{C}$ and $h$ to be the equivalent version of $g$ in the category $\mathbf{Set}$. It follows from Yoneda's Lemma that the pair $(g,h)$ is an $(\mathtt{S}\downarrow \mathbf{2})$-emulation. Specifically, the isomorphism associated with the pair $(g,h)$ may be deduced from the isomorphisms involved in Examples \ref{exam:emulation:precomposition_functor}, \ref{exam:emulation:case_of_limits} and  \ref{exam:emulation:case_of_unit_of_limits}.
\end{example}

\begin{example}\label{exam:emulation:gluing_functor}
Let $\mathtt{S}$ be a small category. For this example, we shall need a small category $A$ together a cone $r:\Delta_{A}(s) \Rightarrow U$ in $\mathtt{S}^A$. Let now $\mathcal{C}$ denote a complete category over $A$. The unit of the adjunction $\Delta_{A} \vdash \mathrm{lim}_{A}$ in $\mathcal{C}$ will be denoted by $\eta$. Now, to define our emulation, take $g$ to be the obvious functor $\mathbf{B}_s^{\mathtt{S}}\mathcal{C} \to \mathcal{C}^{\mathbf{2}}$ mapping an object $(P,e,Q)$ in $\mathbf{B}_s^{\mathtt{S}}\mathcal{C}$ to the arrow
\[
\xymatrix{
P(s) \ar[rr]^-{\eta_{P(s)}}&&\mathrm{lim}_{A} \Delta_{A} P(s) \ar[rr]^-{\mathrm{lim}_{A}\Delta_{A} e}&& \mathrm{lim}_{A} \Delta_{A}Q(s) \ar[rr]^{\mathrm{lim}_{A} Qr}&&
\mathrm{lim}_{A} QU
}
\]
in $\mathcal{C}$ and $h$ to be the equivalent version of $g$ in the category $\mathbf{Set}$. It follows from the isomorphisms involved in Examples \ref{exam:emulation:delta}, \ref{exam:emulation:case_of_limits} and  \ref{exam:emulation:unit_plus_cone} that the pair $(g,h)$ is an $(\mathbf{2}_s[\mathtt{S}]\downarrow \mathbf{2})$-emulation.
\end{example}

\subsection{Cocontinuous Emulations}
Let $\mathtt{S}$ and $\mathtt{T}$ be two small categories, $\mathcal{C}$ be a category and $\kappa$ be a limit ordinal.
An $(\mathtt{S}\downarrow \mathtt{T})$-emulation $(g,h)$ in $\mathcal{C}$ will be said to be \emph{$\kappa$-cocontinuous}, if for every object $t \in \mathtt{T}$, the functor
$h:\mathbf{Set}^{\mathtt{S}} \to \mathbf{Set}^{\mathtt{T}}$ preserves colimits over $\kappa$.

\begin{example}
Since identity functors preverse colimits, the pair $(g,h)$ of Example \ref{exam:emulation:identity} is a $\kappa$-cocontinuous $(\mathtt{T}\downarrow \mathtt{T})$-emulation for every limit ordinal $\kappa$.
\end{example}

\begin{example}\label{exam:emulation:cocontinuous:precomposition_functor}
Consider the same context as that used in Example \ref{exam:emulation:precomposition_functor}. Since $\mathbf{Set}$ is cocomplete over any small category $D$, the colimits of $\mathbf{Set}^{\mathtt{S}}$ are componentwise colimits, which means that for every functor $F:D \to \mathbf{Set}^{\mathtt{S}}$, the following isomorphism holds for every $s \in \mathtt{S}$.
\[
\Big(\mathrm{col}_{d \in D} F\Big)(s) \cong \mathrm{col}_{d \in D}(F(s))
\]

This directly implies that the functor $h:\mathbf{Set}^{\mathtt{S}} \to \mathbf{Set}^{\mathtt{T}}$ preserves colimits, which shows that the $(\mathtt{S}\downarrow \mathtt{T})$-emulation $(g,h)$ is $\kappa$-cocontinuous for every limit ordinal $\kappa$.
\end{example}

\begin{example}\label{exam:emulation:cocontinuous:Delta}
It follows from Example \ref{exam:emulation:cocontinuous:precomposition_functor} that the $(\mathbf{1}\downarrow \mathtt{T})$-emulation $(g,h)$ of Example \ref{exam:emulation:delta} is $\kappa$-cocontinuous for every limit ordinal $\kappa$.
\end{example}

\begin{example}\label{exam:emulation:cocontinuous:case_of_limits}
Consider the same context as that used in Example \ref{exam:emulation:case_of_limits} and suppose to be given a limit ordinal $\kappa$ satisfying the inequality $|\mathtt{T}| \leq \kappa$.
It directly follows from Proposition \ref{prop:SGA_limits_colimits_commute} that the functor $h:\mathbf{Set}^{\mathtt{S}} \to \mathbf{Set}^{\mathtt{T}}$ preserves colimits over $\kappa$. This shows that the $(\mathtt{S}\downarrow \mathbf{1})$-emulation $(g,h)$ is $\kappa$-cocontinuous.
\end{example}

\begin{example}\label{exam:emulation:cocontinuous:case_of_unit_of_limits}
Consider the same context as that used in Example \ref{exam:emulation:case_of_unit_of_limits} and suppose to be given an limit ordinal $\kappa$ satisfying the inequality $|\mathtt{T}| \leq \kappa$.
It follows from Proposition \ref{prop:eta_commutes_with_colimits_O_kappa} that the functor $h:\mathbf{Set}^{\mathtt{S}} \to \mathbf{Set}^{\mathtt{T}}$ preserves colimits over $\kappa$. This shows that the $(\mathtt{S}\downarrow \mathbf{2})$-emulation $(g,h)$ is $\kappa$-cocontinuous.
\end{example}

\begin{example}
By using the cocontinuity involved in Examples \ref{exam:emulation:cocontinuous:case_of_limits} and  \ref{exam:emulation:cocontinuous:case_of_unit_of_limits}, we may show that the
$(\mathtt{S}\downarrow \mathbf{2})$-emulation $(g,h)$ is $\kappa$-cocontinuous for any limit ordinal $\kappa$ satisfying the inequality $|A| \leq \kappa$.
\end{example}

\begin{example}\label{exam:emulation:cocontinuous:gluing_functor}
By using the cocontinuity involved in Examples \ref{exam:emulation:cocontinuous:precomposition_functor}, \ref{exam:emulation:cocontinuous:case_of_limits} and \ref{exam:emulation:cocontinuous:case_of_unit_of_limits}, Example we may show that the
$(\mathbf{2}_s[\mathtt{S}]\downarrow \mathbf{2})$-emulation $(g,h)$ is $\kappa$-cocontinuous for any limit ordinal $\kappa$ satisfying the inequality $|A| \leq \kappa$.
\end{example}

\subsection{Convergent Functors}\label{sssec:convergent_functors}
For any class $\mathcal{G}$ of objects of $\mathcal{C}$, a functor $F:\kappa+1 \to \mathcal{C}$ will be said to 
be \emph{$\mathcal{G}$-convergent in $\mathcal{C}$} if for every object $\mathbb{D}$ in $\mathcal{G}$, the following canonical function (obtained by homing)
is an isomorphism in $\mathbf{Set}$.
\[
\mathrm{col}_{\kappa}\mathcal{C}(\mathbb{D},F\iota_{\kappa}) \to \mathcal{C}(\mathbb{D},F(\kappa)) 
\]

If the class $\mathcal{G}$ turns out to be a singleton $\{\mathbb{D}\}$, the functor will more explicitly be said to be $\mathbb{D}$-convergent. 

\begin{remark}\label{rem:surjectiveness_property_of_convergence_for_SOA}
One of the useful implications of the previous definition is that if a functor $F:\kappa+1 \to \mathcal{C}$ is $\mathcal{G}$-convergent in $\mathcal{C}$, then for every object $\mathbb{D} \in \mathcal{G}$ and morphism $f:\mathbb{D} \to F(\kappa)$ in $\mathcal{C}$, there exist an ordinal $\alpha \in \kappa$ and a morphism $f':\mathbb{D} \to F(\alpha)$ making the following diagram commute  in $\mathcal{C}$.
\[
\xymatrix@R-1.8pc{
&F(\alpha)\ar[rd]^{F(\alpha < \kappa)}&\\
\mathbb{D}\ar[rr]_{f}\ar[ru]^{f'}&&F(\kappa)
}
\]
\end{remark}

Let now $\mathtt{T}$ and $\mathtt{S}$ denote two small categories and $G:\mathtt{T} \to \mathcal{C}$ be a functor. A functor $F:\kappa+1 \to \mathcal{C}^{\mathtt{S}}$ will be said to 
be \emph{unimorly $G$-convergent in $\mathcal{C}$} if for every object $s$ in $\mathtt{S}$ and object $t$ in $\mathtt{T}$, the following canonical function is an isomorphism in $\mathbf{Set}$.
\[
\mathrm{col}_{\kappa}\mathcal{C}(G(t),(F\iota_{\kappa}(\_))(s)) \to \mathcal{C}(G(t),F(\kappa)(s)) 
\]

In other words, the evaluation of $F$ at an object $s$ in $\mathtt{S}$ is $\{G(t)~|~t\in \mathrm{Obj}(\mathtt{T})\}$-convergent.

\begin{lemma}\label{lem:uniform_convergence_to_functorial_convergence}
Let $\mathtt{T}$ and $\mathtt{S}$ be two small categories such that $|\mathtt{T}| \leq \kappa$ and $\mathcal{C}$ be a category. Let $G:\mathtt{T} \to \mathcal{C}$ be a functor and consider a uniformly $G$-convergent functor $F:\kappa+1 \to \mathcal{C}^{\mathtt{S}}$ in $\mathcal{C}$. For every cocontinuous $(\mathtt{S}\downarrow \mathtt{T})$-emulation $(g,h)$, the composite functor $g \circ F:\kappa+1 \to \mathcal{C}^{\mathtt{T}}$ is $G$-convergent in $\mathcal{C}^{\mathtt{T}}$.
\end{lemma}
\begin{proof}
The following series of natural isomorphisms proves the statement.
\begin{align*}
\mathcal{C}^{\mathtt{T}}(G,g \circ F(\kappa)) & \cong \int_{t \in \mathtt{T}}\mathcal{C}\big(G(t),g \circ F(\kappa)(t)\big)&\text{(Definition)}\\
 &\cong  \int_{t \in \mathtt{T}}h\Big(\mathcal{C}\big(G(t),F(\kappa)(\_)\big)\Big)(t)&\text{(Equation (\ref{eq:g_hat_g_natura;_siomorphism_convergence}))} \\
 & \cong  \int_{t \in \mathtt{T}}h\Big(\mathrm{col}_{\kappa}\mathcal{C}\big(G(t),F(\iota_{\kappa}(\_))(\_)\big)\Big)(t)&\text{(Uniform conv.)}\\
 & \cong  \int_{t \in \mathtt{T}}\mathrm{col}_{\kappa}h\Big(\mathcal{C}\big(G(t),F(\iota_{\kappa}(\_))(\_)\big)\Big)(t)&\text{(Cocontinuity)} \\
 & \cong  \mathrm{col}_{\kappa}\int_{t \in \mathtt{T}}h\Big(\mathcal{C}\big(G(t),F(\iota_{\kappa}(\_))(\_)\big)\Big)(t)&\text{(Proposition \ref{prop:SGA_limits_colimits_commute})} \\
 & \cong  \mathrm{col}_{\kappa}\int_{t \in \mathtt{T}}\mathcal{C}\big(G(t),g(F\circ \iota_{\kappa}(\_))(t)\big)&\text{(Equation (\ref{eq:g_hat_g_natura;_siomorphism_convergence}))} \\
 & \cong  \mathrm{col}_{\kappa}\mathcal{C}^{\mathtt{T}}(G,g \circ F\circ \iota_{\kappa}(\_))&\text{(Definition)} 
\end{align*}

This last isomorphism shows that $g \circ F$ is $G$-convergent in $\mathcal{C}^{\mathtt{T}}$.
\end{proof}

\begin{example}
Applying Lemma \ref{lem:uniform_convergence_to_functorial_convergence} to the $(\mathtt{T}\downarrow \mathtt{T})$-emulation $(g,h)$ of Example \ref{exam:emulation:identity} implies that if a functor $F:\kappa+1 \to \mathcal{C}^{\mathtt{T}}$ is uniformly $G$-convergent in $\mathcal{C}$ and the inequality $|\mathtt{T}| \leq \kappa$ holds, then the functor $F:\kappa+1\to \mathcal{C}^{\mathtt{T}}$ is $G$-convergent in $\mathcal{C}^{\mathtt{T}}$.
\end{example}

\begin{example}\label{exam:lemma:convegence_of_functors:gluing_functor}
Applying Lemma \ref{lem:uniform_convergence_to_functorial_convergence} to the $(\mathbf{2}_s[\mathtt{S}]\downarrow \mathbf{2})$-emulation $(g,h)$ of Example \ref{exam:emulation:cocontinuous:gluing_functor} implies that if a functor $(P,e,Q):\kappa+1 \to \mathbf{B}_s^{\mathtt{S}}\mathcal{C}^{\mathtt{S}}$ is uniformly $G$-convergent in $\mathcal{C}$ for some functor $G:\mathbf{2} \to \mathcal{C}$ and the inequality $\mathbf{2} \leq \kappa$ holds, then the functor mapping an ordinal $n$ in $\kappa+1$ to the following composite arrow in $\mathcal{C}$ is $G$-convergent in $\mathcal{C}^{\mathbf{2}}$.
\[
\xymatrix{
P_n(s) \ar[rr]^-{\eta_{P_n(s)}}&&\mathrm{lim}_{A} \Delta_{A} P_n(s) \ar[rr]^-{\mathrm{lim}_{A}\Delta_{A} e_n}&& \mathrm{lim}_{A}\Delta_{A} Q_n(s) \ar[rr]^{\mathrm{lim}_{A} Q_nr}&&\mathrm{lim}_{A} Q_nU
}
\]
\end{example}

\begin{remark}\label{rem:kappa_smallness_colimit_of_uniform_convergent_functor}
It follows from Lemma \ref{lem:uniform_convergence_to_functorial_convergence} that if a functor $F:\kappa+1 \to \mathcal{C}$ is uniformly $G$-convergent in $\mathcal{C}$, then $F:\kappa+1 \to \mathcal{C}$ is $\mathrm{col}_{\mathtt{T}}(G)$-convergent in $\mathcal{C}$. Specifically, this follows from the fact that $\Delta_{\mathtt{T}}$ commutes with hom-sets (see Example \ref{exam:emulation:cocontinuous:Delta}) and the following series of isomorphisms.
\begin{align*}
\mathcal{C}(\mathrm{col}_{\mathtt{T}}G,F(\kappa)) & \cong \mathcal{C}^{\mathtt{T}}(G,\Delta_{\mathtt{T}}(F(\kappa)))&\text{(Adjointness)}\\
 & \cong  \mathrm{col}_{\kappa}\mathcal{C}^{\mathtt{T}}(G,\Delta_{\mathtt{T}} \circ F\circ \iota_{\kappa})&\text{(Lemma \ref{lem:uniform_convergence_to_functorial_convergence})}\\
  & \cong  \mathrm{col}_{\kappa}\mathcal{C}(\mathrm{col}_{\mathtt{T}}G,F\circ \iota_{\kappa})&\text{(Adjointness)}
\end{align*}
\end{remark}

\section{Models for a Croquis}\label{sec:Category_of_models_for_a_croquis}

This section defines the notions of premodel and model for which we want to construct the localisation. We start with the type of theory on which the models are defined.

\subsection{Croquis}\label{sssec:Croquis}
Let $D$ be a small category. Recall that a cone in $D$ over a small category $A$ consists of two functors $d_0:\mathbf{1} \to D$ and $d_1:A \to D$
and a natural transformation $t:\Delta_A d_0 \Rightarrow d_1$.
When such a cone is called $c$, the functor $d_0$ will be denoted by $\mathtt{ou}(c)$, the functor $d_1$ will be denoted by $\mathtt{in}(c)$ and the small category $A$ will be referred to as the \emph{elementary shape} of $c$ and denoted by $\mathtt{Es}(c)$.

\begin{definition}
\emph{A {croquis category} (or {croquis}) in $D$ consists of a set $K$ of cones in $D$ and a functor $\mathtt{rou}:K \to D$ (where $K$ is seen as a discrete category) called the {regular output}.}
\end{definition}
A croquis as above will be denoted by a triple $(D,K,\mathtt{rou})$ and sometimes shortened to the pair $(K,\mathtt{rou})$ when the ambient category $D$ is obvious.

\begin{convention}
For every croquis $(D,K,\mathtt{rou})$, the operation $\mathtt{ou}(\_)$ induces a function from $K$ to $\mathrm{Obj}(D)$.
Alternatively, this may be seen as a functor $K \to D$. If the functor $\mathtt{rou}:K \to D$ is equal to $\mathtt{ou}:K \to D$, then the croquis will be denoted by $(D,K)$ or $K$ and the functor $\mathtt{rou}$ will be said to be \emph{trivial}.
\end{convention}

\begin{example}[Arrow categories]\label{exa:xroquis_arrow_categories}
Let $D$ be a small category, $D'$ be a subcategory of $D$ and $T:D' \to D$ be some given functor. The set $\mathrm{Mor}(D')$ of arrows of $D'$ defines an obvious set of cones of elementary shape $\mathbf{1}$ in $D'$. However, because $D'$ is a subcategory of $D$, we shall in fact see $\mathrm{Mor}(D')$ as a set of cones specifically in $D$. The croquis (in $D$) made of $\mathrm{Mor}(D')$ and the regular output $\mathrm{Mor}(D') \to D$ mapping any arrow $d \to c$ in $D'$ to the object $T(d)$ will later be denoted by $\mathbf{Cr}(D,T)$.
\end{example}

\begin{example}[Spectra]\label{exa:Croquis_for_spectra}
Let $\mathbb{N}$ denote the wide discrete subcategory of the ordinal category $\omega$ and $\mathbb{N}^*$ denote the full subcategory of $\mathbb{N}$ restricted to positive ordinals. Let $\mathtt{pred}:\mathbb{N}^* \to \mathbb{N}$ be the predecessor operation $n \mapsto n-1$. The croquis defined by $\mathbf{Cr}(\mathbb{N},\mathtt{pred})$ will later be used to characterise $\Omega$-spectra.
\end{example}

\begin{example}[Sketches]\label{exa:Croquis_for_a_sketch}
Any limit sketch $(D,K)$ defines an obvious croquis where $K$ stands for the set of chosen cones and where the associated regular output $K \to D$ is the trivial one.
\end{example}

\begin{example}[Grothendieck's pretopologies]
Let $J$ denote a Grothendieck pretopology on a small (opposite) category $D^{\mathrm{op}}$. A covering family $C=\{v_i \to u\}_{i \in A}$ in $J_u$ may be seen as a cone of the form $t:\Delta_{A}(u) \Rightarrow v(\_)$ in $D$ over $A$. If one denotes by $A' \to D^{\mathrm{op}}/d$ the stabilisation of $C$ (see Remark \ref{rem:Grothendieck_pretopologies_rem_stabilisation}), this cone gives rise to another cone $t':\Delta_{A'}(u) \Rightarrow v'(\_)$ over $A'$.
Equipping $D$ with the set of these latest cones, say $K_J$, gives rise to an obvious croquis $(D,K_J)$.
\end{example}

\begin{example}[Flabby pretopologies]\label{ex:Croquis_for_flabby_sheaves}
Let $J$ denote a Grothendieck pretopology on a small (opposite) category $D^{\mathrm{op}}$. The croquis that will later give 
rise to flabby sheaves and the Godement resolution is the union of the two croquis $(D,K_{J})$ and $\mathbf{Cr}(D,\mathrm{id}_D)$. Precisely, this croquis consists of the union of the two sets of cones $K_{J}$ and $\mathrm{Mor}(D)$ and the trivial regular output.
\end{example}

\begin{example}[Segal croquis]\label{ex:Croquis_reedy}
Let $\mathbf{\Delta}$ denote the category of non-zero finite ordinals and preserving-order functions, which is known as the \emph{simplex category}. Denote by $\mathbf{\Delta}^{+}$ the wide subcategory of $\mathbf{\Delta}$ whose arrows are injective functions and, for every object $\mathbf{r} \in \mathbf{\Delta}$, denote by $\partial_{\mathbf{r}}$ the composition of the functor $\partial:\mathbf{\Delta}/\mathbf{r}  \to \mathbf{\Delta}$ (see Section \ref{ssec:Overcategories}) with the obvious inclusion $\mathbf{\Delta}^{+}/\mathbf{r}  \hookrightarrow \mathbf{\Delta}/\mathbf{r} $. The \emph{Segal croquis} of $\mathbf{\Delta}^{\mathrm{op}}$ is of the form $(\mathbf{\Delta}^{\mathrm{op}},K_0 \cup K_1)$ (for a trivial regular output) where 
\begin{itemize}
\item[(i)] $K_0$ contains, for every object $\mathbf{r}  \in \mathbf{\Delta}^{\mathrm{op}}$, the cone $\Delta_{(\mathbf{\Delta}^{+}/\mathbf{r})^{\mathrm{op}}}(\mathbf{r}) \Rightarrow (\partial_{\mathbf{r}})^{\mathrm{op}}$, defined over $(\mathbf{\Delta}^{+}/\mathbf{r})^{\mathrm{op}}$, that stems from the dual transformation described in Remark \ref{rem:functor_in_overcats_give_nat_transf}  for the inclusion functor 
$\mathbf{\Delta}^{+}/\mathbf{r} \hookrightarrow \mathbf{\Delta}/\mathbf{r}$; 
\item[(ii)] $K_1$ contains, for every object $\mathbf{r} \in \mathbf{\Delta}^{\mathrm{op}}$, the cone given below (expressed in $\mathbf{\Delta}$ as a cocone), where, if one denotes $\mathbf{1}:=\{0\}$, $\mathbf{2}:=\{0,1\}$ and $\mathbf{r}:=\{0,\dots,r-1\}$
\begin{itemize}
\item[(1)] $a_i:\mathbf{2} \to \mathbf{r}$ is the function with the mapping rules $0 \mapsto i$ and $1 \mapsto i+1$;
\item[(2)] $s:\mathbf{1} \to \mathbf{2}$ is the function with the mapping rule $0 \mapsto 0$;
\item[(3)] $t:\mathbf{1} \to \mathbf{2}$ is the function with the mapping rule $0 \mapsto 1$.
\end{itemize}
\[
\xymatrix@R-2pc{
&&&\mathbf{r}&&&\\
&&&&&&\\
&&&&&&\\
\mathbf{2}\ar@{..>}[rrruuu]^{a_0}\ar@{<-}[rd]_t&&\mathbf{2}\ar@{..>}[ruuu]|{a_1}\ar@{<-}[ld]^s\ar@{<-}[rd]_t&&\mathbf{2}\ar@{..>}[luuu]|{a_{r-2}}\ar@{<-}[ld]^s\ar@{<-}[rd]_t&&\mathbf{2}\ar@{..>}[llluuu]_{a_{r-1}}\ar@{<-}[ld]^s\\
&\mathbf{1}&&\dots&&\mathbf{1}&\\
}
\]
\end{itemize}

This croquis will be denoted by $\mathbf{Seg}(\mathbf{\Delta}^{\mathrm{op}})$.
\end{example}

\begin{example}[Complete Segal croquis]\label{ex:Croquis_Complete_reedy}
Let $\mathbf{\Delta}$ be the simplex category. The \emph{complete Segal croquis} of $\mathbf{\Delta}^{\mathrm{op}}$ is given by its Segal croquis $(\mathbf{\Delta}^{\mathrm{op}},K)$ to which is added the unique cone 
whose peak is the ordinal $\mathbf{1}$ and whose diagram in $\mathbf{\Delta}$ is given, below, underlying the cocone of dotted arrows, where, if one denotes $\mathbf{2}:=\{0,1\}$ and $\mathbf{4}:=\{0,1,2,3\}$,
\begin{itemize}
\item[(1)] $l:\mathbf{2} \to \mathbf{4}$ is the function with the mapping rules $0 \mapsto 0$ and $1 \mapsto 2$;
\item[(2)] $r:\mathbf{2} \to \mathbf{4}$ is the function with the mapping rules $0 \mapsto 1$ and $1 \mapsto 3$;
\end{itemize}
\[
\xymatrix@R-2pc{
&&\mathbf{1}&&\\
&&&&\\
&&&&\\
\mathbf{1}\ar@{..>}[rruuu]^{!}&&\mathbf{4}\ar@{..>}[uuu]&&\mathbf{1}\ar@{..>}[lluuu]_{!}\\
&\mathbf{2}\ar[lu]^{!}\ar[ru]_{l}&&\mathbf{2}\ar[lu]^{r}\ar[ru]_{!}&\\
}
\]

The induced cone in $\mathbf{\Delta}^{\mathrm{op}}$ will be denoted by $c_{\mathrm{iso}}$ as it is meant to describe the set of isomorphism structures relative to the natural categorical (or nerval) structure of $\mathbf{\Delta}^{\mathrm{op}}$. The resulting croquis will be denoted by $\mathbf{Cseg}(\mathbf{\Delta}^{\mathrm{op}})$.
\end{example}

We shall speak of an \emph{elementary shape} of a croquis $(D,K,\mathtt{rou})$ to refer to the elementary shape of one of its cones.
Because $K$ is a small category, the class of elementary shapes of $(D,K,\mathtt{rou})$ is a set, which will be denoted by 
$\mathrm{Es}(K)$. The \emph{cardinality of a croquis $(D,K,\mathtt{rou})$} is then given by the cardinal of the coproduct of every small category in $\mathrm{Es}(K)$.
\[
|(K,\mathtt{rou})| := |\sum_{A \in \mathrm{Es}(K)}A|
\]

\subsection{Premodels}\label{ssec:premodels}
Let $(D,K,\mathtt{rou})$ be a croquis and $\mathcal{C}$ be a category. For any endofunctor $R:\mathcal{C} \to \mathcal{C}$, denote by $\mathbf{Pr}_{\mathcal{C}}(K,\mathtt{rou},R)$ the category whose objects are triples $(P,S,e)$ where (1) $P$ is a functor $D \to \mathcal{C}$, (2) $S$ is a functor \footnote{To not say a `function valued in a category'. Such a simplification will be common later on} $K \to \mathbf{Set}$ and (3) $e$ denotes a collection of arrows $e_{c,s}:P\mathtt{rou}(c) \to RP\mathtt{ou}(c)$ in $\mathcal{C}$ for every $c \in K$ and $s \in S(c)$ and whose morphisms, say of the form $(P,S,e) \Rightarrow (P',S',e')$, are pairs $(f,a)$ where $f$ and $a$ are two natural transformations of respective forms $P \Rightarrow P'$  and $S \Rightarrow S'$ making the following diagram commute for every $c \in K$ and $s \in S(c)$.
\[
\xymatrix{
*+!R(.7){P\mathtt{rou}(c)}\ar[r]^{f\mathtt{rou}(c)}\ar[d]_{e_{c,s}}&*+!L(.7){P'\mathtt{rou}(c)}\ar[d]^{e'_{c,a_c(s)}}\\
*+!R(.7){RP\mathtt{ou}(c)} \ar[r]_-{Rf \mathtt{ou}(c)}&*+!L(.7){RP'\mathtt{ou}(c)}
}
\]

The objects of $\mathbf{Pr}_{\mathcal{C}}(K,\mathtt{rou},R)$ will be called the \emph{$R$-premodels for $(K,\mathtt{rou})$}. For convenience, the category $\mathbf{Pr}_{\mathcal{C}}(K,\mathtt{rou},R)$ will sometimes be denoted as $\mathbf{Pr}_{\mathcal{C}}(K,R)$ when $\mathtt{rou}$ is trivial and as $\mathbf{Pr}_{\mathcal{C}}(K)$ when $R$ is also an identity.

\begin{example}[Premodels]\label{ex:functors_premodels}
The category of premodels for a sketch $(D,K)$ to a category $\mathcal{C}$ corresponds to the full subcategory of
$\mathbf{Pr}_{\mathcal{C}}(K)$ whose objects $(P,S,e)$ are such that the images of $S$ are equal to $\mathbf{1}$ and the morphism $e_{c}:P\mathtt{ou}(c) \to P\mathtt{ou}(c)$ is an identity for every $c \in Q$. This subcategory is isomorphic to $\mathcal{C}^{D}$.
\end{example}

\begin{example}[Presheaves]\label{ex:presheaves_premodels}
The category of presheaves over a site $(D^{\mathrm{op}},J)$ corresponds to the full subcategory of
$\mathbf{Pr}_{\mathbf{Set}}(K_J)$ whose objects $(P,S,e)$ are such that the images of $S$ are equal to $\mathbf{1}$ and the morphism $e_{c}:P\mathtt{ou}(c) \to P\mathtt{ou}(c)$ is an identity for every $c \in K_J$.  This subcategory is isomorphic to $\mathbf{Set}^D$.
\end{example}

\begin{example}[Prespectra]\label{ex:spectra_premodels}
If $\Omega:\mathbf{pTop} \to \mathbf{pTop}$ denotes the loop space functor on the category of pointed topological spaces and $\mathtt{pred}$ denotes the predecessor operation $n \mapsto n-1$ on $\mathbb{N}^*$, then the category of prespectra is the full subcategory of 
$\mathbf{Pr}_{\mathbf{Top}}(\mathbf{Cr}(\mathbb{N},\mathtt{pred}),\Omega)$ whose objects $(P,S,e)$ are such that the images of $S$ are equal to $\mathbf{1}$. This subcategory will be denoted by $\mathbf{PrSpc}$.
\end{example}

\begin{example}[Pre-localised rings]\label{ex:Prelicalised_ring}
Let $\mathbf{Set}$ denote the category of sets and $\mathtt{Rg}$ be the limit sketch defined in Example \ref{exa:rings_sketch_Rg}.
The category of `pre-localised rings' is defined as the full subcategory of the category $\mathbf{Pr}_{\mathbf{Set}}(\mathtt{Rg})$ whose objects $(P,S,e)$ are such that (1) $P:\mathtt{Rg} \to \mathbf{Set}$ is a model for $\mathtt{Rg}$; (2) the image of $S:K \to \mathbf{Set}$ above the cone $\mathrm{1}_{g_1}:g_1 \to g_1$ is equal to a subset of $P(g_1)$ while its images above all the other cones are equal to $\mathbf{1}$ and (3) the morphism $e_{c,s}:P(g_1) \to P(g_1)$ is given by 
\begin{itemize}
\item[-] the right multiplication map $x \mapsto P(\mu')(x,s)$ for every $s \in S(c)$ if $c = \mathrm{1}_{g_1}$;
\item[-] the identity morphism $e_{c}:P\mathtt{ou}(c) \to P\mathtt{ou}(c)$ otherwise.
\end{itemize}

This subcategory will be denoted by $\mathbf{PrLocRg}$.
\end{example}

\begin{example}[Pre-Segal spaces]\label{ex:PreSegal_spaces}
Let $\mathbf{Top}$ denote the category of topological spaces and continuous functions. The category of \emph{pre-Segal spaces} is the category of simplical topological spaces; it is given as the full subcategory of $\mathbf{Pr}_{\mathbf{Top}}(\mathbf{Seg}(\mathbf{\Delta}^{\mathrm{op}}))$ whose objects $(P,S,e)$ are such that the images of the functor $S$ are equal to $\mathbf{1}$ and the morphism $e_{c}:P\mathtt{ou}(c) \Rightarrow P\mathtt{ou}(c)$ is an identity for every $c \in \mathbf{Seg}(\mathbf{\Delta}^{\mathrm{op}})$. Thecategory of \emph{pre-complete Segal spaces} is defined similarly by replacing $\mathbf{Seg}(\mathbf{\Delta}^{\mathrm{op}})$ with $\mathbf{Cseg}(\mathbf{\Delta}^{\mathrm{op}})$.
\end{example}

\begin{definition}
Let $D$ be a small category and $\mathcal{C}$ be a category. For any given endofunctor $R:\mathcal{C} \to \mathcal{C}$, a \emph{category of $R$-premodels} is a subcategory of the category $\mathbf{Pr}_{\mathcal{C}}(K,\mathtt{rou},R)$.
\end{definition}

\begin{example}
Premodels for a sketch, presheaves on a site, prespectra, pre-localised rings and pre-Segal spaces are examples of such categories (see the previous examples).
\end{example}

\subsection{Models}\label{ssec:Models}
Let $D$ be a small category, $(K,\mathtt{rou})$ be a croquis in $D$ and $\mathcal{C}$ be a complete category over the elementary shapes of $K$. Suppose to be given a right adjoint $R:\mathcal{C} \to \mathcal{C}$. The first goal of this section is to define a functor $\mathcal{G}^K_c:\mathbf{Pr}_{\mathcal{C}}(K,\mathtt{rou},R) \to \mathbf{Fam}(\mathcal{C}^{\mathbf{2}})$ for every cone $c \in K$.
In this respect, for every cone $c$ in $K$ of the form $t:\Delta_A(d_0) \Rightarrow d_1$, for which we shorten the notation $\mathtt{rou}(c)$ to the symbol $r$, the functor $\mathcal{G}^K_c$ maps any premodel $(P,S,e)$ to the family taking any $s \in S(c)$ to the following composite arrow in $\mathcal{C}$.
\[
\xymatrix{
P(r) \ar[r]^-{\eta_{P(r)}}&\mathrm{lim}_{A}\Delta_AP(r) \ar[rr]^-{\mathrm{lim}_{A}\Delta_Ae_{c,s}}&&\mathrm{lim}_{A}\Delta_ARPd_0 \ar[rr]^-{\mathrm{lim}_{A}RPt}&& \mathrm{lim}_{A}RPd_1
}
\]

For every morphism of $R$-premodel of the form $(f,a):(P,S,e) \Rightarrow (P',S',e')$, the image morphism $\mathcal{G}^K_c(f,a):\mathcal{G}^K_c(P,S,e) \Rightarrow \mathcal{G}^K_c(P',S',e')$ is given, for every $s \in S(c)$, by the following morphism in $\mathcal{C}^{\mathbf{2}}$.
\[
\xymatrix{
P(r) \ar[d]_{f(r)}\ar[rr]^-{\eta_{P(r)}}&&\mathrm{lim}_{A}\Delta_AP(r)\ar[d]|{\mathrm{lim}_{A}\Delta_APf(r)} \ar[rr]^-{\mathrm{lim}_{A}\Delta_Ae_{c,s}}&&\mathrm{lim}_{A}\Delta_ARPd_0 \ar[rr]^-{\mathrm{lim}_{A}RPt} \ar[d]|{\mathrm{lim}_{A}f d_1} && \mathrm{lim}_{A}RPd_1 \ar[d]|{\mathrm{lim}_{A}R f d_1}\\
P'(r) \ar[rr]^-{\eta_{P'(r)}}&&\mathrm{lim}_{A}\Delta_AP'(r) \ar[rr]^-{\mathrm{lim}_{A}\Delta_Ae_{c,a_c(s)}'}&&\mathrm{lim}_{A}\Delta_ARP'd_0 \ar[rr]^-{\mathrm{lim}_{A}RP't}&& \mathrm{lim}_{A}RP'd_1
}
\]

\begin{definition}[System of premodels]\label{def:System:premodels}
A \emph{system of $R$-premodels} consists of (1) a croquis $(D,K,\mathtt{rou})$; \linebreak (2) a category $\mathcal{C}$ that is complete on the elementary shapes of $K$ and admits a terminal object; \linebreak (3) a category of $R$-premodels $\mathcal{P} \hookrightarrow \mathbf{Pr}_{\mathcal{C}}(K,\mathtt{rou},R)$ where $R$ is a right adjoint and (4), for every cone $c \in K$, a set $\mathtt{V}_c$ of commutative squares in $\mathcal{C}$, called the \emph{diskads} (see left diagram, below) equipped with a pushout in $\mathcal{C}$ (see right diagram, below). 
\begin{equation}\label{eq:definition:System_of_premodels}
\xymatrix{
\mathbb{S} \ar[r]^{\gamma_1}\ar[d]_{\gamma_2} &\mathbb{D}_1 \ar[d]^{\beta_1}&\\
\mathbb{D}_2 \ar[r]_{\beta_2} &\mathbb{D}'\\
}\quad\quad\quad\quad
\xymatrix{
\mathbb{S}\ar@{}[rd]|>>{\text{\huge{\rotatebox[origin=c]{-90}{$\llcorner$}}}} \ar[r]^{\gamma_1}\ar[d]_{\gamma_2} &\mathbb{D}_1 \ar[d]^{\delta_1}\\
\mathbb{D}_2 \ar[r]_{\delta_2} &\mathbb{S}'\\
}
\end{equation}
\end{definition}
The collection consisting of all the sets $\mathtt{V}_c$ will usually be denoted by $\mathtt{V}$. A system of $R$-premodels will be denoted as a 4-tuple $(K,\mathtt{rou},\mathcal{P},\mathtt{V})$ and said to be \emph{defined over $D$ in $\mathcal{C}$}. The diagrams used in Definition \ref{def:System:premodels} can more efficiently be described as a colimit sketch in $\mathcal{C}$ (i.e. diagram equipped with colimits) of the following form.
\[
\xymatrix{
\mathbb{S}\ar@{}[rd]|>>{\text{\huge{\rotatebox[origin=c]{-90}{$\llcorner$}}}} \ar[r]^{\gamma_1}\ar[d]_{\gamma_2} &\mathbb{D}_1 \ar[d]^{\delta_1}&\\
\mathbb{D}_2 \ar[r]_{\delta_2} &\mathbb{S}' \ar[r]^{\beta} &  \mathbb{D}'\\
}
\]

This type of colimit sketch will be called a \emph{vertebra} and denoted by the symbols $\|\gamma_2,\gamma_1\|\cdot \beta$. For such a vertebra, it will come in handy to refer to the arrows $\gamma_2$, $\gamma_1$, $\beta$ and $\beta \circ \delta_1$ as the \emph{seed}, \emph{coseed}, \emph{stem} and \emph{trivial stem}, respectively.
Finally, the left adjoint of $R:\mathcal{C} \to \mathcal{C}$ will conventionally be denoted by $L$.

\begin{definition}[Model]\label{def:models}
An $R$-premodel $(P,S,e)$ in a system of $R$-premodels $(K,\mathtt{rou},\mathcal{P},\mathtt{V})$ will be said to be an \emph{$R$-model} if, for every cone $c \in K$, every component of the arrow $\mathcal{G}^K_c(P,S,e) \Rightarrow \mathbf{1}$ in $\mathbf{Fam}(\mathcal{C}^{\mathbf{2}})$ has the right lifting property with respect to all the diskads of $\mathtt{V}_c$ when these are seen as arrows $\gamma_1 \Rightarrow \beta_1$ in $\mathcal{C}^{\mathbf{2}}$ with respect to the notations of Equation (\ref{eq:definition:System_of_premodels}).
\[
\xymatrix@C-1.25pc{
&*+!R(.3){P(r)}\ar[rr]^-{\mathcal{G}^K_c(P,S,e)_s}&&*+!L(.7){\mathrm{lim}_{d_1}RP}\\
\mathbb{S}\ar[ur]^x\ar[d]_{\gamma_2}\ar[rr]^{\gamma_1}&&\mathbb{D}_1\ar[ru]^y\ar[d]|{\beta\delta_1}&\\
\mathbb{D}_2\ar@{-->}[ruu]\ar[rr]|{\beta\delta_2}&&\mathbb{D}'\ar@{-->}[ruu]&
}
\]
\end{definition}

\begin{example}[Models for a sketch]\label{ex:Models_for_a_sketch_system}
For every limit sketch $(D,K)$, define the system of premodels consiting of the croquis $K$ (see Example \ref{exa:Croquis_for_a_sketch}); the associated category of premodels $\mathbf{Set}^{D} \hookrightarrow \mathbf{Pr}_{\mathbf{Set}}(K)$ and, for every cone $c$ in $K$, the set made of the following vertebrae in $\mathbf{Set}$.
\begin{equation}\label{eq:vertebrae:set->bijection}
\xymatrix{
\emptyset\ar@{}[rd]|>>>{\text{\huge{\rotatebox[origin=c]{-90}{$\llcorner$}}}} \ar[r]^{!}\ar[d]_{!} &\mathbf{1} \ar[d]^-{\delta_1}&\\
\mathbf{1} \ar[r]_-{\delta_2} &*+!L(.7){\mathbf{1}+\mathbf{1}} \ar[r]^-{!} &  \mathbf{1},\\
}
\quad\quad\quad\quad
\xymatrix{
*+!R(.7){\mathbf{1}+\mathbf{1}}\ar@{}[rd]|>>>{\text{\huge{\rotatebox[origin=c]{-90}{$\llcorner$}}}} \ar[r]^{!}\ar[d]_{!} &\mathbf{1} \ar@{=}[d]&\\
\mathbf{1} \ar@{=}[r] &\mathbf{1} \ar@{=}[r] &  \mathbf{1}\\
}
\end{equation}

The $\mathrm{id}_{\mathbf{Set}}$-models of such a system correspond to the models for the sketch $(D,K)$.
\end{example}

\begin{example}[Sheaves]\label{exa:system_of_models_sheaves}
For every site $(D^{\mathrm{op}},J)$, define the system of premodels consiting of the croquis $K_{J}$ (see Example \ref{ex:presheaves_premodels}); the associated category of premodels $\mathbf{Set}^{D} \hookrightarrow \mathbf{Pr}_{\mathbf{Set}}(K_{J})$ and, for every cone $c$ in $K_{J}$, the set made of the vertebrae given in Equation  (\ref{eq:vertebrae:set->bijection}).
The $\mathrm{id}_{\mathbf{Set}}$-models of such a system correspond to the sheaves over $(D^{\mathrm{op}},J)$.
\end{example}

\begin{example}[Flabby sheaves]\label{exa:system_of_models_sheaves_flabby}
For every site $(D^{\mathrm{op}},J)$, define the system of premodels consiting of the croquis $K_J\cup \mathrm{Mor}(D)$ defined in Example \ref{ex:Croquis_for_flabby_sheaves}; the functor category $\mathbf{Set}^{D} \hookrightarrow \mathbf{Pr}_{\mathbf{Set}}(K_J\cup \mathrm{Mor}(D))$ and 
\begin{itemize}
\item[(i)] for every cone $c$ in $K_J$, the set made of the vertebrae given in Equation (\ref{eq:vertebrae:set->bijection});
\item[(ii)] for every cone $c$ in $\mathrm{Mor}(D)$, the set made of the leftmost vertebra of Equation (\ref{eq:vertebrae:set->bijection}) only.
\end{itemize}

The $\mathrm{id}_{\mathbf{Set}}$-models $F:D \to \mathbf{Set}$ of such a system correspond to the sheaves over $(D^{\mathrm{op}},J)$ whose morphisms $F(U) \to F(V)$ over any arrow $U \to V$ in $D$ are surjective, namely the flabby sheaves over $(D^{\mathrm{op}},J)$.
\end{example}

\begin{example}[Sheaves in categories]\label{exa:system_of_models_sheaves_in Cat}
For every site $(D^{\mathrm{op}},J)$, define the system of premodels consiting of the croquis $K_{J}$ (see Example \ref{ex:presheaves_premodels}); the associated category of premodels $\mathbf{Cat}^{D} \hookrightarrow \mathbf{Pr}_{\mathbf{Cat}}(K_{J})$ and, for every cone $c$ in $K_{J}$, the set made of the following vertebrae for the obvious choices of morphisms, where
\begin{itemize}
\item[(1)] $\mathbf{1} $ is a terminal category;
\item[(2)] $\mathbf{iso}$ is the free living isomorphism category (i.e., two objects, one isomorphism);
\item[(3)] $\mathbf{2}$  is the free living arrow category (i.e., two objects, one arrow);
\item[(4)] $\mathbf{2}\oplus\mathbf{2}$  is category made of two objects and two parallel arrows between them.
\end{itemize}
\begin{equation}\label{eq:vertebrae:cat->equivalences}
\xymatrix{
\emptyset\ar@{}[rd]|>>>{\text{\huge{\rotatebox[origin=c]{-90}{$\llcorner$}}}} \ar[r]^{!}\ar[d]_{!} &\mathbf{1} \ar[d]&\\
\mathbf{1} \ar[r] &*+!L(.7){\mathbf{1}+\mathbf{1}} \ar[r]^-{\subset} &  \mathbf{iso}\\
}
\quad
\xymatrix{
*+!R(.7){\mathbf{1}+\mathbf{1}}\ar@{}[rd]|>>>{\text{\huge{\rotatebox[origin=c]{-90}{$\llcorner$}}}} \ar[r]^-{\subset}\ar[d]_-{\subset} &\mathbf{2} \ar[d]&\\
\mathbf{2} \ar[r] &\mathbf{2}\oplus\mathbf{2} \ar@{->>}[r] &  \mathbf{2}\\
}
\quad
\xymatrix{
*+!R(.7){\mathbf{2}\oplus\mathbf{2}}\ar@{}[rd]|>>>{\text{\huge{\rotatebox[origin=c]{-90}{$\llcorner$}}}} \ar@{->>}[r]\ar@{->>}[d] &\mathbf{2} \ar@{=}[d]&\\
\mathbf{2} \ar@{=}[r] &\mathbf{2}\ar@{=}[r] &  \mathbf{2}\\
}
\end{equation}

The $\mathrm{id}_{\mathbf{Cat}}$-models of such a system correspond to those `sheaves' $D \to \mathbf{Cat}$ for which the sheaf condition is not a bijection but an equivalence of categories.
\end{example}

\begin{example}[Strong stacks]\label{exa:system_of_models_Strong_stacks}
For every site $(D^{\mathrm{op}},J)$, define the system of premodels consiting of the croquis $K_J\cup \mathrm{Mor}(D)$ defined in Example \ref{ex:Croquis_for_flabby_sheaves}; the functor category $\mathbf{Cat}^{D} \hookrightarrow \mathbf{Pr}_{\mathbf{Cat}}(K_J\cup \mathrm{Mor}(D))$ and 
\begin{itemize}
\item[(i)] for every cone $c$ in $K_J$, the set made of the leftmost vertebra of Equation (\ref{eq:vertebrae:set->bijection}) when seen in $\mathbf{Cat}$ (instead of $\mathbf{Set}$) and the rightmost two vertebrae of Equation  (\ref{eq:vertebrae:cat->equivalences});
\item[(ii)] for every cone $c$ in $\mathrm{Mor}(D)$, the set made of the leftmost vertebra of Equation (\ref{eq:vertebrae:cat->equivalences}) only.
\end{itemize}

The $\mathrm{id}_{\mathbf{Cat}}$-models of such a system correspond to the strong stack (see \cite{Joyal_Tierney}). The strong stacks completion constructed in \emph{ibid} corresponds to a special case of the general construction given in this paper.
\end{example}

\begin{example}[Strong stacks up to homotopy]\label{exa:system_of_models_Strong_stacks_bis}
For every site $(D^{\mathrm{op}},J)$, define the system of premodels consiting of the croquis $K_J\cup \mathrm{Mor}(D)$ defined in Example \ref{ex:Croquis_for_flabby_sheaves}; the functor category $\mathbf{Cat}^{D} \hookrightarrow \mathbf{Pr}_{\mathbf{Cat}}(K_J\cup \mathrm{Mor}(D))$ and 
\begin{itemize}
\item[(i)] for every cone $c$ in $K_J$, the set made of the vertebrae given in Equation (\ref{eq:vertebrae:cat->equivalences});
\item[(ii)] for every cone $c$ in $\mathrm{Mor}(D)$, the set made of the leftmost vertebra of Equation (\ref{eq:vertebrae:cat->equivalences}) only.
\end{itemize}

The $\mathrm{id}_{\mathbf{Cat}}$-models of such a system may be identified to the strong stacks of \cite{Joyal_Tierney} up to the notion of homotopy defined thereof.
\end{example}

\begin{example}[Segal spaces]\label{ex:Segal_spaces}
Define the system of premodels consisting of the croquis $\mathbf{Seg}(\mathbf{\Delta}^{\mathrm{op}})$ defined in Example \ref{ex:Croquis_reedy};  the category of pre-Segal spaces $\mathbf{Top}^{\mathbf{\Delta}^{\mathrm{op}}}$, which is included in $\mathbf{Pr}_{\mathbf{Top}}(\mathbf{Seg}(\mathbf{\Delta}^{\mathrm{op}}))$ and
\begin{itemize}
\item[(i)] for every cone $c$ in 
$K_0 \subseteq \mathbf{Seg}(\mathbf{\Delta}^{\mathrm{op}})$, the set of obvious vertebrae induced by the diskads given in Equation (\ref{eq:diskad:top->homotopical_equivalence}), where
\begin{itemize}
\item[-] $n$ runs over the natural numbers;
\item[-] the object $\mathbb{D}^{n}$ is the topological $n$-disc;
\item[-] the map $\iota_n:\mathbb{D}^{n} \to \mathbb{D}^{n+1}$ is the obvious hemisphere inclusion;
\end{itemize}
\begin{equation}\label{eq:diskad:top->homotopical_equivalence}
\xymatrix{
\mathbb{D}^{n}\ar[r]^-{\iota_n}\ar[d]_-{\iota_n} &*+!L(0.7){\mathbb{D}^{n+1}}\ar@{=}[d]\\
\mathbb{D}^{n+1}\ar@{=}[r]&*+!L(0.7){\mathbb{D}^{n+1}}
}
\end{equation}
\item[(ii)]  for every cone $c$ in 
$K_1 \subseteq \mathbf{Seg}(\mathbf{\Delta}^{\mathrm{op}})$, the set of Vertebrae (\ref{eq:vertebra:top->homotopical_equivalence}), where $n$ runs over the positive integers and
\begin{itemize}
\item[-] the object $\mathbb{S}^{n-1}$ is the topological $(n-1)$-sphere;
\item[-] the maps between the different objects are induced by the obvious inclusions;
\end{itemize}
\begin{equation}\label{eq:vertebra:top->homotopical_equivalence}
\xymatrix{
\mathbb{S}^{n-1}\ar[r]^-{\gamma_n}\ar[d]_-{\gamma_n}\ar@{}[rd]|>>>{\text{\huge{\rotatebox[origin=c]{-90}{$\llcorner$}}}} &\mathbb{D}^n\ar[d]^{\delta_1^n}&\\
\mathbb{D}^n\ar[r]_-{\delta_2^n}&\mathbb{S}^{n}\ar[r]_-{\gamma_{n+1}}&\mathbb{D}^{n+1}
}
\end{equation}
\end{itemize}

The $\mathrm{id}_{\mathbf{Top}}$-models of such a system correspond to the Segal spaces in $\mathbf{Top}$ (see \cite{Rezk} for a definition enriched in simplicial sets).
\end{example}

\begin{example}[Complete Segal spaces]\label{ex:Complete_Segal_spaces}
Define the system of premodels consisting of the croquis $\mathbf{Cseg}(\mathbf{\Delta}^{\mathrm{op}})$ defined in Example \ref{ex:Croquis_Complete_reedy};  the category of pre-complete Segal spaces $\mathbf{Top}^{\mathbf{\Delta}^{\mathrm{op}}}$, which is included in $\mathbf{Pr}_{\mathbf{Top}}(\mathbf{Cseg}(\mathbf{\Delta}^{\mathrm{op}}))$, and
\begin{itemize}
\item[(1)] for every cone $c$ in 
$\mathbf{Cseg}(\mathbf{\Delta}^{\mathrm{op}})$ that is in fact in $\mathbf{Seg}(\mathbf{\Delta}^{\mathrm{op}})$, the same set of vertebrae defined in Example \ref{ex:Segal_spaces};
\item[(2)] for the cone $c_{\mathrm{iso}}$ (see Example \ref{ex:Croquis_Complete_reedy}), the set of vertebrae of the form (\ref{eq:vertebra:top->homotopical_equivalence}) for every positive integer $n$.
\end{itemize}

The $\mathrm{id}_{\mathbf{Top}}$-models of such a system correspond to the complete Segal spaces in $\mathbf{Top}$ (see \cite{Rezk} for a definition enriched in simplicial sets).
\end{example}

\begin{example}[Spectra]\label{ex:system_of_premodels_spectra}
For the loop space functor $\Omega:\mathbf{pTop} \to \mathbf{pTop}$, define the system of $\Omega$-premodels consiting of the croquis $\mathbf{Cr}(\mathbb{N},\mathtt{pred})$ defined in Example \ref{exa:Croquis_for_spectra}; the category of prespectra $\mathbf{PrSpc} \hookrightarrow \mathbf{Pr}_{\mathbf{pTop}}(\mathbf{Cr}(\mathbb{N},\mathtt{pred}),\Omega)$ and, for every cone $c$ in $\mathbf{Cr}(\mathbb{N},\mathtt{pred})$, the set of vertebrae of pointed spaces defined in Diagram (\ref{eq:vertebra:pointed_sets->homotopical_equivalence}), where $n$ is a positive number and
\begin{itemize}
\item[-] where the object $\mathbb{S}^{n-1}/\partial$ is the quotient of the $(n-1)$-sphere by itself (i.e., a point);
\item[-] where the object $\mathbb{D}^{n}/\partial$ is the quotient of the $n$-disc by its boundary; 
\item[-] where the object $\mathbb{S}^{n}/p$ is the quotient of the $n$-sphere by its equator;
\item[-] where the object $\mathbb{D}^{n+1}/p$ is the quotient of the $(n+1)$-disc by its equator;
\item[-] where the object $\mathbb{D}^{n+1}/h$ is the quotient of the $(n+1)$-disc by one of its hemispheres;
\item[-] where the object $\mathbb{D}^{n+1}/\partial$ is the quotient of the $(n+1)$-disc by its boundary;
\item[-] where the maps between the different objects are the obvious inclusions.
\end{itemize}
\begin{equation}\label{eq:vertebra:pointed_sets->homotopical_equivalence}
\begin{array}{l}
\xymatrix@C-1.5pc{
*+!R(.7){\mathbb{S}^{n-1}/\partial}\ar[r]^-{\gamma_n}\ar[d]_-{\gamma_n}\ar@{}[rd]|>>>{\text{\huge{\rotatebox[origin=c]{-90}{$\llcorner$}}}} &*+!L(.7){\mathbb{D}^n/\partial}\ar[d]^{\delta_1^n}&\\
*+!R(.7){\mathbb{D}^n/\partial}\ar[r]_-{\delta_2^n}&*+!L(.7){\mathbb{S}^{n}/p}\ar[r]_-{\beta_{n}}&*+!L(1){\mathbb{D}^{n+1}/p}
}\\
$\quad\quad$\includegraphics[width=4cm]{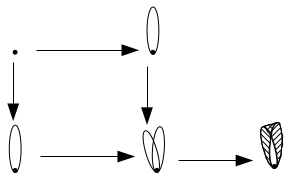}
\end{array}
\quad\quad
\begin{array}{l}
\xymatrix@C-1.5pc{
*+!R(.7){\mathbb{D}^n/\partial}\ar[r]^-{\gamma_n'}\ar[d]_-{\gamma_n'}\ar@{}[rd]|>>>{\text{\huge{\rotatebox[origin=c]{-90}{$\llcorner$}}}} &*+!L(.7){\mathbb{D}^{n+1}/h}\ar[d]^{\delta_1^n}&\\
*+!R(.7){\mathbb{D}^{n+1}/h}\ar[r]_-{\delta_2^n}&*+!L(.7){\mathbb{D}^{n+1}/\partial}\ar@{=}[r]&*+!L(1){\mathbb{D}^{n+1}/\partial}
}
\\
$\quad\quad$\includegraphics[width=4cm]{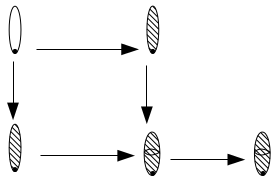}
\end{array}
\end{equation}

The $\Omega$-models of such a system correspond to the $\Omega$-spectra.
\end{example}

\begin{example}[Localisation of rings]\label{ex:system_of_premodels_Localisation_of_rings}
Consider the system of premodels consisting of the croquis $\mathtt{Rg}$ (see Example \ref{ex:Prelicalised_ring}), the subcategory $\mathbf{PrLocRg} \hookrightarrow \mathbf{Pr}_{\mathbf{Set}}(\mathtt{Rg})$ and, for the cone $c$ in $\mathtt{Rg}$, the set made of the vertebrae given in Equation (\ref{eq:vertebrae:set->bijection}).
The $\mathrm{id}_{\mathbf{Rg}}$-models $(P,S,e)$ of such a system correspond to the rings $P$ for which the map $x \mapsto P(\mu')(x,s)$ is invertible for every $s \in S(1_{g_1})$, or in other words those rings that are localised 
at their associated subset of elements $S$. Fields are particular examples.
\end{example}

\begin{remark}\label{rem:Homotopy_functors}
Many other examples could have been provided. Recall that it is common fact (see \cite[Lemma 7.5.1]{Simpson},  \cite{DugIsak} or \cite[Proposition 8]{LafMetWor}) that, in some nice model category $\mathcal{C}$, the notion of weak equivalence may be characterised via the type of right lifting property expressed in Example \ref{def:models}.
For instance, Examples \ref{ex:Segal_spaces} and  \ref{ex:Complete_Segal_spaces} on Segal spaces could have been extended to any nice cofibrantly generated model category, which need not be simplicial (contrary to usual practice). In fact, it is worth noting that the type of localisation described in the present article is an alternative to the usual simplicial Bousfield Localisation process (see \cite{Hirschhorn}). On could also look at the type of localisation discussed in \cite[Corollary 8.8]{B_Toen}, which could be comprised in a more technical generalisation of the present work. Future work will also aim at generalising Example \ref{exa:system_of_models_sheaves_in Cat} to weaker functors in order to charactise the notions of $(\infty,n)$-stack and strong $(\infty,n)$-stack.
\end{remark}

\section{Narratives and the Small Object Argument} \label{sec:From_narratives_to_combinatorial_cat}s

This section aims to introduce the small object argument that will be used for the construction of the localisation. The difference from that given below and the one defined by Quillen \cite{Quillen67} is the notion of `degree' coming along with the concept of narrative (see the table below). The degree is the key ingredient that allows us to obtain our so-called elimination of quotients.

\begin{center}
\begin{tabular}{p{1.7cm}|p{8.8cm}l}
\cellcolor[gray]{0.8}Notions&\multicolumn{2}{c}{\cellcolor[gray]{0.8} Descriptions}\\
\hline
\cellcolor[gray]{0.9}\multirow{4}{*}{Tome} & A collection of commutative squares whose rightmost vertical arrows are all equal: this can be visualised as a `book' whose pages are glued along a spine. The pages can satisfy certain compatibility relations. &
{$
\xymatrix@C-16pt@R-16pt{
\ar[dd]&&\ar@{<-}[rr]\ar@{<-}[ll]\ar[dd]\ar@{<-}[dr]\ar@{<-}[dl]&&\ar[dd]\\
&\ar[dd]&&\ar[dd]&\\
&&\ar@{<-}[rr]|\hole\ar@{<-}[ll]|\hole\ar@{<-}[dr]\ar@{<-}[dl]&&\\
&&&&\\
}
$}\\
\hline
\cellcolor[gray]{0.9}Morphisms & \multicolumn{2}{p{10.5cm}}{\emph{Regular}: relates the spine and the pages of two `books' together.} \\
\cellcolor[gray]{0.9}of tomes& \multicolumn{2}{p{10.5cm}}{\emph{Loose}: only relates the spines.} \\
\hline
\cellcolor[gray]{0.9}\multirow{2}{*}{Oeuvre} & \multicolumn{2}{p{11.3cm}}{An ordered collection of tomes related via loose morphisms; the \emph{theme} is the common object towards which the spines of the books go to.} \\
\hline
\cellcolor[gray]{0.9}\multirow{2}{*}{Narrative} \multirow{2}{*}{of degree $\delta$} & \multicolumn{2}{p{11.3cm}}{An oeuvre that is equipped with sub-diagrams of its tomes, called the \emph{events}, and choices of lifts for these sub-diagrams, called the \emph{viewpoints} These lifts only `commute' from the $k$-th book to the $(k+\delta)$-th book.}\\
\end{tabular}
\end{center}

\subsection{Numbered Categories and Compatibility}\label{sssec:Numbered_categories_and_compatibility}
In the sequel, the term \emph{numbered category} will denominate any pair $(\mathcal{C},\kappa)$
where $\mathcal{C}$ is a category and $\kappa$ is a limit ordinal. 
A small category $\mathtt{T}$ will be said to be \emph{compatible} with $(\mathcal{C},\kappa)$
if (1) the category $\mathcal{C}$ admits colimits over $\mathtt{T}$ and (2) the inequality $|\mathtt{T}| \leq \kappa$ holds.
By extension, a functor $i:\mathtt{T} \to \mathtt{A}$ will be said to be \emph{compatible}
with a numbered category $(\mathcal{C},\kappa)$ if its domain $\mathtt{T}$ is compatible with $(\mathcal{C},\kappa)$.

\subsection{Lifting Systems}
Let us now define in formal terms what will later be seen as a set of generating cofibrations for our small object argument.
Let $(\mathcal{C},\kappa)$ be an numbered category. A \emph{lifting system in $(\mathcal{C},\kappa)$} is a set $J$ of objects of $\mathbf{Cat}/\mathcal{C}^{\mathbf{2}}$ that are compatible with $(\mathcal{C},\kappa)$ as functors.

\subsection{Right Lifting Property}
Let $(\mathcal{C},\kappa)$ be an numbered category and $J$ be a lifting system in $(\mathcal{C},\kappa)$. For every functor $\upvarphi:\mathtt{T} \to \mathcal{C}^{\mathbf{2}}$ in $J$, the image of an object $s$ in $\mathtt{T}$ via $\upvarphi$ will usually be denoted by $\upvarphi(s):A(s) \to B(s)$. A morphism $f:X \to Y$ in $\mathcal{C}$ will be said to have the \emph{right lifting property with respect to the system $J$} if for any functor $i:\mathtt{T} \to \mathtt{S}$ in $J$, the morphism $f:X \to Y$ has the rlp with respect to the arrow $\mathrm{col}_{\mathtt{T}}\upvarphi:\mathrm{col}_{\mathtt{T}}A\to \mathrm{col}_{\mathtt{T}}B$ in $\mathcal{C}$. In the sequel, the class of morphisms of $\mathcal{C}$ that have the right lifting property with respect to
a lifting system $J$ will be denoted by $\mathbf{rlp}(J)$.

\begin{example}
If $J$ is a set of functors of the form $\mathbf{1} \to \mathcal{C}^{\mathbf{2}}$ picking out some objects of $\mathcal{C}^{\mathbf{2}}$, then the preceding right lifting property corresponds to the usual one.
\end{example}

\subsection{Tomes}\label{sssec:Tomes}
Let $\mathcal{C}$ be a category. A \emph{tome} in $\mathcal{C}$ is a triple consisting of a morhism $h:X \to Y$ in $\mathcal{C}$, a small category $\mathtt{S}$ on which $\mathcal{C}$ admits all colimits and a functor $\varphi:\mathtt{S} \to \mathcal{C}^{\mathbf{2}}/h$.
According to Remark \ref{rem:functor_in_overcats_give_nat_transf} applied to the arrow category $\mathcal{C}^{\mathbf{2}}$, a way of seeing a tome in $\mathcal{C}$ is in the form of a cocone $(u,v):\partial \varphi \Rightarrow \Delta_{\mathtt{S}}(h)$ in $\mathcal{C}^{\mathbf{2}}$ over the functor $\partial\varphi:\mathtt{S} \to \mathcal{C}^{\mathbf{2}}$.
Because $\mathcal{C}$ has all colimits over $\mathtt{S}$, the earlier cocone provides an arrow $\mathrm{col}_{\mathtt{S}} \partial \varphi \Rightarrow h$ in $\mathcal{C}^{\mathbf{2}}$ after applying the adjunction property of $\mathrm{col}_{\mathtt{S}} \dashv \Delta_{\mathtt{S}}$ on it. This latest arrow will be referred to as the \emph{content of $(\mathtt{S},\varphi,h)$}.
Note that for any functor $i:\mathtt{T} \to \mathtt{S}$, we may pre-compose the universal shifting induced by $i$ (see section \ref{sssec:Universal_shifting}) with the content of $(h,\mathtt{S},\varphi)$ as follows.
\[
\mathrm{col}_{\mathtt{T}} \partial \varphi i \Rightarrow \mathrm{col}_{\mathtt{S}} \partial \varphi \Rightarrow h
\]

The resulting arrow $\mathrm{col}_{i} \partial \varphi \Rightarrow h$ will later play a central role and be referred to as the \emph{content of $(f,\mathtt{S},\varphi)$ along $i:\mathtt{T} \to \mathtt{S}$}.

\subsection{Morphisms of Tomes}\label{ssec:Morphisms_of_tomes}
Let $\mathcal{C}$ be a category. A \emph{loose morphism of tomes} from $\mathbf{T}_0:=(h_0,\mathtt{S}_0,\varphi_0)$ to $\mathbf{T}_1:=(h_1,\mathtt{S}_1,\varphi_1)$ is given by a morphism $(x,y):h_0 \Rightarrow h_1$ in $\mathcal{C}^{\mathbf{2}}$. A \emph{regular morphism of tomes} $\mathbf{T}_0 \Rightarrow \mathbf{T}_1$ is given by a morphism $(x,y):h_0 \Rightarrow h_1$ in $\mathcal{C}^{\mathbf{2}}$ and a functor $\upsigma:\mathtt{S}_0 \to \mathtt{S}_1$ making the next right diagram commute.
\[
\xymatrix{
X_0\ar[d]_{h_0}\ar[r]^{x}&X_1\ar[d]^{h_1}\\
Y_0\ar[r]^{y}&Y_1
}
\quad\quad\quad\quad\quad\quad
\xymatrix{
\mathtt{S}_0\ar[d]_{\varphi_0}\ar[rr]^{\upsigma}&&\mathtt{S}_1\ar[d]^{\varphi_1}\\
\mathcal{C}^{\mathbf{2}}/h_0\ar[rr]_{\mathcal{C}^{\mathbf{2}}/(x,y)}&&\mathcal{C}^{\mathbf{2}}/h_1
}
\]

The arrow symbol associated with loose morphisms will be denoted as $\mathbf{T}_0\mathop{\Rightarrow}\limits^{\!\!\star}\mathbf{T}_1$. 
The category whose objects are tomes in $\mathcal{C}$ and whose arrows are regular (resp. loose) morphisms of tomes will be denoted by  $\mathbf{Tome}(\mathcal{C})$ (resp. $\mathbf{Ltom}(\mathcal{C})$). For a fixed object $Q$ in $\mathcal{C}$, the wide subcategory of $\mathbf{Ltom}(\mathcal{C})$ that is restricted to the loose morphisms $(x,y):\mathbf{T}_0\mathop{\Rightarrow}\limits^{\!\!\star}\mathbf{T}_1$ whose components $y:Y_0 \to Y_1$ are identities on $Q$ will be denoted by $\mathbf{Ltom}(Q,\mathcal{C})$.

\subsection{Oeuvres and Narratives}\label{sssec:Oeuvres_and_narratives}
Let $(\mathcal{C},\kappa)$ be a numbered category and $Q$ be an object in $\mathcal{C}$. An \emph{oeuvre of theme $Q$} in $(\mathcal{C},\kappa)$
is a functor $\mathfrak{O}:\kappa+1 \to \mathbf{Ltom}(\mathcal{C})$ lifting \footnote{This lifting is formal and is mostly justified by the definition of the morphisms given in Section \ref{sssec:Morphisms_of_oeuvres_and_narratives}} to $\mathbf{Ltom}(Q,\mathcal{C})$ along the obvious inclusion $\mathbf{Ltom}(Q,\mathcal{C}) \hookrightarrow \mathbf{Ltom}(\mathcal{C})$.

\begin{convention}\label{conv:notations_narratives_oeuvres}
In the sequel, the image of an inequality $k<l$ in $\kappa+1$ via an oeuvre $\mathfrak{O}$ will be denoted by $(\upchi_k^l,\mathrm{id}_Q):(h_k,\mathtt{S}_k,\varphi_k) \mathop{\Rightarrow}\limits^{\!\!\star} (h_{l},\mathtt{S}_{l},\varphi_{l})$.
For convenience, when $l$ is successor of $k$ in $\kappa+1$, the notations $\upchi_k^l$ will be shortened to 
$\upchi_k$. For every object $k$ in $\kappa+1$, the morphism $h_k$ will be denoted as an arrow $G_k \to Q$ while the image of the composite functor $\partial\varphi_k:\mathtt{S}_k \to \mathcal{C}^{\mathbf{2}}$ at an object $s$ in $\mathtt{S}_k$ will be denoted as $\partial\varphi_k(s):\mathbf{A}_k(s) \to \mathbf{B}_k(s)$. 
\end{convention}

For every finite ordinal $\delta \in \omega$, a \emph{narrative of theme $Q$ and degree $\delta$} in $(\mathcal{C},\kappa)$ is an oeuvre $\mathfrak{O}:\kappa+1 \to \mathbf{Ltom}(\mathcal{C})$ of theme $Q$ equipped with
\smallskip

(1) (\emph{events}) for every ordinal $k \in \kappa$, a set $J_k$, called the \emph{set of events at rank $k$}, consisting of objects of $\mathbf{Cat}/\mathtt{S}_k$ that are compatible with $(\mathcal{C},\kappa)$ as functors;
\smallskip

(2) (\emph{viewpoint}) for every functor $i:\mathtt{T} \to \mathtt{S}_k$ in the set $J_k$, a lift for the commutative square (living in $\mathcal{C}$) resulting from the pre-composition of the content of $(h_k,\mathtt{S}_k,\varphi_k)$ along $i:\mathtt{T} \to \mathtt{S}_k$ with the arrow $\chi_{k}^{k+\delta}:h_k \Rightarrow h_{k+\delta}$; the square is therefore of the form $\mathrm{col}_{\mathtt{T}} \partial\varphi_k \Rightarrow h_{k+\delta}$ in $\mathcal{C}^{\mathbf{2}}$. The lift will later be referred to as the \emph{viewpoint at rank $k$ along $i$}.

\begin{remark}
It follows from Convention \ref{conv:notations_narratives_oeuvres} that the viewpoint at rank $k$ along $i$ mentioned in item (2) must be of the form $\mathrm{col}_i\mathbf{B}_k \to G_{k+\delta}$.
\end{remark}

\begin{convention}
The functor $\kappa+1 \to \mathcal{C}$ induced by the sequence of arrows $\upchi_k^l:G_k \to G_l$ for every inequality $k < l$ in $\kappa+1$ will be denoted by $G$ and called the \emph{context functor}.
\end{convention}

Observe that any oeuvre and, \emph{a fortiori}, any narrative as defined above provides a factorisation in $\mathcal{C}$ as given below. This factorisation is that used for our small object argument.
\begin{equation}\label{eq:oeuvre_factorisation}
\xymatrix{
G_0\ar@/_1.5pc/[rr]_{h_0}\ar[r]^{\upchi_0^{\kappa}}&G_{\kappa}\ar[r]^{h_{\kappa}}&Q
}
\end{equation}

Also, notice that the set of events $J_k$ induces an obvious lifting system $\{\partial\varphi_k \circ i~|~i \in J_k\}$, which will be denoted by $E_k(\mathfrak{O})$.

\subsection{Small Object Argument} Let $(\mathcal{C},\kappa)$ be a numbered category, $Q$ be an object in $\mathcal{C}$ and $\mathfrak{O}:\kappa+1 \to \mathbf{Ltom}(\mathcal{C})$ be a narrative of theme $Q$ and degree $\delta$.
A lifting system $J$ in $(\mathcal{C},\kappa)$ will be said to \emph{agree with} the narrative $\mathfrak{O}$ if for every ordinal $k \in \kappa$ and functor $\upvarphi:\mathtt{T} \to \mathcal{C}^{\mathbf{2}}$ in $J$ admiting a lift $\psi:\mathtt{T} \to \mathcal{C}^{\mathbf{2}}/h_k$ of $\upvarphi$ along $\partial$ (see left diagram below), there exists a functor $i:\mathtt{T} \to \mathtt{S}_k$ in $J_k$ whose composite with $\varphi_k$ gives the lift $\psi$ (see right diagram below).
\begin{equation}\label{eq:agree_implication}
\xymatrix{
& \mathcal{C}^{\mathbf{2}}/h_k\ar[d]^{\partial}\\
\mathtt{T}\ar@{-->}[ru]^{\psi}\ar[r]_-{\upvarphi}&\mathcal{C}^{\mathbf{2}}
}
\quad\quad\quad\Rightarrow\quad\quad\quad
\xymatrix{
\mathtt{S}_k\ar[r]^-{\varphi_k}& \mathcal{C}^{\mathbf{2}}/h_k\ar[d]^{\partial}\\
\mathtt{T}\ar[u]^-{i}\ar@{-->}[ru]^{\psi}\ar[r]_-{\upvarphi}&\mathcal{C}^{\mathbf{2}}
}
\end{equation}

\begin{proposition}[Small Object Argument]\label{prop:small_object_argument_general}
Let $J$ be a lifting system in
$(\mathcal{C},\kappa)$ agreeing with the narrative $\mathfrak{O}$. If the context functor $G:\kappa+1 \to \mathcal{C}$ is uniformly $(\mathrm{dom} \circ \upvarphi)$-convergent in $\mathcal{C}$ for every $\upvarphi \in J$, then
the morphism $h_{\kappa}:G_{\kappa} \to Q$ appearing in Equation (\ref{eq:oeuvre_factorisation}) is in $\mathbf{rlp}(J)$.
\end{proposition}
\begin{proof} 
The goal of the proof is to show that the morphism $h_{\kappa}:G_{\kappa} \to Q$ is in $\mathbf{rlp}(J)$. To do so, let $\upvarphi:\mathtt{T} \to \mathcal{C}^{\mathbf{2}}$ be a functor in $J$ and consider any arrow $(x,y):\mathrm{col}_{\mathtt{T}}\upvarphi \Rightarrow h_{\kappa}$. The proposition will be proven if the commutative square encoded by this arrow admits a lift.
By assumption, the functor $G:\kappa+1 \to \mathcal{C}$ is uniformly $(\mathrm{dom} \circ \upvarphi)$-convergent in $\mathcal{C}$. It follows from Remark \ref{rem:kappa_smallness_colimit_of_uniform_convergent_functor}, taken from the viewpoint of Remark \ref{rem:surjectiveness_property_of_convergence_for_SOA}, and the fact that $\kappa$ is limit \footnote{Recall that if $\kappa$ is limit, then for every ordinal $k \in \kappa$, the successor $k + \delta$ is also in $\kappa$ for every $\delta \in \omega$} that there exist an ordinal $k \in \kappa$ and an arrow $(x',y):\mathrm{col}_{\mathtt{T}}\upvarphi \Rightarrow h_{k}$ factorising $(x,y):\mathrm{col}_{\mathtt{T}}\upvarphi \Rightarrow h_{\kappa}$ as follows.
\begin{equation}\label{eq:smallness_gives_uprime}
\mathrm{col}_{\mathtt{T}}\upvarphi \Rightarrow h_{k}\Rightarrow h_{k+\delta}\Rightarrow h_{\kappa}
\end{equation}

Note that an application of the universal property of the adjunction $\mathrm{col}_{\mathtt{T}} \dashv \Delta_{\mathtt{T}}$ on the leftmost arrow of Equation (\ref{eq:smallness_gives_uprime}) provides an arrow in $\mathcal{C}^{\mathtt{T}}$ as follows (where the leftmost arrow, given below, is the unit of $\mathrm{col}_{\mathtt{T}} \dashv \Delta_{\mathtt{T}}$).
\begin{equation}\label{eq:rlp_smallness_pickout_diagram}
\upvarphi \Rightarrow \Delta_{\mathtt{T}}\mathrm{col}_{\mathtt{T}} \upvarphi \Rightarrow \Delta_{\mathtt{T}}h_{k}
\end{equation}

According to Remark \ref{rem:functor_in_overcats_give_nat_transf}, Arrow (\ref{eq:rlp_smallness_pickout_diagram}) induces a functor $\psi:\mathtt{T} \to \mathcal{C}^{\mathbf{2}}/h_k$, which makes the leftmost diagram of Equation (\ref{eq:agree_implication}) commute.
Because the lifting system $J$ agrees with the narrative $\mathfrak{O}$, there must exist a functor $i:\mathtt{T} \to \mathtt{S}_k$ making the right diagram of Equation (\ref{eq:agree_implication}) commute. This means, after re-applying the adjunction $\mathrm{col}_{\mathtt{T}} \dashv \Delta_{\mathtt{T}}$, that Equation  (\ref{eq:smallness_gives_uprime}) is in fact of the following form, where the leftmost arrow is precisely the content of the tome $\mathfrak{O}_k$ along $i:\mathtt{T} \to \mathtt{S}_k$.
\begin{equation}\label{eq:soa_diagram_giving_lift_content}
\mathrm{col}_{i}\partial\varphi_k \Rightarrow h_{k}\Rightarrow h_{k+\delta}\Rightarrow h_{\kappa}
\end{equation}

It follows from the viewpoint axiom (see Section \ref{sssec:Oeuvres_and_narratives}) satisfied by $\mathfrak{O}$ that the Composite $\mathrm{col}_{i}\partial\varphi_k \Rightarrow h_{k}\Rightarrow h_{k+\delta}$ admits a lift. This implies that the whole composite (\ref{eq:soa_diagram_giving_lift_content}) admits a lift, which, \emph{a fortiori}, implies that the arrow $(x,y):\mathrm{col}_{\mathtt{T}}\upvarphi \Rightarrow h_{\kappa}$ admits a lift.
\end{proof}

\subsection{Strict Narratives}\label{sssec:strict_narratives}
Let $(\mathcal{C},\kappa)$ be a numbered category and $Q$ be an object in $\mathcal{C}$. For any narrative $\mathfrak{O}:\kappa+1 \to \mathbf{Ltom}(\mathcal{C})$ of theme $Q$, recall that the set of events $J_k$ gives a collection of functors that induces a cocone under the category $S_k$ (see Section \ref{sssec:Oeuvres_and_narratives}).
A narrative $\mathfrak{O}:\kappa+1 \to \mathbf{Ltom}(\mathcal{C})$ of theme $Q$ and degree $\delta$ will be said to be \emph{strict} in $\mathcal{C}$ if 
\begin{itemize}
\item[(1)] for every ordinal $k \in \kappa$, the cocone induced by the elements of $J_k$ is universal in $\mathbf{Cat}$.
\item[(2)] it is equipped with a morphism $\uppi_k:\mathrm{col}_{\mathtt{S}_k} \mathbf{B}_k \to G_{k+\delta}$ factorising the content of $\mathfrak{O}_k$ into a pushout as follows;
\[
\xymatrix@R-1pc{
*+!R(.5){\mathrm{col}_{\mathtt{S}_k}\mathbf{A}_k}\ar@{}[rd]|>>>{\text{\huge{\rotatebox[origin=c]{-90}{$\llcorner$}}}}\ar[rr]\ar[dd]_{\mathrm{col}_{\mathtt{S}_k}\partial\varphi_k}&&G_{k}\ar[dd]^{h_{k}} \ar[ld]_-{\upchi_k}\\
&G_{k+\delta}\ar[rd]^{h_{k+\delta}}&\\
*+!R(.5){\mathrm{col}_{\mathtt{S}_k}\mathbf{B}_k}\ar[rr]\ar[ru]^-{\uppi_k}&&Q
}
\]
\item[(3)] for every functor $i:\mathtt{T} \to \mathtt{S}_k$ in $J_k$, the viewpoint $\uppi_k^i:\mathrm{col}_{i}\mathbf{B}_k \to N_{k}$ along $i$ is equal to the pre-composition of $\uppi_k$ with the universal shifting along $i$ as follows;
\[
\xymatrix{
\mathrm{col}_{i}\mathbf{B}_k\ar[r]^-{\xi_{i}(\mathbf{B}_k)}&\mathrm{col}_{\mathtt{S}_k}\mathbf{B}_k\ar[r]^-{\uppi_k}&G_{k+\delta}\\
}
\] 
\item[(4)] the context functor $G:\kappa+1 \to \mathcal{C}$ is sequential (see Section \ref{ssec:sequential_functor}).
\end{itemize}

\begin{proposition}\label{prop:llp_soa_with_respect_to_tome_event}
If a morphism $f:X \to Y$ is in $\mathbf{rlp}(E_k(\mathfrak{O}))$ (see end of Section \ref{sssec:Oeuvres_and_narratives}) for every $k \in \kappa$, then it has the rlp with respect to the arrow $\upchi_0^{\kappa}:G_0 \to G_{\kappa}$ (see Diagram (\ref{eq:oeuvre_factorisation})).
\end{proposition}
\begin{proof}
Let $f:X \to Y$ be a morphism that has the rlp with respect to the lifting system $E_k(\mathfrak{O})$ for every $k \in \kappa$. For any $k \in \kappa$, this means that it has the rlp with respect to the following arrow in $\mathcal{C}$, for every functor $i:\mathtt{T} \to \mathtt{S}_k$ in $J_k$.
\[
\mathrm{col}_{\mathtt{S}_k}(\partial\varphi_k\circ i):\mathrm{col}_{\mathtt{S}_k}(\mathbf{A}_k\circ i) \to \mathrm{col}_{\mathtt{S}_k}(\mathbf{B}_k\circ i) 
\]

It directly follows that $f$ has the rlp with respect to the coproduct of these arrows over the set $J_k$ (seen as a discrete category), which may be identified to the arrow $\mathrm{col}_{\mathtt{S}_k}(\partial\varphi_k)$ up to isomorphism as shown below.
\begin{align*}
\coprod_{i \in J_k}\mathrm{col}_{\mathtt{S}_k}(\partial\varphi_k\circ i) &\cong \mathrm{col}_{\mathtt{S}_k}(\coprod_{i \in J_k}\partial\varphi_k\circ i)&\text{(colimits commute)}\\
&\cong \mathrm{col}_{\mathtt{S}_k}(\partial\varphi_k)&\text{(universality of $J_k$)}
\end{align*}

It follows from classical facts that, since $f$ has the rlp with respect to $\mathrm{col}_{\mathtt{S}_k}\partial\varphi_k$, it has the rlp with respect to any of its pushouts, and hence with respect to $\upchi_k:G_k \to G_{k+\delta}$ for any $k \in \kappa$. It finally follows from Proposition \ref{prop:strictly_combinatorial_llp_fib_f} and the fact that the context functor $G:\kappa+1 \to \mathcal{C}$ is sequential that $f$ has the rlp with respect to the arrow $\upchi_0^{\kappa}:G_0 \to G_{\kappa}$ in $\mathcal{C}$.
\end{proof}

\subsection{Morphisms of Oeuvres}\label{sssec:Morphisms_of_oeuvres_and_narratives}
Let $(\mathcal{C},\kappa)$ be a numbered category. For every pair of oeuvres $\mathfrak{O}:\kappa+1 \to \mathbf{Ltom}(\mathcal{C})$ and $\mathfrak{O}':\kappa+1 \to \mathbf{Ltom}(\mathcal{C})$, of respective themes $Q$ and $Q'$, a \emph{morphism of oeuvres} from $\mathfrak{O}$ to $\mathfrak{O}'$ consists, for every ordinal $k \in \kappa$, of a regular morphism of tomes
\[
(x_k,y_k,\upsigma_k):\mathfrak{O}_k \Rightarrow \mathfrak{O}_k'\quad\quad(\text{with }y_k:Q \to Q') 
\]
such that the underlying loose morphisms $(x_k,y_k):\mathfrak{O}_k \mathop{\Rightarrow}\limits^{\!\!\star} \mathfrak{O}_k'$ induce a morphism $\mathfrak{O} \Rightarrow \mathfrak{O}$ in the functor category $\mathbf{Ltom}(\mathcal{C})^{\kappa+1}$ (see Remark \ref{rem:morphisms:narratives_and_oeuvres}). The category whose objects are oeuvres for the numbered category $(\mathcal{C},\kappa)$ and whose arrows are morphisms of oeuvres will be denoted by $\mathbf{Oeuv}(\mathcal{C},\kappa)$.

\begin{remark}\label{rem:morphisms:narratives_and_oeuvres}
The previous definition implies that all the arrows $y_k$ are equal to the same morphism $y:Q \to Q'$ for every $k \in \kappa+1$. In addition, it forces the equality $\upchi_k' \circ x_k = x_{k+1} \circ \upchi_k$ to hold in $\mathcal{C}$ for every $k \in \kappa$.
\end{remark}

\section{Constructors and Their Tomes}\label{sec:Constructors}

This section introduces the notion of constructor that allows one to associate systems of premodels with tomes. Constructors contain all the necessary information that permits the `elimination of quotients'. We will see that their definition already brings out what is meant to be analytic (or structural) and what is meant to be  quotiented out. Even if they appear to comprise many components, the main goal of the items defined in Sections \ref{sssec:Pre-constructor} and  \ref{sssec:Constructor} is to be able to define two sums whose forms look like the following type.
\[
\sum_{\vartheta \in J} \mathrm{Hom}(e_0(\vartheta),\_) \times \Big(\sum_{s \in \ell(\vartheta)}\mathrm{Hom}\big(e_1(\vartheta),\Upphi(\vartheta,s)\big)\Big)
\]

The hom-sets $\mathrm{Hom}(e_0(\vartheta),\_)$ -- which are defined in Section \ref{sssec:Constructor}---are meant to ensure a certain functoriality (i.e., they are the monomials for a certain type of species \cite{Joyal06}) while the hom-sets $\mathrm{Hom}(e_1(\vartheta),\Upphi(\vartheta,s))$---which are defined in Section \ref{sssec:Pre-constructor}---are meant to contain the `squares' that will enable us to perform our small object argument. In the sequel, I shall therefore try to give evoking names to the different parameters used to define these sums. In particular, one sum is to encode the structural data of our elimination of quotients while the other one is to encode the quotient acting on this data. To make the reader more confident with the items of Sections \ref{sssec:Pre-constructor} and  \ref{sssec:Constructor}, here is a preluding summary of the different notations used therein.
\smallskip

\begin{center}
\begin{tabular}{ c|c|c|c|c|c|c|c }
 \cellcolor[gray]{0.8} & \cellcolor[gray]{0.8}$J$ & \cellcolor[gray]{0.8}1st $\mathrm{Hom}$ & \cellcolor[gray]{0.8}$e_0$  & \cellcolor[gray]{0.8}$\ell(\vartheta)$ & \cellcolor[gray]{0.8}2nd $\mathrm{Hom}$ & \cellcolor[gray]{0.8}$e_1$ & \cellcolor[gray]{0.8}$\Upphi(\vartheta,s)$\\
\hline
 \cellcolor[gray]{0.9} analytic sum & $J_A$ & \multirow{2}{*}{$D$} & $\upepsilon$  & $\uplambda_{\upiota(\vartheta)}(X)$ & $\mathcal{C}^{\mathbf{2} \times \mathbf{2}}$ & $\upalpha\upiota$ & $\Phi_{\theta}(f)_s$\\
\cellcolor[gray]{0.9} quotient sum & $J_Q$ & & $\upchi$  & $\uplambda_{\updelta(\vartheta)}(X)$ & $\mathcal{C}^{\mathbf{2}}$ & $\upomega\updelta$ & $\Phi_{\theta}^{\bullet}(f)_s$\\
\end{tabular}
\end{center}

\subsection{Some More Notations}
The following conventions are meant to ease the combinatorial description of a constructor and its associated tomes, which will be defined in Section \ref{ssec:Tomes_of_a_constructor}.

\begin{convention}[Vertebrae]\label{conv:some_more_notations_vertebrae}
The diskad of a vertebra $v:=\|\gamma_2,\gamma_1\|\cdot \beta$ will be denoted by $\mathbf{disk}(v)$ and seen as an arrow $\gamma_2 \Rightarrow \beta \delta_1$ in $\mathcal{C}^{\mathbf{2}}$. The other arrow $\gamma_1 \Rightarrow \beta \delta_2$ in $\mathcal{C}^{\mathbf{2}}$, which is induced by the `dual' vertebra $v^{\mathrm{rv}}:=\|\gamma_1,\gamma_2\|\cdot \beta$, will be denoted by $\mathbf{disk}(v^{\mathrm{rv}})$ and called the \emph{codiskad of $v$}. Finally, the stem $\beta$ and seed $\gamma_2$ of $v$ will be referred to by the notations $\mathbf{stem}(v)$ and $\mathbf{seed}(v)$, respectively.
\end{convention}

\begin{convention}[Domains and codomains]\label{conv:circ_bullet_source_target}
Let $\mathcal{A}$ and $\mathcal{C}$ be two categories and $F:\mathcal{A} \to \mathcal{C}^{\mathbf{2}}$ be a functor. In order to avoid too many notations in our reasonings, the image $F(X)$ of an object $X$ of $\mathcal{A}$ in the arrow category $\mathcal{C}^{\mathbf{2}}$ will be denoted as $F(X):F^{\circ}(X) \to F^{\bullet}(X)$. This implies that every morphism $f:X \to Y$ in $\mathcal{A}$ gives a commutative diagram as follows.
\[
\xymatrix{
F^{\circ}(X)\ar[d]_{F^{\circ}(f)}\ar[rr]^{F(X)}&&F^{\bullet}(X)\ar[d]^{F^{\bullet}(f)}\\
F^{\circ}(Y)\ar[rr]_{F(Y)}&&F^{\bullet}(Y)
}
\]

Similarly, for every functor $H:\mathcal{A} \to \mathcal{C}^{\mathbf{2} \times \mathbf{2}}$, we will denote by $H^{\circ}:\mathcal{A} \to \mathcal{C}^{\mathbf{2}}$ and $H^{\bullet}:\mathcal{A} \to \mathcal{C}^{\mathbf{2}}$ the `source' and `target' arrows of the squares involved in the image of $H$. 
\end{convention}

\begin{example}
For every vertebra $v$ in $\mathcal{C}$ as displayed in Equation (\ref{eq:definition:System_of_premodels}), the arrow $\mathbf{disk}(v)^{\circ}$ is equal to $\mathbf{seed}(v)$. Thus, when the reader reads $\upalpha^{\circ}(\_)$ in Section \ref{sssec:Constructor}, where $\upalpha$ is a functor $I \to \mathcal{C}^{\mathbf{2} \times \mathbf{2}}$ mapping any element in $I$ to the diskad of a certain vertebra in $\mathcal{C}$, they should think of the seed of the so-called vertebra.
\end{example}

\begin{convention}[Closedness]
Let $\mathcal{A}$, $\mathcal{B}$ and $\mathcal{C}$ be three categories. The image of any functor of the form $G:\mathcal{A} \times \mathcal{B} \to \mathcal{C}$ will later be denoted as $F_a(b)$ for any pair of objects $(a,b)$ in $\mathcal{A} \times \mathcal{B}$ -- instead of the usual notation $F(a,b)$.
\end{convention}

\begin{convention}[Families]
Let $\mathcal{C}$ be a category. In the sequel, we will denote by $\mathbb{I}$ the obvious functor $\mathbf{Set} \times \mathcal{C} \to  \mathbf{Fam}(\mathcal{C})$ mapping a pair $(S,X)$ to the functor $\Delta_S(X):S \to \mathcal{C}$. Also, mainly for convenience, the images of any object $F:S \to \mathcal{C}$ in $\mathbf{Fam}(\mathcal{C})$ at some $s \in S$ will be denoted by $F_s$. This means that the equation $\mathbb{I}_S(X)_s = X$ holds for every $s \in S$.
\end{convention}

\begin{convention}[Families of arrows]
Convention \ref{conv:circ_bullet_source_target} will be extended to $\mathbf{Fam}(\mathcal{C}^{\mathbf{2}})$ in the obvious way: for every functor $\Phi:\mathcal{A} \to \mathbf{Fam}(\mathcal{C}^{\mathbf{2}})$, we shall denote by $\Phi^{\circ}$ and $\Phi^{\bullet}$ the obvious functors $\mathcal{A} \to \mathbf{Fam}(\mathcal{C}^{\mathbf{2}})$ mapping any object $X \in \mathcal{A}$ to the families $s \mapsto \Phi(X)_s^{\circ}$ and $s \mapsto \Phi(X)_s^{\bullet}$, respectively.
\end{convention}

\begin{convention}
Later on, I shall often identify a set with a discrete category and identify many functions with functors. The reason for this is that we shall pre-compose these functions with functors going from discrete categories to non-trivial categories, which, for their parts, should really be seen as functors. This convention should thus ease the back and forth between set theory and category theory.
\end{convention}

\subsection{Preconstructors}\label{sssec:Pre-constructor}
This section introduces the concept of preconstructor. This notion tries to capture what it takes to specify the data of a localisation. 
For instance, in Modern Algebra, localising a ring $(R,+,\cdot)$ requires one to specify:
\begin{itemize}
\item[$\star$] the underlying set that one wants to act on, which is here the set $R$;
\item[$\star$] the subset $S \subseteq R$ by which one wants to localise the ring; 
\item[$\star$] the operation that one wants to inverse, which is here given by the $S$-indexed family of group morphisms $e_s:R \to R$ defined by the mappings $x \mapsto x \cdot s$; 
\item[$\star$] the type of inversion one wants to see happening on the maps $e_s$.
\end{itemize}

Regarding this last item, the inversion would, for instance, be expressed in terms of a bijection for the type of localisation used in Classical Algebraic Geometry, but it would be expressed in terms of a quasi-isomorphism in the category of unbounded chain complexes in Derived Algebraic Geometry. 

To pass from the earlier description to the formalism of preconstructors, one can try to describe what a preconstructor would be for the previous list of items, so that we could make the following associations (also, see the structure below): the data $\uprho$ would specify the object $R$ while the data $\uplambda$ would give the subset $S$; the data $\Phi$, $\Upsilon$ and $\Psi$ would enumerate the maps $e_s$ with theirs domains $R$ and codomains $R$ (which would be required to be independent of the indices in $S$); and the data $\upalpha$ and $\upomega$ would specify the type of inversion one wants to see happening. We now give a formal definition.
\vspace{3pt}

Let $\mathcal{B}$ and $\mathcal{C}$ be two categories and $D$ be a small category. A \emph{preconstructor of type $D[\mathcal{B},\mathcal{C}]$}, let us call it $\Upgamma$, consists of a discrete category $I$ together with 
\begin{itemize}
\item[(a)] two functors $\uprho:I \to D$ and $\uplambda:I\times \mathcal{B} \to \mathbf{Set}$, called the \emph{regulator} and the \emph{localisor};
\item[(b)] three functors as given below, which satisfy the string diagram axioms given underneath them (or the equations given just after).
\[
\Phi:I\times \mathcal{B} \to \mathbf{Fam}(\mathcal{C}^{\mathbf{2}})\quad\quad \Upsilon:D \times \mathcal{B} \to \mathcal{C}\quad\quad \Psi:I \times \mathcal{B} \to \mathcal{C}
\]
\[
\begin{array}{c}
\xymatrix@R-1pc@C-3pc{
\mathcal{B}\ar@{-}@/_/[dr]&&I\ar@{-}@/^/[dl]\\
&\fbox{$\Phi$}\ar@{-}[d]&\\
&\fbox{$(\_)^{\circ}$}\ar@{-}[d]&\\
&\mathbf{Fam}(\mathcal{C})&\\
}
\end{array}
\begin{array}{c}
=
\end{array}
\begin{array}{c}
\xymatrix@R-1pc@C-3pc{
&\mathcal{B}\ar@{-}[d]&&I\ar@{-}[d]&\\
&\ar@{-}[ddrrr]\ar@{-}@/_1pc/[ddl]&&\ar@{-}@/_/[ddlll]\ar@{-}@/^/[dr]&\\
&&&&\fbox{$\uprho$}\ar@{-}[d]\\
\fbox{$\uplambda$}\ar@{-}@/_/[drr]&&&&\fbox{$\Upsilon$}\ar@{-}@/^/[dll]\\
&&\fbox{$\mathbb{I}$}\ar@{-}[d]&&\\
&&\mathbf{Fam}(\mathcal{C})&&\\
}
\end{array}
\quad\quad\quad\quad
\begin{array}{c}
\xymatrix@R-1pc@C-3pc{
\mathcal{B}\ar@{-}@/_/[dr]&&I\ar@{-}@/^/[dl]\\
&\fbox{$\Phi$}\ar@{-}[d]&\\
&\fbox{$(\_)^{\bullet}$}\ar@{-}[d]&\\
&\mathbf{Fam}(\mathcal{C})&\\
}
\end{array}
\begin{array}{c}
=
\end{array}
\begin{array}{c}
\xymatrix@R-1pc@C-3pc{
&\mathcal{B}\ar@{-}[d]&&I\ar@{-}[d]&\\
&\ar@{-}[drrr]\ar@{-}@/_1pc/[dl]&&\ar@{-}@/_/[dlll]\ar@{-}@/^/[dr]&\\
\fbox{$\uplambda$}\ar@{-}@/_/[drr]&&&&\fbox{$\Psi$}\ar@{-}@/^/[dll]\\
&&\fbox{$\mathbb{I}$}\ar@{-}[d]&&\\
&&\mathbf{Fam}(\mathcal{C})&&\\
}
\end{array}
\]

The previous string diagrams amount to saying that the following equations hold in the functor `category' $[\mathcal{B},\mathbf{Fam}(\mathcal{C})]$ for every $\theta \in I$;
\[
\Phi^{\circ}_{\theta}(\_) = \mathbb{I}_{\uplambda_{\theta}(\_)}\Upsilon_{\uprho(\theta)}(\_)
\quad\quad\quad\quad\quad\quad
\Phi_{\theta}^{\bullet}(\_) = \mathbb{I}_{\uplambda_{\theta}(\_)}\Psi_{\theta}(\_)
\]
\item[(c)] two functors $\upalpha:I \to \mathcal{C}^{\mathbf{2} \times \mathbf{2}}$ and $\upomega:I \to \mathcal{C}^{\mathbf{2}}$, called the \emph{analysor} and the \emph{quotientor}, such that the image $\upalpha(\theta)$ encodes the diskad of a vertebra of stem $\upomega(\theta)$ for every $\theta \in I$;
\end{itemize}

As mentioned in the preamble of Section \ref{sec:Constructors}, a preconstructor contains all the information that is necessary to define the parametrising `squares' on which we will run the small-object-argument algorithm. These so-called parameters will be presented either as families (see Definition \ref{def:analytic_quotient_family}) or as formal sums (see Definition \ref{def:analytic_quotient_species}) -- both presentations being useful.

\begin{definition}[Families]\label{def:analytic_quotient_family}
For any preconstructor $\Upgamma$ as defined above, the \emph{analytic family} of $\Upgamma$ and the \emph{quotient family} of $\Upgamma$ are two functors $\Upgamma_A (\_)\{\_\}:\mathcal{B}^{\mathbf{2}} \times I \to \mathbf{Fam}(\mathbf{Set})$ and $\Upgamma_Q (\_)\{\_\}:\mathcal{B}^{\mathbf{2}} \times I \to \mathbf{Fam}(\mathbf{Set})$ whose images are determined, for every arrow $f:X \to Y$ in $\mathcal{B}$ and  object $\theta \in I$, by the following mappings (or families) over $\uplambda_{\theta}(X)$.
\[
\Upgamma_A(f)\{\theta\}:s \mapsto \mathcal{C}^{\mathbf{2} \times \mathbf{2}}(\upalpha(\theta),\Phi_{\theta}(f)_s)
\quad\quad\quad
\Upgamma_Q(f)\{\theta\}:s \mapsto \mathcal{C}^{\mathbf{2}}(\upomega(\theta),\Phi_{\theta}^{\bullet}(f)_s)
\]
\end{definition}

\begin{remark}[Concept of vertebra]
The relationship between the analytic family and the quotient family is established in item c) via the concept of vertebra. At this stage, this should suggest to the reader that the notion of vertebra subtly encompass both the idea of quotient---or coherence---via its stem and the idea of cellular structure---or \emph{ana}-\emph{lysis}---via its diskad.
\end{remark}

\begin{definition}[Species]\label{def:analytic_quotient_species}
For any preconstructor $\Upgamma$ as defined above, the \emph{analytic species} of $\Upgamma$ and the \emph{quotient species} of $\Upgamma$ are two functors $\Upgamma_A (\_)[\_]:\mathcal{B}^{\mathbf{2}} \times I \to \mathbf{Set}$ and $\Upgamma_Q (\_)[\_]:\mathcal{B}^{\mathbf{2}} \times I \to \mathbf{Set}$ defined as follows, for every arrow $f:X \to Y$ in $\mathcal{B}$ and  object $\theta \in I$.
\[
\Upgamma_A(f)[\theta]:=\sum_{s \in \uplambda_{\theta}(X)}\Upgamma_A(f)\{\theta\}_s
\quad\quad 
\Upgamma_Q(f)[\theta]:=\sum_{s \in \uplambda_{\theta}(X)}\Upgamma_Q(f)\{\theta\}_s
\]
\end{definition}

\subsection{Preconstructor of a System of Premodels}\label{sssec:Preconstructor_system_premodels}
Let $(K,\mathtt{rou},\mathcal{P},\mathtt{V})$ be a system of $R$-premodels over a small category $D$ in a category $\mathcal{C}$. The goal of this section is to associate any such system with a preconstructor of type $D[\mathcal{P},\mathcal{C}]$. In this respect, define the set $I$ to be the following leftmost disjoint sum.
\[
I:=\sum_{c_0 \in K}   \mathtt{V}_{c_0}
\]

\begin{remark}[Encoding]\label{rem:form_of_elt_I_preconstructor}
Any element $\theta$ in $I$ may be presented as a pair $(c_0,v)$ where $c_0$ is a cone in $K$ and $v$ is a vertebra in $\mathtt{V}_{c_0}$.
\end{remark}
By keeping the notational convention suggested by Remark \ref{rem:form_of_elt_I_preconstructor}, one defines the data of the preconstructor for the system of premodels $(K,\mathtt{rou},\mathcal{P},\mathtt{V})$ as follows:
\begin{itemize}
\item[(1)] the regulator is given by the mapping $\uprho:\theta \mapsto \mathtt{rou}(c_0)$;
\item[(2)] the localisor is given by the evaluation $\uplambda:(\theta,(P,S,e)) \mapsto S(c_0)$;
\item[(3)] the analysor is given by the mapping $\upalpha:\theta \mapsto \mathbf{disk}(v)$;
\item[(4)] the quotientor is given by the mapping $\upomega:\theta \mapsto \mathbf{stem}(v)$;

\end{itemize}
and because both equations
\[
\mathcal{G}^K_{c_0}(P,S,e)_s^{\circ} = P(\mathtt{rou}(c_0))\quad\text{ and }\quad\mathcal{G}^K_{c_0}(P,S,e)_s^{\bullet} = \mathrm{lim}RP\mathtt{in}(c_0)
\]
hold for every $s \in S(c_0)$, one may define the functor $\Phi:I\times \mathcal{P} \to \mathbf{Fam}(\mathcal{C}^{\mathbf{2}})$ as the obvious functor satisfying the mapping $(\theta,(P,S,e)) \mapsto \mathcal{G}_{c_0}^K(P,S,e)$ on objects, so that the two associated functors $\Upsilon$ and $\Psi$ are defined as follows.
\[
\Upsilon:
\left(
\begin{array}{ccc}
D \times \mathcal{P} & \to & \mathcal{C}\\
(d,(P,S,e)) &\mapsto & P(d)
\end{array}
\right)
\quad\quad
\Psi:
\left(
\begin{array}{ccc}
I \times \mathcal{P} & \to & \mathcal{C}\\
(\theta,(P,S,e)) &\mapsto & \mathrm{lim}RP\mathtt{in}(c_0)
\end{array}
\right)
\]

\begin{remark}[Encoding]
For every arrow $f:X \to Y$ in $\mathcal{B}$ and element $\theta$ in $I$, the image of the analytic species $\Upgamma^K_A(f)[\theta]$ contains the tuples \footnote{The symbol $s'$ is, here, preferred to the plain letter $s$ as it could be confused with the notation $\mathbf{s}$ (in bold) or thought to be related to the notation $\underline{s}$, which is not the case. I shall sometimes use $s$ instead of $s'$ when no confusion is possible} $(c_0,v,s',\mathbf{c})$ where: $c_0$ is a cone in $K$; $v$ is a vertebra in $\mathtt{V}_{c_0}$; $s'$ is an element in $S(c_0)$ and $\mathbf{c}$ is a commutative square in $\mathcal{C}^{\mathbf{2}}$ of the form given below, on the left, for the notation $\theta:=(c_0,v)$, which may also be seen as the right commutative cube in $\mathcal{C}$ when viewed from the bottom-left corner.
\begin{equation}\label{eq:commutative_cuboid_definition_playground}
\begin{array}{l}
\xymatrix@C-1pc{
\gamma_2\ar@{=>}[r]\ar@{=>}[d]_{\mathbf{disk}(v)}&*+!L(.5){\Phi_{\theta}^{\circ}(f)_{s'}}\ar@{=>}[d]^{\Phi_{\theta}(f)_{s'}}\\
*+!R(.5){\beta \circ \delta_1}\ar@{=>}[r]&*+!L(.5){\Phi_{\theta}^{\bullet}(f)_{s'}}
}
\end{array}
\quad\quad
\begin{array}{l}
\xymatrix@C-0.5pc@R-0.5pc{
&&\cdot\ar[rr]|{\Phi_{\theta}(X)_{s'}}\ar[dd]|\hole&&\cdot\ar[dd]\ar@{}[rdd]|>>>>>>>>>>{\longleftarrow\Phi_{\theta}(f)_{s'}}&\\
&\cdot\ar@{}[ddl]|>>>>>>>>>>{\mathbf{disk}(v)\longrightarrow}\ar[rr]^>>>>>>{\gamma_1}\ar[dd]^<<<<<{\gamma_2}\ar[ru]&&\cdot\ar[dd]|<<<<<{\beta \circ \delta_1}\ar[ru]&&\\
&&\cdot\ar[rr]|\hole&&\cdot&\\
&\cdot\ar[rr]|{\beta \circ \delta_2}\ar[ru]&&\cdot\ar[ru]&&\\
}
\end{array}
\end{equation}

Similarly, the image of the quotient functor $\Upgamma^K_Q(f)[\theta]$ contains the tuples $(c_0,v,s',\mathbf{s})$ where: $c_0$ is a cone in $K$; $v$ is a vertebra in $\mathtt{V}_{c_0}$; $s'$ is an element in $S(c_0)$ and $\mathbf{s}$ is an arrow $\mathbf{stem}(v) \Rightarrow \Psi_{\theta}(f)$ in $\mathcal{C}^{\mathbf{2}}$ for the notation $\theta:=(c_0,v)$.
\end{remark}

\subsection{Constructors}\label{sssec:Constructor}
This section introduces the concept of constructor. In comparison to the informal introduction of Section \ref{sssec:Pre-constructor}, a constructor should be seen as a structure giving all the data that we need to describe the localisation of the ring $R$ by a subset $S$ in terms of freely-added tuples and relations acting on these. 

Specifically, one usually constructs the localisation $S^{-1}R$ by freely adding tuples of the form $(x,s)$, for every $x \in R$ and $s \in S$, to the set $R$. These tuples are often denoted as quotients $x/s$. Because $S$ has not been supposed to be a multiplicative set, one would also need to specify tuples of the form
$(x,s_1,s_2,s_3,\dots,s_n)$
for every $x \in R$ and $s_i \in S$ where $1 \leq i \leq n$. The equivalence relations defined on the pairs $(x,s)$ are quite well-known: two pairs $(x,s)$ and $(x',s')$ are equivalent if there exists $u \in R$ for which the following relation holds.
\[
u\cdot(xs' - x's) = 0
\]

In the case of the elements of the form $(x,s_1,s_2,s_3,\dots,s_n)$, it is less obvious how this should be done. A constructor can help us with this as it contains all the required structure for this type of general description without involving the need of focusing on the encoding.

In terms of the notations given below, in the definition of constructor, the data $\upepsilon$ would specify the set of elements that are to be paired with elements in $S$; the data $\upchi$ would specify the set of elements that are to be subject to relations of the form given earlier; the data $\upiota$ and $\updelta$, which are used for coherence purposes, would be identities; the data $\upmu$ and $\upnu$ would specify the types of quotients one would like to see happening: they provide the seeds and the stems of the vertebrae given by the data $\upalpha$ coming from the preconstructor structure; the maps denoted by $\ell_{\vartheta}$ would map every element $x \in R$ to $x \cdot s$ (for the analytic links) and every pair $(x,x')$ where $x':=(x/s') \cdot s'$ to a pair $(x\cdot s,x'\cdot s)$ (for the quotient links); and the data $\textsc{\j}$ would specify how the set $R$ injects into the localisation $S^{-1}R$. With respect to the definition given below, all of these data would be associated with the canonical ring morphism $f:R \to \mathbf{1}$.
\vspace{3pt}

We now give the definition of constructor. Let $\mathcal{B}$ and $\mathcal{C}$ be two categories and $D$ be a small category. A \emph{constructor of type $D[\mathcal{B},\mathcal{C}]$} consists of a preconstructor $\Upgamma$ of type $D[\mathcal{B},\mathcal{C}]$, say $(I,\uprho, \uplambda, \dots, \upalpha, \upomega)$ as defined in Section \ref{sssec:Pre-constructor}, and a mapping $f \mapsto (J_A,J_Q,\upepsilon,\upchi,\upiota,\updelta,\upmu,\upnu,\textsc{\j})$ that equips every object $f \in \mathcal{B}^{\mathbf{2}}$ with a pair of sets $(J_A,J_Q)$ together with
\begin{itemize}
\item[(1)] two functors  $\upepsilon:J_A \to D$ and $\upchi:J_Q \to D$ called the \emph{analytic} and \emph{quotient exponents};
\item[(2)] two functors $\upiota:J_A \to I$ and $\updelta:J_Q \to I$ called the \emph{analytic} and \emph{quotient indicators};
\item[(3)] a functor $\upmu:J_A \to \mathcal{C}^{\mathbf{2}}$ called the \emph{transitive analysor} and, for every $\vartheta \in J_A$, a function $\ell_{\vartheta}$, called the \emph{analytic link}, of the following form;
\[
\ell_{\vartheta}:\mathcal{C}^{\mathbf{2}}(\upalpha^{\circ}\upiota(\vartheta),\Upsilon_{\uprho\upiota(\vartheta)}(f)) \to \mathcal{C}^{\mathbf{2}}(\upmu(\vartheta),\Upsilon_{\upepsilon(\vartheta)}(f))
\]
\item[(4)] a functor $\upnu:J_Q \to \mathcal{C}^{\mathbf{2}}$ called the \emph{transitive quotientor} and, for every $\vartheta \in J_Q$, a function $\ell_{\vartheta}$, called the \emph{quotient link}, of the following form;
\[
\ell_{\vartheta}:\mathcal{C}^{\mathbf{2}}(\upomega\updelta(\vartheta),\Psi_{\updelta(\vartheta)}(f)) \to \mathcal{C}^{\mathbf{2}}(\upnu(\vartheta),\Upsilon_{\upchi(\vartheta)}(f))
\]
\item[(5)] a functor $\textsc{\j}:I \to J_A$, called the \emph{analytic section}, satisfying the equalities $\upiota \circ \textsc{\j} = \mathrm{id}_I$, $\upepsilon \circ \textsc{\j} = \uprho$ and $\upmu \circ \textsc{\j} = \upalpha^{\circ}$ so that the analytic link $\ell_{\textsc{\j}(\theta)}$ is an identity for every $\theta \in I$;
\end{itemize}

For such a constructor, we define, for every object $f \in \mathcal{B}^{\mathbf{2}}$, an \emph{analytic functor} $\Upgamma_A(f): D \to \mathbf{Set}$ and a \emph{quotient functor} $\Upgamma_Q(f): D \to \mathbf{Set}$ whose images $\Upgamma_A(f)(d)$ and $\Upgamma_Q(f)(d)$ are given by the following formulae, respectively.
\[
\sum_{\vartheta \in J_A} D(\upepsilon(\vartheta),d) \times \Upgamma_A(f)[\upiota(\vartheta)]\quad\quad\quad \sum_{\vartheta \in J_Q} D(\upchi(\vartheta),d) \times \Upgamma_Q(f)[\updelta(\vartheta)]
\]

\subsection{Constructor of a System of Premodels}\label{sssec:Constructor_system_premodels}
Let $(K,\mathtt{rou},\mathcal{P},\mathtt{V})$ be a system of $R$-premodels over a small category $D$ in a category $\mathcal{C}$. The goal of this section is to associate any such system with a constructor of type $D[\mathcal{P},\mathcal{C}]$. We shall, of course, use the preconstructor structure defined in Section \ref{sssec:Preconstructor_system_premodels}. To define the supplementary structure, let us now define the following set (where $\mathrm{Obj}(\mathtt{Es}(c_0))$ denotes the set of objects of the elementary shape of $c_0 \in K$)
\[
I':=\sum_{c_0 \in K} \mathtt{V}_{c_0} \times \mathrm{Obj}(\mathtt{Es}(c_0))
\]
and let us associate every arrow $(f,a):(X,S,e) \Rightarrow (Y,S',e')$ in $\mathcal{P}$ with two sets $J_A$ and $J_Q$ as follows
\[
J_A:=I+\sum_{c_0 \in K}   \mathtt{V}_{c_0} \times \widetilde{J}\big(\mathtt{rou}(c_0)\big)\quad\quad J_Q:=I'+\sum_{c_0 \in K} \mathtt{V}_{c_0} \times \Big( \sum_{z \in \mathrm{Obj}(\mathtt{Es}(c_0))}   \widetilde{J}\big(\mathtt{in}(c_0)(z)\big) \Big)
\]
where the set $\widetilde{J}(d)$ is defined for every $d \in D$ as the following sum, in which $\underline{c}$ denotes a tuple of the form $(c_1, \dots,c_n)$ in $K^n$ and $S(\underline{c})$ stands for the products of sets $S(c_1) \times \dots \times S(c_{n})$.
\[
\sum_{n \geq 1}\sum_{\underline{c} \in K^n}\sum_{\underline{s} \in S(\underline{c})}  D(d,\mathtt{rou}(c_1)) \times \prod_{i=1}^{n-1} D(\mathtt{ou}(c_i),\mathtt{rou}(c_{i+1}))
\]

The initial section $\textsc{\j}:I \to J_A$ is taken to be the canonical monomorphism.

\begin{remark}[Encoding]\label{rem:form_of_elt_I_J}
It will turn out to be convenient to have conventional notations for any element $\theta \in I$, $\theta' \in I'$, $\vartheta_A \in J_A$ or $\vartheta_Q \in J_Q$. In this respect, if one denotes 
\begin{itemize}
\item[-] by $\underline{c} = (c_1, \dots,c_n)$ any tuple of cones in $K^n$, for some positive integer $n$; 
\item[-] by $\underline{s} = (s_1,\dots s_{n})$ any tuple in $S(\underline{c})$, for some tuple of cones $\underline{c}$ as above;
\item[-] by $\underline{t}=(t_1,\dots,t_{n})$ any tuple of morphisms living in $\widetilde{J}(d)$ for some object $d$ in $D$;
\end{itemize}
the elements of the sets $I$, $I'$, $J_A\backslash I$ and $J_Q \backslash I'$ will be described as tuples of the form
\[
\theta:=(c_0,v)\quad\quad\theta':=(c_0,v,z) \quad\quad\vartheta_A:=(c_0,v,n,\underline{c},\underline{s},\underline{t})\quad\text{and}\quad\vartheta_Q:=(c_0,v,z,n,\underline{c},\underline{s},\underline{t})
\]
respectively, where $c_0 \in K$, $v \in \mathtt{V}_{c_0}$, $z \in \mathrm{Obj}(\mathtt{Es}(c_0))$ and, obviously, $n \geq 1$.
\end{remark}
Now, if one denotes by $\theta$, $\theta'$, $\vartheta_A$ and $\vartheta_Q$ any tuple of $I$, $I'$, $J_A$ and $J_Q$ as displayed in Remark \ref{rem:form_of_elt_I_J}, one defines the mappings $\upepsilon$, $\upchi$, $\upiota$, $\updelta$, $\upmu$ and $\upnu$ associated with the constructor structure of $(K,\mathtt{rou},\mathcal{P},\mathtt{V})$ as follows:
\smallskip

\begin{center}
\begin{tabular}{|ll|}
\hline
\multicolumn{2}{|c|}{\cellcolor[gray]{0.8}analytic exponent $\upepsilon$}\\
\hline
$\theta \mapsto \mathtt{rou}(c_0)$&on $I$\\
$\vartheta_A \mapsto \mathtt{ou}(c_n)$&otherwise\\
\hline
\multicolumn{2}{c}{}\\
\hline
\multicolumn{2}{|c|}{\cellcolor[gray]{0.8}analytic indicator $\upiota$}\\
\hline
$\theta \mapsto \theta$&on $I$\\
$\vartheta_A \mapsto \theta$&otherwise\\
\hline
\multicolumn{2}{c}{}\\
\hline
\multicolumn{2}{|c|}{\cellcolor[gray]{0.8}transitive analysor $\upmu$}\\
\hline
$\theta \mapsto \upalpha^{\circ}\upiota(\theta)$&on $I$\\
$\vartheta_A\mapsto L^n\mathbf{seed}(v)$&otherwise\\
\hline
\end{tabular}
\quad\quad
\begin{tabular}{c}
\end{tabular}
\quad\quad
\begin{tabular}{|ll|}
\hline
\multicolumn{2}{|c|}{\cellcolor[gray]{0.8}quotient exponent $\upchi$}\\
\hline
$\theta' \mapsto \mathtt{in}(c_0)(z)$&on $I'$\\
$\vartheta_Q \mapsto \mathtt{ou}(c_n)$&otherwise\\
\hline
\multicolumn{2}{c}{}\\
\hline
\multicolumn{2}{|c|}{\cellcolor[gray]{0.8}quotient indicator $\updelta$}\\
\hline
$\theta' \mapsto \theta$&on $I'$\\
$\vartheta_Q \mapsto \theta$&otherwise\\
\hline
\multicolumn{2}{c}{}\\
\hline
\multicolumn{2}{|c|}{\cellcolor[gray]{0.8}transitive quotientor $\upnu$}\\
\hline
$\theta' \mapsto \upomega\updelta(\theta')$&on $I'$\\
$\vartheta_Q \mapsto L^{n+1}\mathbf{stem}(v)$&otherwise\\
\hline
\end{tabular}
\end{center}
\smallskip

Finally, one produces a constructor of type $D[\mathcal{P},\mathcal{C}]$ by defining the analytic link $\ell_{\vartheta_A}$ as an identity map when $\vartheta_A \in I$, and, otherwise, as a compositional iteration of the form

\begin{equation}\label{eq:iteration_composition_l_for_link_A}
l_{c_n,s_n,t_n} \circ l_{c_{n-1},s_{n-1},t_{n-1}} \circ \dots \circ l_{c_1,s_1,t_1} (\_)
\end{equation}

where the triples $(c_1,s_1,t_1)$, $\dots$, $(c_{n},s_{n},t_{n})$ are made out of the obvious components of $\vartheta_A$ and the functor $l_{c_i,s_i,t_i}$ maps any commutative square as given below, on the left, to the commutative trapezoid given on the right, where $\varepsilon$ denotes the counit of the adjunction $L \vdash R$ and the component $t_i$ is, here, seen as an arrow of the form $\mathtt{ou}[i-1] \to \mathtt{rou}(c_{i})$ with $\mathtt{ou}[0]= \mathtt{rou}(c_0)$ and $\mathtt{ou}[i] = \mathtt{ou}(c_0)$ otherwise.
\[
\xymatrix{
A\ar[r]^-x\ar[d]_{\delta}&*+!L(.7){X\mathtt{ou}[i-1]}\ar[d]^{f\mathtt{ou}[i-1]}\\
B\ar[r]_-y&*+!L(.7){Y\mathtt{ou}[i-1]}
}
\quad\mapsto\quad
\xymatrix@C-0.5pc{
LA\ar[d]_{L\delta}\ar[rrr]^-{\varepsilon \circ L(e_{c_i,s_i} \circ X(t_i) \circ x)}&&&*+!L(.7){X\mathtt{rou}(c_{i})}\ar[d]^{f\mathtt{rou}(c_{i})}\\
LB\ar[rrr]_-{\varepsilon \circ L(e_{c_i,a_{c_i}(s_i)}' \circ Y(t_i) \circ  y)}&&&*+!L(.7){Y\mathtt{rou}(c_{i})}
}
\]

For its part, the quotient link $\ell_{\vartheta_Q}$, which is defined for every $\vartheta_Q \in J_Q$, is given by a first application of the functor $l_z$ that maps any commutative square as given below, on the left, to the commutative trapezoid given on the right, where $\varsigma_z$ is the universal projection of the adjunction $\Delta \vdash \mathrm{lim}$ at $z$,
\[
\xymatrix@C-1pc{
A\ar[r]^-x\ar[d]_{\delta}&*+!L(.7){\mathrm{lim}_zRX\mathtt{in}(c_0)(z)}\ar[d]^{\mathrm{lim}Rf\mathtt{ou}(c_0)}\\
B\ar[r]_-y&*+!L(.7){\mathrm{lim}_zRX\mathtt{in}(c_0)(z)}
}
\quad\mapsto\quad
\xymatrix@C-1pc{
LA\ar[d]_{L\delta}\ar[rrr]^-{\varepsilon \circ L(\varsigma_z \circ x)}&&&*+!L(.7){X\mathtt{in}(c_0)(z)}\ar[d]^{f\mathtt{in}(c_{0})(z)}\\
LB\ar[rrr]_-{\varepsilon \circ L(\varsigma_z \circ y)}&&&*+!L(.7){Y\mathtt{in}(c_0)(z)}
}
\]
and, in the case where $\vartheta_Q$ is not in $I'$, followed by successive iterations of the functor $l_{c_i,s_i,t_i}$ over the triples $(c_i,s_i,t_i)$ made out of the obvious components of $\vartheta_Q$ (see Formula \ref{eq:iteration_composition_l_for_link_A}). It is easy to check that the initial section $\textsc{\j}:I \to J_A$ satisfies the axioms of item 5) of Section \ref{sssec:Constructor}. The constructor associated with $(K,\mathtt{rou},\mathcal{P},\mathtt{V})$ will later be referred to as $\Upgamma^K$.

\begin{remark}\label{rem:J_A_could_have_been_I}
In the case where the associated maps $P\mathtt{rou}(c) \to RP\mathtt{ou}(c)$ of our premodels are identities, the functors $R$ and $\mathtt{rou}$ are trivial and the associated sets $S$ are all equal to a fixed one, the set $\widetilde{J}(d)$ can be set empty for every $d \in\{\mathtt{rou}(c_0),\mathtt{in}(c_0)(z)\}$ and $c_0 \in K$ so that $\textsc{\j}$ can be defined as an identity. 
In this case, the validity of our results still holds for Examples \ref{ex:functors_premodels} and  \ref{ex:presheaves_premodels}, but not for Examples \ref{ex:spectra_premodels} and \ref{ex:Prelicalised_ring}, which require $\widetilde{J}(d)$ to be as above. See Remark \ref{rem:trivial_effectiveness_functor_categories} and the proof of Theorem \ref{th:admissible_quotiented_effective} for more insight.
\end{remark}

\begin{remark}[Encoding]\label{rem:encoding_of_analytic_functors}
For every arrow $f:X \to Y$ in $\mathcal{B}$ and object $d$ in $D$, the image of the analytic functor $\Upgamma^K_A(f)(d)$ contains the tuples $(c_0,v,t,s',\mathbf{c})$ and the tuples $(c_0,v,n,\underline{c},\underline{s},\underline{t},t,s',\mathbf{c})$ where: $c_0$ is a cone in $K$; $v$ is a vertebra in $\mathtt{V}_{c_0}$; $n$ is a positive integer; $\underline{c}$, $\underline{s}$ and $\underline{t}$ are the tuples defined in Remark \ref{rem:form_of_elt_I_J} and used to define the analytic link; $t$ is an arrow in $D$ of the form $\mathtt{rou}(c_0) \to d$ for the first type of tuple and an arrow $\mathtt{ou}(c_n) \to d$ otherwise; $s'$ is an element in $S(c_0)$ and $\mathbf{c}:\mathbf{disk}(v) \Rrightarrow \Phi_{\theta}(f)_{s'}$ is an arrow in $\mathcal{C}^{\mathbf{2}\times \mathbf{2}}$ as displayed in Equation (\ref{eq:commutative_cuboid_definition_playground}) for the notation $\theta:=(c_0,v)$.

Similarly, the image of the quotient functor $\Upgamma^K_Q(f)(d)$ contains the tuples $(c_0,v,z,t,s',\mathbf{s})$ and the tuples $(c_0,v,n,\underline{c},\underline{s},\underline{t},z,t,s',\mathbf{s})$ where: $c_0$ is a cone in $K$; $v$ is a vertebra in $\mathtt{V}_{c_0}$; $n$ is a natural number; $\underline{c}$, $\underline{s}$ and $\underline{t}$ are the tuples defined in Remark \ref{rem:form_of_elt_I_J} and used to define the quotient link; $z$ is an object of $\mathtt{Es}(c_0)$; $t$ is an arrow in $D$ of the form $\mathtt{in}(c_0)(z) \to d$ for the first type of tuple and an arrow $\mathtt{ou}(c_n) \to d$ otherwise; $s'$ is an element in $S(c_0)$ and $\mathbf{s}$ is an arrow $\mathbf{stem}(v) \Rightarrow \Psi_{\theta}(f)$ in $\mathcal{C}^{\mathbf{2}}$ for the notation $\theta:=(c_0,v)$.
\end{remark}

\begin{remark}[Encoding]\label{rem:encoding_of_analytic_functors_gen_shape}
It is not hard to see from Remark \ref{rem:encoding_of_analytic_functors} that any type of tuple in $\Upgamma^K_A(f)(d)$ may be written as a tuple of the form $(\vartheta, t,s',\mathbf{c})$ where the encoding of the parameter $\vartheta$ may vary.
Similarly, it follows from  Remark \ref{rem:encoding_of_analytic_functors} that any tuple in $\Upgamma^K_Q(f)(d)$ may be written as a tuple $(\vartheta, t,s',\mathbf{s})$ where the encoding of the parameter $\vartheta$ may vary.
\end{remark}

\subsection{Tomes of a Constructor}\label{ssec:Tomes_of_a_constructor}

Let $\Upgamma$ denote a constructor of type $D[\mathcal{B},\mathcal{C}]$ as defined in Section \ref{sssec:Constructor}. This section shows that $\Upgamma$ may be associated with a variety of canonical tomes, each of them being used for specific purposes. The first one, called the \emph{operadic tome}, is meant to be used in the small object argument (see Section \ref{sec:From_narratives_to_combinatorial_cat}) and is constructed out of the preconstructor structure of $\Upgamma$ as follows: For every object $\theta \in I$, arrow $f:X \to Y$ in $\mathcal{B}$ and $s \in \uplambda_{\theta}(X)$, it is given by the functor $\varrho_{\theta}^s:\Upgamma_A(f)\{\theta\}_s \to (\mathcal{C}^{\mathbf{2}})^{\mathbf{2}}/\Phi_{\theta}(f)_s$ defined by the following inclusion.
\[
\begin{array}{ccccc}
\mathcal{C}^{\mathbf{2} \times \mathbf{2}}(\upalpha(\theta),\Phi_{\theta}(f)_s) 
&\to&(\mathcal{C}^{\mathbf{2}})^{\mathbf{2}}/\Phi_{\theta}(f)_s
\\
\mathbf{c}&\mapsto&\mathbf{c}
\end{array}
\]

A second tome, called the \emph{analytic tome}, is given by a functor $\varphi_A:\Upgamma_A(f)(d) \to \mathcal{C}^{\mathbf{2}}/\Upsilon_{d}(f)$ and is defined on each term of $\Upgamma_A(f)(d)$---which denoted as $\mathtt{t}_{\vartheta,s}$ below---as follows.
\begin{align*}
\mathtt{t}_{\vartheta,s} & =  D(\upepsilon(\vartheta),d) \times \mathcal{C}^{\mathbf{2} \times \mathbf{2}}(\upalpha\upiota(\vartheta),\Phi_{\upiota(\vartheta)}(f)_s)&\text{(definition)}\\
& \to D(\upepsilon(\vartheta),d) \times\mathcal{C}^{\mathbf{2}}(\upalpha^{\circ}\upiota(\vartheta),\Phi_{\upiota(\vartheta)}^{\circ}(f)_s)&\text{(see Convention \ref{conv:circ_bullet_source_target})}\\
& \to D(\upepsilon(\vartheta),d) \times \mathcal{C}^{\mathbf{2}}(\upalpha^{\circ}\upiota(\vartheta),\Upsilon_{\uprho\upiota(\vartheta)}(f))&\text{(definition of $\Upgamma$)}\\
& \to D(\upepsilon(\vartheta),d) \times \mathcal{C}^{\mathbf{2}}(\upmu(\vartheta),\Upsilon_{\upepsilon(\vartheta)}(f))&\text{(analytic link)}\\
& \to \mathcal{C}^{\mathbf{2}}(\Upsilon_{\upepsilon(\vartheta)}(f),\Upsilon_{d}(f)) \times \mathcal{C}^{\mathbf{2}}(\upmu(\vartheta),\Upsilon_{\upepsilon(\vartheta)}(f))&\text{(func. $ \Upsilon_{\_}(f)$)}\\
& \to \mathcal{C}^{\mathbf{2}}(\upmu(\vartheta),\Upsilon_{d}(f))&\text{(comp. of $\mathcal{C}^{\mathbf{2}}$)}\\
& \to \mathcal{C}^{\mathbf{2}}/\Upsilon_{d}(f)&\text{(inclusion)}
\end{align*}

Explicitly, the functor maps any tuple $(\vartheta,t,s,\mathbf{c})$ in $\Upgamma^K_A(f)(d)$ (see Remark \ref{rem:encoding_of_analytic_functors_gen_shape}) to the composite arrow given, below, by Equation  (\ref{eq:analytic_tome_formula}) in $\mathcal{C}^{\mathbf{2}}$.
\begin{equation}\label{eq:analytic_tome_formula}
\Upsilon_t(f) \circ \ell_{\vartheta}(\mathbf{c}^{\circ}):\upmu(\vartheta) \Rightarrow \Upsilon_{d}(f)
\end{equation}

A third tome, called the \emph{quotient tome}, is given by a functor $\varphi_Q:\Upgamma_Q(f)(d) \to \mathcal{C}^{\mathbf{2}}/\Upsilon_{d}(f)$
and is defined on each term of $\Upgamma_Q(f)(d)$ -- which denoted as $\mathtt{t}_{\vartheta,s}$ below -- as follows.
\begin{align*}
\mathtt{t}_{\vartheta,s}& =  D(\upchi(\vartheta),d) \times \mathcal{C}^{\mathbf{2}}(\upomega\delta(\vartheta),\Phi_{\delta(\vartheta)}^{\bullet}(f)_s)&\text{(definition)}\\
& =  D(\upchi(\vartheta),d) \times \mathcal{C}^{\mathbf{2}}(\upomega\delta(\vartheta),\Psi_{\delta(\vartheta)}(f))&\text{(def. of $\Upgamma$)}\\
& \to   D(\upchi(\vartheta),d) \times \mathcal{C}^{\mathbf{2}}(\upnu(\vartheta), \Upsilon_{\upchi(\vartheta)}(f))&\text{(quotient link)}\\
& \to  \mathcal{C}^{\mathbf{2}}( \Upsilon_{\upchi(\vartheta)}(f), \Upsilon_{d}(f)) \times \mathcal{C}^{\mathbf{2}}(\upnu(\vartheta), \Upsilon_{\upchi(\vartheta)}(f))&\text{(func. $\Upsilon_{\_}(f)$)}\\
& \to  \mathcal{C}^{\mathbf{2}}(\upnu(\vartheta),\Upsilon_{d}(f))&\text{(comp of $\mathcal{C}^{\mathbf{2}}$)}\\
& \to  \mathcal{C}^{\mathbf{2}}/\Upsilon_{d}(f)&\text{(inclusion)}
\end{align*}

Explicitly, the quotient tome $\varphi_Q$ maps any tuple $(\vartheta,t,s,\mathbf{s})$ in $\Upgamma^K_Q(f)(d)$ (see Remark \ref{rem:encoding_of_analytic_functors_gen_shape}) to the composite arrow given, below, by Equation  (\ref{eq:quotient_tome_formula}) in $\mathcal{C}^{\mathbf{2}}$.
\begin{equation}\label{eq:quotient_tome_formula}
\Upsilon_t(f) \circ \ell_{\vartheta}(\mathbf{s}):\upnu(\theta) \Rightarrow \Upsilon_{d}(f)
\end{equation}

The proofs of the following propositions follow from the previous definitions.

\begin{proposition}\label{prop:operadic_tome_functorial}
The operadic tome $\varrho_{\theta}^s:\Upgamma_A(f)\{\theta\}_s \to \mathcal{C}^{\mathbf{2}\times \mathbf{2}}/\Phi_{\theta}(f)_s$ is natural in the variable $f \in \mathcal{B}^{\mathbf{2}}$.
This amounts to saying that the mapping $f \mapsto (\Phi_{\theta}(f)_s,\Upgamma_A(f)[\theta],\varrho_{\theta}^s)$ induces a functor $\mathbb{T}^{\mathrm{op}}_{\theta,s}:\mathcal{B}^{\mathbf{2}} \to \mathbf{Tome}(\mathcal{C})$.
\end{proposition}

\begin{proposition}\label{prop:analytic_tome_functorial}
The analytic tome $\varphi_A:\Upgamma_A(f)(d) \to \mathcal{C}^{\mathbf{2}}/\Upsilon_{d}(f)$ is natural in the variable $d \in D$.
This amounts to saying that the mapping $d \mapsto (\Upsilon_{d}(f),\Upgamma_A(f)(d),\varphi_A)$ induces a functor $\mathbb{T}^{\mathrm{an}}_f:D \to \mathbf{Tome}(\mathcal{C})$.
\end{proposition}

\begin{proposition}\label{prop:quotient_tome_functorial}
The quotient tome $\varphi_Q:\Upgamma_Q(f)(d) \to \mathcal{C}^{\mathbf{2}}/\Upsilon_{d}(f)$ is natural in the variable $d \in D$. This amounts to saying that the mapping $d \mapsto (\Upsilon_{d}(f),\Upgamma_Q(f)(d),\varphi_Q)$ induces a functor $\mathbb{T}^{\mathrm{qu}}_f:D \to \mathbf{Tome}(\mathcal{C})$. 
\end{proposition}

\subsection{Quotiented Arrows}\label{ssec:Quotiented-arrows}
Let $\Upgamma$ denote a constructor of type $D[\mathcal{B},\mathcal{C}]$ as defined in Section \ref{sssec:Constructor}. This section defines the concept of `quotient' whose essential idea is to restrict the quotient family of $\Upgamma$ to certain parametrising `squares' only.
In this respect, a \emph{$\Upgamma$-quotient} for a morphism  $f:X \to Y$ in $\mathcal{B}$ consists of a collection of discrete categories, as given below, on the left, as well as a collection of functors as given on the right
\[
\{E_s(\theta)\}_{\theta \in I,s \in \uplambda_{\theta}(X)}\quad\quad\quad\quad\mathfrak{q}=\{\mathfrak{q}_{\theta}^s\{\_\}:E_s(\theta) \to \mathbf{Set}\}_{\theta \in I,s \in \uplambda_{\theta}(X)}
\]
such that the inclusion $\mathfrak{q}_{\theta}^s\{\upsilon\} \hookrightarrow \Upgamma_Q(f)\{\theta\}_s$ holds for every element $\upsilon\in E_s(\theta)$. 
We may associate any such $\Upgamma$-quotient $\mathfrak{q}$ with a functor $\mathfrak{q}[\_]:I\to \mathbf{Set}$ defined as follows for every $\theta \in I$.
\[
\mathfrak{q}[\theta]:=\sum_{s \in \uplambda_{\theta}(X)}\sum_{\upsilon \in E_s(\theta)} \mathfrak{q}_{\theta}^s\{\upsilon\}
\]

This functor will be called the \emph{species of $\mathfrak{q}$}. In much the same fashion as the quotient species of $\Upgamma$ was used to define its quotient functor, we use the species of $\mathfrak{q}$ to define a third functor $\mathfrak{q}:D \to \mathbf{Set}$ given by the following formula.
\[
\mathfrak{q}(d):=\sum_{\vartheta \in J_Q} D(\upchi(\vartheta),d) \times \mathfrak{q}[\updelta(\vartheta)]
\]

This functor will be referred to as the \emph{quotienting functor of $\mathfrak{q}$}.

\begin{proposition}\label{prop:puotienting_functor_subfunctor}
The inclusions $\mathfrak{q}_{\theta}^s\{\upsilon\} \hookrightarrow \Upgamma_Q(f)\{\theta\}_s$ holding for every $\upsilon\in E_s(\theta)$ and $s \in \uplambda_{\theta}(X)$ induce functions of the form $\mathfrak{q}[\theta] \to \Upgamma_Q(f)[\theta]$ for every $\theta \in I$, which in turn induce a morphism
$\mathfrak{q}(\_) \Rightarrow \Upgamma_Q(f)(\_)$ in $\mathbf{Set}^D$.
\end{proposition}
\begin{proof}
By universality of the coproducts.
\end{proof}

\begin{convention}\label{conv:use_q_tilde_instead_of_q}
The natural transformation of Proposition \ref{prop:puotienting_functor_subfunctor} may be composed with the quotient tome $\varphi_Q$ of $\Upgamma$ to give a natural transformation $\varphi_Q:\mathfrak{q} \Rightarrow \mathcal{C}^{\mathbf{2}}\rotatebox[origin=c]{-90}{$\multimap$}\Upsilon(f)$.
Because this arrow lives in the functor category $\mathbf{Set}^D$, it may be factorised into an epimonomorphism followed by a monomorphism as follows (this is an image factorisation).
\begin{equation}\label{eq:image_factorisation_quotient_of_q}
\xymatrix{
\mathfrak{q} \ar@{=>}[r]^-{\textit{epi.}}& \tilde{\mathfrak{q}}\ar@{=>}[r]^-{\tilde{\varphi}_Q}& \mathcal{C}^{\mathbf{2}}\rotatebox[origin=c]{-90}{$\multimap$}\Upsilon(f)
}
\end{equation}

For every object $d \in D$, the image $\tilde{\mathfrak{q}}(d)$ will be thought of as the set $\mathfrak{q}(d)$, but quotiented by the obvious binary relation. In any case, the elements of $\tilde{\mathfrak{q}}(d)$ and $\mathfrak{q}(d)$ will be denoted as tuples $(\vartheta,t,s,\upsilon,\mathbf{s})$ where $t$ is an arrow of the form $\upchi(\vartheta) \to d$; $s$ is an element in $\uplambda_{\updelta(\vartheta)}(X)$; $\upsilon$ is an element in $E_{s}(\updelta(\vartheta))$ and $\mathbf{s}$ is an element in $\mathfrak{q}_{\theta}^s\{\upsilon\} \hookrightarrow \Upgamma_Q(f)\{\theta\}_s$.
\end{convention}

\begin{remark}[In preparation for Theorem \ref{th:admissible_quotiented_effective}]\label{remark:quotient_are_always_admissible}
Let $f:X \to Y$ be a morphism in $\mathcal{B}$ as above. For every object $d\in D$, denote by $D(J_Q,d)$ the following sum of sets, which is defined with respect to the structure of $f$ provided by the constructor $\Upgamma$.
\[
\sum_{\vartheta \in J_Q}D(\upchi(\vartheta),d)
\]

The definition of $\Upgamma$-quotient for $f:X \to Y$ implies that any function of the form
$h:D(J_Q,d) \to D(J_Q,d')$ that maps a pair $(\vartheta,t)$ in $D(J_Q,d)$ to a pair $(\vartheta',t')$ in $D(J_Q,d')$ so that the equality $\updelta(\vartheta) = \updelta(\vartheta')$ is satisfied
lifts to a function $h:\mathfrak{q}(d) \to \mathfrak{q}(d')$ mapping any tuple
$(\vartheta,t,s,\upsilon,\mathbf{s})$ in $\mathfrak{q}(d)$ to the tuple $(h(\vartheta,t),s,\upsilon,\mathbf{s})$ in $\mathfrak{q}(d')$.
\end{remark}

\begin{example}[In preparation for Theorem \ref{th:admissible_quotiented_effective}]\label{example:quotient_are_always_admissible_h_c_s}
In the case of a constructor $\Upgamma^K$ associated with a system of $R$-premodels $(K,T,\mathcal{P},\mathtt{V})$ over a small category $D$ in a category $\mathcal{C}$, the disjoint sum
$D(J_Q,d)$ associated with a morphism $(f,a):(X,S,e) \Rightarrow (Y,S',e')$ in $\mathcal{P}$ contains two types of tuples, which are of the form $(c_0,v,z,t)$ and $(c_0,v,n,\underline{c},\underline{s},\underline{t},z,t)$ with respect to the same notations given in Remark \ref{rem:encoding_of_analytic_functors}. For every $c \in K$ and $s \in S(c)$, if one takes $r$ to be $\mathtt{rou}(c)$ and $d_0$ to be $\mathtt{ou}(c)$, then it is possible to define a function $h_{c,s}:D(J_Q,r) \to D(J_Q,d_0)$ with the following mapping rules, where $\underline{c}c$ stands for $(c_1,\dots,c_n,c)$, $\underline{s}s$ stands for $(s_0,\dots,s_{n-1},s)$ and $\underline{t}t$ stands for $(t_0,\dots,t_{n-1},t)$.
\[
\begin{array}{lcccr}
h_{c,s}:&D(J_Q,r) &\to& D(J_Q,d_0)\\
&(c_0,v,z,t) &\mapsto & (c_0,v,1,c,s,t,z,\mathrm{id}_{d_0})\\
&(c_0,v,n,\underline{c},\underline{s},\underline{t},z,t) &\mapsto & (c_0,v,n+1,\underline{c}c,\underline{s}s,\underline{t}t,z,\mathrm{id}_{d_0})
\end{array}
\]

Because the following equations hold, it follows from Remark \ref{remark:quotient_are_always_admissible} that the function $h_{c,s}:D(J_Q,r) \to D(J_Q,d_0)$ extends to a function $\mathfrak{q}(r) \to \mathfrak{q}(d_0)$.
\[
\begin{array}{c}
\updelta(c_0,v,z,t) = (c_0,v) = \updelta(c_0,v,1,c,s,t,z,\mathrm{id}_{d_0})\\
\updelta(c_0,v,n,\underline{c},\underline{s},\underline{t},z,t) = (c_0,v) = \updelta(c_0,v,n+1,\underline{c}c,\underline{s}s,\underline{t}t,z,\mathrm{id}_{d_0})
\end{array}
\]

In fact, the function $h_{c,s}:\mathfrak{q}(r) \to \mathfrak{q}(d_0)$ also restricts to a function $h_{c,s}:\tilde{\mathfrak{q}}(r) \to \tilde{\mathfrak{q}}(d_0)$. To see this, take two tuples
$\mathbf{x}_*:=(\vartheta_*,t_*,s'_*,\upsilon_*,\mathbf{s}_*)$ and $\mathbf{x}_{\dagger}:=(\vartheta_{\dagger},t_{\dagger},s'_{\dagger},\upsilon_{\dagger},\mathbf{s}_{\dagger})$ in $\mathfrak{q}(r)$ that are equivalent in $\tilde{\mathfrak{q}}(r)$, that is to say that have the same image under $\varphi_Q$ (see below, according to Formula (\ref{eq:quotient_tome_formula})).
\[
\Upsilon_{t_*}(f) \circ \ell_{\vartheta_*}(\mathbf{s}_*) = \Upsilon_{t_{\dagger}}(f) \circ \ell_{\vartheta_{\dagger}}(\mathbf{s}_{\dagger})
\]

It follows that their images via $h_{c,s}:\tilde{\mathfrak{q}}(r) \to \tilde{\mathfrak{q}}(d_0)$ are also equivalent in $\tilde{\mathfrak{q}}(d_0)$. This comes from the fact that the previous equation gives rise to the following one, after some obvious compositional operations on it (see the definitions for $l_{c,s,t_*}$ and $l_{c,s,t_{\dagger}}$ in Section \ref{sssec:Constructor_system_premodels}).
\[
\Upsilon_{\mathrm{id}}(f) \circ l_{c,s,t_*} \circ \ell_{\vartheta_*}(\mathbf{s}_*) = \Upsilon_{\mathrm{id}}(f) \circ l_{c,s,t_{\dagger}} \circ \ell_{\vartheta_{\dagger}}(\mathbf{s}_{\dagger})
\]

However, this last equation also amounts to saying that the images of $h_{c,s}(\mathbf{x}_*)$ and $h_{c,s}(\mathbf{x}_{\dagger})$ via $\varphi_Q$ are the same, and thus shows that $h_{c,s}$ restricts to a function $\tilde{\mathfrak{q}}(r) \to \tilde{\mathfrak{q}}(d_0)$.
\end{example}

\begin{definition}[Quotiented arrows]\label{def:quotiented_arrow}
From now on, we shall speak of a \emph{$\Upgamma$-quotiented arrow} in $\mathcal{B}$ to refer to any arrow $f:X \to Y$ in $\mathcal{B}$ that is equipped with a $\Upgamma$-quotient $\mathfrak{q}$ for $f$.
\end{definition}
A $\Upgamma$-quotiented arrow as defined above will be denoted either as a pair $(f,\mathfrak{q})$ or as a paired arrow $(f,\mathfrak{q}):X \to Y$. A \emph{morphism of $\Upgamma$-quotiented arrows}, denoted as an arrow $(f,\mathfrak{q}) \Rightarrow (g,\mathfrak{p})$, will be understood as a morphism $f \Rightarrow g$ in $\mathcal{B}^{\mathbf{2}}$. The category of $\Upgamma$-quotiented arrows and their morphisms will be denoted by $\Upgamma\mathcal{B}^{\mathbf{2}}$.

\subsection{Merolytic Functors and Their Tomes}
Let $\Upgamma$ denote a constructor of type $D[\mathcal{B},\mathcal{C}]$ as defined in Section \ref{sssec:Constructor} where $\mathcal{C}$ has coproducts.
For every $\Upgamma$-quotiented arrow $(f,\mathfrak{q}):X \to Y$, define the \emph{merolytic functor of $(f,\mathfrak{q})$} as the coproduct of functors given below.
\[
\Upgamma^{f,\mathfrak{q}}(d) := \Upgamma_A(f)(d) + \tilde{\mathfrak{q}}(d)
\]

Then, define the \emph{merolytic tome} of $(f,\mathfrak{q}):X \to Y$ as the coproduct $\varphi^{\mathfrak{q}}:\Upgamma^{f,\mathfrak{q}}(d) \to \mathcal{C}^{\mathbf{2}}/\Upsilon_d(f)$ of the following cospan whose right leg is given by the rightmost arrow of Equation (\ref{eq:image_factorisation_quotient_of_q}).
\[
\xymatrix{
\Upgamma_A(f)(d)\ar[rr]^-{\varphi_A}&& \mathcal{C}^{\mathbf{2}}/\Upsilon_d(f)&&\tilde{\mathfrak{q}}(d)\ar[ll]_-{\tilde{\varphi}_Q}
}
\]
\begin{proposition}\label{prop:merolytic_tome_functorial_D}
For every $(f,\mathfrak{q}) \in \Upgamma\mathcal{B}^{\mathbf{2}}$, the merolytic tome
$\varphi^{\mathfrak{q}}:\Upgamma^{f,\mathfrak{q}}(d) \to \mathcal{C}^{\mathbf{2}}/\Upsilon_d(f)$
is natural in the variable $d \in D$.
This amounts to saying that the mapping rule $d \mapsto (\Upsilon_{d}(f),\Upgamma^{f,\mathfrak{q}}(d),\varphi^{\mathfrak{q}})$ induces a functor $\mathbb{T}^{f,\mathfrak{q}}:D \to \mathbf{Tome}(\mathcal{C})$.
\end{proposition}
\begin{proof}
Follows from Propositions \ref{prop:analytic_tome_functorial},  \ref{prop:quotient_tome_functorial} and \ref{prop:puotienting_functor_subfunctor}.
\end{proof}

\begin{proposition}\label{prop:merolytic_tome_functorial}
For every $d \in D$, the mapping $(f,\mathfrak{q}) \mapsto (\Upsilon_{d}(f),\Upgamma^{f,\mathfrak{q}}(d),\varphi^{\mathfrak{q}})$ induces a functor $\mathbb{T}_d^{\mathrm{me}}:\Upgamma\mathcal{B}^{\mathbf{2}} \to \mathbf{LTom}(\mathcal{C})$.
\end{proposition}
\begin{proof}
According to the definition of Section \ref{ssec:Morphisms_of_tomes}, it is sufficient to assign any arrow $(\eta_0,\eta_1):(f,\mathfrak{q}) \Rightarrow (f',\mathfrak{q}')$ in $\Upgamma\mathcal{B}^{\mathbf{2}}$ to the arrow $(\Upsilon_d(\eta_0),\Upsilon_d(\eta_1)):\Upsilon_d(f) \Rightarrow \Upsilon_d(f')$ in $\mathcal{C}^{\mathbf{2}}$. This mapping is functorial by functoriality of $\Upsilon_d:\mathcal{B} \to \mathcal{C}$.
\end{proof}

Because the tome $\mathbb{T}^{f,\mathfrak{q}}(d)$ is functorial in $d \in D$, so is its content $\mathrm{col}\partial \varphi^{\mathfrak{q}} \Rightarrow \Upsilon_{d}(f)$ (see Section \ref{sssec:Tomes}). In other words, the content gives us a commutative diagram in $\mathcal{C}^D$ as follows.
\[
\xymatrix{
*+!R(.5){\mathrm{col} \mathbf{A}}\ar@{=>}[r]^-{\mathrm{col}u}\ar@{=>}[d]_{\mathrm{col}\partial\varphi^{\mathfrak{q}}}&\Upsilon(X)
\ar@{=>}[d]^{\Upsilon(f)}\\
*+!R(.5){\mathrm{col} \mathbf{B}}\ar@{=>}[r]_-{\mathrm{col}v}&\Upsilon(Y)
}
\]

The previous diagram will be referred to as the \emph{functorial content of $\mathbb{T}^{f,\mathfrak{q}}$}.

\subsection{Effectiveness of Quotiented Arrows}\label{sssec:Combinatorial_constructors}
The goal of this section is to introduce what logicians could see as a concept of definability. The concept of effectiveness will allow us to designate those arrows that can be equipped with well-defined pushout factorisations in the category associated with a constructor. We prepare the notion of effectiveness by introducing the (almost-trivial) concept of realisability.
Let $\Upgamma$ denote a constructor of type $D[\mathcal{B},\mathcal{C}]$ as defined in Section \ref{sssec:Constructor} where $\mathcal{C}$ has coproducts.
A $\Upgamma$-quotiented arrow $(f,\mathfrak{q}):X \to Y$ in $\mathcal{B}$ will be said to be \emph{$\Upgamma$-realised} if one may form a componentwise pushout square inside the functorial content of its merolytic tome as shown below. 
\begin{equation}\label{eq:combinatorial_constructors_pushout_def}
\xymatrix{
*+!R(.5){\mathrm{col} \mathbf{A}}\ar@{}[rd]|->{\text{\huge{\rotatebox[origin=c]{-90}{$\llcorner$}}}}\ar@{=>}[r]^-{\mathrm{col}u}\ar@{=>}[d]_{\mathrm{col}\partial\varphi^{\mathfrak{q}}}&\Upsilon(X)\ar@/^1pc/@{=>}[rd]^{\Upsilon(f)}\ar@{=>}[d]^{p^{\mathfrak{q}}_f}&\\
*+!R(.5){\mathrm{col} \mathbf{B}}\ar@/_1.8pc/@{=>}[rr]_-{\mathrm{col}v}\ar@{=>}[r]_{\pi^{\mathfrak{q}}_f}&[f,\mathfrak{q}]\ar@{=>}[r]^{h^{\mathfrak{q}}_{f}}&\Upsilon(Y)
}
\end{equation}

The functor $d \mapsto [f,\mathfrak{q}](d)$ will then be called the \emph{$\Upgamma$-realisation of $(f,\mathfrak{q})$} while the pair of arrows $(p_f^{\mathfrak{q}},h_f^{\mathfrak{q}})$ will be referred to as the \emph{$\Upgamma$-prefactorisation} of $(f,\mathfrak{q})$. 

\begin{definition}[Effectiveness]\label{def:effectiveness_quotiented_models}
Let $\Upgamma$ denote a constructor of type $D[\mathcal{B},\mathcal{C}]$ as defined in Section \ref{sssec:Constructor}.
A $\Upgamma$-quotiented arrow $(f,\mathfrak{q}):X \to Y$ in $\mathcal{B}$ will be said to be \emph{effective} if it is $\Upgamma$-realised and its $\Upgamma$-prefactorisation in $\mathcal{C}^D$ lifts to a factorisation of $f:X \to Y$ in $\mathcal{B}$, as shown in Equation 
(\ref{eq:lifting_definition_combinatorial_constructor}), such that the arrow $\uplambda_{\theta}(\{f\}_{\mathfrak{q}}):\uplambda_{\theta}(X)  \to \uplambda_{\theta}([f/\mathfrak{q}])$ is an identity for every $\theta \in I$.
\begin{equation}\label{eq:lifting_definition_combinatorial_constructor}
\xymatrix{
X\ar@/^1.8pc/[rr]^-{f}\ar[r]_-{\{ f\}_{\mathfrak{q}}}&[f/\mathfrak{q}]\ar[r]_-{\lfloor f \rfloor_{\mathfrak{q}}}&Y
}
\quad\quad\mathop{\longmapsto}\limits^{\Upsilon}\quad\quad
\xymatrix{
\Upsilon(X)\ar@/^1.8pc/@{=>}[rr]^-{\Upsilon(f)}\ar@{=>}[r]_-{p^{\mathfrak{q}}_f}&[f,\mathfrak{q}]\ar@{=>}[r]_-{h^{\mathfrak{q}}_f}&\Upsilon(Y)
}
\end{equation}
\end{definition}

The leftmost factorisation of Equation (\ref{eq:lifting_definition_combinatorial_constructor}) will be called the \emph{$\Upgamma$-factorisation of $(f,\mathfrak{q})$}.

\begin{remark}\label{rem:trivial_effectiveness_functor_categories}
Let $S_0$ be a given set and $\Upgamma^K$ be the constructor of a system of $R$-premodels $(K,T,\mathcal{P},\mathtt{V})$ over a small category $D$ in a category $\mathcal{C}$ where every object $(X,S,e)$ in $\mathcal{P}$ is such that $S$ is equal to $S_{0}$ and $e$ is made of identities only. In this case, the underlying functor $\Upsilon:\mathcal{P} \to \mathcal{C}^D$ is fully faithful and it follows that if $\mathcal{C}$ has pushouts, then every $\Upgamma^K$-quotiented arrow in $\mathcal{P}$ is effective. This means that the theorem given below becomes trivial, which explains why the set $\widetilde{J}(d)$ mentioned in Remark \ref{rem:J_A_could_have_been_I} may be set empty since it is not really needed anywhere else in the paper except for Theorem \ref{th:admissible_quotiented_effective} (and Theorem \ref{th:admissible_quotiented_effective_model}, which is a copy of it). See Example \ref{exa:functor_category_fibered} in the case where $\widetilde{J}(d)$ is defined as in Section \ref{sssec:Constructor_system_premodels}. 
\end{remark}

\begin{theorem}\label{th:admissible_quotiented_effective}
Let $(K,\mathtt{rou},\mathcal{P},\mathtt{V})$ be a system of $R$-premodels over a small category $D$ in a category $\mathcal{C}$. If $\mathcal{C}$ has pushouts and the inclusion $\mathcal{P} \hookrightarrow \mathbf{Pr}_{\mathcal{C}}(K,\mathtt{rou},R)$ is an identity, then every $\Upgamma^K$-quotiented arrow in $\mathcal{P}$ is effective. 
\end{theorem}
\begin{proof}
For convenience, the symbol $\Upgamma^K$ will be shortened to $\Upgamma$.
Since $\mathcal{C}$ has pushouts, every $\Upgamma$-quotiented arrow is $\Upgamma$-realised by definition. Let $(f,a,\mathfrak{q}):(X,S,e) \Rightarrow (Y,S',e')$ be an $\Upgamma$-quotiented arrow in $\mathcal{B}$. We are going to prove that the $\Upgamma$-realisation of $(f,\mathfrak{q})$ has an $R$-premodel structure of the form $([f,\mathfrak{q}],S,e^{\mathfrak{q}})$ and that this structure lifts the $\upgamma$-prefactorisation of $f:X \Rightarrow Y$ in $\mathcal{C}^D$ to another one in $\mathcal{P}$.
In this respect, fix $c \in K$ and $s \in S(c)$ and denote $\mathtt{rou}(c)$ and $\mathtt{ou}(c)$ by $r$ and $d_0$, respectively. For simplicity, we will denote by $e_{c,s}(f)$ the obvious morphism $f(r) \Rightarrow Rf(d_0)$ in $\mathcal{C}^{\mathbf{2}}$ whose components are given by the pair of arrows $e_{c,s}$ and $e_{c,a_c(s)}'$ in $\mathcal{C}$.

To prove the statement, we first need to define two functors. The first one is of the form 
$\upzeta_{c,s}:\Upgamma^{f,\mathfrak{q}}(r) \to \Upgamma^{f,\mathfrak{q}}(d_0)$ and is induced by the following mappings, where $\underline{c}c$ stands for $(c_1,\dots,c_n,c)$, $\underline{s}s$ stands for $(s_0,\dots,s_{n-1},s)$ and $\underline{t}t$ stands for $(t_0,\dots,t_{n-1},t)$.
\[
\begin{array}{lcccr}
\upzeta_{c,s}:&\Upgamma^{f,\mathfrak{q}}(r) &\to& \Upgamma^{f,\mathfrak{q}}(d_0)&\\
&(c_0,v,t,s',\mathbf{c}) &\mapsto & (c_0,v,1,c,s,t,\mathrm{id}_{d_0},s',\mathbf{c})&\text{on $\Upgamma_A(r)$}\\ 
&(c_0,v,z,t,s',\upsilon,\mathbf{s}) &\mapsto & (c_0,v,1,c,s,t,z,\mathrm{id}_{d_0},s',\upsilon,\mathbf{s})&\text{on $\tilde{\mathfrak{q}}(r)$}\\
&(c_0,v,n,\underline{c},\underline{s},\underline{t},t,s',\mathbf{c}) &\mapsto & (c_0,v,n+1,\underline{c}c,\underline{s}s,\underline{t}t,\mathrm{id}_{d_0},s',\mathbf{c})&\text{on $\Upgamma_A(r)$}\\ 
&(c_0,v,n,\underline{c},\underline{s},\underline{t},z,t,s',\upsilon,\mathbf{s}) &\mapsto & (c_0,v,n+1,\underline{c}c,\underline{s}s,\underline{t}t,z,\mathrm{id}_{d_0},s',\upsilon,\mathbf{s})&\text{on $\tilde{\mathfrak{q}}(r)$}
\end{array}
\]

Note that the mappings on $\tilde{\mathfrak{q}}(r)$ have already been given in Example \ref{example:quotient_are_always_admissible_h_c_s}.
The second functor is of the form $\upxi_{c,s}:\mathcal{C}^{\mathbf{2}}/\Upsilon_{r}(f) \to \mathcal{C}^{\mathbf{2}}/\Upsilon_{d_0}(f)$ and maps any arrow $x:\delta \Rightarrow f(r)$ to the map $\varepsilon \circ L(e_{c,s}(f) \circ x):L\delta \Rightarrow f(d_0)$, where $\varepsilon$ denotes the unit of the adjunction $L \vdash R$.

We are now going to show that the following diagram commutes.
\begin{equation}\label{eq:quotiented_arrows_are_effective}
\xymatrix{
\Upgamma^{f,\mathfrak{q}}(r)\ar[r]^{\upzeta_{c,s}}\ar[d]_{\varphi^{\mathfrak{q}}_{r}}&\ar[d]^{\varphi^{\mathfrak{q}}_{d_0}}\Upgamma^{f,\mathfrak{q}}(d_0)\\
\mathcal{C}^{\mathbf{2}}/\Upsilon_{r}(f) \ar[r]^-{\upxi_{c,s}}& \mathcal{C}^{\mathbf{2}}/\Upsilon_{d_0}(f)
}
\end{equation}

On the set $\Upgamma_A(r)$, the calculation on a tuple $\underline{x} = (c_0,v,t,s',\mathbf{c})$ goes as follows.
\begin{align*}
\upxi_{c,s}\varphi^{\mathfrak{q}}_{r}(\underline{x})
& = \varepsilon \circ L(e_{c,s}(f) \circ\Upsilon_{t}(f) \circ \mathbf{c}^{\circ}) &\text{(Formula (\ref{eq:analytic_tome_formula}))}\\
& = l_{(c,s,t)}(\mathbf{c}^{\circ}) &\text{(reformulation)}\\
& =\varphi^{\mathfrak{q}}_{d_0}\upzeta_{c,s} (\underline{x}) &\text{(definition of $\upzeta_{c,s}$)}
\end{align*}

On the set $\tilde{\mathfrak{q}}(r)$, the calculation for $\underline{x} = (c_0,v,z,t,\upsilon,s',\mathbf{s})$ goes as follows.
\begin{align*}
\upxi_{c,s}\varphi^{\mathfrak{q}}_{r}(\underline{x})
& = \varepsilon \circ L(e_{c,s}(f) \circ\Upsilon_{t}(f) \circ \ell_{\vartheta}(\mathbf{s})) &\text{(Formula (\ref{eq:analytic_tome_formula}))}\\
& = \varepsilon \circ L(e_{c,s}(f) \circ\Upsilon_{t}(f) \circ l_z(\mathbf{s})) &\text{(definition of the link)}\\
& = l_{(c,s,t)} \circ l_z(\mathbf{s}) &\text{(reformulation)}\\
& =\varphi^{\mathfrak{q}}_{d_0}\upzeta_{c,s} (\underline{x}) &\text{(definition of $\upzeta_{c,s}$)}
\end{align*}

On the set $\Upgamma_A(r)$, the calculation on a tuple $\underline{x} = (c_0,v,n,\underline{c},\underline{s},\underline{t},t,s',\mathbf{c})$ goes as follows.
\begin{align*}
\upxi_{c,s}\varphi^{\mathfrak{q}}_{r}(\underline{x})
& = \varepsilon \circ L(e_{c,s}(f) \circ\Upsilon_{t}(f) \circ \ell_{\vartheta}(\mathbf{c}^{\circ})) &\text{(Formula (\ref{eq:analytic_tome_formula}))}\\
& = \varepsilon \circ L(e_{c,s}(f) \circ\Upsilon_{t}(f) \circ l_{(c_{n},s_{n},t_{n})}\dots \circ l_{(c_1,s_1,t_1)}(\mathbf{c}^{\circ})) &\text{(definition of the link)}\\
& = l_{(c,s,t)} \circ l_{(c_{n},s_{n},t_{n})}\dots \circ l_{(c_1,s_1,t_1)}(\mathbf{c}^{\circ}) &\text{(reformulation)}\\
& =\varphi^{\mathfrak{q}}_{d_0}\upzeta_{c,s} (\underline{x}) &\text{(definition of $\upzeta_{c,s}$)}
\end{align*}

On the set $\tilde{\mathfrak{q}}(r)$, the calculation for $\underline{x} = (c_0,v,n,\underline{c},\underline{s},\underline{t},z,t,\upsilon,s',\mathbf{s})$ goes as follows.
\begin{align*}
\upxi_{c,s}\varphi^{\mathfrak{q}}_{r}(\underline{x})
& = \varepsilon \circ L(e_{c,s}(f) \circ\Upsilon_{t}(f) \circ \ell_{\vartheta}(\mathbf{s})) &\text{(Formula (\ref{eq:analytic_tome_formula}))}\\
& = \varepsilon \circ L(e_{c,s}(f) \circ\Upsilon_{t}(f) \circ l_{(c_{n},s_{n},t_{n})}\dots \circ l_{(c_1,s_1,t_1)}\circ l_z(\mathbf{s})) &\text{(definition of the link)}\\
& = l_{(c,s,t)} \circ l_{(c_{n},s_{n},t_{n})}\dots \circ l_{(c_1,s_1,t_1)}\circ l_z(\mathbf{s}) &\text{(reformulation)}\\
& =\varphi^{\mathfrak{q}}_{d_0}\upzeta_{c,s} (\underline{x}) &\text{(definition of $\upzeta_{c,s}$)}
\end{align*}

Now, the equation $\upxi_{c,s}\varphi^{\mathfrak{q}}_{r} = \varphi^{\mathfrak{q}}_{d_0}\upzeta_{c,s}$ tells us that the content of the tome $\mathbb{T}^{f,\mathfrak{q}}(d_0)$ along $\upzeta_{c,s}$ is equal to the content of $\mathbb{T}^{f,\mathfrak{q}}(r)$ after applying the functor $\upxi_{c,s}$ on it. More specifically, the equation means that the respective composites of Equations (\ref{eq:first_sequence_arrows_effectiveness}) and (\ref{eq:second_sequence_arrows_effectiveness}) are equal.
\begin{equation}\label{eq:first_sequence_arrows_effectiveness}
\xymatrix{
\mathrm{col}_{\upzeta_{c,s}}\partial\varphi^{\mathfrak{q}}_{r} \ar@{=>}[r]^-{\textit{shift.}}& \mathrm{col}\partial\varphi^{\mathfrak{q}}_{d_0} \ar@{=>}[r]^-{\varphi^{\mathfrak{q}}_{d_0}}& f(d_0)
}
\end{equation}
\begin{equation}\label{eq:second_sequence_arrows_effectiveness}
\xymatrix{
\mathrm{col}_{}L\partial\varphi^{\mathfrak{q}}_{r} \ar@{=>}[r]^{\cong}& L\mathrm{col}\partial\varphi^{\mathfrak{q}}_{r} \ar@{=>}[r]^-{L(\varphi^{\mathfrak{q}}_{r})}& Lf(r) \ar@{=>}[r]^-{Le_{c,s}(f)}& LRf(d_0) \ar@{=>}[r]^-{\varepsilon}& f(d_0)
}
\end{equation}

If one denotes by $\eta$ the unit of the adjunction $L \vdash R$, the definition of adjunction implies that the function $R(\_) \circ \eta$ is inverse of $\varepsilon \circ L(\_)$.
Since the content $\mathrm{col}\partial\varphi^{\mathfrak{q}}_{d_0} \Rightarrow f(d_0)$ appearing in Equation (\ref{eq:first_sequence_arrows_effectiveness}) may be factorised as in Diagram (\ref{eq:combinatorial_constructors_pushout_def}) on $d_0$, an application of the inverse function of $\varepsilon \circ L(\_)$ on the arrow represented by Equations (\ref{eq:first_sequence_arrows_effectiveness}) and (\ref{eq:second_sequence_arrows_effectiveness}) provides the following commutative diagram, where Equation  (\ref{eq:first_sequence_arrows_effectiveness}) provides the inside while Equation  (\ref{eq:second_sequence_arrows_effectiveness}) provides the outside.
\[
\xymatrix{
\mathrm{col} \mathbf{A}_{r}\ar@{=>}[r]^{\mathrm{col}u_{r}}\ar@{=>}[dd]_{\mathrm{col}\partial\varphi^{\mathfrak{q}}_{r}}&X(r)\ar@{=>}[r]^{e_{c,s}}&RX(d_0)\ar@{=>}[dd]^{Rf(d_0)}\ar@{=>}[ld]_{Rp^{\mathfrak{q}}_f(d_0)}\\
&R[f,\mathfrak{q}](d_0)\ar@{=>}[dr]^{Rh^{\mathfrak{q}}_f(d_0)}&\\
\mathrm{col}\mathbf{B}_{r}\ar@{=>}[r]_{\mathrm{col}v_{r}}\ar@{=>}[ru]|{}&Y(r)\ar@{=>}[r]_{e'_{c,a_c(s)}}&RY(d_0)
}
\]

Now, because the top left corner of the previous diagram corresponds to the top left corner of the commutative square defining the $\Upgamma$-realisation of $(f,a,\mathfrak{q})$ when evaluated at ${r}$, it follows that there exists a natural transformation $e^{\mathfrak{q}}_{c,s}:[f,\mathfrak{q}](r) \Rightarrow R[f,\mathfrak{q}](d_0)$ making the following diagram commute.
\[
\xymatrix{
\mathrm{col} \mathbf{A}_{r}\ar@{}[rd]|->>{\text{\huge{\rotatebox[origin=c]{-90}{$\llcorner$}}}}\ar@{=>}[rr]^{\mathrm{col}u_{r}}\ar@{=>}[dd]_{\mathrm{col}\partial\varphi^{\mathfrak{q}}_{r}}&&X(r)\ar@{=>}[r]^-{e_{c,s}}\ar@{=>}[ld]_{p^{\mathfrak{q}}_f(r)}&RX(d_0)\ar@{=>}[dd]\ar@{=>}[ld]_{Rp^{\mathfrak{q}}_f(d_0)}\\
&[f,\mathfrak{q}](r)\ar@{=>}[r]^-{e^{\mathfrak{q}}_{c,s}}&R[f,\mathfrak{q}](d_0)\ar@{=>}[dr]^{Rh^{\mathfrak{q}}_f(d_0)}&\\
\mathrm{col}\mathbf{B}_{r}\ar@{=>}[rr]_{\mathrm{col}v_{r}}\ar@{=>}[ru]^{\pi^{\mathfrak{q}}_f(r)}\ar@{=>}[rru]_{}&&Y(r)\ar@{=>}[r]_{e'_{c,a_c(s)}}&RY(d_0)
}
\]

The previous diagram provides a morphism $(p^{\mathfrak{q}}_f,\mathrm{id}):(X,S,e) \Rightarrow ([f,\mathfrak{q}],S,e^{\mathfrak{q}})$ in the category of $R$-premodels $\mathbf{Pr}_{\mathcal{C}}(K,\mathtt{rou},R)$. The universality of $[f,\mathfrak{q}]$ also provides a morphism 
$(h^{\mathfrak{q}}_f,a):([f,\mathfrak{q}],S,e^{\mathfrak{q}}) \Rightarrow (Y,S'e')$ in $\mathbf{Pr}_{\mathcal{C}}(K,\mathtt{rou},R)$.
These two morphisms obviously define a factorisation of the morphism $(f,a):(X,S,e) \Rightarrow (Y,S',e')$ in $\mathbf{Pr}_{\mathcal{C}}(K,\mathtt{rou},R)$.
Finally, since the second component of the morphism $(p^{\mathfrak{q}}_f,\mathrm{id})$ is the identity on $S$, its image via the functor $\uplambda_{\theta}$ is an identity for every $\theta \in I$ (see Section \ref{sssec:Constructor_system_premodels}). In other words, the arrow $\uplambda_{\theta}(X,S,e) \to \uplambda_{\theta}([f,\mathfrak{q}],S,e^{\mathfrak{q}})$ mentioned in Definition \ref{def:effectiveness_quotiented_models} is indeed an identity.
\end{proof}

\begin{definition}[Fibered]
A system of $R$-premodels $(K,\mathtt{rou},\mathcal{P},\mathtt{V})$ over a small category $D$ in a category $\mathcal{C}$ will be said to be \emph{fibered}
if the category $\mathcal{C}$ has pushouts and the $\Upgamma$-factorisation of any $\Upgamma$-quotiented arrow (obtained in Theorem \ref{th:admissible_quotiented_effective}) lifts to $\mathcal{P}$.
\end{definition}

\begin{example}
By Theorem \ref{th:admissible_quotiented_effective}, any system of $R$-premodels $(K,\mathtt{rou},\mathcal{P},\mathtt{V})$ where $\mathcal{C}$ has pushouts and $\mathcal{P}$ is identified with the category $\mathbf{Pr}_{\mathcal{C}}(K,T,R)$ is fibered.
\end{example}

\begin{example}\label{exa:functor_category_fibered}
In the proof of Theorem \ref{th:admissible_quotiented_effective}, note that if the objects $(X,S,e)$ and $(Y,S',e')$ are such that the associated arrows
$e_{c,s}$ and $e_{c,a_c(s)}'$ are identities, then so is $e^{\mathfrak{q}}_{c,s}$. This implies that any system of $R$-premodels $(K,\mathtt{rou},\mathcal{P},\mathtt{V})$ where $\mathcal{C}$ has pushouts and $\mathcal{P}$ may be identified with the functor category $\mathcal{C}^D$ is fibered (e.g., Examples \ref{ex:Models_for_a_sketch_system}--\ref{ex:Complete_Segal_spaces})
\end{example}

\begin{example}\label{ex:spectra_fibered}
In the proof of Theorem \ref{th:admissible_quotiented_effective}, note that if the objects $(X,S,e)$ and $(Y,S',e')$ are such that the images of $S$ and $S'$ are equal to $\mathbf{1}$, then so is the $\Upgamma^K$-realisation $([f,\mathfrak{q}],S,e^{\mathfrak{q}})$. This implies that the system of $\Omega$-premodels given in Example \ref{ex:system_of_premodels_spectra} is fibered.
\end{example}

\begin{remark}
A system of $R$-premodels $(K,\mathtt{rou},\mathcal{P},\mathtt{V})$ is not always fibered (e.g., Example \ref{ex:system_of_premodels_Localisation_of_rings}), which is often due to a too strong restriction of the premodels via the inclusion $\mathcal{P} \hookrightarrow \mathbf{Pr}_{\mathcal{C}}(K,\mathtt{rou},R)$. However, Theorem \ref{th:admissible_quotiented_effective} shows that if $\mathcal{P}$ is too strong, we might want to stay in $\mathbf{Pr}_{\mathcal{C}}(K,\mathtt{rou},R)$ to process most of our calculations. The idea would then be that it is possible to go back to $\mathcal{P}$ at the very end of a transfinite calculation.
\end{remark}

\begin{example}\label{exa:elimination_of_quotients_sketches} This example discusses the form that the $\Upgamma$-realisation takes when considering categories of models for a limit sketch.
Let $(D,Q)$ be a limit sketch seen as a croquis. Consider the system of premodels defined in Example \ref{ex:Models_for_a_sketch_system} for the category $\mathbf{Set}^{D}$. Recall that the vertebrae  associated with any cone $c \in Q$ were of the following form.
\[
v_0:=\left(
\begin{array}{l}
\xymatrix{
\emptyset\ar@{}[rd]|>>>{\text{\huge{\rotatebox[origin=c]{-90}{$\llcorner$}}}} \ar[r]^{!}\ar[d]_{!} &\mathbf{1} \ar[d]^-{\delta_1}&\\
\mathbf{1} \ar[r]_-{\delta_2} &*+!L(.7){\mathbf{1}+\mathbf{1}} \ar[r]^-{!} &  \mathbf{1},\\
}
\end{array}
\right)
\quad\quad\quad\quad
v_1:=\left(
\begin{array}{l}
\xymatrix{
*+!R(.7){\mathbf{1}+\mathbf{1}}\ar@{}[rd]|>>>{\text{\huge{\rotatebox[origin=c]{-90}{$\llcorner$}}}} \ar[r]^{!}\ar[d]_{!} &\mathbf{1} \ar@{=}[d]&\\
\mathbf{1} \ar@{=}[r] &\mathbf{1} \ar@{=}[r] &  \mathbf{1}\\
}
\end{array}
\right)
\]

It follows from the definition of the transitive analysor and quotientor that, for any $\Upgamma^K$-quotiented arrow $(f,\mathfrak{q}):X \to \mathbf{1}$, the $\Upgamma^K$-realisation of $(f,\mathfrak{q})$ evaluated at an object $d \in D$ is defined over the following types of span.
\[
\underbrace{
\xymatrix{
\emptyset\ar[r]\ar[d]&X(d)\\
\mathbf{1}&
}}_{\text{from }\Upgamma_A^K(f)(d)\text{ restricted to }v_0}
\quad\quad\quad
\underbrace{
\xymatrix{
\mathbf{1}+\mathbf{1}\ar[d]\ar[r]^{x,y}&X(d)\\
\mathbf{1}&
}}_{\text{from }\tilde{\mathfrak{q}}(d)\text{ and }\Upgamma_A^K(f)(d)\text{ restricted to }v_1}
\]

The contribution of the left span to the construction of the $\Upgamma^K$-realisation $[f,\mathfrak{q}](d)$ is to add an element to $X(d)$ while the contribution of the right span to the construction of the $\Upgamma^K$-realisation $[f,\mathfrak{q}](d)$ is to quotient a pair of elements in $X(d)$. After unravelling the indices that parametrise the two types of span, we may deduce that the colimit $[f,\mathfrak{q}](d)$ is of the following form, where $\Upgamma_{A,0}^{K}(f)(d)$ and $\Upgamma_{A,1}^{K}(f)(d)$ are the restrictions of $\Upgamma_{A}^{K}(f)(d)$ to the vertebrae $v = v_0$ and $v = v_1$, respectively.
\[
[f/\mathfrak{q}](d) =  X(d)/\Big(\Upgamma_{A,1}^{K}(f)(d)+\tilde{\mathfrak{q}}(d)\Big) + \Upgamma_{A,0}^{K}(f)(d)
\]

After further unravelling the parameterisation of the rightmost summand, we may show that the colimit $[f/\mathfrak{q}](d)$ may be expressed as follows, where $R$ is a binary relation on $X$ in $\mathbf{Set}^D$.
\begin{equation}\label{eq:first_expression_elimination_of_quotients}
[f/\mathfrak{q}](d) =  X(d)/R(d) + \sum_{\vartheta_A \in J_A\text{ for }v_0} D(\upepsilon(\vartheta_A),d) \times \mathcal{C}^{\mathbf{2} \times \mathbf{2}}(\mathbf{disk}(v_0),\Phi_{\upiota(\vartheta_A)}(f))
\end{equation}

Concretely, the set $\mathcal{C}^{\mathbf{2} \times \mathbf{2}}(\mathbf{disk}(v_0),\Phi_{\upiota(\vartheta_A)}(f))$ is nothing but the set $X[c_0] :=\mathrm{lim}X\mathtt{in}(c_0)$ with respect to the notations of $\vartheta_A$ given in Remark \ref{rem:form_of_elt_I_J} while the object $\upepsilon(\vartheta_A)$ is given by $\mathtt{ou}(c_n)$ for the same notations.

Recall that, according to Remarks \ref{rem:J_A_could_have_been_I} and  \ref{rem:trivial_effectiveness_functor_categories}, the set $J_A$ could in fact be given by the set $I$ itself in the present situation (i.e. premodels for a sketch). In this case, the expression of Equation  (\ref{eq:first_expression_elimination_of_quotients}) turns out to be as follows.
\[
[f/\mathfrak{q}](d) =  X(d)/R(d) + \sum_{c \in K} D(\mathtt{ou}(c),d) \times X[c]
\]
\end{example}

\subsection{Rectification of Effective Quotiented Arrows}\label{ssec:Rectification}
Let $\Upgamma$ denote a constructor of type $D[\mathcal{B},\mathcal{C}]$ as defined in Section \ref{sssec:Combinatorial_constructors}, with the usual notations, and~$(f,\mathfrak{q}):X \to Y$ be an effective $\Upgamma$-quotiented arrow in $\mathcal{B}$. Usually, effectiveness does not
mean that the quotiented arrow is as we would like it to be. It is in fact necessary to rectify its defaults via a second quotient. The goal of this section is to define the `rectification' of $(f,\mathfrak{q})$, which is nothing but a $\Upgamma$-quotient $\mathfrak{u}$ for the arrow $\lfloor f \rfloor_{\mathfrak{q}}:[f/\mathfrak{q}] \to Y$.

To do so, let us define, for every element $\theta \in I$ and $s \in \uplambda_{\theta}(X)=\uplambda_{\theta}(\lfloor f \rfloor_{\mathfrak{q}})$, the associated functor of the following form.
\[
\mathfrak{u}_{\theta,s}\{\_\}:E_s(\theta) \to \mathbf{Set}
\]

First, define the discrete category $E_s(\theta)$ to be the set $\Upgamma_A(f)\{\theta\}_s$.
By definition, an element $\upsilon \in E_s(\theta)$ may be identified with an element $\mathbf{c} \in \mathcal{C}^{\mathbf{2} \times \mathbf{2}}(\upalpha(\theta),\Phi_{\theta}(f)_s)$, which may be sent to the arrow 
\begin{equation}\label{eq:construction_obstruction_image_merolytic}
\mathbf{c}^{\circ}:\upalpha^{\circ}(\theta)\Rightarrow \Upsilon_{\uprho(\theta)}(f)
\end{equation}
via the domain restriction $(\_)^{\circ}:\mathcal{C}^{\mathbf{2} \times \mathbf{2}} \to \mathcal{C}^{\mathbf{2}}$. The arrow encoded by $\mathbf{c}^{\circ}$ may be identified with an element in the image of the analytic tome of 
$\varphi_A:\Upgamma_A(f)(\uprho(\theta)) \to \mathcal{C}^{\mathbf{2}}/\Upsilon_{\uprho(\theta)}(f)$ as follows (see Formula (\ref{eq:analytic_tome_formula}) and the assumption of the initial section $\textsc{\j}:I \to J_A$).
\[
\Upsilon_{\mathrm{id}_{\uprho(\theta)}}(f) \circ \ell_{\textsc{\j}(\theta)}(\mathbf{c}^{\circ}):\upmu\textsc{\j}(\theta) \Rightarrow \Upsilon_{\uprho(\theta)}(f)
\]

This therefore defines a function $i:E_s(\theta) \to \Upgamma_A(f)(\uprho(\theta))$ mapping any element $\mathbf{c} \in E_s(\theta)$
 to the tuple $(\textsc{\j}(\theta),\mathrm{id}_{\uprho(\theta)},s,\mathbf{c})$ whose image via the merolytic tome $\varphi^{\mathfrak{q}}:\Upgamma^{f,\mathfrak{q}}(\uprho(\theta)) \to \mathcal{C}^{\mathbf{2}}/\Upsilon_{\uprho(\theta)}(f)$ is the arrow encoded by $\mathbf{c}^{\circ}$. 
 
This being said, denote by $r$ the element $\uprho(\theta)$ and, for every $\mathbf{c} \in E_s(\theta)$, denote by $i_{\mathbf{c}}$ the function $\mathbf{1} \to \Upgamma_A(f)(r)$ that picks out the element $i(\mathbf{c})$. From the point of view of these notations, we have showed that the image of the composite $\varphi^{\mathfrak{q}} \circ i_{\mathbf{c}}$ corresponds to the commutative square $\mathbf{c}^{\circ}$. However, this also means that the content of the merolytic tome of $(f,\mathfrak{q})$ along $i_{\mathbf{c}}:\mathbf{1} \to E_s(\theta)$ is equal to the commutative Square (\ref{eq:construction_obstruction_image_merolytic}) in $\mathcal{C}$ as illustrated below.
\[
\mathbf{c}^{\circ}:\mathrm{col}_{i_{\mathbf{c}}}\partial \varphi^{\mathfrak{q}}_r \Rightarrow \mathrm{col}\partial \varphi^{\mathfrak{q}}_r \Rightarrow \Upsilon_{r}(f)
\]

Because the left arrow $\mathrm{col}\partial \varphi^{\mathfrak{q}}_r \Rightarrow \Upsilon_{r}(f)$ (i.e., the content) may be factorised as shown in Diagram (\ref{eq:combinatorial_constructors_pushout_def}), it follows that the commutative square encoding $\mathbf{c}^{\circ}$ factorises as shown below, on the left.
\begin{equation}\label{eq:point_of_view_structure_strict_narratives_premodels}
\xymatrix{
\mathbb{S}\ar[rr]^x\ar[dd]_{\upalpha^{\circ}(\theta)}&&\Upsilon_r(X)\ar[d]|{p_f^{\mathfrak{q}}(r)}\ar@/^2pc/[dd]^{\Upsilon_r(f)}\\
&&[f,\mathfrak{q}](r)\ar[d]|{h_f^{\mathfrak{q}}(r)}\\
\mathbb{D}_2\ar[rru]^{\uppi_0}\ar[rr]_{x'}&&\Upsilon_r(Y)
}
\quad\quad\quad
\xymatrix{
\Upsilon_r(X)\ar[d]_{p_f^{\mathfrak{q}}(r)}\ar[rr]^-{\Phi_{\theta}(X)_s}&&\Phi_{\theta}^{\bullet}(X)_s\ar[d]^{\Phi_{\theta}^{\bullet}(\{f\}_{\mathfrak{q}})_s}\\
[f,\mathfrak{q}](r)\ar[d]_{h_f^{\mathfrak{q}}(r)}\ar[rr]^-{\Phi_{\theta}([f/\mathfrak{q}])_s}&&\Phi_{\theta}^{\bullet}([f/\mathfrak{q}])_s\ar[d]^{\Phi_{\theta}^{\bullet}(\lfloor f \rfloor_{\mathfrak{q}})_s}\\
\Upsilon_r(X)\ar[rr]^-{\Phi_{\theta}(Y)_s}&&\Phi_{\theta}^{\bullet}(Y)_s
}
\end{equation}

The diagram displayed above, on the right, is for its part the image of the $\Upgamma$-factorisation of $(f,\mathfrak{q})$ in $\mathcal{B}$ via the functor $\Phi_{\theta}:\mathcal{B} \to \mathcal{C}$. The definitions of the diagrams involved in Equation (\ref{eq:point_of_view_structure_strict_narratives_premodels}) imply that the commutative square $\mathbf{c} \in 
\mathcal{C}^{\mathbf{2} \times \mathbf{2}}(\upalpha(\theta),\Phi_{\theta}(f)_s)$ factorises as follows, where the image $\upalpha(\theta)$ is replaced with the diskad of a vertebra $\|\gamma_2,\gamma_1\|\cdot \beta$ for which $\beta = \upomega(\theta)$ by definition.

\begin{equation}\label{eq:factorisation_combi_cat_premodel_step_k-1_cuboid}
\xymatrix@R-0.5pc{
\mathbb{S}\ar[ddd]_{\gamma_2}\ar[rr]^x\ar[rd]_{\gamma_1}&&\Upsilon_r(X)\ar[rd]|{\Phi_{\theta}(X)_s}\ar[dd]|\hole&\\
&\mathbb{D}_1\ar[ddd]_<<<<<<<<<<<{\beta\delta_2}\ar[rr]^<<<<<<<<<<y&&\Phi_{\theta}^{\bullet}(X)_s\ar[dd]^{\Phi_{\theta}^{\bullet}(\{ f\}_{\mathfrak{q}})_s}\\
&&[f,\mathfrak{q}](r)\ar[d]\ar[rd]|{\Phi_{\theta}([f/\mathfrak{q}])_s}&\\
\mathbb{D}_2\ar[rru]|<<<<<<<<<<<<<\hole|>>>>>>>{\uppi_0}\ar[rd]_{\beta\delta_2}\ar[rr]|<<<<<<<<<<<<\hole^>>>>>>>>{x'}&&\Upsilon_r(Y)\ar[rd]|<<<<<<<<{\Phi_{\theta}(Y)_s}&\Phi_{\theta}^{\bullet}([f/\mathfrak{q}])_s\ar[d]^{\Phi_{\theta}^{\bullet}(\lfloor f\rfloor_{\mathfrak{q}})_s}\\
&\mathbb{D}'\ar[rr]^{y'}&&\Phi_{\theta}^{\bullet}(Y)_s\\
}
\end{equation}

Notice that the previous commutative cube provides the following left commutative square. 
\[
\xymatrix{
\mathbb{S}\ar[d]_{\gamma_2}\ar[rr]^-{\gamma_1}&&\mathbb{D}_1\ar[d]^{\Phi_{\theta}^{\bullet}(\{f\}_{\mathfrak{q}})_s \circ y}\\
\mathbb{D}_2\ar[rr]_-{\Phi_{\theta}([f/\mathfrak{q}])_s \circ \uppi_0}&&\Phi_{\theta}^{\bullet}([f/\mathfrak{q}])_s
}\quad\Rightarrow\quad
\xymatrix{
\mathbb{S}\ar@{}[rd]|->>{\text{\huge{\rotatebox[origin=c]{-90}{$\llcorner$}}}}\ar[d]_{\gamma_2}\ar[r]^{\gamma_1}&\mathbb{D}_1\ar[d]^{\delta_1}\ar@/^1pc/[rrd]^>>>>>>>{\Phi_{\theta}^{\bullet}(\{f\}_{\mathfrak{q}})_s \circ y}&&\\
\mathbb{D}_2\ar[r]^{\delta_2}\ar@/_1.5pc/[rrr]|{\Phi_{\theta}([f/\mathfrak{q}])_s \circ \uppi_0}&\mathbb{S}'
\ar@{-->}[rr]^-{w}&&\Phi_{\theta}^{\bullet}([f/\mathfrak{q}])_s
}
\]

By using the structure of the vertebra $\|\gamma_2,\gamma_1\| \cdot \beta$, we may form a pushout $\mathbb{S}'$ inside so that we obtain a canonical arrow $w:\mathbb{S}' \to  \Phi_{\theta}^{\bullet}([f/\mathfrak{q}])_s$ making the preceding right diagram commute.
It is not hard to deduce from the universality of this pushout that both arrows 
\[
\Phi_{\theta}^{\bullet}(\lfloor f\rfloor_{\mathfrak{q}})_s \circ w:\mathbb{S}' \to \Phi_{\theta}^{\bullet}(Y)_s
\quad\text{and}\quad
y' \circ \beta:\mathbb{S}' \to \Phi_{\theta}^{\bullet}(Y)_s
\] 
are solutions for a same universal problem over $\mathbb{S}'$ (Diagram (\ref{eq:factorisation_combi_cat_premodel_step_k-1_cuboid}) might come in handy to visualise this fact). In particular, this means that the following diagram must commute.

\begin{equation}\label{eq:obstruction_square_define_quotient}
\xymatrix@C-.5pc{
\mathbb{S}'\ar[d]_{\beta}\ar[r]^-{w}&*+!L(.5){\Phi_{\theta}^{\bullet}([f/\mathfrak{q}])_s}\ar[d]^{\Phi_{\theta}^{\bullet}(\lfloor f\rfloor_{\mathfrak{q}})_s} \\
\mathbb{D}'\ar[r]_-{y'}& *+!L(.5){\Phi_{\theta}^{\bullet}(Y)_s}
}
\end{equation}

Because $\beta$ corresponds to the image $\upomega(\theta)$, we have defined a functor $\mathfrak{u}_{\theta,s}\{\_\}:E_s(\theta) \to \mathbf{Set}$ mapping a commutative cube $\mathbf{c} \in E_s(\theta)$ to the subset of $\mathcal{C}^{\mathbf{2}}(\upomega(\theta),\Phi_{\theta}^{\bullet}(\lfloor f\rfloor_{\mathfrak{q}})_s)$ consisting of Diagram (\ref{eq:obstruction_square_define_quotient}) only. Thus, the images of $\mathfrak{u}_{\theta,s}\{\_\}$ are sets (or singletons) included in $\Upgamma_Q(\lfloor f\rfloor_{\mathfrak{q}})\{\theta\}_s$ so that the collection of functors given below, denoted by $\mathfrak{u}$, defines a $\Upgamma$-quotient for the arrow $\lfloor f\rfloor_{\mathfrak{q}}:[f/\mathfrak{q}] \to Y$.
\[
\mathfrak{u}:=\{\mathfrak{u}_{\theta,s}\{\_\}:E_s(\theta) \to \mathbf{Set}\}_{\theta \in I, s \in \uplambda_{\theta}(X)}
\]
\begin{definition}[Rectification]
The \emph{$\Upgamma$-rectification} of the $\Upgamma$-quotiented arrow $(f,\mathfrak{q}):X \to Y$ is the $\Upgamma$-quotiented arrow
$(\lfloor f\rfloor_{\mathfrak{q}},\mathfrak{u})$, which will sometimes be denoted by $\mathrm{Rec}(f,\mathfrak{q})$.
\end{definition}

Later on, the diagram obtained in Equation (\ref{eq:obstruction_square_define_quotient}), which is entirely determined by the image of the $\Upgamma$-rectification of $(f,\mathfrak{q})$ above a cube $\mathbf{c} \in E_s(\theta)$ at the parameters $\theta \in I$ and $s \in \uplambda_{\theta}([f/\mathfrak{q}])$, will be referred to as the \emph{obstruction square of $(f,\mathfrak{q})$ for $\mathbf{c}$ at $(\theta,s)$}.

\begin{definition}[Ideal]
A $\Upgamma$-quotiented arrow $(f,\mathfrak{q}):X \to Y$ will be said to be \emph{ideal} if it is effective, its $\Upgamma$-rectification $(\lfloor f\rfloor_{\mathfrak{q}},\mathfrak{u})$ is effective and for every $\theta \in I$, $s \in \uplambda_{\theta}([f/\mathfrak{q}])$ and $\mathbf{c} \in E_s(\theta)$, there exists an arrow $\uppi_1(\theta,s): \mathbb{D}' \to \Phi_{\theta}^{\bullet}([\lfloor f \rfloor_{\mathfrak{q}}/\mathfrak{u}])_s$ factorising the obstruction square of $(f,\mathfrak{q})$ for $\mathbf{c}$ at $(\theta,s)$ as follows.
\begin{equation}\label{eq:factorisation_combi_cat_premodel_rectifying_modifier_square}
\xymatrix{
\mathbb{S}'\ar[dd]_{\beta}\ar[rr]^-{w}&&*+!L(.5){\Phi_{\theta}^{\bullet}([f/\mathfrak{q}])_s}\ar[dd]^{\Phi_{\theta}^{\bullet}(\lfloor f\rfloor_{\mathfrak{q}})_s} \ar[dl]|-{\Phi_{\theta}^{\bullet}(\{ \lfloor f\rfloor_{\mathfrak{q}}\}_{\mathfrak{u}})_s}\\
&\Phi_{\theta}^{\bullet}([\lfloor f \rfloor_{\mathfrak{q}}/\mathfrak{u}])_s
\ar[rd]|-{\Phi_{\theta}^{\bullet}(\lfloor \lfloor f\rfloor_{\mathfrak{q}}\rfloor_{\mathfrak{u}})_s}&\\
\mathbb{D}'\ar[rr]_-{y'}\ar@{-->}[ru]|-{\uppi_1(\theta,s)}&& *+!L(.5){\Phi_{\theta}^{\bullet}(Y)_s}
}
\end{equation}
\end{definition}

\begin{remark}[Structure of narrative of degree 2]\label{rem:structure_narrative_degree_2}
Consider an ideal $\Upgamma$-quotiented arrow $(f,\mathfrak{q}):X \to Y$ and a commutative cube $\mathbf{c}$ in $\mathcal{C}^{\mathbf{2} \times \mathbf{2}}(\upalpha(\theta),\Phi_{\theta}(f)_s)$. According to the previous discussion, this cube $\mathbf{c}$ may be factorised as in Diagram (\ref{eq:factorisation_combi_cat_premodel_step_k-1_cuboid}). Merging this factorisation of $\mathbf{c}$ with: (1) the factorisation of the obstruction square of $(f,\mathfrak{q})$ for $\mathbf{c}$ at $(\theta,s)$ on its front face and (2) the $\Upgamma$-factorisation of the $\Upgamma$-rectification of $(f,\mathfrak{q})$ on its back face leads to the following factorisation of $\mathbf{c}$ (where the top front corner has been forgotten and $r = \uprho(\theta)$).
\[
\xymatrix@R-0.4pc@C-0.4pc{
&\mathbb{S}\ar[ddd]_{\gamma_2}\ar[rr]^x&&\Upsilon_r(X)\ar[rd]|{\Phi_{\theta}(X)_s}\ar[d]_{p^{\mathfrak{q}}_{f}(r)}&&\\
&&&[f,\mathfrak{q}](r)\ar[d]_{p^{\mathfrak{u}}_{\lfloor f \rfloor_{\mathfrak{q}}}(r)}\ar@{..>}[rd]&*+!L(.7){\Phi_{\theta}^{\bullet}(X)_s}\ar[d]&\ar@{}@<+3.8ex>[d]|{\leftarrow~\Phi_{\theta}(\{f\}_{\mathfrak{q}})_s }\\
&&&[\lfloor f\rfloor_{\mathfrak{q}},\mathfrak{u}](r)\ar[d]_{h^{\mathfrak{u}}_{\lfloor f \rfloor_{\mathfrak{q}}}(r)}\ar[rd]&*+!L(.7){\Phi_{\theta}^{\bullet}([f /\mathfrak{q}])_s}\ar[d]&\ar@{}@<+3.8ex>[d]|{\leftarrow~\Phi_{\theta}(\{\lfloor f\rfloor_{\mathfrak{q}}\}_{\mathfrak{u}})_s }\\
&\mathbb{D}_2\ar@/^1.5pc/[rruu]^{\uppi_0}\ar[rd]_{\beta \circ \delta_2}\ar[rr]^-{x'}&&\Upsilon_r(Y)\ar[rd]|<<<<<\hole&*+!L(.7){\Phi_{\theta}^{\bullet}([\lfloor f \rfloor_{\mathfrak{q}}/\mathfrak{u}])_s}\ar[d]&\ar@{}@<+3.8ex>[d]|{\leftarrow~\Phi_{\theta}(\lfloor \lfloor f\rfloor_{\mathfrak{q}}\rfloor_{\mathfrak{u}})_s}\\
&&\mathbb{D}'\ar[rr]|{y'}\ar[rru]|{\uppi_1(\theta,s)}&&*+!L(.7){\Phi_{\theta}^{\bullet}(Y)_s}&\\
}
\]

This means that the composite arrow given in Equation (\ref{eq:composite_structure_for_narrative_degree_2}), whose the leftmost arrow is given by the content of the operadic tome $\varrho_{\theta}^s:\Upgamma_A(f)\{\theta\}_s \to \mathcal{C}^{\mathbf{2}\times \mathbf{2}}/\Phi_{\theta}(f)_s$, admits a lift in $\mathcal{C}^{\mathbf{2}}$
\begin{equation}\label{eq:composite_structure_for_narrative_degree_2}
\xymatrix{
\mathrm{col}\partial \varrho_{\theta}^s \ar@{=>}[r]^{\varrho_{\theta}^s}&\Phi_{\theta}(f)_s\ar@{=>}[rr]^-{\Phi_{\theta}(\{ f \}_{\mathfrak{q}})_s}&&\Phi_{\theta}(\lfloor f \rfloor_{\mathfrak{q}})_s\ar@{=>}[rr]^-{\Phi_{\theta}(\{ \lfloor f \rfloor_{\mathfrak{q}} \}_{\mathfrak{u}})_s}&&
\Phi_{\theta}(\lfloor \lfloor f \rfloor_{\mathfrak{q}} \rfloor_{\mathfrak{u}})_s
}
\end{equation}

This last fact will later imply that we may construct a narrative of degree 2 out of the operadic tome.
\end{remark}

\begin{remark}[About $\uppi_0$]\label{rem:about_uppi_0}
This section discusses the encoding of the arrow that we have denoted $\uppi_0$. We shall use the same notations as that introduced at the beginning of the section. Recall that we defined the element $i_{\mathbf{c}}= (\textsc{\j}(\theta),\mathrm{id}_{\textsc{\j}(\theta)},s,\mathbf{c})$, which we used to shift the merolytic tome of $(f,\mathfrak{q})$ and obtain the leftmost diagram of Equation  (\ref{eq:point_of_view_structure_strict_narratives_premodels}). Therefore, we have the following formula if we use the notation of Diagram (\ref{eq:combinatorial_constructors_pushout_def}).
\[
\uppi_0 = \pi_f^{\mathfrak{q}}(\textsc{\j}(\theta)) \circ \xi_{i_{\mathbf{c}}}(\mathbf{B}_{\textsc{\j}(\theta)})
\]

If we now denote $i'_{\mathbf{c}}= (\textsc{\j}(\theta),t,s,\mathbf{c})$ for some arrow $t:\textsc{\j}(\theta) \to d$, the functionality of $\pi_f^{\mathfrak{q}}$ and the construction of the merolytic tome of $(f,\mathfrak{q})$ gives the following formula.
\[
[f,\mathfrak{q}](t) \circ \uppi_0 = \pi_f^{\mathfrak{q}}(d) \circ \xi_{i_{\mathbf{c}}'}(\mathbf{B}_{d})
\]

This formula will later come in handy in the proof of Theorem \ref{th:admissible_quotiented_factorisable_model}.
\end{remark}

\begin{theorem}\label{th:admissible_quotiented_ideal}
Let $(K,\mathtt{rou},\mathcal{P},\mathtt{V})$ be a system of $R$-premodels over a small category $D$ in a category $\mathcal{C}$. If $\mathcal{C}$ admits pushouts and the inclusion $\mathcal{P} \hookrightarrow \mathbf{Pr}_{\mathcal{C}}(K,\mathtt{rou},R)$ is an identity, then every $\Upgamma^K$-quotiented arrow is ideal.
\end{theorem}
\begin{proof}
For convenience, the symbol $\Upgamma^K$ will be shortened to $\Upgamma$. The present proof uses the construction made in the proof of Theorem \ref{th:admissible_quotiented_effective}. In particular, we shall use the notations defined thereof, such as $h_f^{\mathfrak{q}}$ and $p_f^{\mathfrak{q}}$.
Let $(f,a,\mathfrak{q}):(X,S,e) \Rightarrow (Y,S',e')$ be an $\Upgamma$-quotiented arrow in $\mathcal{B}$. By Theorem \ref{th:admissible_quotiented_effective}, it is effective and so is its $\Upgamma$-rectification $(h_f^{\mathfrak{q}},a,\mathfrak{u}):([f,\mathfrak{q}],S,e^{\mathfrak{q}}) \Rightarrow (Y,S',e')$. There now remains to show the existence of an arrow 
\[
\uppi_1(\theta,s): \mathbb{D}' \to \Phi_{\theta}^{\bullet}([\lfloor f \rfloor_{\mathfrak{q}}/\mathfrak{u}])_s
\]
factorising the obstruction square of $(f,\mathfrak{q})$ for any cube $\mathbf{c} \in E_s(\theta)$ at any parameter $\theta \in I$ and $s \in \uplambda_{\theta}([f/\mathfrak{q}])$ (see Diagram (\ref{eq:factorisation_combi_cat_premodel_rectifying_modifier_square})).

First, recall that, for every $\theta \in I$, $s \in \uplambda_{\theta}([f/\mathfrak{q}])$ and cube $\mathbf{c} \in E_s(\theta)$, the obstruction square of $(f,\mathfrak{q})$ for $\mathbf{c}$ at $(\theta,s)$ is given by an arrow in $\mathcal{C}^{\mathbf{2}}$ of the following form.
\[
\mathbf{s}:\upomega(\theta) \Rightarrow \Phi_{\theta}^{\bullet}(\lfloor f\rfloor_{\mathfrak{q}})_s\quad\quad=\quad\quad\mathbf{s}:\upomega(\theta) \Rightarrow \mathrm{lim}_z Rh^{\mathfrak{q}}_f\mathtt{in}(c_0)(z)
\]

By using the notations of Section \ref{sssec:Constructor_system_premodels} and the adjointness properties of $R$ and $\mathrm{lim}_z$, the preceding righthand arrow may be turned into the following arrow in $\mathcal{C}^{\mathbf{2}}$ for every $z \in \mathtt{Es}(c_0)$.
\begin{equation}\label{eq:obstruction_square_image_tome}
\varsigma_z \circ \varepsilon \circ L(\mathbf{s}):L\upomega(\theta) \Rightarrow h^{\mathfrak{q}}_f\mathtt{in}(c_0)(z)
\end{equation}

Now, observe from the definitions of Section \ref{sssec:Constructor_system_premodels} that, for every $\theta=(c_0,v) \in I$ and $z \in \mathtt{Es}(c_0)$, we may define an object $\theta_z :=(\theta,z)$ in $J_Q$, which precisely lands in the component $I'$ of $J_Q$. From the notations of Section \ref{sssec:Constructor_system_premodels}, the arrow given in Equation (\ref{eq:obstruction_square_image_tome}) may in fact be rewritten as follows \footnote{We have the identities $\upnu(\theta_z) = L\upomega \updelta(\theta_z) = L\upomega(\theta)$ and $\upchi(\theta_z) = \mathtt{in}(c_0)(z)$}.
\[
\Upsilon_{\mathrm{id}_{\upchi(\theta_z)}}(\lfloor f\rfloor_{\mathfrak{q}}) \circ \ell_{\theta_z}(\mathbf{s}):\upnu(\theta_z) \Rightarrow \Upsilon_{\upchi(\theta_z)}(\lfloor f\rfloor_{\mathfrak{q}})
\]

It therefore follows from Formula (\ref{eq:quotient_tome_formula}) that the arrow given in Equation  (\ref{eq:obstruction_square_image_tome}) may be identified with the image of Tuple (\ref{eq:tuple_factorisation_of_obstruction_square}) (see below) via the quotient tome
$\varphi_Q:\mathfrak{u}(\upchi(\theta_z))  \to \mathcal{C}^{\mathbf{2}}/\Upsilon_{\upchi(\theta_z)}(\lfloor f\rfloor_{\mathfrak{q}})$.
\begin{equation}\label{eq:tuple_factorisation_of_obstruction_square}
(\theta_z,\mathrm{id}_{\upchi(\theta_z)},\upsilon,s,\mathbf{s}) \in \mathfrak{u}(\upchi(\theta_z))
\end{equation}

In order to avoid overloading the next diagrams, denote by $d:\mathtt{Es}(c_0) \to D$ the functorial mapping $z \mapsto \upchi(\theta_z)$ and, for every $\mathbf{s} \in \mathfrak{u}_{\theta,s}\{\upsilon\}$, denote by $i_{\mathbf{s},z}$ the function $\mathbf{1} \to \Upgamma^{\lfloor f\rfloor_{\mathfrak{q}},\mathfrak{u}}(d(z))$ that picks out Tuple (\ref{eq:tuple_factorisation_of_obstruction_square}) in $\tilde{\mathfrak{u}}(d(z))$ for every $z \in \mathtt{Es}(c_0)$. 
Now, to resume, the previous discussion showed that the image of the composite $\varphi^{\mathfrak{u}}_{d(z)} \circ i_{\mathbf{s},z}$ corresponds to the arrow $\varsigma_z \circ \varepsilon \circ L(\mathbf{s})$. However, this is equivalent to saying that the content of the merolytic tome of $(\lfloor f\rfloor_{\mathfrak{q}},\mathfrak{u})$ along $i_{\mathbf{s},z}:\mathbf{1} \to \Upgamma^{\lfloor f\rfloor_{\mathfrak{q}},\mathfrak{u}}d(z)$ is equal to the arrow $\varsigma_z \circ \varepsilon \circ L(\mathbf{s})$ as illustrated below.
\[
\xymatrix{
\mathrm{col}_{i_{\mathbf{s},z}}\partial \varphi^{\mathfrak{u}}_{d(z)} \ar@{=>}[r] \ar@{=>}@/_1.2pc/[rr]_{\varsigma_z \circ \varepsilon \circ L(\mathbf{s})}& \mathrm{col}\partial \varphi^{\mathfrak{u}}_{d(z)}  \ar@{=>}[r] & \Upsilon_{d(z)}(\lfloor f\rfloor_{\mathfrak{q}})
}
\]

Because the rightmost arrow $\mathrm{col}\partial \varphi^{\mathfrak{u}}_{d(z)} \Rightarrow \Upsilon_{d(z)}(\lfloor f\rfloor_{\mathfrak{q}})$ may be factorised as shown in Diagram (\ref{eq:combinatorial_constructors_pushout_def}), it follows that the commutative square encoded by $\varsigma_z \circ \varepsilon \circ L(\mathbf{s})$ factorises as follows.
\[
\xymatrix{
L\mathbb{S}'\ar[rr]^w\ar[dd]_{L\upomega(\theta)}&&[f/\mathfrak{q}]d(z)\ar[dl]|{p_{\lfloor f\rfloor_{\mathfrak{q}}}^{\mathfrak{u}}(d(z))}\ar[dd]^{h_f^{\mathfrak{q}}d(z)}\\
&[\lfloor f\rfloor_{\mathfrak{q}},\mathfrak{u}]d(z)\ar[dr]|{h_{\lfloor f\rfloor_{\mathfrak{q}}}^{\mathfrak{u}}(d(z))}&\\
L\mathbb{D}'\ar[ru]^{\pi_1(z)}\ar[rr]_{y'}&&Yd(z)
}
\]

The idea is now to obtain a factorisation of the form given in Equation  (\ref{eq:factorisation_combi_cat_premodel_rectifying_modifier_square}) by reconstructing the obstruction square $\mathbf{s}$ (from which the previous diagram is derived) without losing the factorisation.

First, note that, by definition of the quotient acting on $\tilde{\mathfrak{u}}$ (see Convention \ref{conv:use_q_tilde_instead_of_q}), the collection of arrows $\{i_{\mathbf{s},z}\}_{z \in \mathtt{Es}(c_0)}$ is natural in $z \in\mathtt{Es}(c_0)$ since the following tuples have the same images via the functor $\varphi_Q$ for every arrow $t:z \to z'$ in $\mathtt{Es}(c_0)$.
\[
(\theta_{z'},\mathrm{id}_{d(z')},\upsilon,s,\mathbf{s})\quad\quad\quad(\theta_z,d(t),\upsilon,s,\mathbf{s})
\]

The functoriality of Diagram (\ref{eq:combinatorial_constructors_pushout_def}) over $D$
and the naturality of $i_{\mathbf{s},z}:\mathbf{1} \to \Upgamma^{\lfloor f\rfloor_{\mathfrak{q}},\mathfrak{u}}(d(z))$ in $z \in \mathtt{Es}(c_0)$ then
implies that the earlier commutative diagram is natural over $z \in \mathtt{Es}(c_0)$. Forming the limit of that diagram over $\mathtt{Es}(c_0)$ and then applying the inverse of the function $\varepsilon \circ L(\_)$ (which is given by the function $R(\_) \circ \eta$ if $\eta$ denotes the unit of $L \vdash R$) provides a factorisation of the original obstruction square $\mathbf{s}$ as follows.
\[
\xymatrix{
\mathbb{S}'\ar[dd]_{\beta}\ar[rr]^-{w}&&\ar[dd]^{\Phi_{\theta}^{\bullet}(\lfloor f\rfloor_{\mathfrak{q}})_s} \Phi_{\theta}^{\bullet}([f/\mathfrak{q}])_s\ar[dl]|{\Phi_{\theta}^{\bullet}(\{ \lfloor f\rfloor_{\mathfrak{q}}\}_{\mathfrak{u}})}\\
&\Phi_{\theta}^{\bullet}([\lfloor f \rfloor_{\mathfrak{q}}/\mathfrak{u}])_s
\ar[rd]|{\Phi_{\theta}^{\bullet}(\lfloor \lfloor f\rfloor_{\mathfrak{q}}\rfloor_{\mathfrak{u}})_s}&\\
\mathbb{D}'\ar[rr]_-{y'}\ar[ru]|-{\mathrm{lim}R(\pi_1(z)) \circ \eta}&& \Phi_{\theta}^{\bullet}(Y)_s
}
\]

This finally shows that the $\Upgamma$-quotiented arrow $(f,a,\mathfrak{q}):(X,S,e) \Rightarrow (Y,S',e')$ is ideal.
\end{proof}

\begin{example}\label{exa:elimination_of_quotients_sketches_the_binary_relations}
This example continues the discussion started in Example \ref{exa:elimination_of_quotients_sketches} (we shall use the same notations as those used thereof) in order to describe, in more details, the binary relation $R(d)$ acting on $X(d)$ (see Formula (\ref{eq:first_expression_elimination_of_quotients})) in the case where $f$ is taken to be the canonical map $!_X:X \to \mathbf{1}$. Recall that the quotient $X(d)/R(d)$ was meant to simplify the following expression.
\[
X(d)/(\Upgamma_{A,1}^{K}(!_X)(d)+\tilde{\mathfrak{q}}(d))
\]

Also, recall that, by definition, the binary relations contained in $\Upgamma_{A,1}^{K}(!_X)(d)$ (see Remark \ref{rem:encoding_of_analytic_functors} for the encoding of $\Upgamma_{A}^{K}$) are those pairs $(x,y):\mathbf{1}+\mathbf{1} \to X(d)$ that may be related to commutative diagrams as follows.
\[
\xymatrix @C-1pc @R-1pc{
\mathbf{1}+\mathbf{1}\ar@/^1pc/[rrrr]^{(x,y)}\ar@{-->}[dd]\ar[rd]\ar[rr] &&X\mathtt{ou}(c)\ar[rr]\ar[rd]\ar@{-->}[dd]&&X(d)\ar@{-->}[dd]\\
&\mathbf{1}\ar@{-->}[dd]\ar[rr]&&\mathrm{lim}X\mathtt{in}(c)\ar@{-->}[dd]&\\
\mathbf{1}\ar@{-->}[rd]\ar@{-->}[rr]&&\mathbf{1}\ar@{-->}[rr]\ar@{-->}[rd]&&\mathbf{1}\\
&\mathbf{1}\ar@{-->}[rr]&&\mathbf{1}&\\
}
\]

\emph{Precisely:} The above diagram says that two elements $x,y \in X(d)$ will be identified if there exist a cone $c \in K$, a morphism $t:\mathtt{ou}(c) \to d$ and two elements $x'$ and $y'$ in $X\mathtt{ou}(c)$
such that the identities $X(t)(x') = x $ and $X(t)(y') = y$ hold and the elements $x'$ and $y'$ have the same image via the canonical map $X\mathtt{ou}(c) \to \mathrm{lim}X\mathtt{in}(c)$.

On the other hand, the binary relations contained in $\tilde{\mathfrak{q}}(d)$ were given as part of our assumptions. However, in the sequel, the idea will be to define $\tilde{\mathfrak{q}}(d)$ either as the empty binary relation or as we defined the set $\tilde{\mathfrak{u}}(d)$ in Section \ref{ssec:Rectification}. In the latter case, in order to make sense of $\tilde{\mathfrak{q}}(d)$, we need to suppose that the image $X(d)$ takes the form given below for some functor $Y:D \to \mathbf{Set}$ and binary relation $R':D \to \mathbf{Set}$.
\[
X(d):=Y(d)/R'(d)+\sum_{c \in K}D(\mathtt{ou}(c),d)\times Y[c]
\]

The quotient $Y(d)/R'(d)$, which will later be shortened as $Y'(d)$, is supposed to identify pairs of elements coming from a previous $\Upgamma^K$-quotient $\tilde{\mathfrak{p}}(d)$. In this case, the pairs contained in the relation $\tilde{\mathfrak{q}}(d) = \mathrm{Rec}(!_X,\tilde{\mathfrak{p}})$ are those pairs $(x,y):\mathbf{1}+\mathbf{1} \to X(d)$ that are the top parts of commutative diagrams of the form displayed below, where the leftmost commutative square is one of those obstruction squares constructed in Section \ref{ssec:Rectification}.
\[
\xymatrix{
\mathbf{1}+\mathbf{1}\ar[d]\ar[r]_-{(x',y')}\ar@/^1.5pc/[rrr]^-{(x,y)}&\mathrm{lim}_zX(\mathtt{in}(c)(z))\ar[d]\ar[r]_-{\varsigma_z}&X(\mathtt{in}(c)(z))\ar[d]\ar[r]&X(d)\ar[d]\\
\mathbf{1}\ar[r]&\mathbf{1}\ar[r]&\mathbf{1}\ar[r]&\mathbf{1}
}
\]

\emph{Precisely:} After unravelling the details of the construction of the corresponding obstruction square, the above diagram says that two elements 
\[
x,y \in Y'(d)+\sum_{c \in K}D(\mathtt{ou}(c),d)\times Y[c]
\] will be identified if there exist a cone $c \in K$, say encoded by a natural transformation $\rho:\Delta\mathtt{ou}(c) \Rightarrow \mathtt{in}(c)$, an element $z \in \mathtt{Es}(c)$, a morphism $t:\mathtt{in}(c)(z) \to d$ and two elements $x'$ and $y'$ living in $X(\mathtt{in}(c)(z))$ of the form
\[
\left\{
\begin{array}{lll}
x' &= (\rho_z,(x_z)_{z \in \mathtt{Es}(c)}) &\in D(\mathtt{ou}(c),\mathtt{in}(c)(z))\times Y[c]\\
y' &= x_z &\in  Y'(\mathtt{in}(c)(z))
\end{array}
\right.
\]
such that the following relations hold.
\[
x =(t \circ \rho_z,(x_z)_{z \in \mathtt{Es}(c)}) \in D(\mathtt{ou}(c),d)\times Y[c]
\quad\quad\text{ and }\quad\quad
y = Y'(t)(y') \in  Y'(d)
\]

We can clearly see that the role of two binary relations $\tilde{\mathfrak{q}}(d)$ and $\Upgamma_{A,1}^{K}(!_X)(d)$ is to turn the canonical arrow $X(\mathtt{ou}(c)) \to X[c]$ into a surjection and an injection, respectively.
\end{example}

\begin{example}[Comparison with Kelly's construction]
\label{exa:clarifies_description_colimits}
Let us compare the quotients acting on the pushout object $[!_X/\mathfrak{q}]$, as described in Examples \ref{exa:elimination_of_quotients_sketches_the_binary_relations} and \ref{exa:elimination_of_quotients_sketches} (where $!_X$ denotes the canonical arrow $X \to \mathbf{1}$), with those acting on the pushout object of Kelly's construction \cite{Kelly}.
Recall that, for each cone $c \in K$, the latter is given by a well-pointed endofunctor $\mathrm{id} \Rightarrow P_c$ in $\mathbf{Set}^D$. More specifically, if we take $c$ to be a cone of the usual the form $$\rho:\Delta_{\mathtt{Es}(c)}(\mathtt{ou}(c)) \Rightarrow \mathtt{in}(c)$$ in $D$, then for every functor $X:D \to \mathbf{Set}$, the object $P_c(X)$ can be computed in $\mathbf{Set}^D$ as the pushout object of the following span \cite[diag. (10.1), p. 31]{Kelly}, whose components are further detailed below, while the natural transformation $\mathrm{id} \Rightarrow P_c$ is the bottom arrow of the resulting pushout square.
\[
\xymatrix{
D(\mathtt{ou}(c),\_)\times X(\mathtt{ou}(c)) + \mathrm{col}_{\mathtt{Es}(c)}D(\mathtt{in}(c),\_) \times X[c] \ar@{=>}[r]\ar@{=>}[d]& *+!L(.6){D(\mathtt{ou}(c),\_) \times X[c]}\\
X(\_)&
}
\]

For every object $d \in D$, we can decompose the previous span in four parts as follows:

(1) The arrow given below, part of the vertical leg, maps every pair $(t,x)$, where $t$ is an arrow $\mathtt{ou}(c) \to d$ and $x \in X(\mathtt{ou}(c))$, to the element $X(t)(x)$ in $X(d)$; 
\[
D(\mathtt{ou}(c),d) \times X(\mathtt{ou}(c)) \to X(d)
\]

(2) The arrow given below, also part of the vertical leg, maps every pair $(t,(x_z)_{z \in \mathtt{Es}(c)})$, where $t$ is an arrow $\mathtt{in}(c)(z) \to d$ in the colimit $\mathrm{col}_z\,D(\mathtt{in}(c)(z),d)$ and $(x_z)_{z \in \mathtt{Es}(c)}$ is a tuple in $X[c] = \mathrm{lim}\, X \circ \mathtt{in}(c)$, to the element $X(t)(x_z)$ in $X(d)$;
\[
\mathrm{col}_{\mathtt{Es}(c)}D(\mathtt{in}(c),d) \times X[c] \to X(d)
\]

(3) The arrow given below, part of the horizontal leg,  is induced by the canonical arrow $X(\mathtt{ou}(c)) \to X[c]$ and maps every pair $(t,x)$ to the pair $(t,x')$, where $x'$ is the tuple $(\rho_z(x))_{z \in \mathtt{Es}(c)}$ in the limit object $X[c]$; 
\[
D(\mathtt{ou}(c),d) \times X(\mathtt{ou}(c)) \to D(\mathtt{ou}(c),d) \times X[c]
\]

(4) The arrow given below, also part of the horizontal leg, is induced by the canonical arrow $\mathrm{col}_{z}D(\mathtt{in}(c)(z),d) \to D(\mathtt{ou}(c),d)$ and maps every pair $(t,x)$, where $t$ is an arrow $\mathtt{in}(c)(z) \to d$ for some object $z \in \mathtt{Es}(c)$ and $x \in X[c]$, to the pair $(t \circ \rho_z,x)$; 
\[
\mathrm{col}_{\mathtt{Es}(c)}D(\mathtt{in}(c),d) \times X[c] \to D(\mathtt{ou}(c),d) \times X[c]
\]

It takes a few lines of calculations to see that the pushout $P_c(X)(d)$ of the previous span evaluated at $d$ can be described as a quotiented sum of the form
\begin{equation}\label{eq:expression_P_c_Kelly}
\Big(X(d) + D(\mathtt{ou}(c),d) \times X[c]\Big)/(R_0+R_1)
\end{equation}
where:

$\triangleright$ $R_0$ identifies all pairs $(x,y)$, where $x \in D(\mathtt{ou}(c),d) \times X[c]$ and $y \in X(d)$, such that there exist $a \in X(\mathtt{ou}(c))$ and an arrow $t:\mathtt{ou}(c) \to d$ for which the following identities hold.
\[
x = (t,(\rho_z(a))_{z \in \mathtt{Es}(c)})\quad\quad\quad\quad\quad y = X(t)(a)
\]

$\triangleright$ $R_1$ identifies all pairs $(x,y)$, where $x \in D(\mathtt{ou}(c),d) \times X[c]$ and $y \in X(d)$, such that there exist $z \in \mathtt{Es}(c)$ and an arrow $t:\mathtt{in}(c)(z) \to d$ for which the following identities hold.
\[
x = (t \circ \rho_z,(x_z)_{z \in \mathtt{Es}(c)})\quad\quad\quad\quad\quad y = X(t)(x_z) 
\]

We can see that the definition of the relation $R_1$ exactly matches that of  the relation $\tilde{\mathfrak{q}}(d) = \mathrm{Rec}(!_X,\tilde{\mathfrak{p}})$ given in Example 
\ref{exa:elimination_of_quotients_sketches_the_binary_relations}. On the other hand, we can check that for every relation $(x_1,x_2) \in \Upgamma_{A,1}^{K}(!_X)(d)$, as described in Example \ref{exa:elimination_of_quotients_sketches_the_binary_relations}, there is an (obvious) element $y$ for which both relations $x_1R_0 y$ and $x_2R_0 y$ are satisfied. 

However, a relation of the form $x R_0 y$ cannot be retrieved from the union of the relations $\Upgamma_{A,1}^{K}(!_X)(d)$ and $\tilde{\mathfrak{q}}(d)$, given in Example \ref{exa:elimination_of_quotients_sketches_the_binary_relations}. It can only be retrieved if one allows a use of these relations up to quotients. Indeed, the reader can check that the identification of the second line, below, cannot be made unless the one given in the fist line has already occured.
\[
\begin{array}{lcccc}
\cellcolor[gray]{0.8}&\cellcolor[gray]{0.8}&\cellcolor[gray]{0.8}\text{elt.}&\cellcolor[gray]{0.8}\text{Relation}&\cellcolor[gray]{0.8}\text{elt.}\\
\vspace{-5pt}
&&&&\\
\text{first identify}&&(\rho_z,(\rho_z(a))_{z \in \mathtt{Es}(c)})& \tilde{\mathfrak{q}}(d)  & X(\rho_z(a))\\
\vspace{-9pt}
&&&&\\
\text{which then allows us to identify} &&(t,(\rho_z(a))_{z \in \mathtt{Es}(c)})& \Upgamma_{A,1}^{K}(!_X)(d) & X(t)(a)\\
\end{array}
\]

As mentioned in Section \ref{ssec:result_2}, Kelly's construction is pursued by pushing out all the maps $X \Rightarrow P_c(X)$ to give a natural transformation $X \Rightarrow P(X)$ where $P(X)$ identifies each component $X$ appearing in the expression of the objects $P_c(X)$ for every $c \in K$. We therefore obtain an expression as follows, for very object $d \in D$.
\[
P(X)(d) = \Big(\sum_{c \in K}\Big(X(d) + D(\mathtt{ou}(c),d) \times X[c]\Big)/(R_0+R_1) \Big)/X(d)
\]

This expression should be compared with the (similar) expression of the $\Upgamma^K$-realisation $[!_X/\emptyset]$ obtained in Example \ref{exa:elimination_of_quotients_sketches_the_binary_relations}, whose sum over $K$ is, here, quotient-free.
\[
[!_X/\emptyset](d) = X(d)/\Upgamma_{A,1}^{K}(!_X)(d)+\sum_{c \in K}D(\mathtt{ou}(c),d) \times X[c]
\]

Because the relations contained in $\Upgamma_{A,1}^{K}(!_X)(d)$ can be written as a zigzag of relations in $R_0$, we can construct an obvious arrow from $[!_X/\emptyset]$ to $P(X)$ matching all the components $D(\mathtt{ou}(c),d) \times X[c]$ together (here, the symbol $\sim$ stands for the obvious relation).
\[
\begin{array}{ccc}
[!_X/\emptyset] &\Rightarrow& P(X)\\
X/\sim + \sum_{c \in K} D(\mathtt{ou}(c),\_) \times X[c] &\Rightarrow& P(X)
\end{array}
\]

In fact, our earlier discussion showed that, if we denote $X_1=[!_X/\emptyset]$ and $X_{n+1}=[!_{X_n}/\mathfrak{u}_n]$ where $\mathfrak{u}_1=\mathrm{Rec}(!_{X},\emptyset)$ and $\mathfrak{u}_{n+1}=\mathrm{Rec}(!_{X_n},\mathfrak{u}_n)$, then we can continue this process iteratively, by matching the components of the sum over $K$, so that we have arrows as follows.
\[
\begin{array}{ccc}
[!_{X_1}/\mathfrak{u}_1] &\Rightarrow& P(P(X))\\
X_1/\sim + \sum_{c \in K} D(\mathtt{ou}(c),\_) \times X_1[c] &\Rightarrow& P(P(X))\vspace{3pt}\\
\vdots&\vdots&\vdots\vspace{3pt}\\
{[!_{X_n}/\mathfrak{u}_n]} &\Rightarrow& P^{n+1}(X)\\
X_n/\sim + \sum_{c \in K} D(\mathtt{ou}(c),\_) \times X_n[c] &\Rightarrow& P^{n+1}(X)
\end{array}
\]

One can check that all these arrows are compatible, in an obvious way, with the arrows $X_n \Rightarrow [!_{X_n}/\mathfrak{u}_n]$ and $P^{n}(X) \Rightarrow P^{n+1}(X)$. However, one of our previous remarks on the fact that $R_0$ can only be retrieved from the relations $\Upgamma_{A,1}^{K}(!_X)(d)$ and $\tilde{\mathfrak{q}}(d)$ up to quotients indicates us that if there exists a pair of dashed arrows making the following diagram commute

\[
\xymatrix@C-15pt@R-5pt{
&&{X_n(\_)}\ar[rr]\ar[d]&&*+!L(.7){[!_{X_n}/\mathfrak{u}_n](\_)}\ar[d]\\
&&P^{n}(X)(\_)\ar[rr]|\hole&&P^{n+1}(X)(\_)\\
P^{n}(X)(\_)\ar@{-->}[rruu]\ar@{=}[rru]\ar[rr]&&*+!L(.9){P_cP^{n}(X)(\_)}\ar@{-->}[rruu]\ar[rru]&&
}\quad\quad(\text{for $n \geq 1$})
\]
then the front arrow must factorise through the following canonical arrow (see the reason below).
\[
\xymatrix{
X_n(\_)/\sim \ar[r]& X_n(\_)/\sim + \sum_{c \in K} D(\mathtt{ou}(c),\_) \times X_n[c],
}
\]
Indeed, otherwise we could derive a contradiction from the elements of the form 
\[
(\rho_z,(x_z)_{z \in \mathtt{Es}(c)}) \in D(\mathtt{ou}(c),\_) \times X_n[c],
\]
which must be identified with the elements $x_z$ in $P_cP^{n}(X)(\_)$ via the relation $R_1$, but must be left free in the expression of $[!_{X_n}/\mathfrak{u}_n](\_)$. The empty case $X_n[c] = \emptyset$ obviously leads to the same conclusion. 

If we now look at Formula (\ref{eq:expression_P_c_Kelly}), this factorisation means that that all the elements in the component 
$
D(\mathtt{ou}(c),d) \times P^n(X)[c]
$ 
of $P_cP^n(X)(d)$ must be identified with elements in the other component $P^n(X)(d)$. From the point of view of the relation $R_0$ at $d=\mathtt{ou}(c)$ where $t$ is taken to be the identity on $\mathtt{ou}(c)$, this means that the canonical arrow $P^n(X)(\mathtt{ou}(c)) \to P^n(X)[c]$ must be a surjection. 

Finally, observe that, when $n>0$, the arrow $P^n(X)(\mathtt{ou}(c)) \to P^n(X)[c]$ is also an injection because the images of $P^n(X)$ are quotiented by the relations $R_0$ and hence the relation $\Upgamma_{A,1}^{K}(!_{P^{n-1}(X)})(d)$, which precisely characterises its injectiveness (see Example \ref{exa:elimination_of_quotients_sketches_the_binary_relations}). In other words, the canonical arrow $P^n(X)(\mathtt{ou}(c)) \to P^n(X)[c]$ is a bijection, which makes the object $P^n(X)$ a model for $(D,\{c\})$.
\end{example}

\section{Combinatorial Categories and Their Oeuvres}\label{sec:Combinatorial_cat}

The notion of combinatorial category encompasses all the assumptions that are necessary to the application of the small object argument in the case of systems of premodels.

\subsection{Numbered Constructor}
Let $\mathcal{B}$, $\mathcal{C}$ be two categories and $D$ be a small category. A \emph{numbered constructor of type $D[\mathcal{B},\mathcal{C}]$} consists of a constructor $\Upgamma$ of type $D[\mathcal{B},\mathcal{C}]$, where $\mathcal{C}$ has coproducts, together with a limit ordinal $\kappa$ such that the category $\mathcal{B}$ admits colimits over every limit ordinal 
$\lambda \in \kappa+1$ when seen as a preorder category. Such a structure will be denoted as a pair $(\Upgamma,\kappa)$ where $\Upgamma$ will be equipped with its usual notational conventions.

\subsection{Factorisable Morphisms}\label{ssec:Factorisable_morphisms}
Let $(\Upgamma,\kappa)$ be a numbered constructor of type $D[\mathcal{B},\mathcal{C}]$. A morphism $f:X \to Y$ in $\mathcal{B}$ will be said to be \emph{$(\Upgamma,\kappa)$-factorisable} if it is equipped with a sequence $(f_n,\mathfrak{u}_{n})_{n \in \kappa+1}$ of ideal $\Upgamma$-quotiented arrows in $\mathcal{B}$ satisfying the following conditions:
\begin{itemize}
\item[$\triangleright$] \textbf{initial case:} $f_0= f$;
\item[$\triangleright$] \textbf{successor cases:} $\mathrm{Rec}(f_n,\mathfrak{u}_{n}) = (f_{n+1},\mathfrak{u}_{n+1})$;
\item[$\triangleright$] \textbf{limit cases:} for any (infinite) limit ordinal $\lambda \in \kappa+1$, the arrow $f_{\lambda}$ is the colimit $\mathrm{col}_{n \in \lambda} f_n$ in $\mathcal{B}$ of the following diagram over the category $\lambda$.
\begin{equation}\label{eq:definition_f_n_constructor_end}
\xymatrix{
X\ar[d]_{f_0}\ar[r]^-{\{ f_0 \}_{\mathfrak{u}_0}}&[f_0/\mathfrak{u}_0]\ar[d]_{f_1}\ar[r]^-{\{ f_1 \}_{\mathfrak{u}_1}}&[f_1/\mathfrak{u}_1]\ar[d]_{f_2}\ar[r]^-{\{ f_2 \}_{\mathfrak{u}_2}}&\dots \ar[r]^-{\{ f_n \}_{\mathfrak{u}_n}}&[f_{n}/\mathfrak{u}_{n}]\ar[d]_{f_{n+1}}\ar[r]&\dots\\
Y\ar@{=}[r]&Y\ar@{=}[r]&Y\ar@{=}[r]&\dots\ar@{=}[r]&Y\ar@{=}[r]&\dots
}
\end{equation}
\end{itemize}

\begin{convention}\label{conv:construction_G_f_sequential_functor}
For every (infinite) limit ordinal $\lambda \in \kappa+1$, the domain of the arrow $f_{\lambda}$ will be denoted by $[f/\mathfrak{u}]_{\lambda}$.
The object $[f/\mathfrak{u}]_{\lambda}$ is by definition the colimit of the sequence of arrows $\{f_n \}_{\mathfrak{u}_n}$ where $n$ runs over $\lambda$ (see Diagram (\ref{eq:definition_f_n_constructor_end})). We will later denote by $\chi_n^{\lambda}(f)$ the associated canonical arrow $[f_{n}/\mathfrak{u}_{n}] \to [f/\mathfrak{u}]_{\lambda}$.
\end{convention}

By induction, we may show that the arrows $\chi_n^{\lambda}(f)$ and $\{f_n \}_{\mathfrak{u}_n}$ define a sequential functor $G(f):\kappa+1 \to \mathcal{B}$ with the following mapping rules.
\[
\begin{array}{llll}
n+1 &\mapsto & [f_{n}/\mathfrak{u}_{n}]&\text{(succ. objects)}\\
\lambda &\mapsto & X\text{ if }\lambda = 0\text{ and }[f/\mathfrak{u}]_{\lambda}\text{ otherwise.}&\text{(lim. objects)}\\
n+1 < n+2 & \mapsto & \{ f_n \}_{\mathfrak{u}_n}&\text{(succ. arrows)}\\
n+1 < \lambda &\mapsto &\chi_n^{\lambda}(f)&\text{(lim. arrows)}\\
\lambda < \lambda+1 &\mapsto &  \{ f_{\lambda} \}_{\mathfrak{u}_{\lambda}}&\text{(lim. arrows)}
\end{array}
\]
\begin{remark}\label{rem:spc_of_f_factorisable_functor}
The functor $G(f):\kappa+1 \to \mathcal{B}$ turns the mapping $n \mapsto f_n$ into an obvious functor $G'(f):\kappa+1 \to \mathcal{B}^{\mathbf{2}}$, which also lifts to the category $\Upgamma\mathcal{B}^{\mathbf{2}}$ via the mapping $n \mapsto (f_n,\mathfrak{u}_n)$ (see Diagram (\ref{eq:definition_f_n_constructor_end})).
\end{remark}

\begin{theorem}\label{th:admissible_quotiented_factorisable}
Let $\kappa$ denote a limit ordinal and $(K,T,\mathcal{P},\mathtt{V})$ be a system of $R$-premodels over a small category $D$ in a category $\mathcal{C}$. If $\mathcal{C}$ is cocomplete, $R$ preserves colimits over every limit ordinal $\lambda \in \kappa+1$ and the inclusion $\mathcal{P} \hookrightarrow \mathbf{Pr}_{\mathcal{C}}(K,T,R)$ is an identity, then every morphism in $\mathcal{P}$ may be equipped with the structure of a $(\Upgamma^K,\kappa)$-factorisable morphism. 
\end{theorem}
\begin{proof}
First, the assumption that $\mathcal{C}$ is cocomplete and $R$ preserves colimits over every limit ordinal $\lambda \in \kappa+1$ implies that $\mathbf{Pr}_{\mathcal{C}}(K,T,R)$ admits colimits over every limit ordinal $\lambda \in \kappa+1$.
We are now going to show that every morphism $(f,a):(X,S,e) \Rightarrow (Y,S',e')$ of the category $\mathbf{Pr}_{\mathcal{C}}(K,T,R)$ may be equipped with the structure of a $(\Upgamma^K,\kappa)$-factorisable morphism by induction. 
Let us define the sequence of $\Upgamma^K$-quotiented arrow $(f_n,a_n,\mathfrak{u}_{n})_{n \in \kappa+1}$ as follows: 
\begin{itemize}
\item[$\triangleright$] For the initial case, take $(f_0,a_0)$ to be the morphism $(f,a):(X,S,e) \Rightarrow (Y,S',e')$ and $\mathfrak{u}_{0}$ to be given by the collection of empty functors $\{\emptyset: \mathbf{1} \to \mathbf{Set}\}_{\theta \in I,s \in \uplambda_{\theta}(X)}$;
\item[$\triangleright$] By Theorem \ref{th:admissible_quotiented_ideal}, the $\Upgamma^K$-quotiented arrow $(f_n,a_n,\mathfrak{u}_n)$ is ideal and we can take the next $\Upgamma^K$-quotiented arrow $(f_{n+1},a_{n+1},\mathfrak{u}_{n+1})$ to be $\mathrm{Rec}(f_n,a_n,\mathfrak{u}_{n})$;
\item[$\triangleright$] For any (infinite) limit ordinal $\lambda \in \kappa+1$, the arrow $(f_{\lambda},a_{\lambda})$ is given by the colimit $\mathrm{col}_{n \in \lambda} (f_{n},a_{n})$ in $\mathbf{Pr}_{\mathcal{C}}(K,T,R)$ of  Diagram (\ref{eq:definition_f_n_constructor_end_prop}) over the category $\lambda$ while $\mathfrak{u}_{\lambda}$ is given by the collection of empty functors $\{\emptyset: \mathbf{1} \to \mathbf{Set}\}_{\theta \in I,s \in \uplambda_{\theta}(X)}$
\begin{equation}\label{eq:definition_f_n_constructor_end_prop}
\xymatrix{
(X,S,e)\ar[d]_{(f_0,a_0)}\ar[r]^-{\{ f_0 \}_{\mathfrak{u}_0}}&[f_0/\mathfrak{u}_0]\ar[d]_{(f_1,a_1)}\ar[r]^-{\{ f_1 \}_{\mathfrak{u}_1}}&[f_1/\mathfrak{u}_1]\ar[d]_{(f_2,a_2)}\ar[r]^-{\{ f_2 \}_{\mathfrak{u}_2}}&\dots \ar[r]^-{\{ f_n \}_{\mathfrak{u}_n}}&[f_{n}/\mathfrak{u}_{n}]\ar[d]_{(f_{n+1},a_{n+1})}\ar[r]&\dots\\
(Y,S',e')\ar@{=}[r]&(Y,S',e')\ar@{=}[r]&(Y,S',e')\ar@{=}[r]&\dots\ar@{=}[r]&(Y,S',e')\ar@{=}[r]&\dots
}
\end{equation}
\end{itemize}

By Principle of Transfinite Induction, the preceding construction equip the morphism $(f,a):(X,S,e) \Rightarrow (Y,S',e')$ with the structure of a $(\Upgamma^K,\kappa)$-factorisable morphism.
\end{proof}

\begin{corollary}\label{cor:admissible_quotiented_factorisable}
Let $\kappa$ denote a limit ordinal and $(K,T,\mathcal{P},\mathtt{V})$ be a fibered system of $R$-premodels over a small category $D$ in a category $\mathcal{C}$. If $\mathcal{C}$ is cocomplete and $R$ preserves colimits over every limit ordinal $\lambda \in \kappa+1$, then~every morphism in $\mathcal{P}$ may be equipped with the structure of a $(\Upgamma^K,\kappa)$-factorisable morphism. 
\end{corollary}
\begin{proof}
Follows from fiberedness and Theorem \ref{th:admissible_quotiented_factorisable}.
\end{proof}

\begin{example}[Systems of premodels]\label{exa:Categories_of_premodels_factorisable_morphism_subcategories}
Let $\kappa$ denote a limit ordinal and $(K,T,\mathcal{P},\mathtt{V})$ be a system of $R$-premodels over a small category $D$ in a category $\mathcal{C}$ where $\mathcal{P}$ may be identified with the category of $R$-premodels $\mathcal{C}^D \hookrightarrow \mathbf{Pr}_{\mathcal{C}}(K,T,R)$ -- hence $R$ is an identity. It follows from Example \ref{exa:functor_category_fibered} and Corollary \ref{cor:admissible_quotiented_factorisable} that the morphisms of $\mathcal{P}$ are all $(\Upgamma^K,\kappa)$-factorisable.
\end{example}

\begin{proposition}\label{prop:narrative_constructor_mathfrak_I_strict}
Let $f:X \to Y$ be a $(\Upgamma,\kappa)$-factorisable morphism. For every object $d$ in $D$, the mapping $n \mapsto \mathbb{T}^{f_n,\mathfrak{u}_n}(d)$ induces an oeuvre $\mathfrak{O}[f](d):\kappa+1 \to \mathbf{Ltom}(\mathcal{C})$ of theme $\Upsilon_d(Y)$. This induces a functor $\mathfrak{O}[f]:D \to \mathbf{Oeuv}(\mathcal{C},\kappa)$ whose images are strict narratives of degree 1.
\end{proposition}
\begin{proof}
The fact that the mapping $n \mapsto \mathbb{T}^{f_n,\mathfrak{u}_n}(d)$ induces an oeuvre
follows from Proposition \ref{prop:merolytic_tome_functorial} and Remark \ref{rem:spc_of_f_factorisable_functor}. 
One thus obtains an oeuvre $\mathfrak{O}[f](d):\kappa+1 \to \mathbf{Ltom}(\mathcal{C})$ of theme $\Upsilon_d(Y)$. It follows from Proposition \ref{prop:merolytic_tome_functorial_D} that the mapping $d \mapsto \mathfrak{O}[f](d)$ defines a functor 
$D \to \mathbf{Oeuv}(\mathcal{C},\kappa)$.
The narrative structure is defined as follows:
\begin{itemize}
\item[(1)] for every $n \in \kappa+1$, the set of events $J_d^n$ contains all the functors $\mathbf{1} \to \Upgamma^{f_n,\mathfrak{u}_n}(d)$;
\item[(2)] for every $n \in \kappa+1$ and functor $i:\mathbf{1} \to \Upgamma^{f_n,\mathfrak{u}_n}(d)$ in $J_d^n$, the viewpoint associated with the arrow
\[
\xymatrix{
\mathrm{col}_{i}\partial \varphi^{\mathfrak{u}_n}_d \ar@{=>}[r]\ar@{}@<+0.1ex>[r]^-{\textit{shift}}& \mathrm{col}\partial \varphi^{\mathfrak{u}_n}_d 
\ar@{=>}[r]\ar@{}@<+0.3ex>[r]^-{\textit{content}}&  \Upsilon_{d}(f_n)  \ar@{=>}[rr]^-{\Upsilon_{d}(\{f_{n}\}_{\mathfrak{u}_n})}&&  \Upsilon_{d}(f_{n+1})
}
\]
is given by the $\Upgamma$-realisation of $(f_n,\mathfrak{u}_n)$ (see Diagram (\ref{eq:combinatorial_constructors_pushout_def})) that may be inserted in the content $ \mathrm{col}\partial \varphi^{\mathfrak{u}_n}_d \Rightarrow \Upsilon_{d}(f_n)$, so that we obtain a lift $\uppi_0$ for the previous composite that makes the following diagram commute.
\[
\xymatrix@R-0.5pc@C-0.5pc{
*+!R(.5){\mathrm{col}_{\mathbf{1}}(\mathbf{A}_k \circ i)}\ar[rr]\ar[dd]_{\mathrm{col}_{\mathbf{1}}(\partial\varphi^{\mathfrak{u}_n} \circ i)}&&*+!L(.4){[f_{n-1},\mathfrak{u}_{n-1}](d)}\ar[dl]|{p_{f_{n}}^{\mathfrak{u}_{n}}(d)}\ar[dd]^{\Upsilon_d(f_{n})}\\
&[f_{n},\mathfrak{u}_{n}](d)\ar[dr]|{\Upsilon_d(f_{n+1})}&\\
*+!R(.5){\mathrm{col}_{\mathbf{1}}(\mathbf{B}_k \circ i)}\ar[ru]^{\uppi_0}\ar[rr]&&\Upsilon_d(Y)
}
\]

Note that the object $[f_{n-1}/\mathfrak{u}_{n-1}]$ must stand for $X$ when $n=0$.
\end{itemize}
By definition (see Section \ref{sssec:strict_narratives}), the previous narrative is strict.
\end{proof}

\begin{proposition}
Let $f:X \to Y$ be a $(\Upgamma,\kappa)$-factorisable morphism. For every object $\theta \in I$ and $s \in \uplambda_{\theta}(X)$, the mapping $n \mapsto \mathbb{T}^{\mathrm{op}}_{\theta,s}(f_n)$
induces an oeuvre $\mathfrak{O}^{\star}_{\theta,s}[f]:\kappa+1 \to \mathbf{Ltom}(\mathcal{C}^{{\mathbf{2}}})$ of theme $\Phi_{\theta}(Y)$ that is equipped with the structure of a narrative of degree 2.
\end{proposition}
\begin{proof}
The fact that the mapping $n \mapsto \mathbb{T}^{\mathrm{op}}_{\theta,s}(f_n)$ induces an oeuvre
follows from Proposition \ref{prop:operadic_tome_functorial} and Remark \ref{rem:spc_of_f_factorisable_functor}. 
One thus obtains an oeuvre $\mathfrak{O}^{\star}_f(\theta):\kappa+1 \to \mathbf{Ltom}(\mathcal{C}^{{\mathbf{2}}})$ of theme $\Phi_{\theta}(Y)$. 
The narrative structure is defined as follows:
\begin{itemize}
\item[(1)] for every $n \in \kappa+1$, the set of events $J_{\theta,s}^n$ contains all the functors $\mathbf{1} \to \Upgamma_A(f_n)\{\theta\}_s$;
\item[(2)] for every $n \in \kappa+1$ and functor $i:\mathbf{1} \to \Upgamma_A(f_n)\{\theta\}_s$ in $J_n^{\star \theta}$, the viewpoint is given by the pair $(\uppi_0,\uppi_1(\theta,s))$ defined in Section \ref{ssec:Rectification} if one replaces the functor $i_{\mathbf{c}}$ with $i$ and the $\Upgamma$-quotiented arrow $(f,\mathfrak{q})$ with $(f_n,\mathfrak{u}_n)$. As noticed in Remark \ref{rem:structure_narrative_degree_2}, the version of Diagram (\ref{eq:factorisation_combi_cat_premodel_step_k-1_cuboid}) for these parameters provides the wanted lift. 
\end{itemize}

This finishes the proof.
\end{proof}

\subsection{Combinatorial Categories}\label{sssec:Combinatorial_categories}
\smallskip

\begin{convention}
Let $\Upgamma$ be a constructor as in Section \ref{sssec:Constructor}. Recall that for every $\theta \in I$, the image $\upalpha(\theta) \in \mathcal{C}^{\mathbf{2} \times \mathbf{2}}$ encodes the diskad of a vertebra whose stem is given by $\upomega(\theta) \in \mathcal{C}^{\mathbf{2}}$. According to the conventions set in Section \ref{conv:some_more_notations_vertebrae}, if this vertebra is denoted by $v_{\upalpha}(\theta):=\|\gamma_2,\gamma_1\| \cdot \beta$, the diskad $\upalpha(\theta)$ is seen as an arrow $\gamma_2 \Rightarrow \beta\delta_1$. We shall let $\mathcal{G}\mathbf{en}(\Upgamma)$ denote the set consisting of the domain and codomain of the coseed of the vertebra $v_{\upalpha}(\theta)$ (i.e., the domain and codomain of $\gamma_1$) for every object $\theta  \in I$. Similarly, we shall let $\mathcal{C}\mathbf{osd}(\Upgamma)$ denote the set consisting of the coseeds of every vertebra $v_{\upalpha}(\theta)$ for every object $\theta \in I$. 
\end{convention}

\begin{remark}
For every object $\theta  \in  I$, the set $\mathcal{C}\mathbf{osd}(\Upgamma)$ may alternatively be seen as the set of domains of every codiskad $\mathbf{disk}(v_{\upalpha}(\theta)^{\mathrm{rv}}):\gamma_1 \Rightarrow \beta \circ \delta_2$ for every object $\theta \in I$.
\end{remark}

For every limit ordinal $\kappa$, a category $\mathcal{B}$ will be said to be \emph{$\kappa$-combinatorial} in a category $\mathcal{C}$ if it is equipped with a numbered constructor $(\Upgamma,\kappa)$ of type $D[\mathcal{B},\mathcal{C}]$ such that
\begin{itemize}
\item[(1)] every morphism in $\mathcal{B}$ is $(\Upgamma, \kappa)$-factorisable;
\item[(2)] for every object $f$ in $\mathcal{B}^{\mathbf{2}}$ and object $\theta$ in $I$, the functor 
$\Phi_{\theta,s} \circ G(f):\kappa+1 \to \mathcal{C}^{\mathbf{2}}$, which is the context functor of the oeuvre $\mathfrak{O}_{\theta,s}^{\star}[f]$, is $\mathcal{C}\mathbf{osd}(\Upgamma)$-convergent.
\end{itemize}

\begin{remark}\label{rem:context_functor_usually_convergent}
In practice, it is easy to prove that for every morphism $f:X \to Y$ in $\mathcal{B}$ and object $d$ in $D$, the context functor 
\[
\Upsilon_d \circ G(f):\kappa+1 \to \mathcal{C}
\]
of the oeuvre $\mathfrak{O}[f](d)$ is $\mathcal{G}\mathbf{en}(\Upgamma)$-convergent. This is generally due to the fact that the context functor $\Upsilon_d \circ G(f)$ is sequential and the vertebrae $\{v_{\upalpha}(\theta)\}_{\theta \in I}$ are rather `small'.
\end{remark}

\begin{example}\label{exa:convergence_functor_limit_sketch}
The following discussion continues the discussion began in Examples \ref{exa:elimination_of_quotients_sketches} and  \ref{exa:elimination_of_quotients_sketches_the_binary_relations}. In this respect, let $(D,K)$ be a limit sketch seen as a croquis and consider the system of premodels defined in Example \ref{ex:Models_for_a_sketch_system} for the category $\mathbf{Set}^{D}$. If one numbers the constructor $\Upgamma^K$ with an ordinal $\kappa \geq \omega$, then for every morphism $f:X \to Y$ in $\mathbf{Set}^{D}$ and object $d$ in $D$, the context functor 
\[
\Upsilon_d \circ G(f):\kappa+1 \to \mathbf{Set}
\]
of the oeuvre $\mathfrak{O}[f](d)$ is $U$-convergent for any finite set $U$. This comes from the fact that any sequential functor of the form $\kappa+1 \to \mathbf{Set}$ where $\kappa \geq \omega$ is convergent with respect to finite sets. Now, in the case of the constructor $\Upgamma^K$, the set $\mathcal{G}\mathbf{en}(\Upgamma^K)$ is made of the finite sets $\emptyset$, $\mathbf{1}$ and $\mathbf{2}=\mathbf{1}+\mathbf{1}$, so the functor $\Upsilon_d \circ G(f):\kappa+1 \to \mathbf{Set}$ is $\mathcal{G}\mathbf{en}(\Upgamma^K)$-convergent.
\end{example}

\begin{example}\label{exa:convergence_functor_topological_spaces}
Let $\mathbf{CW}$ denote the wide subcategory of $\mathbf{Top}$ restricted to inclusions $A \hookrightarrow B$ defining relative CW-complex structures (see \cite{Hatcher}). It is well-known that any sequential functor of the form $\kappa+1 \to \mathbf{CW}$, where $\kappa \geq \omega$, is convergent with respect to compact topological spaces (see Appendix of \cite{Hatcher}). Since topological spheres and discs are compact, it follows that the functor $\Upsilon_d \circ G(f):\kappa+1 \to \mathbf{Top}$ associated with the constructors of the systems of $\Omega$-premodels defined in Examples \ref{ex:Segal_spaces} and \ref{ex:Complete_Segal_spaces} is $\mathcal{G}\mathbf{en}(\Upgamma^K)$-convergent when $K$ is taken to be equal to $\mathbf{Seg}(\mathbf{\Delta}^{\mathrm{op}})$ and $\mathbf{Cseg}(\mathbf{\Delta}^{\mathrm{op}})$, respectively.
\end{example}

\begin{example}[Systems of premodels]\label{exa:category_of_models_are_combinatorial_in_nice_cases}
Let $(K,T,\mathcal{P},\mathtt{V})$ be a fibered system of $R$-premodels over a small category $D$ in a category $\mathcal{C}$. In addition, suppose that $\mathcal{C}$ is cocomplete and $R$ preserves colimits over every limit ordinal $\lambda \in \kappa+1$. Corollary \ref{cor:admissible_quotiented_factorisable} shows that every morphism in $\mathcal{P}$ is $(\Upgamma^K, \kappa)$-factorisable for any limit ordinal $\kappa$. Let us prove that the category $\mathcal{P}$ becomes $\kappa$-combinatorial if 
\begin{itemize} 
\item[-] $\kappa$ is a well-chosen limit ordinal;
\item[-] the statement of Remark \ref{rem:context_functor_usually_convergent} holds.
\end{itemize}

As specified by Remark \ref{rem:context_functor_usually_convergent}, for every morphism $f:(X,S,e) \Rightarrow (Y,S',e')$ in $\mathcal{P}$ and object $d$ in $D$, the context functor $G(f)(d):\kappa+1 \to \mathcal{C}$ of the oeuvre $\mathfrak{O}[f](d)$ is generally $\mathcal{G}\mathbf{en}(\Upgamma^K)$-convergent. Recall that this functor lifts to a functor landing in $\mathcal{P}$ as follows.
\[
G(f):
\left\{
\begin{array}{ccc}
\kappa+1 &\to& \mathcal{P}\\
0 &\mapsto& (X,S,e)\\
n &\mapsto& ([f_{n-1},\mathfrak{u}_{n-1}],S,e^{\mathfrak{u}_{n-1}})\\
\lambda &\mapsto& (\mathrm{col}_{n \in \lambda}[f_{n},\mathfrak{u}_{n}],S,\mathrm{col}_{n \in \lambda} e^{\mathfrak{u}_{n}})
\end{array}
\right.
\]

Let $c$ denote a cone of the form $t:\Delta_A(d) \Rightarrow d_{1}$ in $K$ where $d_{1}$ is a functor $A \to D$. 
Let also $g$ denote the functor $(\mathcal{C}^D)^{\mathbf{2}} \to \mathcal{C}^{\mathbf{2}}$ defined in Example \ref{exam:emulation:gluing_functor} where the cone `$r$' used thereof is replaced with the natural transformation $t:\Delta_{A}(d) \Rightarrow d_{1}$. By definition, the following equations hold for every ordinal $n \in \kappa+1$, cone $c \in K$, vertebra $v \in \mathtt{V}$ and element $s \in S(c)$.
\[
\Phi_{(c,v),s}\circ G(f)(n) = \mathcal{G}^K_c(G(f)(n))_s = 
\left\{
\begin{array}{ll}
g (e_{c,s})&\text{ if }n=0\\
g (e^{\mathfrak{u}_{n-1}}_{c,s})&\text{ if }n\text{ is succ.}\\
g(\mathrm{col}_{n \in \lambda} e^{\mathfrak{u}_{n}}_{c,s})&\text{ if }n\text{ is limit}.
\end{array}
\right.
\]

In the case where the inequalities $\mathbf{2} \leq \kappa$ and $|A| \leq \kappa$ hold, Example \ref{exam:lemma:convegence_of_functors:gluing_functor} says that the composite of the functor $G(f):\kappa+1 \to \mathcal{P}$ with the functor $\Phi_{(c,v),s}:\mathcal{P} \to \mathcal{C}^{\mathbf{2}}$ is $\mathcal{C}\mathbf{osd}(\mathtt{V})$-convergent.
In other words, this shows that if $\kappa$ is greater than or equal to the cardinality $|(K,T)|$ and $\mathbf{2}$, then the context functor of the oeuvre $\mathfrak{O}_f^{\star}(\theta)$ is $\mathcal{C}\mathbf{osd}(\mathtt{V})$-convergent.
\end{example}

\begin{definition}[Lifting system]
Let $\mathcal{B}$ be a combinatorial category as defined above. For every morphism $f:X \to Y$ in $\mathcal{B}$, every $\theta \in K$ and $s \in \uplambda_{\theta}(X)$, denote by $J_{\theta,s}^{\mathrm{soa}}$ the lifting system consisting of the functors $\mathbf{1} \to (\mathcal{C}^{\mathbf{2}})^{\mathbf{2}}$ picking out the codiskad $\mathbf{disk}(v_{\upalpha}(\theta)^{\mathrm{rv}}):\gamma_1 \Rightarrow \beta \circ \delta_2$.
\end{definition}

\begin{proposition}\label{prop:J_soa_agrees_with_tome}
For every morphism $f:X \to Y$ in $\mathcal{B}$, every $\theta \in K$ and $s \in \uplambda_{\theta}(X)$, the lifting system $J_{\theta,s}^{\mathrm{soa}}$ agrees with the narrative $\mathfrak{O}^{\star}_{\theta,s}[f]:\kappa+1 \to \mathbf{Ltom}(\mathcal{C}^{\mathbf{2}})$ in the numbered category $(\mathcal{C},\kappa)$.
\end{proposition}
\begin{proof}
To show that the lifting system $J_{\theta,s}^{\mathrm{soa}}$ in $(\mathcal{C},\kappa)$ agrees with the narrative $\mathfrak{O}^{\star}_{\theta,s}[f]$, which is generated by the operadic tomes $\varrho_{\theta}^s:\Upgamma_A(f_n)\{\theta\}_s \to (\mathcal{C}^{\mathbf{2}})^{\mathbf{2}}/\Phi_{\theta}(f_n)_s$ for $n$ running over $\kappa+1$, consider an ordinal $n \in \kappa+1$ and suppose that the functor $\upvarphi:\mathbf{1} \to (\mathcal{C}^{\mathbf{2}})^{\mathbf{2}}$ in $J_{\theta,s}^{\mathrm{soa}}$ that admits a lift $\psi:\mathbf{1} \to (\mathcal{C}^{\mathbf{2}})^{\mathbf{2}}/\Phi_{\theta}(f_n)_s$ along $\partial$ as follows.
\[
\xymatrix{
& (\mathcal{C}^{\mathbf{2}})^{\mathbf{2}}/\Phi_{\theta}(f_n)_s\ar[d]^{\partial}\\
\mathbf{1}\ar@{-->}[ru]^{\psi}\ar[r]_-{\upvarphi}&(\mathcal{C}^{\mathbf{2}})^{\mathbf{2}}
}
\]

By definition, the functor $\psi$ picks out an element in $\mathcal{C}^{\mathbf{2} \times \mathbf{2}}(\upalpha(\theta),\Phi_{\theta}(f_n)_s)$ which is therefore an element of $\Upgamma_A(f_n)\{\theta\}_s$. This means that we found a functor $i:\mathbf{1} \to \Upgamma_A(f_n)\{\theta\}_s$ in the set of events $J_{\theta,s}^n$ whose composite with $\varrho_{\theta}^s$ gives the lift $\psi$ as follows.
\[
\xymatrix{
\Upgamma_A(f_n)\{\theta\}_s\ar[r]^-{\varrho_{\theta}^s}& (\mathcal{C}^{\mathbf{2}})^{\mathbf{2}}/\Phi_{\theta}(f_n)_s\ar[d]^{\partial}\\
\mathbf{1}\ar[u]^-{i}\ar@{-->}[ru]^{\psi}\ar[r]_-{\upvarphi}&(\mathcal{C}^{\mathbf{2}})^{\mathbf{2}}
}
\]

This exactly shows the statement of the proposition.
\end{proof}

\begin{theorem}\label{theo:factorisation_SOA}
Let $\kappa$ be a limit ordinal and $\mathcal{B}$ be a $\kappa$-combinatorial category as defined above. Every morphism $f:X \to Y$ may be factorised into two arrows
\[
\xymatrix{
X\ar[r]^-{\chi_0^{\kappa}(f)}&G(f)(\kappa)\ar[r]^-{f_{\kappa}}&Y
}
\]
such that, for every $\theta \in I$ and $s \in \uplambda_{\theta}(X)$, the arrow $\Phi_{\theta}(f_{\kappa})_s:\Phi_{\theta}(G(f)(\kappa))_s \to \Phi_{\theta}(Y)_s$ in $\mathcal{C}^{\mathbf{2}}$ has the rlp with respect to the codiskad of $v_{\upalpha}(\theta)$ and, for every object $d$ in $D$, the arrow $\Upsilon_d(\chi_0^{\kappa}(f)):\Upsilon_d(X) \to \Upsilon_d(G(f)(\kappa))$ has the llp with respect to every morphism in $\mathbf{rlp}_{\kappa}(E_n(\mathfrak{O}[f](d)))$ (see end of Section \ref{sssec:Oeuvres_and_narratives}) for every $n \in \kappa+1$.
\end{theorem}
\begin{proof}
The factorisation is given by the image of the arrow $0 \to \kappa$ via the functor $G'(f):\kappa+1 \to \mathcal{B}^{\mathbf{2}}$ defined in Remark \ref{rem:spc_of_f_factorisable_functor}. The statement on the arrow $\Phi_{\theta}(f_{\kappa})_s:\Phi_{\theta}(G(f)(\kappa))_s \to \Phi_{\theta}(Y)_s$ follows from Propositions \ref{prop:small_object_argument_general}  and  \ref{prop:J_soa_agrees_with_tome}  since the context functor 
$$\Phi_{\theta} \circ G(f):\kappa+1 \to \mathcal{C}^{\mathbf{2}}$$
of the oeuvre $\mathfrak{O}_{\theta,s}^{\star}[f]$ is $\mathcal{C}\mathbf{osd}(\mathtt{V})$-convergent (and hence $(\mathrm{dom} \circ \upvarphi)$-convergent for every $\upvarphi \in J^{\mathrm{soa}}_{\theta,s}$). The~statement on the arrow $\Upsilon_d(\chi_0^{\kappa}(f)):\Upsilon_d(X) \to \Upsilon_d(G(f)(\kappa))$ follows from Propositions \ref{prop:llp_soa_with_respect_to_tome_event} and  \ref{prop:narrative_constructor_mathfrak_I_strict}, which ensures that $\mathfrak{O}[f](d)$ is a strict narrative for every object $d \in D$.
\end{proof}

Let $\kappa$ be a limit ordinal. A category $\mathcal{C}$ will be said to be \emph{trivially $\kappa$-combinatorial} over a set $\mathtt{G} \subseteq \mathrm{Obj}(\mathcal{C}^{\mathbf{2}})$ if it is $\kappa$-combinatorial when equipped with the numbered constructor $(\Upgamma^{\mathbf{1}},\kappa)$ associated with the obvious category of $\mathrm{id}_{\mathcal{C}}$-premodels $\mathcal{C}^{\mathbf{1}} \hookrightarrow \mathbf{Pr}(\mathbf{1},\mathrm{id},\mathrm{id})$ whose set of vertebrae consists of the following degenerate vertebrae for every arrow $\delta \in \mathtt{G}$.
\[
\xymatrix{
\mathbb{S}\ar@{}[rd]|->>{\text{\huge{\rotatebox[origin=c]{-90}{$\llcorner$}}}}\ar[d]_{\delta}\ar@{=}[r]&\mathbb{S}\ar[d]^{\delta}&\\
\mathbb{D}\ar@{=}[r]&\mathbb{D}\ar@{=}[r]&\mathbb{D}
}
\] 

\begin{corollary}[Quillen's small object argument]\label{cor:factorisation_SOA}
Let $\mathcal{C}$ be a trivially $\kappa$-combinatorial category over a set of arrows $\mathtt{G}$ in $\mathcal{C}$.
Every morphism $f:X \to Y$ in $\mathcal{C}$ may be factorised into two arrows $\chi_0^{\kappa}(f):X \to G(f)(\kappa)$ and
$f_{\kappa}:G(f)(\kappa) \to Y$ where the arrow $f_{\kappa}$ is in the class
$\mathbf{rlp}(\mathtt{G})$ and the arrow $\chi_0^{\kappa}(f)$ is in the class $\mathbf{llp}(\mathbf{rlp}(\mathtt{G}))$.
\end{corollary}
\begin{proof}
Theorem \ref{theo:factorisation_SOA} implies that every morphism $f:X \to Y$ in $\mathcal{C}$ may be factorised into two arrows $\chi_0^{\kappa}(f):X \to G(f)(\kappa)$ and
$f_{\kappa}:G(f)(\kappa) \to Y$ where the arrow $f_{\kappa}$ is in the class $\mathbf{rlp}(\mathtt{G})$
and the arrow $\chi_0^{\kappa}(f)$ has the llp with respect to every morphism in $\mathbf{rlp}(E_n(\mathfrak{O}[f](d)))$ for every $n \in \kappa+1$. However, because of the triviality of our data, it follows that the equality $\mathbf{rlp}(E_n(\mathfrak{O}[f](d))) = \mathbf{rlp}(\mathtt{G})$ holds for every $n \in \kappa+1$.
\end{proof}

\begin{remark}\label{rem:localisation}
For every system of $R$-premodels $(K,T,\mathcal{P},\mathtt{V})$ where: $\mathcal{C}$ is cocomplete; $R:\mathcal{C} \to \mathcal{C}$ preserves colimits over every limit ordinal $\lambda \in \kappa+1$ and $\mathcal{P} \hookrightarrow \mathbf{Pr}_{\mathcal{C}}(K,T,R)$ is combinatorial (see Example \ref{exa:category_of_models_are_combinatorial_in_nice_cases}), Theorem \ref{theo:factorisation_SOA} provides every arrow $!:(X,S,e) \Rightarrow \mathbf{1}$ in $\mathcal{P}$ with a factorisation
\[
\xymatrix{
(X,S,e)\ar@{=>}[r]^-{\chi(f)}&G(X,S,e)\ar@{=>}[r]^-{!}&\mathbf{1}
}
\]
where $G(X,S,e)$ is an $R$-model and the arrow $\chi(f)$ satisfies nice lifting properties. In the case of a category of premodels for a sketch,
Example \ref{exa:elimination_of_quotients_sketches} shows that the `localisation' $(X,S,e) \Rightarrow G(X,S,e)$ admits a presentation as given in Theorem \ref{theo:Elimination_quotient_intro}. There now remains to show that the arrow $(X,S,e) \Rightarrow G(X,S,e)$ is universal. This is the goal of the next and last section.
\end{remark}

\section{Universal Property}\label{sec:Universal_property}

This section discusses the universal properties of the factorisations provided by Theorem \ref{theo:factorisation_SOA}. To do so, we shall require our constructor to be `normal' (see Section \ref{ssec:Normal_constructor}). An existential resut is given in Theorem \ref{th:admissible_quotiented_factorisable_model} while a universal one is given in Theorem \ref{th:localisation_universal}.

\subsection{Normal Constructors}\label{ssec:Normal_constructor}
A constructor $\Upgamma$ of type $D[\mathcal{B},\mathcal{C}]$ will be said to be \emph{normal} if 
\begin{itemize}
\item[(1)] the categories $\mathcal{B}$ and $\mathcal{C}$ possess terminal objects (denoted by $\mathbf{1}$);
\item[(2)] for every $\theta \in I$ and $s \in \uplambda_{\theta}(X)$, the functor $\Phi_{\theta,s}:\mathcal{B} \to \mathcal{C}^{\mathbf{2}}$ preserves $\mathbf{1}$.
\item[(3)] the mappings $f \mapsto J_A$ and $f \mapsto J_Q$ (see Section \ref{sssec:Constructor}) induce functors from $\mathcal{B}^{\mathbf{2}}$ to $\mathbf{Set}$ that extend the mappings  $f \mapsto \upepsilon$, $f \mapsto \upchi$, $f \mapsto \upiota$ and $f \mapsto \updelta$ into obvious functors from $\mathcal{B}^{\mathbf{2}}$  to $\mathbf{Cat}/D$, $\mathbf{Cat}/D$, $\mathbf{Cat}/I$ and $\mathbf{Cat}/I$, respectively;
\end{itemize}

\begin{example}\label{exa:Constructor_normal_for_systems}
The constructor associated with a system of $R$-premodels $(K,T,\mathcal{P},\mathtt{V})$ over a small category $D$ in a category $\mathcal{C}$ that possesses a terminal object $\mathbf{1}$ is normal. Item (1) is straightforward and item (2) follows from the fact that $R:\mathcal{C} \to \mathcal{C}$ preserves any terminal object by adjointness. The functoriality of the sets $J_A$ and $J_Q$ is induced by the action of a morphism $(f,a):(X,S,e) \Rightarrow (X,S',e')$ on the sets $S$ and $S'$ (see Section \ref{sssec:Constructor_system_premodels}) while the functoriality of the functors $\upepsilon$, $\upchi$, $\upiota$ and $\updelta$ is straightforward .
\end{example}

\begin{remark}\label{rem:analytic_quotient_functors_functorial_in_B}
A consequence of item (3) is that the mappings $(f,d) \mapsto \Upgamma_A(f)(d)$ and $(f,d) \mapsto \Upgamma_Q(f)(d)$
now induce functors $\Upgamma_A(\_)(\_):\mathcal{B}^{\mathbf{2}} \times D \to \mathbf{Set}$ and $\Upgamma_A(\_)(\_):\mathcal{B}^{\mathbf{2}} \times D \to \mathbf{Set}$.
\end{remark}

\subsection{Quasi-Models and Models}
Let $\Upgamma$ be a normal constructor of type $D[\mathcal{B},\mathcal{C}]$. 
For every object $X$ in $\mathcal{B}$, there is an obvious morphism in $\mathcal{B}^{\mathbf{2}}$ given by the following commutative square.
\begin{equation}\label{eq:id_X:shrink_X:in_arrow_B}
\xymatrix@R-0.6pc@C-0.6pc{
X\ar@{=}[rr]\ar@{=}[d]&&X\ar[d]^{!_X}\\
X\ar[rr]^{!_X}&&\mathbf{1}
}
\end{equation}

Applying the functor $\Upgamma_A(\_)(\_):\mathcal{B}^{\mathbf{2}} \times D \to \mathbf{Set}$ (see Remark \ref{rem:analytic_quotient_functors_functorial_in_B}) on this morphism provides the following natural transformation in $\mathbf{Set}$ over $D$, which is natural in $X \in \mathcal{B}$.
\[
\wp_A(X,\_):\Upgamma_A(\mathrm{id}_X)(\_) \Rightarrow \Upgamma_A(!_X)(\_)
\]

An object $X$ in $\mathcal{B}$ will be said to be a \emph{quasi-model of $\Upgamma$} if for every object $d$ in $D$, the function
$\wp_A(X,d):\Upgamma_A(\mathrm{id}_X)(d) \Rightarrow \Upgamma_A(!_X)(d)$ is surjective. A \emph{model of $\Upgamma$} is a
quasi-model $X$ of $\Upgamma$ that is equipped with a natural section $\sigma(\_):\Upgamma_A(!_X)(\_) \to \Upgamma_A(\mathrm{id}_X)(\_)$ of the natural surjection $\wp_A(X,\_):\Upgamma_A(\mathrm{id}_X)(\_) \to \Upgamma_A(!_X)(\_)$. Such a structure will be denoted as a pair $(X,\sigma)$.

\begin{remark}\label{rem:quasi_model_no_choice}
It follows from the definition of a surjection that an object $X \in \mathcal{B}$ is a quasi-model of $\Upgamma$ if and only if for every object $d \in D$ and tuple $(\vartheta,t,s,\mathbf{c})$ in $\Upgamma_A(!_X)(d)$, every commutative cube $\mathbf{c} \in \Upgamma_A(!_X)\{\iota(\vartheta)\}_s$ admits a lift as follows (where $\theta$ stands for $\iota(\vartheta)$).
\[
\xymatrix@R-0.5pc@C-0.5pc{
&&\cdot\ar[rr]|{\Phi_{\theta}(X)_s}\ar[dd]|\hole&&\cdot\ar[dd]\ar@<+1.5ex>@{}[rdd]|>>>>>>>>>>{\leftarrow\Phi_{\theta}(!_X)_s}&\\
&\cdot\ar@<-1ex>@{}[ddl]|>>>>>>>>>>{\!\!\!\mathbf{disk}(v)\rightarrow}\ar[rr]^>>>>>>{\gamma_1}\ar[dd]\ar[ru]&&\cdot\ar[dd]|<<<<<<<{\beta \circ \delta_1}\ar[ru]&&\\
&&\mathbf{1}\ar[rr]|\hole|>>>>>\hole&&\mathbf{1}&\\
&\cdot\ar[uuur]_{h}|>>>>>>>>>\hole\ar[rr]|{\beta \circ \delta_2}\ar[ru]&&\cdot\ar[uuur]_{\overline{h}}\ar[ru]&&\\
}
\]
\end{remark}

\begin{remark}\label{rem:algebraic_model_choice}
The difference between a quasi-model and a model is that the lifts are chosen. Specifically, any model $(X,\sigma)$ of $\Upgamma$ is determined by a collection of lifts $(h,\overline{h})$ chosen for every object $\vartheta \in J_A$, element $s \in \uplambda_{\iota(\vartheta)}(X)$ as follows.
\[
\xymatrix@R-0.5pc@C-0.5pc{
&&\cdot\ar[rr]|{\Phi_{\theta}(X)}\ar[dd]|\hole&&\cdot\ar[dd]\ar@<+1.5ex>@{}[rdd]|>>>>>>>>>>{\leftarrow\Phi_{\theta}(!_X)}&\\
&\cdot\ar[rr]^>>>>>>{\gamma_1}\ar[dd]_<<<<<<<{\gamma_2}\ar[ru]&&\cdot\ar[dd]|<<<<<<<{\beta \circ \delta_1}\ar[ru]&&\\
&&\mathbf{1}\ar[rr]|\hole|>>>>>\hole&&\mathbf{1}&\\
&\cdot\ar[uuur]_{h}|>>>>>>>>>\hole\ar[rr]|{\beta \circ \delta_2}\ar[ru]&&\cdot\ar[uuur]_{\overline{h}}\ar[ru]&&\\
}
\]

Indeed, if one denotes the previous commutative cube by $\mathbf{c}$ and its upper commutative part seen as a degenerate commutative cube by $\mathtt{lift}(\vartheta,s,\mathbf{c})$, the section $\sigma(\_):\Upgamma_A(!_X)(\_) \to \Upgamma_A(\mathrm{id}_X)(\_)$ is determined by the following mapping rules. 
\[
(\vartheta,t,s,\mathbf{c}) \mapsto (\vartheta,t,s,\mathtt{lift}(\vartheta,s,\mathbf{c}))
\]

The fact that such a mapping defines a natural section of the natural surjection $\wp_A(X,\_):\Upgamma_A(\mathrm{id}_X)(\_) \Rightarrow \Upgamma_A(!_X)(\_)$ is straightforward. Conversely, if a natural section $\wp_A(X,\_)$ was not of this form, we could find two arrows $t:\upepsilon(\theta) \to d$ and $t':\upepsilon(\theta) \to d'$ such that the elements $(\vartheta,t,s,\mathbf{c})$ and $(\vartheta,t',s,\mathbf{c})$ would be sent to elements of the following form via the section $\sigma$.
\[
(\vartheta,t,s,\mathtt{lift}(\vartheta,t,s,\mathbf{c}))\quad\quad\quad(\vartheta,t',s,\mathbf{c},\mathtt{lift}(\vartheta,t',s,\mathbf{c}))
\]

However, the naturality of $\sigma(\_)$ above the arrows $t$ and $t'$ also implies the equalities 
\[
\mathtt{lift}(\vartheta,t,s,\mathbf{c}) = \mathtt{lift}(\theta,v,\mathrm{id}_{\upepsilon(\vartheta)},\mathbf{c}) = \mathtt{lift}(\vartheta,t',s,\mathbf{c}),
\]
which show that the section has to be of the form previously given in the remark.
\end{remark}

\begin{example}[System of premodels]\label{exam:models_are_models}
Let $(K,T,\mathcal{P},\mathtt{V})$ be a system of $R$-premodels as in Example \ref{exa:Constructor_normal_for_systems}.
The $R$-models are exactly given by the quasi-models of $\Upgamma^K$.
By Remark \ref{rem:algebraic_model_choice}, it is always possible to turn a quasi-model $X$ into a model $(X,\sigma)$ by using the axiom of choice on the different possible lifts. 
\smallskip
\end{example}

Let now $A$ be an object in $\mathcal{B}$. An \emph{$A$-quasi-model} for the constructor $\Upgamma$ consists of a morphism $f:A \to X$ in $\mathcal{B}$ where $X$ is equipped with the structure of a quasi-model $X$. Similarly, an \emph{$A$-model} for the constructor $\Upgamma$ consists of a morphism $f:A \to X$ in $\mathcal{B}$ where $X$ is equipped with the structure of a model $(X,\sigma)$. The latter structure will be denoted as a triple $(f,X,\sigma)$.

\subsection{Quotiented Models}
Let $\Upgamma$ be a normal constructor of type $D[\mathcal{B},\mathcal{C}]$ and $A$ be an object in $\mathcal{B}$. A \emph{$\Upgamma$-quotiented $A$-quasi-model} consists of an $\Upgamma$-quotiented arrow $(!_A,\mathfrak{u}):A \to \mathbf{1}$ in $\mathcal{B}$ together with an $A$-quasi-model $f:A \to X$. Such a structure will be denoted as an arrow $f:(A,\mathfrak{u}) \to X$.

\begin{remark}[In preparation of Definition \ref{def:quotiented_models_associated_image_quotient}]
In Definition \ref{def:quotiented_models_associated_image_quotient}, we define two new quotients that relies on the definition of $\mathfrak{u}$. There is the  quotient denoted by $f[\mathfrak{u}]$, which should be thought of as the collections of all commutative squares contained in $\mathfrak{u}$ (when viewed  in $\mathcal{C}^{\mathbf{2}}$) whose top horizontal arrows are post-composed with the morphism $\Phi_{\theta}(f):\Phi_{\theta}(A) \to \Phi_{\theta}(X)$ (in $\mathcal{C}^{\mathbf{2}}$) while the bottom horizontal arrows consist of identities on the terminal object $\mathbf{1}$. The other quotient, denoted by $(f|\mathfrak{u})$, should be thought of as the collections of commutative squares contained in $f[\mathfrak{u}]$ that admit lifts.
\end{remark}

\begin{definition}\label{def:quotiented_models_associated_image_quotient}
For every $\Upgamma$-quotiented quasi-model $f:(A,\mathfrak{u}) \to X$ where $\mathfrak{u}$ is given by a collection of functors $\{\mathfrak{u}_{\theta}^s\{\_\}:E_s(\theta) \to \mathbf{Set}\}_{\theta \in I, s \uplambda_{\theta}(A)}$, we may define a collection of functors
\[
f[\mathfrak{u}]:=\{f[\mathfrak{u}]_{\theta}^s\{\_\}:E_s(\theta) \to \mathbf{Set}\}_{\theta \in I, s \in \uplambda_{\theta}(A)}
\]
whose component at the parameters $\theta \in I$ and $s \in \uplambda_{\theta}(A)$ is given by the following image factorisation for every $\upsilon \in E_s(\theta)$.
\[
\xymatrix@C+1pc{
&f[\mathfrak{u}]_{\theta}^s\{\upsilon\}\ar[rd]^-{\subseteq}&\\
\mathfrak{u}_{\theta}^s\{\upsilon\} \ar[r]_-{\subseteq} \ar[ru]^{f^{*}}&\Upgamma_Q(!_A)\{\theta\}_s\ar[r]_-{\Upgamma_Q(f)\{\theta\}_s}&\Upgamma_Q(!_X)\{\theta\}_{\uplambda_{\theta}(f)(s)}
}
\]

Then, we may define another collection of functors of the form
\[
(f|\mathfrak{u}):=\{(f|\mathfrak{u})_{\theta}^s\{\_\}:E_s(\theta) \to \mathbf{Set}\}_{\theta \in I, s \in \uplambda_{\theta}(A)}
\]
whose component at the parameters $\theta \in I$ and $s \in \uplambda_{\theta}(A)$ is obtained by pulling back the inclusion $f[\mathfrak{u}]_{\theta}^s\{\upsilon\} \hookrightarrow \Upgamma_Q(!_X)\{\theta\}_{\uplambda_{\theta}(f)(s)}$ along the image of the morphism given in Equation (\ref{eq:id_X:shrink_X:in_arrow_B}) via the functor $\Upgamma_Q(\_)\{\theta\}_{\uplambda_{\theta}(f)(s)}$ (see diagram below).
\[
\xymatrix{
*+!R(0.7){(f|\mathfrak{u})_{\theta}^s\{\upsilon\}}\ar[d]_{\subseteq}\ar[rr]^{\wp_{\mathfrak{u},f}}\ar@{}[rd]|<<{\text{\huge{\rotatebox[origin=c]{90}{$\llcorner$}}}}&&
*+!L(0.7){f[\mathfrak{u}]_{\theta}^s\{\upsilon\}}\ar[d]^{\subseteq}\\
*+!R(0.7){\Upgamma_Q(\mathrm{id}_X)\{\theta\}_{\uplambda_{\theta}(f)(s)}}\ar[rr]^-{\Upgamma_Q(\mathrm{id}_X)\{\theta\}_{\uplambda_{\theta}(f)(s)}}&&*+!L(0.7){\Upgamma_Q(!_X)\{\theta\}_{\uplambda_{\theta}(f)(s)}}
}
\]
\end{definition}

\begin{definition}[Quotiented model]
A \emph{$\Upgamma$-quotiented $A$-model} consists of a $\Upgamma$-quotiented arrow $(!_A,\mathfrak{u}):A \to \mathbf{1}$ in $\mathcal{B}$ together with an $A$-model $(f:A \to X,\sigma)$ such that for every element $\theta \in I$ and $s \in \uplambda_{\theta}(A)$, the transformation $\wp_{\mathfrak{u},f}:(f|\mathfrak{u})_{\theta}^s\{\_\} \to f[\mathfrak{u}]_{\theta}^s\{\_\}$ has a section $\textit{\ss}_{\theta}^s:f[\mathfrak{u}]_{\theta}^s\{\_\} \to (f|\mathfrak{u})_{\theta}^s\{\_\}$. Such a structure will be denoted as an arrow $f:(A,\mathfrak{u}) \to (X,\sigma,\textit{\ss})$.
\end{definition}

For every $\Upgamma$-quotiented $A$-model $f:(A,\mathfrak{u}) \to (X,\sigma,\textit{\ss})$, we may define two functors $f[\mathfrak{u}][\_]:I \to \mathbf{Set}$ and $(f|\mathfrak{u})[\_]:I \to \mathbf{Set}$ given by the following sums for every $\theta \in I$.
\[
f[\mathfrak{u}][\theta]:=\sum_{s \in \uplambda_{\theta}(A)}\sum_{\upsilon \in E_s(\theta)}f[\mathfrak{u}]_{\theta}^s\{\upsilon\} 
\quad\quad\quad
(f|\mathfrak{u})[\theta]:=\sum_{s \in \uplambda_{\theta}(A)}\sum_{\upsilon \in E_s(\theta)}(f|\mathfrak{u})_{\theta}^s\{\upsilon\}
\]

These two functors give rise to two others $f[\mathfrak{u}](\_):I \to \mathbf{Set}$ and $(f|\mathfrak{u})(\_):I \to \mathbf{Set}$ defined as the following sums over the set $J_Q$ associated with $!_A:A \to \mathbf{1}$.
\[
f[\mathfrak{u}](d):=\sum_{\vartheta \in J_Q}f[\mathfrak{u}][\theta]
\quad\quad\quad
(f|\mathfrak{u})(d):=\sum_{\vartheta \in J_Q}(f|\mathfrak{u})[\theta]
\]

It follows from the structure of $f:(A,\mathfrak{u}) \to (X,\sigma,\textit{\ss})$ that the functions $f^*:\mathfrak{u}_{\theta}^s\{\_\}  \to f[\mathfrak{u}]_{\theta}^s\{\_\}$, $\textit{\ss}_{\theta}^s:f[\mathfrak{u}]_{\theta}^s\{\_\} \to (f|\mathfrak{u})_{\theta}^s\{\_\}$ and $(f|\mathfrak{u})_{\theta}^s\{\_\} \hookrightarrow \Upgamma_Q(\mathrm{id}_X)\{\theta\}_s$ induce an obvious sequence of natural transformations as follows.
\[
\xymatrix{
\mathfrak{u}_{\theta}(\_) \ar@{=>}[r]^{f^*}&f[\mathfrak{u}](\_)\ar@{=>}[r]^{\textit{\ss}_{\mathfrak{u},f}}& (f|\mathfrak{u})(\_)\ar@{=>}[r]&\Upgamma_Q(\mathrm{id}_X)(\_)
}
\]

Applying an image factorisation on the three composite arrows of codomain $\Upgamma_Q(\mathrm{id}_X)$ that results from the previous sequence leads to a new sequence of arrows
as follows.
\[
\xymatrix{
\tilde{\mathfrak{u}}_{\theta}(\_) \ar@{=>}[r]^{f^*}&\tilde{f[\mathfrak{u}]}(\_)\ar@{=>}[r]^{\textit{\ss}_{\mathfrak{u},f}}& \tilde{(f|\mathfrak{u})}(\_)\ar@{=>}[r]^-{\tilde{\varphi}_Q}& \mathcal{C}^{\mathbf{2}}\rotatebox[origin=c]{-90}{\text{$\multimap$}}\Upsilon(\mathrm{id}_X)
}
\]

\subsection{Tomes for Quotiented Models}
Let $\Upgamma$ be a normal constructor of type $D[\mathcal{B},\mathcal{C}]$ where $\mathcal{C}$ has coproducts, $A$ be an object in $\mathcal{B}$ and $f:(A,\mathfrak{u}) \to (X,\sigma,\textit{\ss})$ be a $\Upgamma$-quotiented $A$-model. The \emph{merolytic tome of $f:(A,\mathfrak{u}) \to (X,\sigma,\textit{\ss})$} is the functor $\psi_{d}^{\mathfrak{u}}:\Upgamma^{!_A,\mathfrak{u}}(d) \to \mathcal{C}^{\mathbf{2}}/\Upsilon_{d}(\mathrm{id}_X)$ resulting from the coproduct of the following two functors for every $d \in D$.
\begin{equation}\label{eq:playground_not_modified_tome_for_modified_model}
\xymatrix{
\Upgamma_A(!_A)(d) \ar[rr]^-{\Upgamma_A(f)(d)}&&\Upgamma_A(!_X)(d) \ar[r]^-{\sigma(d)}&\Upgamma_A(\mathrm{id}_X)(d) \ar[r]^-{\varphi_{\Upgamma,d}}& \mathcal{C}^{\mathbf{2}}/\Upsilon_{d}(\mathrm{id}_X)
}
\end{equation}
\begin{equation}\label{eq:playground_modifier_for_modified_model}
\xymatrix{
\tilde{\mathfrak{u}}(d) \ar[r]^-{f^*(d)}& \tilde{f[\mathfrak{u}]}(d)\ar[r]^-{\textit{\ss}(d)} & \tilde{(f|\mathfrak{u})}(d)\ar[r]^-{\tilde{\varphi}_{Q}} & \mathcal{C}^{\mathbf{2}}/\Upsilon_{d}(\mathrm{id}_X)
}
\end{equation}

This functor is therefore equipped with the following mapping rules.
\begin{equation}\label{eq:formula_quotiented_model_tome}
\left\{
\begin{array}{lllll}
(\vartheta,t,s,\mathbf{c})&\mapsto &
\Upsilon_{t}(\mathrm{id}_X) \circ \ell_{\vartheta}(\mathtt{lift}(\vartheta,s,\Phi_{\upiota(\vartheta)}(f)_s \circ \mathbf{c})^{\circ})&\text{ on }\Upgamma_A(!_A)(d)\\
(\vartheta,t,s,\mathbf{s}) &\mapsto &
\Upsilon_{t}(\mathrm{id}_X) \circ \ell_{\vartheta}(\mathtt{lift}(\vartheta,s,\Phi_{\updelta(\vartheta)}^{\bullet}(f)_s \circ \mathbf{s}))&\text{ on }\tilde{\mathfrak{u}}(d)
\end{array}
\right.
\end{equation}

\begin{proposition}\label{prop:modified_tome_natural_in_D}
The merolytic tome $\psi_{d}^{\mathfrak{u}}:\Upgamma^{!_A,\mathfrak{u}}(d) \to \mathcal{C}^{\mathbf{2}}/\Upsilon_{d}(\mathrm{id}_X)$ is natural in the variable $d \in D$. This amounts to saying that the mapping $d \mapsto (\Upsilon_{d}(\mathrm{id}_X),\Upgamma^{!_A,\mathfrak{u}}(d),\psi_{d}^{\mathfrak{u}})$ induces a functor
$\mathbb{T}^{\mathrm{mod}}_{f,\mathfrak{u}}:D \to \mathbf{Tome}(\mathcal{C})$.
\end{proposition}
\begin{proof}
Follows from the naturality of the arrows given in Equations (\ref{eq:playground_not_modified_tome_for_modified_model}) and (\ref{eq:playground_modifier_for_modified_model}).
\end{proof}

\begin{remark}\label{rem:tome:modified_models}
The naturality of $\mathbb{T}^{\mathrm{mod}}_{f,\mathfrak{u}}$ over $D$ extends to its content. In particular, it takes a few lines of straightforward calculations to see from the definitions of the functor $\psi_d^{\mathfrak{u}}$ and the functor $\varphi^{\mathfrak{u}}_d$ that the top-left corner of the content of $\mathbb{T}^{\mathrm{mod}}_{f,\mathfrak{u}}$ is equal to the top-left corner of the content of the tome $\mathbb{T}^{!_A,\mathfrak{u}}$.
\[
\underbrace{
\xymatrix{
*+!R(.5){\mathrm{col} \mathbf{A}_{d}}\ar[r]^-{\mathrm{col}u}\ar[d]_{\mathrm{col}\partial\varphi_{d}^{\mathfrak{u}}}&\Upsilon_{d}(A)\ar[r]^-{\Upsilon_{d}(f)}&\Upsilon_{d}(X)\ar[d]^{\Upsilon_{d}(\mathrm{id}_X)}\\
*+!R(.5){\mathrm{col} \mathbf{B}_{d}}\ar[rr]_-{\mathrm{col}h}&&\Upsilon_{d}(X)
}}_{\text{content of }\mathbb{T}^{\mathrm{mod}}_{f,\mathfrak{u}}}
\quad\quad\quad
\underbrace{
\xymatrix{
*+!R(.5){\mathrm{col} \mathbf{A}_{d}}\ar[r]^-{\mathrm{col}u}\ar[d]_{\mathrm{col}\partial\varphi_{d}^{\mathfrak{u}}}&\Upsilon_{d}(A)\ar[d]^{\Upsilon_{d}(!_A)}\\
*+!R(.5){\mathrm{col} \mathbf{B}_{d}}\ar[r]_-{\mathrm{col}v}&\Upsilon_{d}(\mathbf{1})
}}_{\text{content of }\mathbb{T}^{!_A,\mathfrak{u}}}
\]
\end{remark}
The diagram given on the left of Remark \ref{rem:tome:modified_models} induces a commutative diagram in $\mathcal{C}^D$ of the form given below. This diagram will be referred to as the  \emph{functorial content of $\mathbb{T}^{\mathrm{mod}}_{f,\mathfrak{u}}$}.
\[
\xymatrix{
*+!R(.5){\mathrm{col} \mathbf{A}}\ar@{=>}[r]^-{\mathrm{col}u}\ar@{=>}[d]_{\mathrm{col}\partial\varphi^{\mathfrak{u}}}&\Upsilon(A)\ar@{=>}[r]^-{\Upsilon(f)}&\Upsilon(X)\ar@{=>}[d]^{\Upsilon(\mathrm{id}_X)}\\
*+!R(.5){\mathrm{col} \mathbf{B}}\ar@{=>}[rr]_-{\mathrm{col}h}&&\Upsilon(X)
}
\]

\subsection{Effectiveness of Quotiented Models}
Let $\Upgamma$ be a normal constructor of type $D[\mathcal{B},\mathcal{C}]$ where $\mathcal{C}$ has coproducts and $A$ be an object in $\mathcal{B}$. A $\Upgamma$-quotiented $A$-model $f:(A,\mathfrak{u}) \to (X,\sigma,\textit{\ss})$ will be said to be \emph{$\Upgamma$-realised} if one may form a pushout square inside the functorial content of its merolytic tome as shown below. 
\begin{equation}\label{eq:Effectiveness_of_quotiented_models:realisation}
\xymatrix{
*+!R(.5){\mathrm{col} \mathbf{A}}\ar@{}[rd]|->{\text{\huge{\rotatebox[origin=c]{-90}{$\llcorner$}}}}\ar@{=>}[r]^-{\mathrm{col}u}\ar@{=>}[d]_{\mathrm{col}\partial\varphi^{\mathfrak{u}}}&\Upsilon(A)\ar@{=>}[r]^{\Upsilon(f)}\ar@{=>}[d]^{p^{\mathfrak{u}}_f}&\Upsilon(X)\ar@{=>}[d]^{\Upsilon(\mathrm{id}_X)}\\
*+!R(.5){\mathrm{col} \mathbf{B}}\ar@/_1.8pc/@{=>}[rr]_-{\mathrm{col}h}\ar@{=>}[r]_{\pi^{\mathfrak{u}}_f}&[!_A,\mathfrak{u}]\ar@{=>}[r]^{b^{\mathfrak{u}}_{f}}&\Upsilon(X)
}
\end{equation}

By Remark \ref{rem:tome:modified_models}, the pushout square may be supposed to be exactly the same as that defined for the $\Upgamma$-realisation of $(!_A,\mathfrak{u})$. In particular, the following result holds.

\begin{proposition}
A $\Upgamma$-quotiented $A$-model $f:(A,\mathfrak{u}) \to (X,\sigma,\textit{\ss})$ is $\Upgamma$-realised if and only if so is the $\Upgamma$-quotiented arrow $(!_A,\mathfrak{u}):A \to \mathbf{1}$.
\end{proposition}

\begin{definition}[Effectiveness]\label{def:effectiveness_relative_quotiented_models}
Let $\Upgamma$ denote a constructor of type $D[\mathcal{B},\mathcal{C}]$ as defined in Section \ref{sssec:Constructor}.
A $\Upgamma$-quotiented $A$-model $f:(A,\mathfrak{u}) \to (X,\sigma,\textit{\ss})$  will be said to be \emph{effective} if it is $\Upgamma$-realised and
it is equipped with a factorisation of  $f:A \to Y$ in $\mathcal{B}$, as given on the left of Equation 
(\ref{eq:lifting_definition_effective_constructor_model}), that lifts the factorisation of $\Upsilon(f)$ through the $\Upgamma$-realisation of $(f,\mathfrak{q})$ along $\Upsilon:\mathcal{B} \to \mathcal{C}^D$.
\begin{equation}\label{eq:lifting_definition_effective_constructor_model}
\xymatrix{
A\ar@/^1.8pc/[rr]^-{f}\ar[r]_-{\{ !_A\}_{\mathfrak{u}}}&[!_A/\mathfrak{u}]\ar[r]_-{\langle f \rangle_{\mathfrak{u}}}&X
}
\quad\quad\mathop{\longmapsto}\limits^{\Upsilon}\quad\quad
\xymatrix{
\Upsilon(A)\ar@/^1.8pc/@{=>}[rr]^-{\Upsilon(f)}\ar@{=>}[r]_-{p^{\mathfrak{u}}_f}&[!_A,\mathfrak{u}]\ar@{=>}[r]_-{b^{\mathfrak{u}}_f}&\Upsilon(X)
}
\end{equation}
\end{definition}
The leftmost factorisation of Equation (\ref{eq:lifting_definition_combinatorial_constructor}) will be called the \emph{$\Upgamma$-factorisation of $f:(A,\mathfrak{u}) \to (X,\sigma,\textit{\ss})$}.

\begin{theorem}\label{th:admissible_quotiented_effective_model}
Let $(K,\mathtt{rou},\mathcal{P},\mathtt{V})$ be a system of $R$-premodels over a small category $D$ in a category $\mathcal{C}$. If $\mathcal{C}$ has pushouts and the inclusion $\mathcal{P} \hookrightarrow \mathbf{Pr}_{\mathcal{C}}(K,\mathtt{rou},R)$ is an identity, then every $\Upgamma$-quotiented relative model is effective.
\end{theorem}
\begin{proof}
Consider a relative model given by a morphism $(f,a):(A,S,e) \Rightarrow (X,S',e')$. The goal is to show that this morphism satisfies to the lifting conditions expressed in Equation (\ref{eq:lifting_definition_effective_constructor_model}) where the arrow $\{ !_A\}_{\mathfrak{u}}:(A,S,e) \Rightarrow ([!_A,\mathfrak{u}],S,e^{\mathfrak{u}})$ is already constructed in Theorem \ref{th:admissible_quotiented_effective}. In fact, the proof of the present theorem is very similar to that of Theorem \ref{th:admissible_quotiented_effective}, except that it uses Diagram (\ref{eq:quotiented_relative_models_are_effective}) instead of Diagram (\ref{eq:quotiented_arrows_are_effective}) for every $c \in K$ and $s \in S(c)$. As in the proof of Theorem \ref{th:admissible_quotiented_effective}, the symbols $r$ and $d_0$ stand for the objects $\mathtt{rou}(c)$ and $\mathtt{ou}(c)$ in $D$, respectively.
\begin{equation}\label{eq:quotiented_relative_models_are_effective}
\xymatrix{
\Upgamma^{!_A,\mathfrak{u}}(r)\ar[r]^{\upzeta_{c,s}}\ar[d]_{\psi^{\mathfrak{u}}_{r}}&\ar[d]^{\psi^{\mathfrak{u}}_{d_0}}\Upgamma^{!_A,\mathfrak{u}}(d_0)\\
\mathcal{C}^{\mathbf{2}}/\Upsilon_{r}(\mathrm{id}_X) \ar[r]^-{\upxi_{c,a_c(s)}}& \mathcal{C}^{\mathbf{2}}/\Upsilon_{d_0}(\mathrm{id}_X)
}
\end{equation}

The proof that Diagram (\ref{eq:quotiented_relative_models_are_effective}) commutes goes as in the proof of Theorem \ref{th:admissible_quotiented_effective} by using Formula (\ref{eq:formula_quotiented_model_tome}).
Then, Diagram (\ref{eq:quotiented_relative_models_are_effective}) may be used to show that the following diagram commutes.
\[
\xymatrix{
\mathrm{col} \mathbf{A}_{r}\ar@{}[rd]|->>{\text{\huge{\rotatebox[origin=c]{-90}{$\llcorner$}}}}\ar@{=>}[r]^{\mathrm{col}u_{r}}\ar@{=>}[dd]_{\mathrm{col}\partial\varphi^{\mathfrak{u}}_{r}}&A(r)\ar@{=>}[r]^{f(r)}\ar@{=>}[d]&X(r)\ar@{=>}[r]^-{e_{c,a_c(s)}'}&RX(d_0)\ar@{=>}[dd]^{\mathrm{id}}\\
&[!_A,\mathfrak{u}](r)\ar@{=>}[rd]^{b^{\mathfrak{u}}_f}\ar@{=>}[r]^-{e^{\mathfrak{u}}_{c,s}}&R[!_A,\mathfrak{u}](d_0)\ar@{=>}[dr]^{Rb^{\mathfrak{u}}_f(d_0)}&\\
\mathrm{col}\mathbf{B}_{r}\ar@{=>}[rr]_{\mathrm{col}v_{r}}\ar@{=>}[ru]^{\pi^{\mathfrak{u}}_f(r)}&&X(r)\ar@{=>}[r]_{e_{c,a_c(s)}'}&RX(d_0)
}
\]

The substantial information given by the previous diagram is the inner bottom commutative trapezoid, which shows that the lift $([!_A,\mathfrak{u}],S,e^{\mathfrak{u}}) \Rightarrow (X,S',e')$ exists; the desired factorisation is  deduced by universality.
\end{proof}

\begin{definition}[Strongly Fibered]\label{def:strongly_fibered_relative_quotiented_model}
A system of $R$-premodels $(K,\mathtt{rou},\mathcal{P},\mathtt{V})$ over a small category $D$ in a category $\mathcal{C}$ will be said to be \emph{strongly fibered}
if it is fibered and, for every $\Upgamma$-quotiented arrow $(!_A,\mathfrak{u}):(A,S,e) \to \mathbf{1}$, the $\Upgamma$-factorisation of any corresponding $\Upgamma$-quotiented $(A,S,e)$-model  (obtained in Theorem \ref{th:admissible_quotiented_effective}) lifts to $\mathcal{P}$.
\end{definition}

\begin{remark}\label{rem:Many_examples_are_strongly_fibered}
Let $\mathcal{C}$ be a category with all pushouts. By Definitions \ref{def:effectiveness_quotiented_models} and  \ref{def:effectiveness_relative_quotiented_models}, any subcategory $\mathcal{P}$ of $\mathbf{Pr}_{\mathcal{C}}(K,\mathtt{rou},R)$ whose associated functor $\Upsilon:\mathcal{P} \to \mathcal{C}^D$ is fully faithful is necessarily strongly fibered. As a result, the category $\mathcal{C}^D$ is strongly fibered. Thus, examples of strongly fibered systems of premodels are : Example \ref{ex:Models_for_a_sketch_system}--\ref{ex:Complete_Segal_spaces}.
\end{remark}

\begin{example}
For the same reasons put forward in Example \ref{ex:spectra_fibered}, the system of $\Omega$-premodels defined in Example \ref{ex:system_of_premodels_spectra} is strongly fibered.
\end{example}

\begin{proposition}\label{prop:factorisability_provides_new_model}
Let $\Upgamma$ be a normal constructor of type $D[\mathcal{B},\mathcal{C}]$ where $\mathcal{C}$ has coproducts. Let $(!_A,\mathfrak{u}):A \to \mathbf{1}$ be an effective $\Upgamma$-quotiented arrow in $\mathcal{B}$ and $f:(A,\mathfrak{u}) \to (X,\sigma,\textit{\ss})$ be some effective $A$-model of $\Upgamma$. There exists a section $\textit{\ss}_{\dagger}$ turning the $\Upgamma$-quotiented $[A/\mathfrak{u}]$-quasi-model $\langle f \rangle_{\mathfrak{u}}:\mathrm{Rec}(!_A,\mathfrak{u}) \to (X,\sigma)$ into a $\Upgamma$-quotiented $[A/\mathfrak{u}]$-model.
\end{proposition}
\begin{proof}
Let us define the section $\textit{\ss}_{\dagger}$, which must be a function of the following form for every $\theta \in I$ and $s \in \uplambda_{\theta}(A)$.
\begin{equation}\label{eq:form_of_ss_dagger}
\langle f \rangle_{\mathfrak{u}}[\mathrm{Rec}(!_A,\mathfrak{u})]_{\theta}^s\{\_\} \to \Big(\langle f \rangle_{\mathfrak{u}}\big|\mathrm{Rec}(!_A,\mathfrak{u})\Big)_{\theta}^s\{\_\}
\end{equation}

The idea is that the section $\textit{\ss}_{\dagger}$ is induced by the action of the section $\sigma$ on the obstruction squares contained in the domain of Equation (\ref{eq:form_of_ss_dagger}). On the other hand, the other section $\textit{\ss}$ mentioned in the statement does play any role here. 

First, recall that an obstruction square in 
$[\mathrm{Rec}(!_A,\mathfrak{u})]_{\theta}^s\{\upsilon\}$ is given by the lower front commutative square of a commutative cuboid as follows.
\[
\xymatrix@C-0.5pc@R-0.5pc{
\mathbb{S}\ar[ddd]_{\gamma_2}\ar[rr]^x\ar[rd]_{\gamma_1}&&\Upsilon_d(A)\ar[rd]|{\Phi_{\theta}(A)_s}\ar[dd]|\hole&\\
&\mathbb{D}_1\ar[dd]|<<<<<<<<<<<{\delta_1}\ar[rr]^<<<<<<<<<<y&&\Phi_{\theta}^{\bullet}(A)_s\ar[dd]^{\Phi_{\theta}^{\bullet}(\{ !_A\}_{\mathfrak{u}})_s}\\
&&[!_A,\mathfrak{u}](d)\ar[rd]|{\Phi_{\theta}([!_A/\mathfrak{u}])_s}&\\
\mathbb{D}_2\ar@/_/[r]|{\delta_2}\ar[rru]|<<<<<<<<<<\hole|>>>>>>>{\uppi_0}\ar[rd]_{\beta \circ \delta_2}&\mathbb{S}'\ar[d]^{\beta}\ar[rr]^{w}&&\Phi_{\theta}^{\bullet}([!_A,\mathfrak{u}])_s\ar[d]^{\Phi_{\theta}^{\bullet}(\lfloor !_A\rfloor_{\mathfrak{u}})_s}\\
&\mathbb{D}'\ar[rr]&&\mathbf{1}\\
}
\]

By using the factorisation $f = \langle f \rangle_{\mathfrak{u}} \circ \{ !_A \}_{\mathfrak{u}}$, we can also obtain Diagram (\ref{eq:big_diagram_pre_universal_property_existential}), whose lower trapezoid going from the arrow $\beta:\mathbb{S}' \to \mathbb{D}'$ to the arrow $\Phi_{\theta}^{\bullet}(\lfloor !_A\rfloor_{\mathfrak{u}})_s$, on the front face, is the image of our previous obstruction square via the canonical map
\[
\Upgamma_{Q}(\mathrm{id}_X)\{\theta\}_{\uplambda_{\theta}(f)(s)} \longrightarrow \Upgamma_{Q}(!_X)\{\theta\}_{\uplambda_{\theta}(f)(s)}
\]
induced by Diagram (\ref{eq:id_X:shrink_X:in_arrow_B}).
By definition, this lower trapeziod is an element in the quotient $\langle f \rangle_{\mathfrak{u}}[\mathrm{Rec}(!_A,\mathfrak{u})]_{\theta}^s\{\upsilon\}$.
\begin{equation}\label{eq:big_diagram_pre_universal_property_existential}
\xymatrix@C-0.5pc@R-0.5pc{
\mathbb{S}\ar[ddd]_{\gamma_2}\ar[rr]^x\ar[rd]_{\gamma_1}&&\Upsilon_d(A)\ar[rd]|{\Phi_{\theta}(A)_s}\ar[dd]|\hole\ar[rr]&&\Upsilon_d(X)\ar@{=}[dd]|\hole\ar[rd]^{\Phi_{\theta}(X)_s}&\\
&\mathbb{D}_1\ar[dd]|<<<<<<<<<<<{\delta_1}\ar[rr]^<<<<<<<<<<y&&\Phi_{\theta}^{\bullet}(A)\ar[dd]|<<<<<<{\Phi_{\theta}^{\bullet}(\{ !_A\}_{\mathfrak{u}})_s}\ar[rr]^<<<<<<<<<{\Phi_{\theta}^{\bullet}(f)_s}&&\Phi_{\theta}^{\bullet}(X)_s\ar@{=}[dd]\\
&&[!_A,\mathfrak{u}](d)\ar[rr]|<<<<<<<<<<<<<<<<\hole\ar[rd]|{\Phi_{\theta}([!_A/\mathfrak{u}])_s}&&\Upsilon_d(X)\ar[rd]^{\Phi_{\theta}(X)_s}&\\
\mathbb{D}_2\ar@/_/[r]|{\delta_2}\ar[rru]|<<<<<<<<<<\hole|>>>>>>>{\uppi_0}\ar[rd]_{\beta \circ \delta_2}&\mathbb{S}'\ar[d]^{\beta}\ar[rr]^{w}&&\Phi_{\theta}^{\bullet}([!_A/\mathfrak{u}])_s\ar@{..>}[d]\ar[rr]^{\Phi_{\theta}^{\bullet}(\langle f \rangle_{\mathfrak{u}})_s}&&\Phi_{\theta}^{\bullet}(X)_s\ar[d]^{\Phi_{\theta}^{\bullet}(\lfloor !_A\rfloor_{\mathfrak{u}})_s}\\
&\mathbb{D}'\ar[rr]&&\mathbf{1}\ar@{=}[rr]&&\mathbf{1}\\
}
\end{equation}

To define our section $\textit{\ss}_{\dagger}$, we simply need to explain how the lower trapezoid going from the arrow $\beta:\mathbb{S}' \to \mathbb{D}'$ to the arrow $\Phi_{\theta}^{\bullet}(\lfloor !_A\rfloor_{\mathfrak{u}})_s$, on the front face of the previous diagram, is mapped to an element in the following quotient.
\[
\Big(\langle f \rangle_{\mathfrak{u}}\big|\mathrm{Rec}(!_A,\mathfrak{u})\Big)_{\theta}^s\{\upsilon\}
\]

We will do so by simply showing that this lower trapezoid admits a lift.

By assumption on the $A$-model $f:A \to (X,\sigma)$, the outer cuboid of Diagram (\ref{eq:big_diagram_pre_universal_property_existential}) admits a lift $(h,\overline{h})$ (see Remark \ref{rem:algebraic_model_choice}). By universality of $\mathbb{S}'$, this implies that the commutative diagram given below, on the left, must commute. The corresponding square given on the right then encodes an element in $(\langle f \rangle_{\mathfrak{u}}|\mathrm{Rec}(!_A,\mathfrak{u}))_{\theta}^s\{\upsilon\}$.
\[
\xymatrix{
\mathbb{S}'\ar[d]_{\beta}\ar[rrr]^-{\Phi_{\theta}^{\bullet}(\langle f \rangle_{\mathfrak{u}})_s \circ w}&&&\Phi_{\theta}^{\bullet}(X)_s\ar[d]\\
\mathbb{D}'\ar[rrr]\ar[rrru]_{\overline{h}}&&&\mathbf{1}
}
\quad\Rightarrow\quad
\xymatrix{
\mathbb{S}'\ar[d]_{\beta}\ar[rrr]^-{\Phi_{\theta}^{\bullet}(\langle f \rangle_{\mathfrak{u}})_s \circ w}&&&\Phi_{\theta}^{\bullet}(X)_s\ar@{=}[d]\\
\mathbb{D}'\ar[rrr]_{\overline{h}}&&&\Phi_{\theta}^{\bullet}(X)_s
}
\]

This finishes the description of the section $\textit{\ss}_{\dagger}$ for the parameters $\theta \in I$ and $s \in \uplambda_{\theta}(A)$.
\end{proof}

The following theorem provides a universal property that only makes sense in the case of Examples \ref{ex:Models_for_a_sketch_system} and  \ref{exa:system_of_models_sheaves}. However, possible extensions of its assumptions (to a homotopical context) may be discussed so that the examples that were provided in Section \ref{ssec:Models} may be equipped with universal properties too; this will be discussed in a future work.

\begin{theorem}[Universality]\label{th:localisation_universal}
Let $f:(A,\mathfrak{u}) \to (X,\sigma,\textit{\ss})$ be an effective $\Upgamma$-quotiented $A$-model such that the $\Upgamma$-quotiented arrow $\langle f \rangle_{\mathfrak{u}}:([!_A/\mathfrak{u}],\mathfrak{u}') \to (X,\sigma,\textit{\ss}_{\dagger})$ is also effective. If 
\begin{itemize}
\item[(i)] the transitive quotientor $\upnu(\vartheta)$ is a epimorphism in $\mathcal{C}$  for every $\vartheta \in J_Q$;
\item[(ii)] the arrow $\Phi_{\theta}(X):\Phi_{\theta}^{\circ}(X) \to \Phi_{\theta}^{\bullet}(X)$ is a monomorphism is $\mathcal{C}$;
\item[(iii)] the trivial stem of $v_{\upalpha}(\vartheta)$ is an epimorphism for every $\vartheta \in J_A$; 
\item[(iv)] the functor $\Upsilon:\mathcal{P} \to \mathcal{C}^D$ is faithful;
\item[(v)] the initial section $\textsc{\j}:I \to J_A$ is an isomorphism,
\end{itemize}
then every arrow $g:[!_A/\mathfrak{u}] \to X$ in $\mathcal{P}$ that factorises as shown below, on the left, and makes the succeeding diagram, on the right, commute must be equal to $\langle f \rangle_{\mathfrak{u}}:[!_A/\mathfrak{u}] \to X$.
\[
\xymatrix{
[!_A/\mathfrak{u}]\ar[r]^g\ar[d]_{\{\lfloor !_A \rfloor_{\mathfrak{u}}\}_{\mathfrak{u}'}}&X\\
[\lfloor !_A \rfloor_{\mathfrak{u}}/\mathfrak{u}']\ar[ru]_{g'}&
}
\quad\quad\quad\quad\quad\quad
\xymatrix{
A\ar[r]^f\ar[d]_{\{!_A\}_{\mathfrak{u}}}&X\\
[!_A/\mathfrak{u}]\ar[ru]_{g}&
}
\]
\end{theorem}
\begin{proof}
Consider some arrow $g:[f/\mathfrak{u}] \to X$ making the right diagram of the statement commute. Let $d$ be an object in $D$. After application of $\Upsilon_d$ on it and using the definition of the $\Upgamma$-realisation of $(!_A,\mathfrak{u})$, we obtain the following commutative diagram.
\[
\xymatrix{
*+!R(.5){\mathrm{col} \mathbf{A}_d}\ar@{}[rd]|->{\text{\huge{\rotatebox[origin=c]{-90}{$\llcorner$}}}} \ar[r]^-{\mathrm{col}u_d}\ar[d]_{\mathrm{col}\partial\varphi_d^{\mathfrak{u}}}&\Upsilon_d(A)\ar[r]^{\Upsilon_d(f)}\ar[d]^{p^{\mathfrak{u}}_f(d)}&\Upsilon_d(X)\ar[d]^{\Upsilon_d(\mathrm{id}_X)}\\
*+!R(.5){\mathrm{col} \mathbf{B}_d}\ar@/_1.8pc/[rr]_{\Upsilon_d(g) \circ \pi^{\mathfrak{u}}_{!_A}(d)}\ar[r]_{\pi^{\mathfrak{u}}_{!_A}(d)}&*+!L(0.4){[!_A,\mathfrak{u}](d)}\ar[r]^-{\Upsilon_d(g)}&\Upsilon_d(X)
}
\]

By universality of the $\Upgamma$-realisation $[!_A,\mathfrak{u}]$ and faithfulness of $\Upsilon:\mathcal{P} \to \mathcal{C}^D$, the statement is proven if we show that the composite $\Upsilon_d(g) \circ \pi^{\mathfrak{u}}_{!_A}(d)$ is equal to the arrow $\mathrm{col}h_d$ of Diagram (\ref{eq:Effectiveness_of_quotiented_models:realisation}) for every object $d \in D$. 
Equivalently, we need to show that, for every functor $i:\mathbf{1} \to \Upgamma^{!_A,\mathfrak{u}}(d)$, the universal shifting 
of $\Upsilon_d(g) \circ \pi^{\mathfrak{u}}_{!_A}(d)$ along $i$ is equal to the composite $h_d \circ i$.
By definition, the universal shifting of the preceding diagram along any functor $i:\mathbf{1} \to \tilde{\mathfrak{u}}(d)$ provides a diagram as given below, on the left. On the right is given the shifting of Diagram (\ref{eq:Effectiveness_of_quotiented_models:realisation}) along that same functor $i$.
\begin{equation}\label{eq:universal_property_decomposition_h_i_g}
\xymatrix@C-0.05pc{
\mathbb{S}'\ar@{-->}[r]\ar[d]_{\upnu(\vartheta)}&*+!R(.5){\mathrm{col} \mathbf{A}_d}\ar[rr]^-{\Upsilon_d(f) \circ \mathrm{col}u_d}\ar[d]_{\mathrm{col}\partial\varphi_d^{\mathfrak{u}}}&&\Upsilon_d(X)\ar@{=}[d]\\
\mathbb{D}'\ar@{-->}[r]&*+!R(.5){\mathrm{col} \mathbf{B}_d}\ar[rr]^-{\Upsilon_d(g) \circ \pi^{\mathfrak{u}}_{!_A}(d)}&&\Upsilon_d(X)
}\quad\quad
\xymatrix@C-0.05pc{
\mathbb{S}'\ar@{-->}[r]\ar[d]_{\upnu(\vartheta)}&*+!R(.5){\mathrm{col} \mathbf{A}_d}\ar[rr]^-{\Upsilon_d(f) \circ \mathrm{col}u_d}\ar[d]_{\mathrm{col}\partial\varphi_d^{\mathfrak{u}}}&&\Upsilon_d(X)\ar@{=}[d]\\
\mathbb{D}'\ar@{-->}[r]\ar@/_1.5pc/[rrr]^-{h_d \circ i}&*+!R(.5){\mathrm{col} \mathbf{B}_d}\ar[rr]^-{\Upsilon_d(\langle f \rangle_{\mathfrak{u}}) \circ \pi^{\mathfrak{u}}_{!_A}(d)}&&\Upsilon_d(X)
}
\end{equation}

Because the top vertical arrows of the previous two diagrams are the same and the transitive quotientor $\upnu(\vartheta)$ is an epimorphism, the shifting of $\Upsilon_d(g) \circ \pi^{\mathfrak{u}}_{!_A}$ along $i:\mathbf{1} \to \tilde{\mathfrak{u}}(d)$ must be equal to $h_d \circ i$.
There only remains to check the same property for functors of the form $\mathbf{1} \to \Upgamma_A(f)(d)$. 

Consider a functor $i:\mathbf{1} \to \Upgamma_A(f)(d)$. By assumption (v), the image of this functor is of the form $(\textsc{\j}(\theta),t,s,\mathbf{c})$ where $t:\uprho(\theta) \to d$ (since $\uprho = \upepsilon\circ \textsc{\j}$) and $s \in \uplambda_{\theta}(A)$. For this particular parameter $\theta \in I$, apply the functor $\Phi_{\theta}$ on the two factorisations given in the statement. With these two factorisations, the diagram obtained in Remark \ref{rem:structure_narrative_degree_2} for the parameters $(\theta,s,\mathbf{c})$ gives a commutative diagram as follows (where $s'=\uplambda_{\theta}(g)(s)$).
\begin{equation}\label{eq:universality_big_diagram_analytic_part}
\xymatrix@R-0.4pc@C-0.4pc{
&\mathbb{S}\ar[rd]\ar[ddd]_{\gamma_2}\ar[rr]^x&&\Upsilon_{\uprho(\theta)}(A)\ar[rd]^{\Phi_{\theta}(A)_s}\ar[d]|>>>\hole\ar[r]&\Upsilon_{\uprho(\theta)}(X)\ar[rrd]&&\\
&&\mathbb{D}_1\ar@/^1.2pc/[rr]^<<<<y\ar[ddd]|{\beta\delta_1}&[!_A,\mathfrak{u}](\uprho(\theta))\ar[d]\ar@{..>}[rd]&*+!L(.7){\Phi_{\theta}^{\bullet}(A)_s}\ar[d]\ar[rr]|{\Phi_{\theta}^{\bullet}(f)}&&\Phi_{\theta}^{\bullet}(X)_{s'}\\
&&&[\lfloor f\rfloor_{\mathfrak{u}},\mathfrak{u}'](\uprho(\theta))\ar[d]\ar[rd]&*+!L(.7){\Phi_{\theta}^{\bullet}([!_A /\mathfrak{u}])_s}\ar[d]\ar@/_/[rru]_{\Phi_{\theta}^{\bullet}(g)}&&\\
&\mathbb{D}_2\ar@/^1.5pc/[rruu]|<<<<<<<<<<{\uppi_0}|<<<<<<<<<<<<<<<\hole\ar[rd]_{\beta \circ \delta_2}\ar[rr]|<<<<<<<<<<\hole&&\mathbf{1}\ar[rd]|<<<<<<<<<<<\hole&*+!L(.7){\Phi_{\theta}^{\bullet}(
[\lfloor !_A \rfloor_{\mathfrak{u}}/\mathfrak{u}'])_s}\ar[d]\ar@/_2pc/[rruu]_{\Phi_{\theta}^{\bullet}(g')}&&\\
&&\mathbb{D}'\ar[rr]\ar[rru]|{\uppi_1(\theta,s)}&&\mathbf{1}&&\\
}
\end{equation}

As shown in the diagram above, the composite $\Phi_{\theta}^{\bullet}(g') \circ \uppi_1(\theta,s) \circ \beta\delta_1$ is equal to $\Phi_{\theta}^{\bullet}(f) \circ y$. Because the two factorisations of the statement are also true when replacing $g$ and $g'$ with $\langle f \rangle_{\mathfrak{u}}$ and $\langle \langle f \rangle_{\mathfrak{u}}  \rangle_{\mathfrak{u}'}$, we similarly deduce that the composite $\Phi_{\theta}^{\bullet}(\langle \langle f \rangle_{\mathfrak{u}}\rangle_{\mathfrak{u}'}) \circ \uppi_1(\theta,s) \circ \beta\delta_1$ is equal to $\Phi_{\theta}^{\bullet}(f) \circ y$. Because the trivial stem $\beta\delta_1$ of $v_{\upalpha}(\vartheta)$ is an epimorphism, the following equality must hold.
\[
\Phi_{\theta}^{\bullet}(g') \circ \uppi_1(\theta,s) = \Phi_{\theta}^{\bullet}(\langle \langle f \rangle_{\mathfrak{u}}\rangle_{\mathfrak{u}'}) \circ \uppi_1(\theta,s)
\]

It is not hard to see that this equality implies the next one (Diagram (\ref{eq:universality_big_diagram_analytic_part}) might help visualise this point if one imagines the arrows that were forgotten in the background).
\[
\Phi_{\theta}(X) \circ \Upsilon_{d}(g) \circ \uppi_0 = \Phi_{\theta}(X)\circ \Upsilon_{\uprho(\theta)}(\langle f \rangle_{\mathfrak{u}}) \circ \uppi_0
\]

Because $\Phi_{\theta}(X)$ is a monomorphism, we obtain the equation $\Upsilon_{\uprho(\theta)}(g) \circ \uppi_0 = \Upsilon_{\uprho(\theta)}(\langle f \rangle_{\mathfrak{u}}) \circ \uppi_0$, which leads to the following one after post-composing with the arrow $\Upsilon_{t}(\mathrm{id}_X):\Upsilon_{\uprho(\theta)}(\mathrm{id}_X) \to \Upsilon_d(\mathrm{id}_X)$ and using the bifunctoriality of $\Upsilon$.
\[
\Upsilon_{d}(g) \circ [!_A,\mathfrak{u}](t) \circ \uppi_0 = \Upsilon_{d}(\langle f \rangle_{\mathfrak{u}}) \circ [!_A,\mathfrak{u}](t) \circ \uppi_0
\] 

Now, by Remark \ref{rem:about_uppi_0}, we know that the composite $[!_A,\mathfrak{u}](t) \circ \uppi_0$ is equal to the composite $\pi^{\mathfrak{u}}_{!_A}(d) \circ \xi_i(\mathbf{B}_d)$. According to the bottom part of the rightmost diagram of Equation (\ref{eq:universal_property_decomposition_h_i_g}), this means that the right-hand side of the previous equation corresponds to the component of the natural transformation $h$ evaluated above the element picked out by the functor $i:\mathbf{1} \to \Upgamma_A(f)(d)$. This therefore concludes the proof of the statement.
\end{proof}

\begin{example}\label{exa:universality_models_and_sheaves}
The vertebrae of Examples \ref{ex:Models_for_a_sketch_system} and  \ref{exa:system_of_models_sheaves} satisfy assumptions (i), (iii) and (iv) of Theorem \ref{th:localisation_universal}. Similarly, the $\mathrm{id}_{\mathbf{Set}}$-models generated by these examples, which are quasi-models of the associated constructor by Remark \ref{rem:quasi_model_no_choice}, or, in fact, actuals models, by Remark \ref{rem:algebraic_model_choice} and the axiom of choice, satisfy condition (ii). Finally, it follows from Remark \ref{rem:J_A_could_have_been_I} that condition (v) can also be satisfied in the case of these examples.
\end{example}

\subsection{Factorisable Models}
Let $(\Upgamma,\kappa)$ be a normal numbered constructor of type $D[\mathcal{B},\mathcal{C}]$ where $\mathcal{C}$ has coproducts, $A$ be an object in $\mathcal{B}$ and $(X,\sigma)$ be a model of $\Upgamma$. An $A$-model $f:A \to (X,\sigma)$ will be said to be \emph{$(\Upgamma,\kappa)$-factorisable} if the morphism $!_A:A \to \mathbf{1}$ is equipped with the structure of a $(\Upgamma,\kappa)$-factorisable arrow, say $(!_{A_n},\mathfrak{u}_n)$, together with a sequence $\{f_n:(A_n,\mathfrak{u}_n) \to (X,\sigma,\textit{\ss}_n)\}_{n \in \kappa+1}$ of effective $\Upgamma$-quotiented  relative models satisfying the following conditions:
\begin{itemize}
\item[$\triangleright$] \textbf{initial case:} $f=f_0$;
\item[$\triangleright$] \textbf{successor cases:} $f_{n+1}$ is given by the arrow $\langle f_n \rangle_{\mathfrak{u}_n}:[!_{A_n}/\mathfrak{u}_n] \to X$;
\item[$\triangleright$] \textbf{limit cases:} for any (infinite) limit ordinal $\lambda \in \kappa+1$, the arrow $f_{\lambda}$ is the colimit $\mathrm{col}_{n \in \lambda} f_n$ in $\mathcal{B}$ of the following diagram over the category $\lambda$.
\[
\xymatrix{
A\ar[d]_{f_0}\ar[r]^-{\{ !_{A_0} \}_{\mathfrak{u}_0}}&[!_{A_0}/\mathfrak{u}_0]\ar[d]_{f_1}\ar[r]^-{\{ !_{A_1} \}_{\mathfrak{u}_1}}&[!_{A_1}/\mathfrak{u}_1]\ar[d]_{f_2}\ar[r]^-{\{ !_{A_2} \}_{\mathfrak{u}_2}}&\dots \ar[r]^-{\{ !_{A_n} \}_{\mathfrak{u}_n}}&[!_{A_n}/\mathfrak{u}_{n}]\ar[d]_{f_{n+1}}\ar[r]&\dots\\
X\ar@{=}[r]&X\ar@{=}[r]&X\ar@{=}[r]&\dots\ar@{=}[r]&X\ar@{=}[r]&\dots
}
\]
\end{itemize}

\begin{remark}\label{rem:elimination_of_quotients_localisation}
Every $(\Upgamma,\kappa)$-factorisable $A$-model $f:A \to (X,\sigma)$ is equipped with a factorisation of the form given in (\ref{eq:elimination_of_quotients_localisation}), where the arrow $\chi_0^{\kappa}(!_A):A \to G(!_A)(\kappa)$ is the image of the sequential functor $G(!_A):\kappa+1 \to \mathcal{B}$ (defined after Convention \ref{conv:construction_G_f_sequential_functor}) above the arrow $0 \to \kappa$.
\begin{equation}\label{eq:elimination_of_quotients_localisation}
\xymatrix{
A\ar[r]^f\ar[d]_{\chi_0^{\kappa}(!_A)}&X\\
G(!_A)(\kappa)\ar[ru]_{f_{\kappa}}&
}
\end{equation}

Later on, the arrow $\chi_0^{\kappa}(!_A):A \to G(!_A)(\kappa)$ will be denoted as $\rho_A:A \to G(A)$ and called the \emph{localisation of $f:A \to (X,\sigma)$}.
According to Theorem \ref{theo:factorisation_SOA} and Remark \ref{rem:quasi_model_no_choice}, if the category $\mathcal{B}$ is equipped with the structure of a $\kappa$-combinatorial category for the constructor $\Upgamma$, then the object $G(A)$ must be a quasi-model of $\Upgamma$.
\end{remark}

\begin{theorem}[Weak localisation]\label{th:admissible_quotiented_factorisable_model}
Let $\kappa$ denote a limit ordinal and $(K,\mathtt{rou},\mathcal{P},\mathtt{V})$ be a system of $R$-premodels over a small category $D$ in a category $\mathcal{C}$. If $\mathcal{C}$ is cocomplete, $R$ preserves colimits over every limit ordinal $\lambda \in \kappa+1$ and the inclusion $\mathcal{P} \hookrightarrow \mathbf{Pr}_{\mathcal{C}}(K,\mathtt{rou},R)$ is an identity, then every relative model of $\Upgamma^K$ may be equipped with the structure of a $(\Upgamma^K,\kappa)$-factorisable relative model.
\end{theorem}
\begin{proof}
Let $(f,a):(A,S,e) \Rightarrow (X,S',e',\sigma)$ be a relative model of $\Upgamma^K$.
By Theorem \ref{th:admissible_quotiented_factorisable}, the morphism $!_A:A \to \mathbf{1}$ may be equipped with the structure of a $(\Upgamma,\kappa)$-factorisable arrow
$(!_{A_n},\mathrm{id},\mathfrak{u}_{n})_{n \in \kappa+1}$ where
\begin{itemize}
\item[$\triangleright$] $\mathfrak{u}_0$ is a collection of empty functors $\emptyset: \mathbf{1} \to \mathbf{Set}$;
\item[$\triangleright$] the object $A_{n+1}$ is given by $[A_{n}/\mathfrak{u}_n]$;
\item[$\triangleright$] For any (infinite) limit ordinal $\lambda \in \kappa+1$, the object $A_{\lambda}$ is the colimit $\mathrm{col}_{n \in \lambda}\{ !_{A_n} \}_{\mathfrak{u}_n}$ in $\mathbf{Pr}_{\mathcal{C}}(K,\mathtt{rou},R)$ of Diagram (\ref{eq:definition_f_n_constructor_end_prop_model}) over the category $\lambda$ while $\mathfrak{u}_{\lambda}$ is given by the collection of empty functors $\{\emptyset: \mathbf{1} \to \mathbf{Set}\}_{\theta \in I,s \in \uplambda_{\theta}(X)}$
\begin{equation}\label{eq:definition_f_n_constructor_end_prop_model}
\xymatrix{
(A,S,e)\ar[r]^-{\{ !_{A_0} \}_{\mathfrak{u}_0}}&[!_{A_0}/\mathfrak{u}_0]\ar[r]^-{\{ !_{A_1} \}_{\mathfrak{u}_1}}&[!_{A_1}/\mathfrak{u}_1]\ar[r]^-{\{ !_{A_2} \}_{\mathfrak{u}_2}}&\dots \ar[r]^-{\{ !_{A_n} \}_{\mathfrak{u}_n}}&[!_{A_n}/\mathfrak{u}_{n}]\ar[r]&\dots
}
\end{equation}
\end{itemize}

Because $\mathfrak{u}_0$ is made of empty functors, the identity on the empty set provides a section $\textit{\ss}_0$ that turns $(f,a):(A,S,e) \Rightarrow (X,S',e',\sigma)$ into an obvious $\Upgamma^K$-quotiented $(A,S,e)$-model. By Proposition \ref{prop:factorisability_provides_new_model}, this model structure generates new model structures $(f_n,a_n):(A_n,S,e_n) \Rightarrow (X,S',e',\sigma,\textit{\ss}_n)$ for all the finite successor ordinals of $\kappa+1$. These structures of relative model give rise to an $(A_{\omega},S,e_{\omega})$-model by forming the colimit of the previous ones along the arrows $\{ !_{A_n} \}_{\mathfrak{u}_n}:A_n \to [!_{A_{n}}/\mathfrak{u}_{n}]$. The same argument can be repeated for all ordinals of $\kappa+1$, since, for every infinite limit ordinal $\lambda \in \kappa+1$, the $\Upgamma$-quotient $\mathfrak{u}_{\lambda}$ is made of empty functors. By Principle of Transfinite Induction, this shows that we can define a sequence of $\Upgamma$-quotiented relative models $(f_n,a_n):(A_n,S_n,e_n) \Rightarrow (X,S',e',\sigma,\textit{\ss}_n)$, which must be effective by Theorem \ref{th:admissible_quotiented_effective_model} and where $(f_0,a_0)$ is given by $(f,a)$. This concludes the proof by definition of a factorisable $(A,S,e)$-model.
\end{proof}

\begin{corollary}\label{cor:admissible_quotiented_factorisable_model}
Let $\kappa$ denote a limit ordinal and $(K,\mathtt{rou},\mathcal{P},\mathtt{V})$ be a strongly fibered system of $R$-premodels over a small category $D$ in a category $\mathcal{C}$. If $\mathcal{C}$ is cocomplete and $R$ preserves colimits over every limit ordinal $\lambda \in \kappa+1$, then every relative model of $\Upgamma^K$ may be equipped with the structure of a $(\Upgamma^K,\kappa)$-factorisable relative model. 
\end{corollary}
\begin{proof}
This corollary is an obvious generalisation of Theorem \ref{th:admissible_quotiented_factorisable_model} that takes advantage of the notion of strong fiberedness (see Definition \ref{def:strongly_fibered_relative_quotiented_model}).
\end{proof}

\subsection{Elimination of Quotients}\label{ssec:Elimination_of_quotients}
A normal numbered constructor $(\Upgamma,\kappa)$ of type $D[\mathcal{B},\mathcal{C}]$ will be said to 
\emph{eliminate quotients} if the category $\mathcal{B}$ is a $\kappa$-combinatorial category for the constructor $\Upgamma$ and every canonical arrow $A \to \mathbf{1}$ is equipped with the structure of a $(\Upgamma,\kappa)$-factorisable morphism such that every $A$-model
$f:A \to (X,\sigma)$ is $(\Upgamma,\kappa)$-factorisable for this structure.

\begin{remark}
For every object $A$ in $\mathcal{B}$, all $A$-models $f:A \to (X,\sigma)$ are equipped with the same localisation $\rho_A:A \to G(A)$ where $G(A)$ is a quasi-model (see Remark \ref{rem:elimination_of_quotients_localisation}). The way in which this arrow has been defined from the data of $\Upgamma$ is the key of the so-called `elimination of quotients'.
\end{remark}

\begin{theorem}\label{theorem:final}
Let $(\Upgamma,\kappa)$ be a normal numbered constructor of type $D[\mathcal{B},\mathcal{C}]$ that eliminates quotients. For every object $A$ in $\mathcal{B}$, every quasi-model $X$ and arrow $f:A \to X$ in $\mathcal{B}$, there exists an arrow $f':G(A) \to X$ in $\mathcal{B}$ making the following diagram commute.
\[
\xymatrix{
A\ar[r]^f\ar[d]_{\rho_A}&X\\
G(A)\ar[ru]_{f'}&
}
\]
\end{theorem}
\begin{proof}
By Remark \ref{rem:algebraic_model_choice} and the axiom of choice, every quasi-model $X$ may be equipped with the structure of a model $(X,\sigma)$. It follows from Remark \ref{rem:elimination_of_quotients_localisation} that the diagram of the statement commutes.
\end{proof}

\begin{example}\label{exa:elimination_of_quotients_final}
Save for Example \ref{ex:system_of_premodels_Localisation_of_rings}, all the examples of Section \ref{ssec:Models} satisfy Corollary \ref{cor:admissible_quotiented_factorisable_model} (see Remark \ref{rem:Many_examples_are_strongly_fibered}). Following Examples \ref{exa:convergence_functor_limit_sketch} and  \ref{exa:convergence_functor_topological_spaces} for premodels valued in $\mathbf{Set}$ and $\mathbf{Top}$ and considering similar arguments for premodels valued in $\mathbf{Cat}$ and $\mathbf{pTop}$, one can deduce from Example \ref{exa:category_of_models_are_combinatorial_in_nice_cases} that these examples are $\kappa$-combinatorial for some well-chosen ordinal $\kappa$. This means that these examples eliminate quotients and are equipped with a localisation of the form given in Theorem \ref{theorem:final}. In particular, this localisation tends to organise the different sorts of data appearing in the diskads of the systems in the form of bundles---this was explicited in Examples \ref{exa:elimination_of_quotients_sketches} and  \ref{exa:elimination_of_quotients_sketches_the_binary_relations} in the case of the models for a limit sketch.
\end{example}

\begin{remark}\label{rem:final_universality}
Under the conditions of Theorem \ref{th:localisation_universal}, the factorisation of Theorem \ref{theorem:final} may be shown to be unique by using an obvious transfinite induction.
\end{remark}

\section{Concusions}\label{sec:conclusion}

\subsection{Conclusions for Motivation 1}
In Section \ref{ssec:result_1}, one of our main goals was to provide a language that would allow us to show strict universal properties from weak definitions. In this paper, we address this question in the form of Theorem \ref{th:localisation_universal}. This theorem shows us what the main ingredients that are responsible for universal properties look like and most of them pertain to the sets of vertebrae associated with our systems of premodels (see conditions (i), (ii) and (iii)). 

In fact, many sections and concepts were introduced in this paper because of these vertebrae. The need for each of these sections can be explained by the following storyline. At the centre of things is Section \ref{ssec:Models}, which introduces the concept of system of premodels. This structure is a formal way to present the lifting problems associated with our vertebrae. To handle these lifting problems, we have to introduce the analytic and quotient species given in Definition \ref{def:analytic_quotient_species}. 
However, because these species need some formal setting, the concept of constructor is introduced in Section \ref{sssec:Constructor}, which \emph{a fortiori} motivates the introduction of preconstructors in Section \ref{sssec:Pre-constructor}. Note that the main purpose of the latter is to allow the handling of the premodel structure (e.g., the maps $e_{c,s}:P\mathtt{rou}(c) \to RP\mathtt{ou}(c)$ defined in Section \ref{ssec:premodels}) while the purpose of the former is to allow the handling of the vertebrae associated with systems of premodels. 
The way one handles the species is formalised via the tools of Section \ref{sec:From_narratives_to_combinatorial_cat}, in which is expressed our small object argument (Proposition \ref{prop:small_object_argument_general}). This section really allows us to see the big picture without introducing too much detail. On the other hand, from Section \ref{ssec:Tomes_of_a_constructor} to the end of Section \ref{sec:Constructors}, we give all the details of this big picture in the case of systems of premodels. We also use Section \ref{sec:Combinatorial_cat} to explain what it takes, in terms of required assumptions, to be able to apply the small object argument of Section \ref{sec:From_narratives_to_combinatorial_cat}. 
The need for Section \ref{sec:Universal_property} naturally presents itself if one is interested to know more about the universal properties satisfied by the models living in systems of premodels. As one is able to see there, this section heavily relies on the concept of species introduced in Section \ref{sssec:Constructor} and hence the concept of vertebra.

The fact that this last section relies so heavily on the vertebrae is not so surprising when one knows that vertebrae are meant to encode some sort of homotopical information and that, on the other hand, Homotopy Theory is all about coherence property. In fact, this idea of coherence---and universal property---coming from vertebrae is already discussed in my thesis \cite{Thesis_Tuy} and this is exactly the spirit in which Theorem \ref{th:localisation_universal} should be regarded. In this respect, I will use the rest of this conclusion to explain why the formalism of systems of premodels is something that one might want to consider if one wants to solve higher coherence problem.

A way to put it would be to ask what happens if one starts changing the assumptions of Theorem \ref{theo:elimitation_reflection_simple_case} (see Section \ref{sec:introduction}) in terms of homotopical properties. The notion of epimorphism used thereof could be replaced with a notion of epimorphism up to homotopy. For instance, the arrow $\beta:\mathbb{S}' \to \mathbb{D}'$ could be called a \emph{weak epimorphism} if for every pair of arrows $f,g:\mathbb{D}' \to X$ for which the equation $f \circ \beta = g \circ \beta$ holds, we can form the pushout $\mathbb{S}^{\prime \prime}$ of $\beta$ with itself (see below) so that a given arrow $\beta':\mathbb{S}^{\prime \prime} \to \mathbb{D}^{\prime \prime}$ factorises the universal arrow induced by $f$ and $g$ under $\mathbb{S}'$ as follows.
\begin{equation}\label{eq:homotopy_conclusion}
\xymatrix{
\mathbb{S}' \ar@{}[rd]|>>{\rotatebox[origin=c]{-90}{\huge{\text{$\llcorner$}}}}\ar[r]^{\beta}\ar[d]_{\beta} & \mathbb{D}' \ar[d]^{\delta_1'} \ar@/^1.1pc/[rrd]^{f}&&\\
\mathbb{D}' \ar[r]^{\delta_2'} \ar@/_1.5pc/[rrr]_{g} & \mathbb{S}^{\prime \prime} \ar[r]^{\beta^{\prime\prime}}& \mathbb{D}^{\prime \prime}\ar@{-->}[r]&X
}
\end{equation}

A quick look at the beginning of the proof of Theorem \ref{th:localisation_universal}, in which $\beta$ should be viewed as the transitive quotientor $\upnu(\vartheta)$,  shows that such a notion of weak epimorphism would imply that the universal solution of the reflection would be unique up to a homotopy relation as defined in Equation (\ref{eq:homotopy_conclusion}). However, this type of statement would only hold if the vertebra  
\begin{equation}\label{eq:vertebra_dim_2}
\xymatrix{
\mathbb{S}' \ar@{}[rd]|>>{\rotatebox[origin=c]{-90}{\huge{\text{$\llcorner$}}}}\ar[r]^{\beta}\ar[d]_{\beta} & \mathbb{D}' \ar[d]^{\delta_1'}&\\
\mathbb{D}' \ar[r]^{\delta_2'} & \mathbb{S}^{\prime \prime} \ar[r]^{\beta^{\prime\prime}}& \mathbb{D}^{\prime \prime}
}
\end{equation}
satisfies some nice compositional properties, and, more specifically, compositional properties of the type defined in \cite{Thesis_Tuy}. In other words, our vertebra would need to satisfy axioms of the same type as those usually considered in the case of (co)limit sketches -- the compositional properties would try to recapture the idea of composition of cells in (Higher) Category Theory. 

Interestingly, these axioms would also mingle different vertebrae together. For instance, it is interesting to note that our current discussion has made us consider two vertebrae: one for which $\beta'$ is a stem (as usual) and one for which $\beta'$ is both a seed and a coseed, given in Equation (\ref{eq:vertebra_dim_2}). This pair of vertebrae can be arranged in the form of the following diagram.
\[
\xymatrix@C-10pt@R-10pt{
&\mathbb{D}_1\ar[rd]^{\delta_1}&&\mathbb{D}'\ar[rd]^{\delta_1'}&&\\
\mathbb{S}\ar[ru]^{\gamma_1}\ar[rd]_{\gamma_2}&&\mathbb{S}'\ar[ru]^{\beta}\ar[rd]_{\beta}&&\mathbb{S}^{\prime\prime}\ar[r]^{\beta'}&\mathbb{D}^{\prime\prime}\\
&\mathbb{D}_2\ar[ru]_{\delta_2}&&\mathbb{D}'\ar[ru]_{\delta_2'}&&\\
}
\]

Such a commutative diagram defines what is called a \emph{spine} (of degree 1) in \cite{Thesis_Tuy}. There, spines are shown to be essential in the understanding of higher coherence results of the type mentioned above and one can see that these structures arise very naturally once one starts talking about universal properties.  The degree of a spine hides a dimensional nature and it is interesting to note that this dimensional aspect already arises among the examples discussed in \cite[Section 4]{TholenInjnotNat} when it is asked whether weak reflections can possess strict universal property such as functoriality and naturality.

In conclusion, the idea of universal property and coherence fits the language of systems of premodels nicely, so that these structures appear to be the right setting to address the question whether a class of algebraic objects defined via weak lifting properties can satisfy strict (or at least stricter than expected) universal properties---and an important part of the work to be done in this direction can already be perceived in \cite{Thesis_Tuy}. 

\subsection{Conclusions for Motivation 2}\label{ssec:conclusion_motivation_2}
In Section \ref{ssec:result_2}, our other main goal was to prove Theorem \ref{theo:Elimination_quotient_intro}, along with Propositions \ref{prop:Elimination_quotient_intro_1} and  \ref{prop:Elimination_quotient_intro_2}. These results were proven in different sections of the present paper. Before addressing the usefulness of these results, we briefly recapitulate their proof below.

Let $(D,K)$ be a limit sketch, seen as a croquis, and consider the system of premodels defined in Example \ref{ex:Models_for_a_sketch_system} for the inclusion $\mathbf{Set}^{D}  \hookrightarrow \mathbf{Pr}_{\mathbf{Set}}(K)$. First, Example \ref{exa:elimination_of_quotients_final} tells us that the reflector $\rho_A:A \to G(A)$ associated with a premodel \footnote{Also called a `presentation' in Section \ref{ssec:result_2}} $A$ in $\mathbf{Set}^{D}$ is given by Theorem \ref{theorem:final}. Its strict universal property then follows from Remark \ref{rem:final_universality}, where one needs to look at Example \ref{exa:universality_models_and_sheaves} in order to use Theorem \ref{th:localisation_universal}. The functoriality of the reflection $A \mapsto G(A)$  and the naturality of the reflector $\rho_A$ obviously follows from this (strict) universal property. 

Now, if one consider the transfinite construction of the reflector $\rho_A:A \to G(A)$ given in Section \ref{ssec:Factorisable_morphisms}, one may see that the transfinite sequence that gives rise to the reflector $\rho_A$ is of the desired form 
\begin{itemize}
\item[-] for Proposition \ref{prop:Elimination_quotient_intro_1} by using Example \ref{exa:elimination_of_quotients_sketches};
\item[-] for Proposition \ref{prop:Elimination_quotient_intro_2} by using 
Example \ref{exa:clarifies_description_colimits};
\item[-] for Theorem \ref{theo:Elimination_quotient_intro} by using Examples \ref{exa:elimination_of_quotients_sketches}--\ref{exa:clarifies_description_colimits}, for which one needs to realise that a sum of the form
\[
E(X)(\_):=\sum_{c \in K} D(\mathtt{ou}(c),\_) \times X[c]
\]
is the same thing as a left Kan extension $E(X)$ of the form given in Equation (\ref{eq:conclusion_left_Kan_extension}), where $K$ must regarded as a discrete category.
\begin{equation}\label{eq:conclusion_left_Kan_extension}
\xymatrix{
D\ar@{-->}[rd]^{E(X)}&\\
K\ar@{}[ru]|<<<<{\rotatebox[origin=c]{115}{$\Rightarrow$}a_i}\ar[u]^{\mathtt{ou}}\ar[r]_-{X[\_]}&\mathbf{Set}
}
\end{equation}
\end{itemize}

The question that now remains to be answered is: what is the combinatorial presentation given by Theorem \ref{theo:Elimination_quotient_intro} useful for? Recall that, according to Theorem \ref{theo:Elimination_quotient_intro}, the reflector associated with a presentation $X$ in $\mathbf{Set}^D$ is the transfinite composition of a sequence of arrows as follows
\[
\xymatrix{
B_0(X)+E_0(X) \ar[r]^-{p_0}& B_1(X)+E_1(X) \ar[r]^-{p_1} & B_2(X)+E_2(X) \ar[r]^-{p_2} & \dots
}
\]
where, for every $i \geq 0$, the object $E_{i+1}(X)$ is the left Kan extension of the functor
\[
\begin{array}{lllll}
\hat{S}_i[\_]&:&K& \to &\mathbf{Set}\\
&&c& \mapsto &\mathrm{lim}_{\mathtt{Es}(c)}\, S_i \circ \mathtt{in}(c)\\
\end{array}
\]
where, here, the functor $S_i:D \to \mathbf{Set}$ denotes the sum $E_i(X)+B_i(X)$.
\[
\xymatrix{
D\ar@{-->}[rd]^{E_{i+1}(X)}&\\
K\ar@{}[ru]|<<<<{\rotatebox[origin=c]{115}{$\Rightarrow$}a_i}\ar[u]^{\mathtt{ou}}\ar[r]_-{\hat{S}_i[\_]}&\mathbf{Set}
}
\]

The restriction of the quotient map $p_{i+1}$ (see Section \ref{ssec:result_2}) to the object $E_{i+1}(X)$ gives us a way to organise the data of $E_{i+1}(X)$ with respect to its fibres. Of course, this organisation is also present in Kelly's construction via the quotients acting on $E_{i+1}(X)$ (see Example \ref{exa:clarifies_description_colimits}), but this organisation is also unlikely to be the one that one wants to consider if one decides to study the combinatorial properties of the models. In fact, while Kelly's construction forces us to consider an actual quotient of the object $E_{i+1}(X)$, the elimination of quotients leaves the object $E_{i+1}(X)$ free of quotients, so that one can now use any other type of relations on $E_{i+1}(X)$ without being forced to deal with the relations of the localisation. Furthermore, the formalism of quotient maps (formalised in terms of quotiented arrows in Section \ref{ssec:Quotiented-arrows}) makes compatibility and distributivity questions between potential new relations and those forced by the localisation much easier to study.

For instance, one could want to study the colimits of the category of models for $(D,K)$. Recall that colimits in this category are given by the images of the reflection $G$ on the corresponding colimits in the category of premodels $\mathbf{Set}^D$, as shown below.
\[
\mathrm{col}\, F = G(\mathrm{col}\, F)
\]

In addition, recall that a colimit of the form $\mathrm{col}\, F$ in $\mathbf{Set}^D$ can be seen as a quotiented sum. 
\[
\Big(\sum_{x} F(x) \Big)/\sim
\]

The relations $\sim$ acting on the sum $\sum_{x} F(x)$ usually generates the type of identifications that one wants to study. Specifically, one usually wants to understand how these propagate through the transfinite constructions building the models. However, their propagation is usually non-obvious and requires some more-or-less non-trivial case-by-case analysis, depending on how complicated the theory $(D,K)$ is. This case-by-case analysis might not even depend on the quotients implied by the localisation and might instead depend on the properties of the objects $F(x)$. In order to be efficient and clear, this case-by-case analysis needs to be processed in a quotient-free environment separated from the quotients generated by the localisation, but what is better than a quotient map whose domain is a quotient-free left Kan extension of the form given in Equation (\ref{eq:conclusion_left_Kan_extension}) to make such a separation? Interestingly, the construction of the quotient maps $p_{i+1}$ has motivated the formalisation of the concept of \emph{quotient} (in Section \ref{ssec:Quotiented-arrows}), so that our results open the door to the development of a new language to talk about quotients living in algebraic objects in general.

\appendix
\section{\label{Appendix:first}}

Recall that the category of sets $\mathbf{Set}$ is complete and cocomplete. The limit $\mathrm{lim}_D F$ of
a functor $F:D \to \mathbf{Set}$ for some small category $D$ is given by the set
\begin{equation}\label{eq:description_of_limits_in_set}
\{(x_d)_{d \in \mathrm{Obj}(D)}~|~x_d \in F(d)\text{ and for any }t:d\to d'\text{ in }D:~F(t)(x_d) = x_{d'}\}
\end{equation}
while the colimit $\mathrm{col}_D F$ of
a functor $F:D \to \mathbf{Set}$ for some small category $D$ is given by the quotient set
\[
\{(d,x)~|~d \in \mathrm{Obj}(D);~ x \in F(d)\}/\sim
\]
where $\sim$ denotes the binary relation whose relations $(d,x) \sim (d',x')$ are defined when
there exists an object $e$ and two arrow $t:d \to e$ and $t':d' \to e$ in $D$ such that the equation $F(t)(x) = F(t')(x')$ holds. Note that in the case where $D$ is a preorder category $\kappa$ for some ordinal $\kappa$, the binary relation $\sim$ is an equivalence relation.

\begin{proof}[Proof of Proposition \ref{prop:SGA_limits_colimits_commute}]
A proof may be found in \cite[Corollaire 9.8]{SGA4}. For the sake of self-contained\-ness, the proof is recalled in this appendix.  Let $F_{\_}(\_):\kappa \times D \to \mathbf{Set}$ be a functor. An equivalence class for the equivalence relation $\sim$ will be denoted into brackets, i.e. $[(k,x)]$. The notation $$(x_d)^F_{d \in \mathrm{Obj}(D)}$$ will be used to mean that the collection $(x_d)_{d \in \mathrm{Obj}(D)}$ is compatible with the action of the functor $F$ in the appropriate way (see Equation (\ref{eq:description_of_limits_in_set})). By definition, the following equations hold.
\[
\mathrm{col}_{\kappa}\mathrm{lim}_D F =\{[k,(x_d)^F_{d \in \mathrm{Obj}(D)}]~|~(x_d)^F_{d  \in \mathrm{Obj}(D)} \in \mathrm{lim}_D F_k(\_)\}
\]
\[
\mathrm{lim}_D\mathrm{col}_{\kappa} F =\{([k_d,x_d])^F_{d \in \mathrm{Obj}(D)}~|~[k_d,x_d] \in \mathrm{col}_{\kappa} F_{\_}(d)\}
\]
The natural transformation $\mathrm{col}_{\kappa}\,\mathrm{lim}_{D} \Rightarrow \mathrm{lim}_{D}\,\mathrm{col}_{\kappa}(\_)$ is given by the following mapping.
\[
[k,(x_d)^F_{d  \in \mathrm{Obj}(D)}] \mapsto ([k,x_d])^F_{d \in \mathrm{Obj}(D)}
\]
Let us prove its surjectiveness. Consider an element in $\mathrm{lim}_D\mathrm{col}_{\kappa} F$ of the following form. 
\[
([k_d,x_d])^F_{d \in \mathrm{Obj}(D)}
\]
By definition of the compatibility with the action of $F$, for any arrow $t:d \to d'$ in $D$, there exist arrows $s_{d}:k_{d} \to e_{t}$ and $s_{d'}:k_{d'} \to e_{t}$ in $\kappa$ such that the next equation holds.
\begin{equation}\label{eq:compatibility_col_lim_F_commutativity}
F_{s_{d}}(d) \circ F_{k_d}(t)(x_d) = F_{s_{d'}}(d')(x_{d'})
\end{equation}
Since $\kappa$ is a limit ordinal greater than or equal to $|D|$, we may define the following supremum in $\kappa$.
\[
\underbrace{\xymatrix{
&&\cup_{t \in \mathrm{Ar}(D)} e_{t}&&\\
e_{t_0}\ar[rru]^{g_{t_0}}&e_{t_1}\ar[ru]|{g_{t_1}}&e_{t_2}\ar[u]|{g_{t_2}}&\dots&e_{t}\ar[llu]_{g_t}
}}_{\text{cardinality given by }|D|}
\]
Denote the supremum $\cup_{t \in \mathrm{Ar}(D)} e_{t}$ by $e$. Note that for any pair of arrows $t:d \to d'$ and $t':d^{\prime\prime} \to d$ in $D$, the arrows $g_t \circ s_d:k_{d} \to e$ and $g_{t'} \circ s_d:k_{d} \to e$ are equal in $\kappa$. The family made of the elements
$
F_{g_t \circ s_d}(d)(x_d)
$
for every object $d$ in $D$
is then compatible with the action of $F$, since, for any arrow $t:d \to d'$ in $D$, the following equation holds from Equation (\ref{eq:compatibility_col_lim_F_commutativity}).
\[
F_{e}(t) \circ F_{g_t \circ s_{d}}(d)(x_d) =  F_{g_t \circ s_{d}}(d) \circ F_{k_d}(t)(x_d) = F_{g_{t}}(d) \circ F_{s_{d'}}(d')(x_{d'})
\]
In addition, it is not hard to check that the mapping rule of the natural transformation $\mathrm{col}_{\kappa}\,\mathrm{lim}_{D}(\_) \Rightarrow \mathrm{lim}_{D}\,\mathrm{col}_{\kappa}(\_)$ includes the rule
\[
[e,(F_{g_t \circ s_d}(d)(x_d))^F_{d \in \mathrm{Obj}(D)}] \mapsto ([k_d,x_d])^F_{d \in \mathrm{Obj}(D)}
\]
since $(k_d,x_d) \sim (e,F_{g_t \circ s_{d}}(d)(x_d))$.
Let us now prove its injectiveness. Note that any equality $([k,x_d])^F_{d \in \mathrm{Obj}(D)} = ([k',x_d'])^F_{d \in \mathrm{Obj}(D)}$ implies the existence of cospans
\[
\xymatrix{
&e_d&\\
k\ar[ru]^{s_d}&&k'\ar[lu]_{s'_d}
}
\]
such that the identity $F_{s_d}(d)(x_d) = F_{s'_d}(d)(x_d')$ holds for every object $d$ in $D$. Now, define the following supremum, which will be denoted by $e'$.
\[
\underbrace{\xymatrix{
&&\cup_{d \in \mathrm{Obj}(D)} e_{d}&&\\
e_{d_0}\ar[rru]^{g_0}&e_{d_1}\ar[ru]|{g_1}&e_{d_2}\ar[u]|{g_2}&\dots&e_{d}\ar[llu]_{g_d}
}}_{\text{cardinality below }|D|}
\]
For every object $d$ in $D$, the arrows $g_d \circ s_d:k \to e'$ are equal in $\kappa$. The same is true for 
$g_d \circ s_d':k' \to e'$. It follows that the equation
\[
\mathrm{lim}_d F_{g_d \circ s_d}(d)((x_d)^F_{d}) = \mathrm{lim}_d F_{g_d \circ s'_d}(d)((x_d')^F_{d})
\]
holds, which implies the identity $[k,(x_d)^F_{d  \in \mathrm{Obj}(D)}]=[k',(x_d')^F_{d  \in \mathrm{Obj}(D)}]$.
\end{proof}

\begin{proof}[Proof of Proposition \ref{prop:eta_commutes_with_colimits_O_kappa}]
We keep the convention set in the proof of Proposition \ref{prop:SGA_limits_colimits_commute}. We only need to check that the diagram of the statement commutes. For any set $X$, the unit $\eta_X:X \to \mathrm{lim}_D\Delta_D(X)$ maps an element of $x \in X$ to the constant collection $(x)_{d \in \mathrm{Obj}(D)}$. Similarly, for any functor $X:\kappa \to \mathbf{Set}$, the unit $\eta_{X(\_)}:X(\_) \to \mathrm{lim}_D\Delta_D(X(\_))$ maps an element of $x \in X(k)$ to the constant collection $(x)_{d \in \mathrm{Obj}(D)}$ in $\mathrm{lim}_D\Delta_D(X(k))$. The diagram of the statement is therefore encoded by the following mapping rules.
\[
\xymatrix{
[(k,x)]\ar@{=}[d]\ar@{|->}[rr]^-{\mathrm{col}_{\kappa}\eta_{F(\_)}}&&([k,(x)_{d \in \mathrm{Obj}(D)}]) \ar@{|->}[d]^{\cong}\\
[(k,x)]\ar@{|->}[rr]_-{\eta_{\mathrm{col}_{\kappa}F(\_)}}&&([(k,x)])_{d \in \mathrm{Obj}(D)}
}
\]
In particular, this shows that the diagram commutes.
\end{proof}

\bibliographystyle{amsplain}

\end{document}